\documentclass[12pt,reqno]{amsart}%[normalheadings]{scrartcl}

\usepackage{amssymb}
\usepackage[T1]{fontenc}
\usepackage[latin1]{inputenc} % utf8b
\usepackage{dsfont, mathrsfs}
\usepackage{a4wide}  
\usepackage{stmaryrd}
\usepackage{amsthm}
\usepackage{graphicx}
\usepackage{color,soul}  
\usepackage{xspace}
\usepackage{yhmath, epigraph}
\usepackage{mleftright}

\definecolor{pink}{RGB}{219, 48, 122}
\definecolor{purple}{RGB}{128, 0, 128}

\usepackage[
        colorlinks,
        hyperfootnotes=true,
        pagebackref,
        hyperindex,
        citecolor=purple,
        linkcolor=blue]{hyperref}
\usepackage{bbm}%mathbbol
\usepackage[bbgreekl]{mathbbol}
\DeclareSymbolFontAlphabet{\mathbb}{AMSb}
\DeclareSymbolFontAlphabet{\mathbbl}{bbold}
\allowdisplaybreaks

\theoremstyle{plain}
\newtheorem{theorem}{Theorem}[section]
\newtheorem{lemma}[theorem]{Lemma}

\newtheorem{corollary}[theorem]{Corollary}

\theoremstyle{definition}
\newtheorem{definition}[theorem]{Definition}
\newtheorem{example}[theorem]{Example}

\newtheorem{question}[theorem]{Question}
\newtheorem{conjecture}[theorem]{Conjecture}

\theoremstyle{remark}
\newtheorem{remark}[theorem]{Remark}

\makeatletter
\newcommand*{\trimabs}[1]{%
  \mathpalette{\@trimleftright{|}{|}}{#1}%
}
\newcommand*{\@trimleftright}[4]{%
  \sbox0{$#3#4\m@th$}%
  \sbox2{$#3\vcenter{}$}% Math axis: \ht2
  \dimen0=\ht0 %
  \ifdim\dimen0<\ht2 %
    \dimen0=\ht2 %
  \fi
  \dimen2=\dp0 %
  \ifdim\dimen2<\z@
    \dimen2=\z@
  \fi
  \dimen0=\dimexpr(\dimen0-\dimen2)/2 -\ht2\relax
  \raisebox{\dimen0}{%
    $#3\mleft#1\raisebox{-\dimen0}{\box0}\mright#2\m@th$%
  }%
}
\makeatother

\newcommand{\Cc}{\mathbb{C}}
\newcommand{\Dd}{\mathbb{D}}
\newcommand{\Ee}{\mathbb{E}}
\newcommand{\Ff}{\mathbb{F}}
\newcommand{\Hh}{\mathbb{H}}

\newcommand{\Nn}{\mathbb{N}}
\newcommand{\Pp}{\mathbb{P}}

\newcommand{\Rr}{\mathbb{R}}

\newcommand{\Tt}{\mathbb{T}}
\newcommand{\Uu}{\mathbb{U}}

\newcommand{\Yy}{\mathbb{Y}}
\newcommand{\Zz}{\mathbb{Z}}

\newcommand{\Un}{\mathds{1}}

\newcommand{\Ae}{\mathcal{A}}
\newcommand{\Be}{\mathcal{B}}
\newcommand{\Ce}{\mathcal{C}}
\newcommand{\De}{\mathcal{D}}
\newcommand{\Eee}{\mathcal{E}}
\newcommand{\Fe}{\mathcal{F}}

\newcommand{\He}{\mathcal{H}}
\newcommand{\Ie}{\mathcal{I}}
\newcommand{\Je}{\mathcal{J}}
\newcommand{\Ke}{\mathcal{K}}
\newcommand{\Le}{\mathcal{L}}
\newcommand{\Me}{\mathcal{M}}

\newcommand{\Pe}{\mathcal{P}}
\newcommand{\Qe}{\mathcal{Q}}
\newcommand{\Ree}{\mathcal{R}}
\newcommand{\Se}{\mathcal{S}}
\newcommand{\Te}{\mathcal{T}}
\newcommand{\Ue}{\mathcal{U}}
\newcommand{\We}{\mathcal{W}}
\newcommand{\Xe}{\mathcal{X}}
\newcommand{\Ze}{\mathcal{Z}}

\newcommand{\Ns}{\mathscr{N}}

\newcommand{\Us}{\mathscr{U}}

\newcommand{\Mg}{\mathfrak{M}}

\newcommand{\Rg}{\mathfrak{R}}
\newcommand{\Sg}{\mathfrak{S}}

\newcommand{\eg}{\mathfrak{e}}

\newcommand{\alphab}{{\boldsymbol{\alpha}}}
\newcommand{\betab}{{\boldsymbol{\beta}}}
\newcommand{\gammab}{{\boldsymbol{\gamma}}}

\newcommand{\lambdab}{{\boldsymbol{\lambda}}}
\newcommand{\nub}{{\boldsymbol{\nu}}}
\newcommand{\varphib}{{\boldsymbol{\varphi}}}

\newcommand{\sigmab}{{\boldsymbol{\sigma}}}
\newcommand{\thetab}{{\boldsymbol{\theta}}}

\newcommand{\Gammab}{{\boldsymbol{\Gamma}}}
\newcommand{\Lambdab}{{\boldsymbol{\Lambda}}}

\newcommand{\Ab}{{\boldsymbol{A}}}
\newcommand{\Bd}{{\boldsymbol{B}}}

\newcommand{\Db}{{\boldsymbol{D}}}

\newcommand{\Pb}{{\boldsymbol{P}}} 
\newcommand{\Tb}{{\boldsymbol{T}}}
\newcommand{\Vb}{{\boldsymbol{V}}}
\newcommand{\Xb}{{\boldsymbol{X}}}
\newcommand{\Yb}{{\boldsymbol{Y}}}
\newcommand{\Zb}{{\boldsymbol{Z}}}

\newcommand{\ab}{{\boldsymbol{a}}}
\newcommand{\bb}{{\boldsymbol{b}}}

\newcommand{\hb}{{\boldsymbol{h}}}

\newcommand{\sbb}{{\boldsymbol{s}}}
\newcommand{\tb}{{\boldsymbol{t}}}
\newcommand{\ub}{{\boldsymbol{u}}}
\newcommand{\vb}{{\boldsymbol{v}}}
\newcommand{\wb}{{\boldsymbol{w}}}
\newcommand{\xb}{{\boldsymbol{x}}}
\newcommand{\yb}{{\boldsymbol{y}}}
\newcommand{\zb}{{\boldsymbol{z}}}

\newcommand{\ellb}{{\boldsymbol{\ell}}}
\def\ee{ { \mathbbm{e} } }

\newcommand{\ensemble}[1]{ \left\lbrace #1 \right\rbrace } 
\newcommand{\prth}[1]{\!\left( #1 \right) }
\newcommand{\crochet}[1]{\!\left[ #1 \right] }  
\newcommand{\intcrochet}[1]{\llbracket #1 \rrbracket} 
\newcommand{\abs}[1]{\left| #1 \right|}  
\newcommand{\norm}[1]{\left| \! \left| #1 \right| \! \right|}

\newcommand{\Esp}[1]{ \Ee  \prth{ #1 } }  
\newcommand{\Prob}[1]{ \Pp \prth{ #1 } }  
\newcommand{\Espr}[2]{ \Ee_{#1}\! \prth{ #2 } } 

\def\inv{^{-1}}

\def\longlongrightarrow{\hspace{+0.1ex} - \hspace{-1.1ex} - \hspace{-1.1ex} - \hspace{-1.1ex}\longrightarrow  } 

\newcommand{\tendvers}[2]{ \underset{#1 \rightarrow #2}{\longlongrightarrow} }  
\newcommand{\cvlaw}[2]{\stackrel{\Le}{\underset{#1 \, \rightarrow \, #2}{\longlongrightarrow}}}  

\newcommand{\equivalent}[1]{ {\underset{#1 }{\sim} } }

\newcommand{\bracket}[1]{\left\langle #1 \right\rangle}
\newcommand{\Unens}[1]{ \Un_{ \ensemble{#1} } }
\def\eqlaw{\stackrel{\Le}{=}}
\newcommand{\pe}[1]{\left[ #1 \right] }
\newcommand{\pf}[1]{\left\{ #1 \right\} }
\newcommand{\tr}[1]{\operatorname{tr}\prth{ #1 } }

\def\trace{ \operatorname{tr} }

\def\sinc{ \operatorname{sinc} }

\def\GUE{ \operatorname{GUE} }
\def\vol{ \operatorname{vol} }
\def\covol{ \operatorname{covol} }
\def\sc{\operatorname{sc}}
\def\Mom{\operatorname{MoM}}

\def\tensorsum{\oplus}
\def\plusInOne{}

\def\MoMf{\mathfrak{MoM}}

\def\GUE{ \operatorname{GUE} }

\def\Exp{ \operatorname{Exp} }

\def\Geom{ \operatorname{Geom} }

\def\Gumbel{ \operatorname{Gb} }

\def\NegBin{\operatorname{NegBin}}

\def\geq{\geqslant}
\def\leq{\leqslant}

\def\Re{\Rg \eg}

\let\oldforall\forall
\def\forall{\oldforall\,} %\addto{\forall}{\,}

\let\oldexists\exists
\def\exists{\oldexists\,}

\def\AOP{Ann. Probab. }
\def\CMP{Commun. Math. Phys. }
\def\CPAM{Comm. Pure Appl. Math. }
\def\Inventiones{Invent. Math. }
\def\JTP{J Theoret. Probab. }
\def\PTRF{Probab. Theory Related Fields }

\newcommand{\emailhref}[1]{ \email{\href{mailto:#1}{#1}} }

\definecolor{rougeclair}{rgb}{1,.65,.65}

\sethlcolor{rougeclair}

\newcommand{\white}[1]{  \color{white} #1 \color{black}  }

\begin{document}

\title[A new approach to the CUE characteristic polynomial]{A new approach to the characteristic \\ polynomial of a random unitary matrix}
\author[Y. Barhoumi-Andr\'eani]{Yacine Barhoumi-Andr\'eani}
\address{Ruhr-Universit\"at Bochum, Fakult\"at f\"ur Mathematik, Universit\"atsstrasse 150, 44780 Bochum, Germany.}
\emailhref{yacine.barhoumi@rub.de}
\date{\today}

% == Abstract

\begin{abstract}
Since the seminal work of Keating and Snaith, the characteristic polynomial of a random Haar-distributed unitary matrix has seen several of its functional studied or turned into a conjecture; for instance:
\begin{itemize}
\item its value in $1$ (Keating-Snaith theorem),
\item the truncation of its Fourier series up to any fraction of its degree,
\item the computation of the relative volume of the Birkhoff polytope,
\item its products and ratios taken in different points, %at a macroscopic, mesoscopic or microscopic distance from each other,
\item the product of its iterated derivatives in different points, 
\item functionals %used in a recent theorem of Keating, Rodgers, Roditty-Gershon and Rudnick 
in relation with sums of divisor functions in $ \Ff_q[X] $.
\item its mid-secular coefficients,  
\item the ``moments of moments'', etc.
\end{itemize}

\medskip
We revisit or compute for the first time the asymptotics of the integer moments of these last functionals and several others. The method we use is a very general one based on reproducing kernels, a symmetric function generalisation of some classical orthogonal polynomials interpreted as the Fourier transform of particular random variables and a local Central Limit Theorem for these random variables. We moreover provide an equivalent paradigm based on a randomisation of the mid-secular coefficients to rederive them all. These methodologies give a new and unified framework for all the considered limits and explain the apparition of Hankel determinants or Wronskians in the limiting functional.
\end{abstract}

\dedicatory{To the loving memory of Jeanne Andr\'eani (1923-2020)}

\maketitle

\vspace{-0.5cm}

\setcounter{tocdepth}{2}
\tableofcontents 

\newpage

% ==
\section{Introduction}

\subsection{Historical motivations}\label{Subsec:Motivations}

The connection between random unitary matrices and number theory is famously known to have emerged during an afternoon tea \cite{SabbaghOnMontgomeryDyson}, when Montgomery discussed with Dyson their respective work on the pair correlation of the Riemann $\zeta$ function zeroes and the eigenangles of a random Haar-distributed unitary matrix \cite{Montgomery, Dyson}. 

The years that followed this meeting saw several reinforcements of this connection (see e.g. \cite{KatzSarnak, MezzadriSnaith} and references cited), and culminated in the seminal work of Keating and Snaith \cite{KeatingSnaith} that opened the path to the study of the characteristic polynomial of a random Haar-distributed unitary matrix as a toy-model for the Riemann $\zeta$ function. 

The original motivation of \cite{KeatingSnaith} was the computation of moments of the $\zeta$ function on a random interval of the critical line. Let $ U $ be a random variable uniformly distributed on $ [0, 1] $. Keating and Snaith conjecture that for all $ k \in \Nn $
\begin{align}\label{Eq:MomentsConjecture}%$
\Esp{ \abs{\zeta\prth{\frac{1}{2} + i T U } }^{2k} } \equivalent{T\to+\infty} A_k M_k \, e^{ k^2 \log\log T  }
\end{align}

Here, $A_k$ is the \textit{arithmetic factor} defined by 
\begin{align}\label{Def:ArithmeticFactor}%$
A_k = e^{ \gamma_\Pe k^2 } \prod_{p \in \Pe} e^{ - \frac{k^2}{p } } \Esp{ \abs{ 1 - p^{-1/2} Z_p  }^{-2k} },  \qquad \gamma_\Pe := \sum_{p \in \Pe} p\inv + \log(1 - p\inv)   
\end{align}
with $ \Pe $ the set of ordered prime numbers and $ (Z_p)_{p \in \Pe} $ a sequence of i.i.d. uniform random variables on the unit circle $ \Uu $, and $ M_k $ is the \textit{random matrix factor}, given by
\begin{align}\label{Def:MatrixFactor}%$
M_k = \prod_{\ell = 0}^{k - 1} \frac{\ell !}{ (\ell + k )!} = \frac{G(1 + k)^2}{G(1 + 2k)}
\end{align}
with $G$ the Barnes $G$-function or double Gamma function (see e.g. \cite{KeatingSnaith}). 

The appearance of $ M_k $ in \eqref{Eq:MomentsConjecture} finds its origin in the study of the characteristic polynomial $ Z_{U_N} $ of a random Haar-distributed unitary matrix $ U_N $, more precisely in the following theorem due to Keating and Snaith \cite{KeatingSnaith}:
\begin{align}\label{Eq:KSmomentsCharpol}%$
\Esp{ \abs{ Z_{U_N}(1) }^{2k} } \equivalent{N\to+\infty} M_k \, e^{ k^2 \log N  }
\end{align}

Here, $ Z_{U_N}(x) := \det\prth{I_N - x U_N} $ is the characteristic polynomial of $ U_N $, a Haar-distributed random unitary matrix, and $ \Pp $ denotes the normalised Haar measure of $ \Ue_N $ ; the underlying probability space $ (\Ue_N, \Pp) $ is called \textit{Circular Unitary Ensemble} (in short, $ CUE_N $).

Conjecture \eqref{Eq:MomentsConjecture} has been numerically verified up to $ k = 12 $ (Hiary and Odlyzko \cite{HiaryOdlyzko}), the only proven cases being $ k = 1 $ (Hardy and Littlewood \cite{HardyLittlewoodMoments}) and $ k = 2 $ (Ingham \cite{Ingham}). 

Since the publication of \cite{KeatingSnaith}, the paradigm that consists in computing functionals of $ Z_{U_N} $ to add a universal correction to the equivalent problem on the $ \zeta $ function (or more generally $L$-functions) became a successful heuristic to produce conjectures often verified numerically: the \textit{Keating-Snaith philosophy}.

The study of the characteristic polynomial became then an independent subject of its own, with specific methods at the junction of several fields, such as: analysis of Toeplitz determinants \cite{JohanssonSzego, JohanssonGroupe} and orthogonal polynomials on the unit circle \cite{KillipNenciu, KillipRyckman, BourgadeNikeghbaliRouault, Hisakado}, mathematical physics \cite{KostovCFTtoRMT} and supersymmetry \cite{ConreyFarmerZirnbauer, Efetov, FyodorovSupersym}, algebra and integrable systems \cite{AdlerVanMoerbeke, AdlerShiotaVanMoerbeke, VanMoerbeke}, algebraic combinatorics \cite{BianeCUE, PODfpsac, FerayJM}, representation theory and symmetric functions \cite{BumpGamburd, POD, DiaconisShahshahani}, probability theory \cite{BarhoumiHughesNajnudelNikeghbali, BourgadeCondHaar, BHNY, ChhaibiNajnudelNikeghbali}, Weingarten calculus \cite{CollinsSniady, MatsumotoNovak, Weingarten}, (free) It\^o calculus \cite{BianeFreeBM, BruDiff, DalqvistBMU, KatoriTanemura, MeliotCutOff, UlrichUnitary}, etc. The conjunction of such a diversity of methods shows the richness of the topic and the variety of strategies one can use to perform computations on $ Z_{U_N}(x) $.

% ==
\subsection{Problems investigated in this article}\label{Subsec:AllProblems}

The following problems had an important impact on the field of random unitary matrices in the last two decades, and this is the goal of this article to rederive them with a novel method:

% ==
\subsubsection{Autocorrelations}\label{Subsubsec:Autocorrelations}
The first random variable studied in \cite{KeatingSnaith} was $ Z_{U_N}(1) $ through its moments, to produce conjecture \eqref{Eq:MomentsConjecture}, but several refinements were proposed, in particular, \cite{BrezinHikamiCharpol, BumpGamburd, ConreyFarmerKeatingRubinsteinSnaith, ConreyForresterSnaith} compute the \textit{autocorrelations} of the characteristic polynomial, i.e. joint moments of the form
\begin{align}\label{FuncZ:AutocorrelationsCharPol}%$
\Ae'_N(X, Y) := \Esp{ \prod_{j = 1}^k Z_{U_N}\prth{ x_j } \prod_{\ell = 1}^m \overline{ Z_{U_N}\prth{ y_\ell } }  } 
\end{align}

Several results and expressions were given for $ \Ae'_N(e^{X/N}, e^{Y/N}) $ for several types of matrix models that sometimes include the $ CUE $ (see \cite{AkemannVernizzi, BorodinOlshanskiStrahov, BorodinStrahov, BrezinHikamiCharpol, BrezinHikamiAutocorrs, DesrosiersLiu2011, FyodorovStrahovCorrChiralGUE, FyodorovStrahovCorrGUE, FyodorovStrahovRHP, GronqvistGuhrKohler} and references cited) with $ e^{X/N} := (e^{x_1/N},  \dots, e^{x_k/N}) $ and similarly for $ e^{Y/N} $. In particular, for $ k = m = 1 $, one gets $ \Ae'_N(e^{ix/N}, e^{iy/N}) \sim N \sinc(x-y) $ where $ \sinc(x) := \frac{\sin(x)}{x} $ if $ x \neq 0 $ and $ 1 $ if $ x = 0 $.

We will perform an asymptotic study of $ \Ae'_N(X, Y) $ when $ N\to+\infty $ in section~\ref{Subsec:Autocorrelations}.

% ==
\subsubsection{Ratios}\label{Subsubsec:Ratios}
A natural generalisation of the autocorrelations is the computation of ratios 
\begin{align}\label{FuncZ:Ratios}%$
\Ree'_N(X, X', Y, Y') := \Esp{ \frac{ \prod_{j = 1}^{\ell_1} Z_{U_N}\prth{ x_j  }  }{ \prod_{r = 1}^{m_1} Z_{U_N}\prth{ y_r } } \overline{ \frac{ \prod_{j = 1}^{\ell_2} Z_{U_N}\prth{ x'_j }  }{ \prod_{r = 1}^{m_2} Z_{U_N}\prth{ y'_r } } } }  
\end{align}

Such a study was performed by \cite{BasorForrester, BumpGamburd, ConreyFarmerKeatingRubinsteinSnaith, ConreyFarmerZirnbauer, ChhaibiNajnudelNikeghbali} in the unitary case following several other studies in other matrix ensembles (see \cite{AkemannPottier, BaikDeiftStrahov, Bergere, BorodinOlshanskiStrahov, BorodinStrahov, BrezinHikamiCharpol, BrezinHikamiCharpolEdge, DesrosiersLiu2011, FyodorovStrahovCorrChiralGUE} and references cited ; see also section~\ref{Subsec:Ratios} for an enlarged historical perspective that includes related works before the connection with the $ CUE $ was revealed). The importance of ratios stems from the possibility of accessing the \textit{$k$-point correlation functions} by performing an appropriate limit (see e.g. the introduction of \cite{BergereEynard, FyodorovStrahovCorrChiralGUE}). Since the knowledge of the correlation functions amounts to the knowledge of the whole eigenvalues point process, this gives an additional motivation to study the characteristic polynomial for its own sake.

An asymptotic study of $ \Ree'_N(X, X', Y, Y') $ when $ N\to+\infty $ will be done in section~\ref{Subsec:Ratios}.

% ==
\subsubsection{Mid-secular coefficient}\label{Subsubsec:Midcoeff}
One key feature to study $ Z_{U_N}(x) $ is to write it as a product over its eigenvalues and to use their determinantal structure. Nevertheless, this is not the only approach worth an investigation. Following a study by Haake et al. \cite{HaakeEtAl}, Diaconis and Gamburd \cite{DiaconisGamburd} note that the data of the \textit{secular coefficients} or Fourier coefficients of $ Z_{U_N} $ defined in \eqref{FuncZ:Truncation} are equivalent to the data of the eigenvalues:
\begin{quote}
\textit{(...) we remark that a unitary matrix $M$ is conjugate on the one hand to the
diagonal matrix with eigenvalues on the diagonal, and on the other hand to the
Frobenius, or companion matrix, with first row consisting of the secular coefficients,
ones below the main diagonal, and remaining entries zero. This strongly suggests
that secular coefficients are (as Gian-Carlo Rota might have put it) ``nearly equiprimordial'' with the eigenvalues.}
\end{quote}

Contrary to the eigenangles that all have the same order, the secular coefficients $ \sc_m(U_N) $ have a different behaviour depending on their order $m$. For instance, $ \sc_1(U_N) = \tr{U_N} $ converges in law to a Gaussian random variable, the next remaining secular coefficients of fixed order converge in law to a polynomial evaluated in independent complex Gaussian random variables \cite[Prop. 4]{DiaconisGamburd}, and $ \sc_N(U_N) = \det(U_N) $ is uniform on the unit circle. A natural question raised in \cite[\S 2.4]{DiaconisGamburd} (see also \cite{DiaconisRMTSurvey}) is then the following: find the limiting behaviour of the mid-secular coefficient
\begin{align}\label{FuncZ:MidSecularCoeff}%$
\sc_{ \pe{N/2} }(U_N) := \oint_\Uu z^{ -\pe{N/2} - 1 } Z_{U_N}(z) \frac{dz}{2 i \pi }
\end{align}

As Diaconis and Gamburd point out, the formuli $ \Esp{ \abs{\sc_m(U_N)}^2 } = 1 $ and $ \Esp{  \sc_m(U_N)  } = 0 $ hold for all $ m \geq 0 $, ``\textit{making normality questionable}''.

We will derive an asymptotic expression for $ \Esp{ \abs{ \sc_{ \pe{\rho N } }(U_N) }^{2k} } $ when $ N\to+\infty $ for all $ \rho \in (0, 1) $ in section~\ref{Subsec:Midcoeff}. The results proven there will indeed show that the mid-coefficient is neither Gaussian, nor even log-normal, but lies in the universality class of a critical \textit{random multiplicative function} that we introduce now.

% ==
\subsubsection{Truncated characteristic polynomial}\label{Subsubsec:TruncatedCharpol}
An interesting case where number-theoretic computations are possible concerns the following truncation of the Riemann $\zeta$ function: $ \zeta_X(s) := \sum_{n  = 1}^X n^{-s} $. Conrey and Gamburd \cite{ConreyGamburd} studied the analogous problem on $ Z_{U_N}(x) $ when the truncation is taken over the Fourier expansion of $ Z_{U_N} $, i.e. the truncation of the polynomial $ Z_{U_N}(X) $ in the canonical basis $ (1, X, X^2, \dots) $. First passing to the limit (in law) in $  \zeta_X(\sigma + i T U) $  when $T\to+\infty$ (for $ \sigma\in \Rr $), then using the Bohr-Jessen theorem \cite{BohrJessen} or the Bohr correspondence (see e.g. \cite[ch. 3.2]{KowalskiRandonnee} or the introduction of \cite{HeapLindqvist}), one gets at the limit the \textit{random multiplicative function}
\begin{align}\label{Def:RandomMultiplicativeFunction}%$
\Ze_{X, \sigma} := \sum_{k = 1}^X k^{-\sigma} \prod_{p \in \Pe} Z_p^{v_p(k) } 
\end{align}
where $ v_p(k) $ is the $p$-adic valuation of $k \in \Nn^* $, i.e. the exponent of $k$ in its prime decomposition $ k = \prod_{p \in \Pe} p^{v_p(k) } $ and $ (Z_p)_{p \in \Pe} $ is a sequence of i.i.d. uniform random variables on the unit circle $ \Uu $.

Conrey and Gamburd are interested in the behaviour of $ \Esp{ \vert \Ze_{X, 1/2 } \vert^{2k} } $ when $ X \to +\infty $. Their theorem \cite[thm. 1]{ConreyGamburd} analogous to conjecture \eqref{Eq:MomentsConjecture} writes
\begin{align}\label{Eq:ConreyGamburd}%$
\Esp{ \abs{ \Ze_{X, 1/2 } }^{2k} } \equivalent{X\to+\infty} A_k \gamma_k \, e^{ k^2 \log\log X }
\end{align}
with a factor $ \gamma_k  $ different from $ M_k $, but coming instead from the truncated characteristic polynomial\footnote{The definition in \cite{ConreyGamburd} uses $ \det(U_N - z I) $ ; since the value in \eqref{Eq:ConreyGamburdCV} is taken in $1$, it changes nothing to our exposition for obvious reasons of invariance by rotation. }
\begin{align}\label{FuncZ:Truncation}%$
Z_{U_N, \ell}(z) := \sum_{j = 0}^\ell \sc_j(U_N) (-z)^j, \qquad \sc_j(U_N) := \oint_\Uu z^{-j} Z_{U_N}(z) \frac{dz}{2 i \pi z}
\end{align}
namely for all $ N \geq k\ell $ (one can take $ N = k\ell $)
\begin{align}\label{Eq:ConreyGamburdCV}%$
\frac{\Esp{ \abs{ Z_{U_N, \ell }(1) }^{2k} } }{ \ell^{k^2 } } \tendvers{\ell}{+\infty} \gamma_k   
\end{align}

This convergence was moreover generalised to $ Z_{U_N, \ell }(\lambda) $ for $ \lambda > 1 $ by Heap and Lindqvist \cite{HeapLindqvist} who also study the behaviour of $ \Ze_{X, \sigma} $ for $ \sigma < 1/2 $.

We will rederive an expression of $ \gamma_k  $ by computing another asymptotic expression of $ \Esp{ \abs{ Z_{U_N, \ell }(1) }^{2k} } $ when $ \ell\to+\infty $ and $ N = k\ell $ in section~\ref{Subsec:TruncatedCharpol}. We will moreover consider the case of $ Z_{U_N, \pe{\rho N} }(z) $ when $ N \to +\infty $ for $ \rho \in (0, 1) $ and $ z \in \Cc^* $ in addition to $ Z_{U_N, \pe{\rho N} }(e^{x/N}) $ for $ x \in \Cc $. The comparison of the different behaviours of these random variables will show the existence of a phase transition in the truncated characteristic polynomial, and that one of these truncated characteristic polynomials belongs to the universality class of the mid-secular coefficients \eqref{FuncZ:MidSecularCoeff} and the random multiplicative function \eqref{Def:RandomMultiplicativeFunction}. Moreover, the analysis of the ``microscopic landscape'' (in points of the form $ e^{ x/N} $) will lead to the question of a phase transition of $ \Ze_{X, 1/2 - \lambda_X } $ that will be treated in Annex~\ref{Sec:NumberTheory}.

% ==
\subsubsection{The volume of the Birkhoff polytope}\label{Subsubsec:Birkhoff}
Often, computations with random unitary matrices can be expressed as volumes of particular polytopes, see e.g. \cite{AssiotisKeating, ConreyGamburd, DiaconisGangolli, KeatingR3G}. In \cite[Thm. 1.4]{KeatingR3G} for instance, it comes from an expression of the autocorrelations \eqref{FuncZ:AutocorrelationsCharPol} as a sum over plane partitions (or Young tableaux). In several cases, a dictionary of equivalent expressions can be provided between volumes of polytopes and functionals of the characteristic polynomial \cite{AssiotisKeating, ConreyGamburd, KeatingR3G}. 

Arguably, the most famous polytope is the \textit{Birkhoff polytope} \cite{BirkhoffPolytope} defined by the equations of bi-stochastic matrices constraints given in definition~\ref{Def:BirkhoffPolytope}. Its volume has been the subject of several studies\footnote{As of February 2020, an internet search reveals more than 2200 papers on the topic.} \cite{CanfieldMcKay, DeLoeraLiuYoshida, PakBirkhoff} and several explicit expressions are known, albeit of several computational difficulty. 

In their study of the truncated characteristic polynomial \cite{ConreyGamburd}, Conrey and Gamburd study the volume of a related polytope. We will derive another expression of these volumes and some generalisations in section~\ref{Subsec:Birkhoff}.

% ==
\subsubsection{Sums of divisor function in $ \Ff_q[X] $}\label{Subsubsec:KR3G}
In \cite{KeatingR3G}, Keating, Rodgers, Roditty-Gershon and Rudnick study the following functional of the characteristic polynomial in link with sums of divisor functions in $ \Ff_q[X] $ when $ q \to +\infty $:
\begin{align}\label{FuncZ:KR3Gfunc}%$
\begin{aligned}
\raisebox{-9pt}{$I_k(m , N):= \Ee\!\Bigg( $ } \! \trimabs{ \rule{0cm}{0.5cm}\sum_{ \substack{ 1 \leq j_1, \dots, j_k \leq N \\ j_1 + \cdots + j_k = m } } \raisebox{-5pt}{ $\sc_{j_1}(U_N) \dots \sc_{j_k}(U_N) $ } }^2 \raisebox{-9pt}{$ \Bigg) $ }
\end{aligned}
\end{align}

The link between random polynomials of $ \Ff_q\crochet{X} $ when $ q \to +\infty $ and random unitary matrices is given by Katz's equidistribution theorem \cite{KatzCUE}. Here, $ m = \pe{cN} $ for a certain $ c \in \crochet{0, k} $. The asymptotic expression $ I_k(\pe{cN} , N) \sim_{N\to+\infty} N^{k^2 - 1} \Ie_c(k) $ is derived in \cite{KeatingR3G}  with a limiting constant $ \Ie_c(k) $ expressed as a combinatorial sum over partitions or with a $k$-fold integral. 

We will give another proof of this result and derive a new expression for $ \Ie_c(k) $ in section~\ref{Subsec:KR3G}.

% ==
\subsubsection{``Moments of moments''}\label{Subsubsec:MoMo}
The following functional of $ Z_{U_N} $ for $ \beta \in \Nn^* $
\begin{align}\label{FuncZ:LpNorm}%$
\Eee_\beta(Z_{U_N}) := \int_0^1 \abs{Z_{U_N}(e^{2 i \pi \theta}) }^{2\beta} d\theta  
\end{align}
is the $2\beta$-th power of the $ L^{2\beta}(\Uu) $ norm of $ Z_{U_N} $, or equivalently, the $ 2\beta $-th moment of $ Z_{U_N} $ for the uniform measure on $ \Uu $. The motivations for its study comes from the fact that $ \Eee_\beta(Z_{U_N})^{1/\beta} $ is a \textit{detropicalisation} of $ \max_{\, \Uu}\! \abs{ Z_{U_N} }^2  $. Indeed, one has $ \Pp $-almost surely
\begin{align*}%$
\frac{1}{\beta}\log\Eee_\beta(Z_{U_N}) = \frac{1}{\beta} \log\prth{ \int_0^1 e^{\beta \log\abs{Z_{U_N}(e^{2 i \pi \theta}) }^2} d\theta }  \tendvers{\beta}{+\infty} \max_{z \in \Uu} \log  \abs{Z_{U_N}(z) }^2
\end{align*}

The $ k $-th integer moments of $ \Eee_\beta( Z_{U_N} ) $ or ``moments of moments''
\begin{align}\label{FuncZ:MoMo}%$
\Mom(N\vert k, \beta) := \Esp{ \Eee_\beta(Z_{U_N})^k } 
\end{align}
where investigated for $ k \beta^2 < 1 $ in \cite{FyodorovHiaryKeating, FyodorovKeating} in link with $ \max_{z \in \Uu} \abs{Z_{U_N}(z) }^2 $ ; in accordance to the Keating-Snaith philosophy, one can infer from this result a conjecture on $ \max_{s \in [0, T] } \abs{\zeta\prth{\frac{1}{2} + i s U} }^2 $ in the equivalent regime (see \cite[(21), (34)]{FyodorovKeating} and \cite{FyodorovHiaryKeating}). 

In the regime $ k\beta^2 < 1 $, the general result was of the form
\begin{align*}%$
\Mom(N\vert k, \beta) \equivalent{N\to +\infty} c_-(k, \beta) N^{k \beta^2}
\end{align*}

Due to the method of proof using Toeplitz determinant analysis with Fisher-Hartwig singularities, the regime $ k\beta^2 > 1 $ of ``coalescing singularities'' was left to a further study in \cite{FyodorovKeating} with the following conjecture: 
\begin{align}\label{Conj:FyodorovKeatingMoMo}%$
\Mom(N\vert k, \beta) \equivalent{N\to +\infty} c_+(k, \beta) N^{k^2 \beta^2 + 1 - k}
\end{align}

This conjecture was recently proven in \cite{BaileyKeating} for $ k, \beta \in \Nn^* $ using combinatorics of Young tableaux and an analysis of \eqref{Eq:SchurCFKRS}, in \cite{AssiotisKeating} with combinatorics of Gelfand-Tsetlin patterns and in \cite{Fahs} for $ k, \beta \in \Rr^+ $ (including the transition $ k\beta^2 = 1 $, but without an expression of $ c_+(k, \beta) $) using Toeplitz determinant analysis.

We will rederive \eqref{Conj:FyodorovKeatingMoMo} for $ k, \beta \in \Nn^* $ with a new expression of $ c_+(k, \beta) $ in section~\ref{Subsec:MoMo}.

% ==
\subsubsection{Derivatives}\label{Subsubsec:Derivatives}

A classical question pertaining to the Riemann Zeta function is the probabilistic behaviour of its derivative on the critical axis. For $ \Re(s) > 1 $, the derivative is defined by $ \zeta'(s) = \sum_{n \geq 1} \log(n) n^{-s} $. In the strip $ \ensemble{ 0 < \Re(s) < 1 } $, it can be analytically extended using other formulas. The Riemann hypothesis is equivalent to the non existence of zeros of $ \zeta' $ on $ \ensemble{\Re > 1/2} $ \cite{Speiser}.

A first probabilistic study of $ \zeta'\prth{\frac{1}{2} + i TU} $ for $ U \sim \Us([0, 1]) $ was performed by Hadamard, Landau and Schnee in the years 1905-1909 (see the introduction of \cite{Ingham} and the references given) with the gradual computation of the following equivalent, for all $ \alpha, \beta > -\frac{1}{2} $:
\begin{align*}%$
\Esp{ \zeta'\prth{ \alpha + i T U } \overline{\zeta'\prth{ \beta + i T U } }  } \equivalent{T\to+\infty} \zeta''(\alpha + \beta) (\log(T))^3
\end{align*}

The years that followed saw several reinforcements of this study in particular with the replacement of $ TU \sim \Us([0, T]) $ by $ V_T \sim \Us\prth{ \ensemble{ \gamma_k, k \leq T } } $ where $ (\gamma_k)_{k \geq 1} $ denotes the sequence of zeroes of $ \zeta $ in the strip $ \ensemble{ 0 < \Re(s) < 1 } $ (see e.g. \cite{GonekDerivatives1, GonekDerivatives2, HejhalLogZetaPrime, HughesKeatingOConnell2} and \cite[\S 14]{Titchmarsh}), the case of $ U $ being treated in \cite{ConreyFourth, MezzadriDerivatives}, but this is only in \cite{ConreyRubinsteinSnaith} (following a conjecture in \cite{HughesThesis}) that the equivalent of \eqref{Eq:MomentsConjecture} was given for $ \zeta' $. It writes
\begin{align}\label{Eq:DerivativeMomentsConjecture}%$
\Esp{ \abs{\zeta'\prth{\frac{1}{2} + i T U } }^{2k} } \equivalent{T\to+\infty} A_k D_k e^{k^2 \log\log(T)}
\end{align}
where $ A_k $ is the arithmetic factor defined in \eqref{Def:ArithmeticFactor} and $ D_k $ comes from the following computation involving $ Z_{U_N}'(x) := \frac{d}{dx} Z_{U_N}(x) $
\begin{align*}%$
\Esp{ \abs{ Z_{U_N}'(1) }^{2k} } \equivalent{T\to+\infty} D_k e^{k^2 \log(N)}
\end{align*}

Several expressions of $ D_k $ are given in \cite{ConreyRubinsteinSnaith, ForresterWittePainleve3p5} following studies in \cite{BourgadeCondHaar, SnaithDerivatives}. In this article, we will focus on the following functional that generalises it: for $ \hb := (h_1, \dots, h_m) \in \Nn^m $ and $ k \geq \sum_r   h_r $, set
\begin{align}\label{FuncZ:Derivatives}%$
\De_N(k ; \hb ; X) := \Esp{ \abs{Z_{U_N}(e^{2 i \pi x_0/N})}^{2(k - \sum_r   h_r )} \prod_{r = 1}^m \abs{\partial^r Z_{U_N}(e^{2 i \pi x_r / N })}^{2h_r} }
\end{align}
where $ \partial^r Z_{U_N} (x)  := \prth{\frac{d}{dt}}^r Z_{U_N}(t)\vert_{t = x } $. Note that $ \Esp{ \abs{Z_{U_N}'(1) }^{2k} } = \De_N(k ; k ; (0, 0)) $.

We will derive an asymptotic expression for this quantity in section~\ref{Subsec:Derivatives}.

\begin{remark}
Riedtmann \cite{RiedtmannMixedRatios} recently investigated the asymptotic behaviour of the following functional that combines ratios and derivatives of $ Z_{U_N} $ (with restrictions ; see \cite{RiedtmannMixedRatios})
\begin{align}\label{FuncZ:MixedRatios}%$
\hspace{-0.27cm} \Me_N(X, X', Y, Y', Z, Z') := \Esp{ \frac{ \prod_{j = 1}^{\ell_1} Z_{U_N}\prth{ x_j  }  }{ \prod_{r = 1}^{m_1} Z_{U_N}\prth{ y_r } } \prod_{q = 1}^{d_1} Z_{U_N}'(z_q) \overline{ \frac{ \prod_{j = 1}^{\ell_2} Z_{U_N}\prth{ x'_j }  }{ \prod_{r = 1}^{m_2} Z_{U_N}\prth{ y'_r } } \prod_{q = 1}^{d_2} Z_{U_N}'(z'_q) }  } 
\end{align}

Its study by means of the combinatorial method of Dehaye \cite{POD} leads to several theorems and conjectures on the equivalent functional on the Riemann $ \zeta $ function. The previous method of computation applies to derive an asymptotic expression of $ \Me_N $ evaluated in points of the form $ e^{i x/N} $ but we leave it for a further study.
\end{remark}

\subsection{A prototype of result}

The Keating-Snaith theorem \cite[(15), (16)]{KeatingSnaith} concerns the asymptotic behaviour of $ \Esp{ \abs{Z_{U_N}(1) }^{2k} } $ and is given in \eqref{Eq:KSmomentsCharpol}. It is a particular case of e.g. the autocorrelations \eqref{FuncZ:AutocorrelationsCharPol}, the truncated characteristic polynomial \eqref{FuncZ:Truncation} or the ``moments of moments'' \eqref{FuncZ:MoMo} (see remark~\ref{Rk:KSandMoMo}).

Note that it was proven for $ k \in \Cc $ using the Selberg integral whose proof consisted in an analytic extension of the result for $ k \in \Nn $. Here, we only propose a proof in this last case, with the advantage of being general enough to treat all the remaining cases previously exposed. The form that we give is the following:

\begin{theorem}[Value of the characteristic polynomial in 1]\label{Thm:KSwithDuality} 
For all $ k \geq 2 $, one has 
\begin{align}\label{Eq:KSCUEwithDuality}%$
\frac{ \Esp{ \abs{Z_{U_N}(1) }^{2k} } }{ N^{k^2 } } \tendvers{N}{+\infty} \widetilde{L}_1(k)
\end{align}
where
\begin{align}\label{EqPhi:KS}%$
\begin{aligned}
\widetilde{L}_1(k) & := \frac{(2\pi)^{k(k - 1) } }{k!} \int_{\Rr^{k - 1} } \Phi_{\widetilde{L}_1}(0, x_2, \dots, x_k) \Delta(0, x_2, \dots, x_k)^2 dx_2 \dots dx_k \\
\Phi_{\widetilde{L}_1}(0, x_2, \dots, x_k) & :=  e^{ 2 i \pi (k^2 - 1) \sum_{ j = 2 }^k x_j } \, h_{k, \infty}^{(2k) }(0, x_2, \dots, x_k)
\end{aligned} 
\end{align} 
\end{theorem}

The function $ h_{c, \infty}^{(\kappa) }(x_1, x_2, \dots, x_k) $ is defined in lemma \ref{Lemma:RescaledGegenbauer} and is a limiting version of the homogeneous complete symmetric function $ h_N $ (see section~\ref{Sec:Notations}) seen as a multivariate extension of the Gegenbauer polynomials. It has several expressions coming from several representations (integral in lemma~\ref{Lemma:RescaledGegenbauer}, probabilistic with beta-gamma random variables in lemma~\ref{Lemma:BetaProbRepr}, exponential-polynomial in lemma~\ref{Lemma:RescaledSupersym}, etc.). 

Note that the integral is on $ \Rr^{k - 1} $, hence the restriction on $ k \geq 2 $. The case $ k = 1 $ can be computed by hand and gives $ \widetilde{L}_1(1) = 1 $. 

This theorem is proven in section~\ref{Subsec:KS}.

% ==
\subsection{Comparison with the literature}\label{Subsec:Intro:ComparisonLiterature}

A plethora of ad-hoc methods can be used to treat some (but not all) of the previous problems~:
\begin{itemize}

\item The \textit{Toeplitz method} is based on the fact that $ \int_{\Ue_N} \det(f(U)) dU = \det(\Tt_N(f)) $ (Andr\'eieff-Heine-Szeg\"o identity or continuous Cauchy-Binet identity, see \cite[\textit{fact five}]{DiaconisRMTSurvey}) where $ \Tt_N(f) $ is the $ N \times N $ Toeplitz matrix of general term $ A_{k, j} := \crochet{z^{k - j } } f(z) $ (see section~\ref{Sec:Notations} for notations) and $ dU $ is the normalised Haar measure of $ \Ue_N $. It is often coupled with a \textit{Riemann-Hilbert} analysis \cite{DeiftIntegrable, DeiftItsKrasovsky} to analyse the orthogonal polynonomials on the unit circle (OPUC) that arise in relation with $f$. It is useful to study \eqref{FuncZ:AutocorrelationsCharPol}, \eqref{FuncZ:Ratios} and \eqref{FuncZ:MoMo} ; for instance, in the case of \eqref{FuncZ:Ratios}, a formula due to Day \cite{Day} was first derived in the framework of Toeplitz determinants, but it was rederived in \cite{ConreyForresterSnaith} with an ad-hoc method from \cite{BasorForrester} and in \cite[rk. 6]{BumpGamburd} using symmetric function theory. Moreover, as reveals \cite{Fahs} in the case of \eqref{FuncZ:MoMo}, the singularity analysis can become extremely involved and the limiting constant become out of reach.

\medskip
\item The \textit{polytope method} consists in expressing the quantity of interest using a polytope, usually with its Ehrhart polynomial defined in \eqref{Def:EhrhartPol} whose higher degree coefficient is the volume of the polytope. It was used in the study of \eqref{FuncZ:Truncation}, \eqref{FuncZ:KR3Gfunc} and \eqref{FuncZ:MoMo} and could give possibly a result in the study of \eqref{FuncZ:MidSecularCoeff}.

\medskip
\item It is related with Dehaye's \textit{combinatorial method} that expresses expectations of the previous functionals as sums over Young diagrams (see e.g. \cite{POD, RiedtmannMixedRatios} and references cited for an overview), hence, not necessarily of polytopial type. Such a method was used to analyse \eqref{FuncZ:Derivatives} and \eqref{FuncZ:MixedRatios} and can be used for \eqref{FuncZ:AutocorrelationsCharPol}, \eqref{FuncZ:Ratios}, \eqref{FuncZ:MixedRatios} and perhaps \eqref{FuncZ:Truncation} but is unlikely to give any result for \eqref{FuncZ:MidSecularCoeff} for instance. Moreover, some restrictions apply in the study of \eqref{FuncZ:MixedRatios} (see \cite{RiedtmannMixedRatios}).

\medskip
\item The \textit{binomial method} was imported by Dehaye \cite[\S~3]{PODfpsac} from asymptotic representation theory to study \eqref{FuncZ:Derivatives}~; it is a variation of the (generalised) binomial theorem of Okounkov-Olshanski \cite[thm. 5.1]{OkounkovOlshanski1} (see also \cite{OkounkovOlshanski2}) itself a consequence of the Cauchy-Binet formula (see \S~\ref{Subsec:Ultimate:Expansions}). It shares similarities with correlation function expansions in integrable and determinantal probability (see \S~\ref{Subsec:Intro:PhilosophyMethod}). It was so far only used to study \eqref{FuncZ:Derivatives} but could be extended to study all previous functionals possibly with the exception of \eqref{FuncZ:Ratios} and \eqref{FuncZ:MixedRatios}.

\medskip
\item \textit{Determinantal techniques} are based on formuli of Fredholm determinants and the specification of a determinantal point process by means of a particular kernel that contains the whole information on the process. Here, the key point lies in the ability to translate the problem into a pure determinantal problem. So far, it was used by Gorodetsky and Rodgers \cite{GorodetskyRodgers} to tackle a generalisation of \eqref{FuncZ:KR3Gfunc} (see remark~\ref{Rk:GorodetskyRodgers}) and by Assiotis, Keating and Warren \cite{AssiotisKeatingWarren} (after a Cayley transform that maps unitary matrices to hermitian matrices) to tackle \eqref{FuncZ:Derivatives} (see \S~\ref{Subsubsec:Applications:Derivatives:Motivations}). In each of these cases, the relevant kernels arose from harmonic analysis on an infinite group \cite{Kerov, KerovOlshanskiVershik, OlshanskiSymInf2, OlshanskiVershik}~: the case treated in \cite{GorodetskyRodgers} uses $ z $-measures on the infinite symmetric group $ \Sg_\infty $ \cite{BorodinInfSym, BorodinOlshanskiZmeas1, BorodinOlshanskiZmeas2, OkounkovZmeas, OlshanskiSymInf1}, and the case treated in \cite{AssiotisKeatingWarren} uses Hua-Pickrell measures on the infinite Hermitian group $ \He_\infty $ \cite{AssiotisHuaPickrell, AssiotisNajnudel, BorodinOlshanskiErg, OlshanskiVershik, Pickrell}.

Note also that all these problems can be transformed in full generality into problems of hermitian matrices using the Cayley transform given in \S~\ref{Subsec:Theory:DualityHermitian}, in which case, hermitian techniques are available (see e.g. \cite{DesrosiersDuality, Neretin} and references cited).

\medskip
\item Techniques of \textit{integrable systems} (and more precisely $ \sigma $-forms of Painlev\'e equations) have recently proven useful to study \eqref{FuncZ:Derivatives} and \eqref{FuncZ:MixedRatios} \cite{BaileyBettinBlowerConreyProkhorovRubinsteinSnaith, BasorBleherBuckinghamGravaIts2Keating, ForresterWittePainleve3p5}. For instance, \cite{BasorBleherBuckinghamGravaIts2Keating} exhibits a Painlev\'e V equation in \eqref{FuncZ:Derivatives} starting with Winn's approach \cite{Winn} (that amounts to a Cayley transform and a $ LUE $ formula that writes \eqref{FuncZ:AutocorrelationsCharPol} with a Toeplitz/size-dependent Hankel determinant) by means of a Riemann-Hilbert technique~; the moments of (logarithmic) derivatives akin to \eqref{FuncZ:MixedRatios} and studied in \cite{ConreyRubinsteinSnaith} were shown in \cite{ForresterWittePainleve3p5} to be related with a Painlev\'e III' equation again with a Toeplitz/Hankel connection. 

%\item Plancherel measures expansions were pioneered by Dehaye (jamais publié, hélas...). 
 
\medskip
\item \&c. 

\end{itemize}

Apart from these methods that have their restricted domain of applicability, a last general method emerges to derive the asymptotic behaviour of the previous quantities. It was used on \eqref{FuncZ:AutocorrelationsCharPol}, \eqref{FuncZ:Ratios}, \eqref{FuncZ:Derivatives}, \eqref{FuncZ:KR3Gfunc} and can be used in a straightforward way for \eqref{FuncZ:Truncation}, \eqref{FuncZ:MixedRatios}, \eqref{FuncZ:LpNorm} and maybe \eqref{FuncZ:MidSecularCoeff} (with some additional tricks and difficulties, though). It is based on the following contour integral formula that has first appeared in \cite[(2.17)]{ConreyFarmerKeatingRubinsteinSnaith} (see also \cite{BrezinHikamiCharpol} for a related formula in the limit $ N \to +\infty $):
\begin{align}\label{Eq:SchurCFKRS}%$
\begin{aligned}
& s_{N^m}(e^{-\alpha_1}, \dots, e^{-\alpha_m} , e^{-\alpha_{m + 1} } , \dots, e^{-\alpha_k}) \\ 
& \hspace{+1cm} = \frac{ (-1)^{\frac{k(k - 1)}{2} } }{k!} {k \choose m} \!\oint_{\Ce}  e^{-N (z_{m + 1} + \cdots + z_k ) }\! \!\!\!\! \prod_{ \substack{ 1 \leq \ell \leq m, \\ m + 1 \leq r \leq k   }  }\!\! \!\!\Ze(z_r - z_\ell ) \, \frac{\Delta(Z)^2}{ \prod_{ 1 \leq i, j \leq k } (z_i - \alpha_j) } \frac{dZ}{(2 i \pi)^k}  
\end{aligned}
\end{align}
%
%
%
%with
%%
%%
%%
\begin{align*}%$
\mbox{with }   \Ze(x) := \frac{1}{1 - e^{-x} } \, . \hspace{+12.5cm}
\end{align*}
%
%
%
%\begin{align}\label{Eq:SchurCFKRS}%$
%\begin{aligned}
%& s_{N^m}(e^{-\alpha_1}, \dots, e^{-\alpha_m} , e^{-\alpha_{m + 1} } , \dots, e^{-\alpha_k}) \\ 
%& \hspace{+1cm} = \frac{ (-1)^{\frac{k(k - 1)}{2} } }{k!} {k \choose m} \!\oint_{\Ce}  e^{-N (z_{m + 1} + \cdots + z_k ) }\! \!\!\!\! \prod_{ \substack{ 1 \leq \ell \leq m, \\ m + 1 \leq r \leq k   }  }\!\! \!\!\Ze(z_r - z_\ell ) \, \frac{\Delta(Z)^2}{ \prod_{ 1 \leq i, j \leq k } (z_i - \alpha_j) } \frac{dZ}{(2 i \pi)^k}
%\end{aligned}
%\end{align}
%%
%%
%%
%with
%%
%%
%%
%\begin{align*}%$
%\Ze(x) := \frac{1}{1 - e^{-x} }
%\end{align*}

Here, we have used the fact, noticed in \cite{BumpGamburd} and \cite{ConreyFarmerZirnbauer} that the autocorrelations \eqref{FuncZ:AutocorrelationsCharPol} can be written as a Schur function indexed by a rectangular diagram up to a factor $ e^{-N (\alpha_{m + 1} + \cdots + \alpha_k ) } $ (see section~\ref{Sec:Notations} for notations and definitions and section~\ref{Subsec:SchurAndCUE} for the derivation of this result). Up to a dilation of $ s_{N^m} $, the $ \alpha_j \in \Cc $ can be taken to have a modulus smaller than $1$ ; in this case, the contour $ \Ce $ which is supposed to enclose the variables $ \alpha_j $ can be taken as the product of unit circles $ \Uu^k $.

The formula \eqref{Eq:SchurCFKRS} is a contour integral reformulation of a sum over particular permutations on disjoint subsets. Note that the integrand in \eqref{Eq:SchurCFKRS} is \textit{not} a symmetric function\footnote{See nevertheless \S~\ref{Subsec:Theory:CFKRSwithRKHS} for an analysis of this fact.} in $ Z $ due to the term $ e^{-N (z_{m + 1} + \cdots + z_k ) } $. It was originally designed to stress similarities with the case of the Riemann $ \zeta $ function and other $L$-functions (see \cite[conj. 2.2]{ConreyFarmerKeatingRubinsteinSnaith} and \cite{CFKRSZeta}: the function $ \zeta(1 + \cdot) $  conjecturally plays a r\^ole analogous to $ \Ze $). Nevertheless, its use in practical computations such as e.g. \eqref{FuncZ:KR3Gfunc} leads to several complications, see e.g. \cite[\S 4.4]{KeatingR3G} or \cite{BaileyKeating}.

Let us analyse \cite[\S 4.4]{KeatingR3G} in details: the quantity of interest is $ \crochet{x^{\pe{cN} } } x^{kN} s_{N^k}(x 1^k, 1^k) $ where $ 1^k = (1, \dots, 1) $ ($k$ times), $ x 1^k = (x, \dots, x) $ and $ c \in (0, k) $ (see the proof of Theorem~\ref{Thm:KR3G:withFormula} for a derivation). The formula \eqref{Eq:SchurCFKRS} is thus first written for $x$ fixed, and the Fourier coefficient is taken at the end. A first step shrinks the contour $ \Ce $ into $ 2^{2k} $ small circles around $ 0 $ or $ \log(x) $ and leads to a sum of the form $ \sum_{0 \leq \ell \leq k} {k \choose \ell}^2 P_{k, \ell}(x) + R_N(x)$ \cite[(4.26)]{KeatingR3G}. A second step consists in showing that the remainder $ R_N(x) $ (namely a sum of integrals where different contours surround $\log(x)$ and $ 0 $) does not contribute \cite[lem. 4.11]{KeatingR3G}. The main term is then analysed with a change of variables of the form $ z_i = \log(x) + v_i/N $ or $ z_i = v_i/N $ and several tricks involving products defined by exclusive conditions on indices \cite[(4.28)]{KeatingR3G}. Finally, picking out the $ \crochet{x^{\pe{cN} } } $-Fourier coefficient leads to an expansion into the critical quantities \cite[(4.30) or (4.31)]{KeatingR3G}. The analysis of these last quantities uses the limiting behaviour of a particular rescaling of the homogeneous symmetric functions $ h_N $ \cite[lem. 4.12]{KeatingR3G}, in the same way we will do in this article~; but these computations somehow hide the fact that this rescaling is the critical step, lost as it is in the middle of other manipulations. All these steps are aguably quite involved and ad-hoc, but an analysis in the same flavour is also performed in \cite[(3.6)]{ConreyRubinsteinSnaith}, and in \cite[thm. 1.2]{BaileyKeating}. In each cases, the final form given by this procedure has a certain complexity (see e.g. \cite[thm. 1.6]{KeatingR3G}). 

We analyse in \S~\ref{Subsec:Theory:CFKRSwithRKHS} the origin of \eqref{Eq:SchurCFKRS} with the general theory presented in this paper. 

\begin{remark}\label{Rk:OtherApproachesSchur}
Since the previous problems are equivalent to a particular rescaling of a Schur function, one can also look at the literature on this topic. This is done in \S~\ref{Subsec:Ultimate:OtherApproachesSchur}.
\end{remark}

% ==
\subsection{Presentation of the method}

We propose a general framework based on symmetric functions and multivariate generalisations of classical orthogonal polynomials written in a probabilistic fashion to compute the limiting behaviour of all the previous functionals of $ Z_{U_N} $. One of the main formuli that replaces \eqref{Eq:SchurCFKRS} is given in \eqref{FourierRepr:SchurRectangle}/\eqref{FourierRepr:SchurRectangleWithHn} and writes \textit{for any alphabet $ \Ae $} (see section~\ref{Sec:Notations} for notations and concepts)
\begin{align}\label{Eq:SchurWithHn}%$
\boxed{ s_{N^k}\crochet{\Ae} = \frac{1}{k!} \oint_{ \Uu^k }  U^{-N}  h_{Nk}\crochet{U \Ae} \abs{\Delta(U)}^2 \frac{d^*U}{ U} }
\end{align}

Based on this methodology, we prove that all the previous functionals of the characteristic polynomial can be written and rescaled by simply rescaling in a particular regime $ h_{Nk } \crochet{U \Ae} $ which is a symmetric function generalisation of some classical orthogonal polynomials. This article will only be concerned with the case of the Gegenbauer (or ultraspherical) polynomials, and their particular Tchebychev specialisation, but other problems could lead to other natural generalisations. Last, the method of rescaling will be of probabilistic nature: we will find a probabilistic representation of these polynomials and show that the regime of limit taken so far in all previous examples is, in probabilistic terms, a local limit theorem or local Central Limit Theorem (local CLT, in short).

We stress the fact that our framework is very general and allows all previous functionals $ \Ie_N(k) $ to satisfy an asymptotic equality of the form
\begin{align*}%$
\Ie_N(k) \sim N^{ s_{\Ie}(k) } \Ie(k) 
\end{align*}
with $ s_{\Ie}(k) \in \Nn $ and
\begin{align}\label{Def:FormPhi}%$
\Ie(k) := \frac{(2\pi)^{k(k - 1)}}{k!} \int_{ \Rr^{k - 1} } \Phi_\Ie(x_2, \dots, x_k) \Delta(0, x_2, \dots, x_k)^2 dx_2 \dots dx_k
\end{align}
for a suitable function $ \Phi_\Ie : \Rr^{k - 1} \to \Cc $ (sometimes, we can also have $ \Rr^k $ and $ \Delta(x_1, \dots, x_k) $).

% ==
\subsection{Philosophy of the method}\label{Subsec:Intro:PhilosophyMethod}

The method can be described as Hilbertian: formula \eqref{Eq:SchurWithHn} is better understood in the setting of \textit{reproducing kernel Hilbert spaces} (RKHS, see e.g. \cite{Aronszajn1, Aronszajn2, SchwartzRKHS}) of symmetric functions if one writes (see remark~\ref{Rk:RestrictedKernel})
\begin{align*}%$
s_{N^k}\crochet{\Ae} = \bracket{s_{N^k}, K_N\prth{\Ae, \cdot} }_{ L^2_{sym}(\Uu^k) }
\end{align*}
with a scalar product on $ \Ue_k $ or equivalently on symmetric functions in $ L^2(\Uu^k) $ (namely with $ k $ effective variables) since we manipulate class functions (see \S~\ref{Subsec:ProdScal}). There exists a general reproducing kernel on the space of all symmetric functions, the \textit{Cauchy kernel} (see \S~\ref{Subsec:ProdScal}) but we will use its restriction on the space of symmetric functions of degree $ kN $ given by $ K_N\prth{\Ae, X} := h_{kN}\crochet{\Ae X} $ (see remark~\ref{Rk:RestrictedKernel}) as it is more suitable for asymptotic analysis.

\medskip

The method also shares conceptual similarities with determinantal and integrable probabilities (see e.g. \cite{JohanssonHouches, BorodinPetrov}). In these last fields, one often wants to compute fluctuations of a random variable $ \lambdab_N $ starting with a formula of the form $ \Prob{ \lambdab_N \leq x} = \int_{\Xe^N } f_x(Y) d\mu(Y) $ that gets transformed into an absolutely convergent expansion of the form $ \Prob{ \lambdab_N \leq x} = \sum_{k \geq 0} \int_{\Xe^k} \! K_{N, x ; k}(y) d\nu_k(y) $ for certain functions $ K_{N, x ; k} : \Xe^k \to \Rr $. For instance, the case of determinantal point processes gives $ K_{N, x ; k}(y_1, \dots, y_k) = \frac{(-1)^k}{k!} \det(\Ke_{N, x}(y_i, y_j))_{1 \leq i, j \leq k} $ for a certain function $ \Ke_{N, x} $ (also called ``kernel'' in the determinantal case). Once the main problem of probability theory (the growing number of integrals) is resolved (since each term in the expansion has a fixed number of integrals), one is reduced to rescale the functions $ K_{N, x ; k } $ (i.e. the kernel $ \Ke_{N, x} $ in the determinantal case) in a regime involving the rescaling of $x$ and the $ y_i $'s. 

This is exaclty what we do here\footnote{Nevertheless, the analogy in terms of expansion is given by the binomial formula used by Dehaye \cite{POD}.}: the initial integral over $ \Ue_N $ (hence with a growing number $N$ of integrals) is reduced to an integral over $ \Uu^k $ that writes $ \int_{\Uu^k } K_N(U ,\Ae) d\mu(U) $ and one then focuses on the sole rescaling of the ``kernel'' $ K_N(U, \Ae) $. This last terminology is in addition the correct one in the theory of RKHS~: $ K_N(U ,\Ae) $ is called a \textit{reproducing kernel}.

% ==
\subsection{Advantages and drawbacks}

One general drawback that this method shares with its original counterpart but not with some ad-hoc manipulations is the impossibility to extend analytically in $k$ the functional. 

There are nevertheless several advantages in using \eqref{Eq:SchurWithHn} in place of \eqref{Eq:SchurCFKRS} to perform computations in the previous problems:
\begin{itemize}

\item it is more ergonomic in the sense that the obtained expressions remain exact until the very moment one passes to the limit, using dominated convergence. It avoids manipulations of contours, intermediate approximations, etc. Moreover, the domination we employ comes from the probabilistic representation of $ h_N\crochet{\Ae} $ as a Fourier transform and amounts to do almost sure manipulations under the form of the simple inequalities \eqref{Ineq:BoundXi}.

\medskip
\item It is more likely to be generalised into two different directions: on the one hand, the ``alphabet'' $ \Ae $ in \eqref{Eq:SchurWithHn} can be specialised in an abstract sense (see \S~\ref{Subsec:FuncSym}) that comprises several cases ; this is not the case of \eqref{Eq:SchurCFKRS} which, even if it writes under a ``similar'' form in the supersymmetric version of \cite{ConreyFarmerZirnbauer}%
\footnote{The residue formula equivalent to \eqref{Eq:SchurCFKRS} is given in the supersymmetric case in e.g. \cite[Thm. 1.3]{BaileyBettinBlowerConreyProkhorovRubinsteinSnaith}. % ~: $ \frac{ (-1)^{k(k - 1)/2 } }{m! (k - m)!}   \!\oint_{\Uu^k}  e^{-N (z_{m + 1} + \cdots + z_k ) }  \prod_{ \substack{ 1 \leq \ell \leq m, \\ m + 1 \leq r \leq k   }  }  \Ze(z_r - z_\ell ) \, \Delta(Z)^2 \prod_{ 1 \leq i, j \leq k } \frac{ z_i - \beta_j }{ z_i - \alpha_j } \frac{dZ}{(2 i \pi)^k} $ in the case of a ratio with variables $ (e^{\alpha_j})_{1 \leq j \leq k} $ in the numerator and $ (e^{\beta_j})_{1 \leq j \leq k} $ in the denominator, with $ 1 \leq m \leq k $.
Note that our definition of $ Z_{U_N} $ is $ \Lambda_{U_N} $ in \cite{ConreyFarmerZirnbauer} and \cite{BaileyBettinBlowerConreyProkhorovRubinsteinSnaith}.} 
is not a straighforward replacement of quantities in the vein of an abstract specialisation due to the dichotomy between classical and exponential variables and other obstructions (see \S~\ref{Subsec:Theory:CFKRSwithRKHS}). On the other hand, formula~\eqref{Eq:SchurWithHn} and the underlying philosophy of proof can be extended in a straightforward way to the orthogonal and symplectic circular ensembles with a zonal polynomial or more generally to circular $ \beta $ ensembles, replacing the Schur function by a Jack polynomial expansion (see \S~\ref{Subsec:DualityWithRKHSandPlethysm}). In fact, the philosophy could also be applied to any ensemble, if one defines the correct equivalent of rectangular Schur function as in \cite{JonnadulaKeatingMezzadri}, and it could even be generalised to other random variables associated to more general symmetric functions such as the Hall-Littlewood ones (see remark~\ref{Rk:HLtoSchur}).

\medskip
\item The phenomenon of ``hyperplane concentration'' is explained in a clear way, using the scaling of the reproducing kernel $ h_{kN}\crochet{X\Ae} $~: one sees that the integral in \eqref{Def:FormPhi} is on $ \Rr^{k - 1} $ when one starts from a formula on $ \Uu^k $. Here again, the initial formula can be transformed into an integral over $ \Uu^{k - 1} $ with an equality, \textit{before} passing to the limit. More generally, the scaling property of the Schur function, namely $ s_\lambda(t X) = t^{\abs{\lambda} } s_\lambda(X) $ is clear in \eqref{Eq:SchurWithHn} as a consequence of the scaling of $ h_{Nk} $ and hidden in \eqref{Eq:SchurCFKRS}.

\medskip
\item One can continue the expansion of the limit in a straightforward way (although with tedious computations) using the natural expansion of all involved quantities (see \S~\ref{Subsec:Ultimate:Expansions}).

\medskip
\item It gives a general explanation for the appearance of Hankel determinants in the limit (see \S~\ref{Subsec:Ultimate:Hankel}).

\medskip
\item All the limiting functionals have the same form \eqref{Def:FormPhi}, coming from a unifying framework for the asymptotic analysis of all the previous problems.

\medskip
\item There are three possible alternative ways to conclude the computations in each case, one of which being more probabilistic than the others (the \textit{randomisation paradigm}, see \S~\ref{Subsec:RandomisationParadigm}).

\end{itemize}

On these two last points, let us precise that although the Keating-Snaith moments computation given in theorem \ref{Thm:KSwithDuality} is e.g. a particular case of theorem \ref{Theorem:TruncatedCharpolIn1} (and others), we will see that the limiting expressions will be different in the two theorems. By unicity of the limit, these expressions will thus be equal. Only the function $ \Phi_{\widetilde{L}_1} $ will be changed, keeping thus the general form intact. These differences can be explained for fixed $N$ by equivalent formulas of symmetric functions and the different ways to rescale them.

We hope this plethora of new expressions of the same nature will bring new bridges between the existing results and shed new lights on their unified nature. In each of the treated cases, the expressions under the form \eqref{Def:FormPhi} can be compared with the existing ones, bringing more instances to the zoo of known expressions and the zoo of different methods to analyse the characteristic polynomial such as described in section \ref{Subsec:Motivations}. We recapitulate all these results in section \ref{Subsec:FinalTable}.

% ==
\subsection{Philosophy of the rescaling}

The formula \eqref{Eq:SchurWithHn} will be rescaled by finding a probabilistic representation of $ h_{Nk}\crochet{U \Ae } $ under the form $ \Ee( [ \sum_{ j = 1 }^k u_j Z_j(\Ae) ]^N ) $ for independent random variables $ Z_j(\Ae) $ (see annex~\ref{Subsec:ProbaHn}). Changing variables in \eqref{Eq:SchurWithHn} by setting $ u_j = e^{2 i \pi \theta_j /N } $ will amount to a local CLT. We recall that if $ S_n $ is an integer-valued random variable with expectation $ \mu_n $ and variance $ \sigma_n^2 $ that satisfies a Gaussian CLT, one can estimate $ \Prob{S_n = \pe{nx} } $ using Fourier inversion:
\begin{align*}%$
\Prob{ S_n = \pe{nx} } & = \int_{-1/2}^{1/2} \Esp{ e^{2 i \pi \theta (S_n - \pe{nx}) } } d\theta = \int_{-\sigma_n/2}^{\sigma_n/2} \Esp{ e^{2 i \pi \alpha (S_n - \mu_n)/\sigma_n } } e^{2 i \pi \alpha (\mu_n- \pe{nx})/\sigma_n } \frac{d\alpha}{\sigma_n} \\
              & \approx \frac{1}{\sigma_n} \int_\Rr e^{-(2\pi\alpha)^2/2} e^{2 i \pi \alpha (\mu_n - \pe{nx})/\sigma_n } d\alpha = \frac{1}{\sigma_n \sqrt{2\pi} } e^{ -(\mu_n - \pe{nx})^2/2\sigma_n^2 }
\end{align*}

The only step to justify is the approximation $ \approx $ that comes from the Central Limit Theorem. This is a consequence of the Dominated Convergence Theorem if one can bound $ \abs{ \Esp{ e^{ i \theta (S_n - \mu_n)/\sigma_n } } - e^{-\theta^2/2} } \Unens{ \abs{\theta } \leq \pi \sigma_n } $. This will also be our strategy, leading to the dominations of Annex~\ref{Sec:ProbReprGegenbauer}. In fact, we will be concerned with the simpler case where $ \mu_n = 0 $ and $ X_n/\sigma_n \to X_\infty $ in law, with an absolutely continuous Lebesgue density $ f_{X_\infty} $:
\begin{align}\label{Eq:PhilosophyCLT}%$
\begin{aligned}
\Prob{ S_n = \pe{ x \sigma_n } } & = \int_{-1/2}^{1/2} \Esp{ e^{2i\pi \theta (S_n - \pe{ x\sigma_n }) } } d\theta = \int_{- \sigma_n/2 }^{ \sigma_n /2 } \Esp{ e^{2i\pi \alpha (S_n - \pe{ x\sigma_n })/\sigma_n } } \frac{d\alpha}{\sigma_n} \\
              & \approx \frac{1}{\sigma_n} \int_\Rr \Esp{ e^{2 i \pi \alpha(X_\infty - x) } } d\alpha = \frac{1}{\sigma_n  } f_{X_\infty}(x)
\end{aligned}
\end{align}

We will not be in the Gaussian setting, the limit using instead beta-gamma random variables (lemmas~\ref{Lemma:GammaProbRepr}, \ref{Lemma:BetaProbRepr}).

The approach that consists in translating the asymptotic estimation of a Fourier coefficient into a local CLT can be traced back to Hayman \cite{Hayman} who used it to prove the Stirling approximation formula as an alternative to the Laplace method or the steepest descent analysis, but it was only later remarked by Rosenbloom \cite{Rosenbloom} that it amounts to prove a local limit theorem. This philosophy culminates in the work of B\'aez-Duarte \cite{BaezDuartePn} who uses it to give a probabilistic proof of the Hardy-Ramanujan estimate of the number of partitions of a large integer (\textit{partitio numerorum}). It was morever rediscovered independently by Romik \cite{RomikPn} who used it to reprove a result of Szekeres \cite{SzekeresPn} itself simplified by Canfield \cite{CanfieldPn} about the asymptotic behaviour of partitions restricted to a rectangle. 

This philosophy is very well described in \cite{BaezDuartePn}. Since it is primarily based on Fourier inversion, we will start by expressing the quantities of interest as Fourier coefficients\footnote{One can already see that \eqref{Eq:SchurWithHn} is a Fourier coefficient with the convention $ U^{-N} := \prod_{j = 1}^k u_j^{-N} $.}. 

\medskip
We stress the fact that this last philosophy consists in injecting probability in a domain a priori unrelated with any probabilistic feature ; this is not the case of random matrix theory, but in the particular case of random unitary matrices, the set of methods described in section~\ref{Subsec:Motivations} are mainly not probabilistic (with the notable exception of \cite{BarhoumiHughesNajnudelNikeghbali, BourgadeCondHaar, BHNY, ChhaibiNajnudelNikeghbali}). Another goal of this article is thus to show that this whole field can be reshaped \textit{in a more probabilistic fashion}.

\medskip
% ==
\subsection{The randomisation paradigm}\label{Subsec:Intro:RandomisationParadigm}

A last point of importance this article will highlight is the prominence of one particular functional amongst all the functionals previously described~: the mid-secular coefficients. Indeed, all previous functionals (with the exception of the ratios) write as products of linear functionals of $ Z_{U_N} $, and as a result can be written as linear statistics of the secular coefficients under the form
\begin{align*}%$
\Le_A(Z_{U_N}) = \sum_{k = 0}^N a_k \sc_k(U_N) = \prth{ \sum_{k = 0}^N \abs{a_k} } \Esp{ \sc_{A_N}(U_N) e^{i \Theta_{A_N}} \vert U_N }, \qquad a_k := \abs{a_k} e^{i \Theta_k}
\end{align*}
where $ A_N \! \in \intcrochet{0, N} $ is the random variable given by $ \Prob{A_N = k} := \abs{a_k} ( \sum_{\ell = 0}^N \abs{a_\ell} )\inv \Unens{ 0 \leq k \leq N } $. 

Suppose to simplify that $ a_k \in \Rr_+ $. Choosing a sequence of i.i.d. $ (A_N^{(j)})_{j \geq 1} $ gives 
\begin{align*}%$
\Esp{ \abs{ \Le_A(Z_{U_N}) }^{2k} } = \prth{ \sum_{\ell = 0}^N  a_\ell  }^{\!\! 2k} \Esp{ \prod_{j = 1}^k \sc_{A_N^{(j)} }(U_N) \overline{ \sc_{A_N^{(j + k)} }(U_N) } }
\end{align*}
and one is reduced to study the behaviour of the randomised sequence $ \sc_{A_N}(U_N) $ (see \S~\ref{Subsec:RandomisationParadigm} for a more complete description). In the particular case where $ A_N/N \to A_\infty $ in law or in $ L^1 $, one has $ A_N \approx \pe{N A_\infty} $ which is exactly a ``quenched'' version of the mid-secular coefficients $ \sc_{\pe{\rho N} }(U_N) $ if one conditions on $ A_\infty $ or if $ A_\infty $ is almost surely a constant, and an ``annealed'' one when integrating on all present random variables (which is ultimately what is done). Playing on these two levels of randomness will show that this is precisely the behaviour of this randomisation that is of importance to understand the global behaviour of the moments. This will particularly be well illustrated with the discrepancy between the Conrey-Gamburd theorem~\ref{Theorem:TruncatedCharpolIn1} and the Heap-Lindqvist theorem~\ref{Theorem:TruncatedCharpolHeapLindqvist}, but all the previous functionals (with the exception of the ratios) enter into this framework. As Gian-Carlo Rota might have put it, the secular coefficients are ``completely equiprimordial'' with the eigenvalues.

% ==
\subsection{Organisation of the paper}

The paper is organised as follows: 

\begin{itemize}

\medskip
\item \textbf{Section \ref{Sec:Notations} :} we gather some notations used throughout the text and introduce the important \textit{plethystic/$ \lambda $-ring formalism} we will make a constant use of. Section~\ref{Subsec:FuncSym} in particular gives some of the basics and several references about this convenient language that aims at manipulating symmetric functions (and that constitutes the modern vision of the theory). We emphasize its importance to understand any phenomenon concerning symmetric functions and hope this article will contribute to popularise it to the probabilistic and mathematical physics community.

\medskip
\item \textbf{Section \ref{Sec:MainFormulas} :} We describe the general theories behind the machinery we employ (duality, RKHS, etc.) and replace them in a general probabilistic, algebraic and physical context. A comparison between the CFKRS formula \eqref{Eq:SchurCFKRS} and the main formula of this article \eqref{Eq:SchurWithHn} is also provided.

\medskip
\item \textbf{Section \ref{Sec:Applications} :} we revisit the problems described in \S~\ref{Subsec:AllProblems} with the tools introduced in section~\ref{Sec:MainFormulas}~:
\begin{itemize}

\item \textbf{\S~\ref{Subsec:KS} :} we give a prototype of proof describing \textit{precisely} the structure of all future proofs in the case of the Keating-Snaith theorem \ref{Thm:KSwithDuality},

\item \textbf{\S~\ref{Subsec:Autocorrelations} :} we treat the autocorrelations described in \S~\ref{Subsubsec:Autocorrelations}, 

\item \textbf{\S~\ref{Subsec:Ratios} :} we treat the ratios described in \S~\ref{Subsubsec:Ratios}, 

\item \textbf{\S~\ref{Subsec:Midcoeff} :} we treat the mid-secular coefficients described in \S~\ref{Subsubsec:Midcoeff}. 

\end{itemize}

Two proofs will be presented every time it is relevant, with a possible alternative end coming from a variation of the involved formuli. The study of the mid-coefficients will then motivate the introduction of the randomisation paradigm described in \S~\ref{Subsec:Intro:RandomisationParadigm} and we will introduce it plainly in \textbf{\S~\ref{Subsec:RandomisationParadigm}} on the example of the autocorrelations. Starting from this example, a third proof by randomisation will be added. We will then revisit~:
\begin{itemize}

\item \textbf{\S~\ref{Subsec:TruncatedCharpol} :} the behaviour of the truncated characteristic polynomial in $1$ (Conrey-Gamburd theorem, \S~\ref{Subsubsec:ConreyGamburd}), in $ \lambda > 1 $ (Heap-Lindqvist theorem, \S~\ref{Subsubsec:HeapLindqvist}) and in the microscopic scaling (\S~\ref{Subsubsec:TruncatedMicro}),

\item \textbf{\S~\ref{Subsec:Birkhoff} :} the Beck-Pixton approach \cite{BeckPixton} to the Birkoff polytope ; we obtain in theorem \ref{Theorem:VolumeBirkoff} a new expression for its relative volume, a known difficult open problem. One representation that we use involves the function $ h_{1, \infty} $ in a $ k $-fold integral. This changes drastically from the representations given in \cite{HarperNikeghbaliRadziwill, HeapLindqvist} and is possibly new. We also treat the case of the polytope of sub-stochastic matrices and the transportation polytopes, two natural generalisations of the Birkhoff polytope. Note also that remark~\ref{Rk:SubBirkhoffWithUniforms} gives a probabilistic explanation to the phenomenon of piecewise polynomiality that appears in \cite{AssiotisKeating, BasorGeRubinstein, BettinConrey, KeatingR3G}.

\item \textbf{\S~\ref{Subsec:Derivatives} :} the joint moments of the iterated derivatives of $ Z_{U_N} $, 

\item \textbf{\S~\ref{Subsec:KR3G} :} the result of Keating, Rodgers, Roditty-Gershon and Rudnick \cite{KeatingR3G} described in \S~\ref{Subsubsec:KR3G}, 

\item \textbf{\S~\ref{Subsec:MoMo} :} the ``moments of moments'' described in \S~\ref{Subsubsec:MoMo}. 
\end{itemize}

\medskip
\item \textbf{Section \ref{Sec:Ultimate} :} We link the theory described in section \ref{Sec:MainFormulas} with Hankel determinants and Wronskians such as developed in \cite{BaileyBettinBlowerConreyProkhorovRubinsteinSnaith, BasorBleherBuckinghamGravaIts2Keating, BasorGeRubinstein} and show that it is the unifying framework that explains the apparition of such determinants in \S~\ref{Subsec:Ultimate:Hankel}. We explain how to push the method presented in this paper to obtain general expansions in the vein of \cite{PODfpsac, BasorBleherBuckinghamGravaIts2Keating} in \S~\ref{Subsec:Ultimate:Expansions}, analyse the literature on the rescaling of the Schur function in \S~\ref{Subsec:Ultimate:OtherApproachesSchur} and give a summary of the encountered functionals $ \Phi $ defined in \eqref{Def:FormPhi} in \S~\ref{Subsec:FinalTable}. We conclude in \S~\ref{Subsec:Ultimate:Conclusion} with general remarks, questions of interest and forthcoming work.

\medskip
\item \textbf{Annex \ref{Sec:ProbReprGegenbauer} :} we introduce $ h_N\crochet{X\Ae} $ as multivariate extensions of  Tchebychev/Gegen-bauer polynomials, compute their limit rescaled in the microscopic regime that will be used to prove the thorems in section~\ref{Sec:Applications} ; we give in particular several probabilistic representations of these functions as Fourier transforms of vectors of independent random variables conditionned on their sum equal to a constant, a classical representation for many probabilistic functionals. 

\medskip
\item \textbf{Annex \ref{Sec:NumberTheory} :} we answer question \ref{Question:NTphase} and show the existence of a phase transition between the two universality classes previously mentionned in \S~\ref{Subsubsec:Midcoeff} and \S~\ref{Subsubsec:TruncatedCharpol}.
\end{itemize}

$ $

\medskip
% ==
\section{Notations and prerequisites}\label{Sec:Notations}

% ==
\subsection{General notations and conventions}

We denote by $ \Uu := \ensemble{z \in \Cc : \abs{z} = 1}$ the unit circle. Define
\begin{align*}%$
d^*z := \frac{dz}{2i \pi}, \qquad \frac{d^*Z}{Z} := \prod_{k = 1}^N \frac{d^*z_k}{  z_k } \quad \text{ if } \quad Z := (z_1, \dots, z_N) 
\end{align*}

We will use the notation 
\begin{align*}%$ 
\crochet{z^n} f(z) = a_n \quad \mbox{if} \quad f(z) = \sum_n a_n z^n
\end{align*}
for the $ n $-th Fourier coefficient of a Laurent series $ f $. Note that $ \crochet{z^n}f(z) = \crochet{z^0}z^{-n} f(z) $. If $ \operatorname{rad}(f) $ denotes the radius of convergence of $f$, then
\begin{align}\label{Eq:FourierCoeff}%$ 
\crochet{z^n} f(z) = \oint_{r \Uu} f(z) z^{-n} \frac{d^*z }{z} = \frac{1}{r^n} \int_0^1 f(r e^{2 i \pi \theta}) e^{-2 i \pi n \theta} d\theta, \qquad r < \operatorname{rad}(f)
\end{align}

We will also use the multi-index notation $ X^\alpha := \prod_{k \geq 1} x_k^{\alpha_k} $ and subsequently the notation $ \crochet{X^\alpha} f(X) $. If $ N \in \Zz $ and $ X = (x_1, \dots, x_k) $, we will write $ X^N $ for $ X^{N \Un_k} $ with $ \Un_k = (1, \dots, 1) $ ($k$ times). In particular, we will write $ X\inv $ for $  x_1\inv x_2\inv \dots  x_k\inv $.

Last, we denote by $ \Pe $ the set of prime numbers, by $ v_p(n) $ the $p$-adic valuation of an integer $ n \in \Nn^* $, and use the convention
\begin{align}\label{Convention:Bar}%$
\overline{\rho} := 1 - \rho
\end{align}

% ==

\begin{remark}\label{Rk:MultivariateFourierWithHomogeneity}

Note the following property of multivaluate Fourier coefficients if $ f(z) = \sum_n a_n z^n $
\begin{align*}%$ 
\crochet{x^n y^n} f(xy) = \crochet{x^n} \crochet{y^n} f(xy) = \crochet{x^n} a_n x^n = a_n
\end{align*}
and the following corollary if $ f $ is homogeneous of degree $ k\alpha $, namely $ f(\lambda X) = \lambda^{k\alpha} f(X) $ with $ X = (x_1, \dots, x_k) $
\begin{align*}%$ 
\crochet{X^{\alpha }} f(X) & := \crochet{X^{\alpha\Un_k}} f(X) = \crochet{X^0}X^{-\alpha \Un_k} x_1^{\alpha k} f(1, x_2/x_1, \dots, x_k/x_1) \\
               & = \crochet{ X^0 } (x_2/x_1)^{-\alpha} \cdots (x_k/x_1)^{-\alpha} f(1, x_2/x_1, \dots, x_k/x_1) \\
               & = \crochet{x_1^0} \crochet{T^0} T^{-\alpha \Un_{k - 1}} f(1, t_2, \dots, t_k) \\
               & = \crochet{T^{ \alpha  }} f(1, t_2, \dots, t_k)
\end{align*}

Here, the ``change of variable'' $ t_i = x_i/x_1 $ for $ i \in \intcrochet{2, k} $ amounts to a repeated application of the previous property. 

Equivalently, one can use the representation \eqref{Eq:FourierCoeff} with $ \theta_1, \dots, \theta_k \in [-\pi, \pi] $, change variable $ \varphi_j := \theta_j - \theta_1 $ for $ j \in \intcrochet{2, k} $ and then integrate on $ \theta_1 $.
\end{remark}

\medskip
% ==
\subsection{Reminders on symmetric functions}\label{Subsec:FuncSym}

The partitions of an integer $n$ are weakly decreasing sequences $ \lambda := (\lambda_1 \geq \lambda_2 \geq \cdots \geq \lambda_m) $ such that $ \abs{\lambda} := \sum_{i = 1}^m \lambda_i = n $ (see \cite[ch. I.1]{MacDo}). We will write $ \lambda \vdash n $ for $ \abs{\lambda} = n $. The length $m$ of $ \lambda $ will be denoted by $ \ell(\lambda) $. The transpose of a partition will be denoted by $ \lambda' $.

% ==

For notations an definitions concerning symmetric functions, the reference is \cite{MacDo}. The Schur functions will be denoted by $  s_\lambda  $, the power functions by $ p_k $, the homogeneous complete symmetric functions by $ h_n $ and the elementary symmetric functions by $ e_n $. We will give a relevant definition of these functions when necessary throughout the chapters.

% ==

The use of $ \lambda $-ring/plethystic notations and alphabets will be done in the same vein as \cite{Ram1991, LascouxSym, HaimanMacDo}, using in particular the convention of \cite{HaimanMacDo} that writes plethysm with a bracket $ \crochet{\cdot}$: the plethysm of a symmetric function $ f $ in an abstract alphabet $ \Ae $ will be denoted $ f[\Ae] $. This convenient language will be necessary to manipulate the ``abstract'' alphabets that we define now. For a quick introduction that shows brilliantly the advantage of such a formalism, we refer to \cite[\S~2]{HaimanMacDo} ; for examples of this formalism in effective computations, we refer to \cite[\S~4]{Ram1991} and for a more advanced version of it, we refer to \cite{LascouxSym} (see also the thesis \cite[ch. 2]{ZabrockiThesis} and \cite{PODmomentsCombin} for an accessible introduction although with different conventions). 

\begin{remark}
We emphasize the reading of \cite[\S~2]{HaimanMacDo} as a critical introduction to plethysm in case the reader is not familiar with it.
\end{remark}

% ==

Given the sets of variables or ``alphabets'' $ X := \ensemble{ x_k }_{k \geq 1} $ and $ Y := \ensemble{ y_k }_{k \geq 1} $, one defines 
\begin{align*}%$
p_k\crochet{X} & =  \sum_{\ell \geq 1} x_\ell^k   \\
p_k\crochet{X + Y}  & =  p_k \crochet{X}  + p_k\crochet{Y}  \\
p_k\crochet{ X \cdot  Y } & =  p_k\crochet{X}  p_k\crochet{Y}  
\end{align*}

Note the difference of convention between $ p_k( X + Y ) = \sum_\ell (x_\ell + y_\ell)^k $ and $ p_k\crochet{ X + Y } $ as defined above.

We can easily check that with these definitions, $ X + Y $ is the alphabet obtained by concatenating $ X $ and $ Y $ and $ X \cdot Y = \ensemble{ x_k y_\ell }_{k, \ell \geq 1} $ is a ``tensor'' alphabet. One can also see that $ + $ and $ \cdot $ are associative operations on the alphabets. From now on, we will forget the dot when considering a tensor alphabet and we will write $ X Y $ for it. In particular, we will write $ t X $ for $ \ensemble{t} \cdot X $. We also define the tensor-sum alphabet by $ X \tensorsum Y := \ensemble{ x_k + y_\ell }_{k, \ell \geq 1} $ and the exponentiated alphabet by $ e^{a X} := \ensemble{ e^{a x}, x \in X } $ for all $ a \in \Cc $. Note that $ e^{a (X \tensorsum Y)} = e^{aX} e^{a Y} $ and that $ \tensorsum $ is associative with $ + $, i.e. $ X\tensorsum(Y + Z) = X\tensorsum Y + X\tensorsum Z $. 

% ==

We define the generating series of the complete homogeneous symmetric functions by
\begin{align}\label{Def:Hfunctor}%$
\boxed{H\crochet{X} := \sum_{n \geq 0} h_n(X) = \prod_{k \geq 1} \frac{1}{ 1 -  x_k }}
\end{align}

This series is an important object that will be constantly referred to throughout this article. It is denoted by $ \Omega $ in \cite{Ram1991, HaimanMacDo} and by $ \sigma $ in \cite{LascouxSym} and \cite[I-2, rk. 2.15]{MacDo}. Here, we prefer to stick to the classical notation of \cite[I-2 (2.5)]{MacDo} to avoid disturbing the reader.

Note in particular that replacing $ X $ by $ tX $ and taking the $ n $-th Fourier coefficient yields
\begin{align}\label{FourierRepr:hn}%$
h_n\crochet{X} = \crochet{t^n} H\crochet{tX }
\end{align}

% ==

With the usual convention $ p_\lambda := \prod_{k \geq 1} p_{\lambda_k} $, we have the relation \cite[I (2.14)]{MacDo}
\begin{align*}%$
\sum_{\lambda } \frac{t^{ \abs{\lambda} }}{ z_\lambda } p_\lambda\crochet{X} = \prod_{k \geq 1} \frac{1}{ 1 - t x_k } = H\crochet{t X}
\end{align*}

Replacing $ X $ by $ XY $, we get
\begin{align*}%$
\sum_{\lambda } \frac{ t^{ \abs{\lambda} } }{ z_\lambda } p_\lambda\crochet{X} p_\lambda\crochet{Y} = \prod_{k, \ell \geq 1} \frac{1}{ 1 - t x_k y_\ell } = H\crochet{t X Y}
\end{align*}

% == Difference

We now define the difference of two alphabets $ X - Y $ as the ``abstract'' alphabet such that \cite[1-3, ex. 23]{MacDo}
\begin{align*}%$
H\crochet{ X - Y} = \frac{ H\crochet{ X } }{ H\crochet{ Y } }
\end{align*}

Using \cite[I (2.10)]{MacDo}, this last definition is equivalent to
\begin{align}\label{Eq:SupersymPn}%$
p_k\crochet{ X - Y } = p_k \crochet{ X } - p_k \crochet{ Y }
\end{align}

Note that there is associativity of $ \tensorsum, \cdot $ with $ - $, i.e. $ X\tensorsum(Y - Z) = X\tensorsum Y -  X\tensorsum Z $ and $ X(Y - Z) = X  Y -  X  Z $. Note also that $ f(-X) := f(-x_1, -x_2, \dots) \neq f\crochet{-X} $. To differentiate these two operations and in accordance with the convention of \cite[after I.2]{GarsiaHaimanTesler}, we define the alphabet % \href{https://www.mat.univie.ac.at/~slc/s/s42garsia.pdf}{article}%\commentaire{changer le nom de la variable par \MoinsUn au lieu de \varepsilon}
\begin{align}\label{Def:AlphabetEpsilon}%$
\varepsilon := \ensemble{-1}
\end{align}
so that $ f\crochet{\varepsilon X} = f(-X) $.

% == Involution \omega

The fundamental involution $ \omega $ \cite[I-2, (2.7)]{MacDo} is defined by $ \omega(h_n) = e_n $ or $ \omega(e_n) = h_n $. Equivalently, it can be defined by $ \omega(s_\lambda) = s_{\lambda'} $ \cite[I-3, (3.8)]{MacDo}. Since $ s_\lambda\crochet{-X} = (-1)^{\abs{\lambda}} s_{\lambda'}\crochet{X} = s_{\lambda'}\crochet{\varepsilon X} $ (see \cite{HaimanMacDo} or \cite[I-3, (3.10)]{MacDo}), we thus have $ \omega(s_\lambda\crochet{X}) = s_\lambda\crochet{-\varepsilon X} $ and one can define an abstract alphabet $ \widehat{\omega} $ by
\begin{align}\label{Eq:InvolutionOmega}%$
\omega(s_\lambda\crochet{X}) = s_\lambda\crochet{\widehat{\omega} X}, \qquad \widehat{\omega} := -\varepsilon
\end{align}

% == Cauchy

The Cauchy identity is \cite[I (4.3)]{MacDo}
\begin{align}\label{Eq:CauchyIdentity}%$
H \crochet{t X  Y} = \sum_\lambda t^{ \abs{ \lambda } } s_\lambda \crochet{ X } s_\lambda \crochet{ Y }
\end{align}

According to \cite[p. 9]{LascouxSym}, this is ``the most important formula in the theory of symmetric functions''. Changing $ X $ into $ -X $ in \eqref{Eq:CauchyIdentity} and using $ s_\lambda\crochet{-X} = (-1)^{\abs{\lambda} } s_{\lambda'}\crochet{X} $ gives the \textit{dual} Cauchy identity
\begin{align}\label{Eq:DualCauchyIdentity}%$
H \crochet{- t X  Y} = \sum_\lambda (-t)^{ \abs{ \lambda } } s_{\lambda'} \crochet{ X } s_\lambda \crochet{ Y }
\end{align}

Using this last identity, one has for all $ n \geq 0 $
\begin{align*}%$
\sum_{\lambda \vdash n } s_\lambda\crochet{ X }  s_\lambda\crochet{ Y }  = \sum_{\lambda \vdash n } \frac{1}{z_\lambda } p_\lambda\crochet{ X }  p_\lambda\crochet{ Y } 
\end{align*}
from what one can deduce that 
\begin{align}\label{Eq:PlethisticHnSchur}%$
h_n\crochet{ XY } = \sum_{\lambda \vdash n } \frac{1}{z_\lambda } p_\lambda\crochet{ XY } = \sum_{\lambda \vdash n } s_\lambda\crochet{ X }  s_\lambda\crochet{ Y }
\end{align}

% == A_{n - 1} roots alphabet

We define the following \textit{root} alphabet
\begin{align}\label{Eq:ARootAlphabet}%$
X^R := \ensemble{ x_i x_j\inv, 1 \leq i < j \leq n }
\end{align}

Let $ (e_i)_i $ be the canonical basis of $ \Cc^n $. We have used here the set of \textit{roots} $ R := \ensemble{ e_i - e_j, i < j } $ denoted by $ R_0^+ $ in \cite[I-5 ex. 16]{MacDo} ; this is in fact the set of \textit{positive roots of type $ A_{n - 1} $} (to simplify the notation, we abusively write $ R $ in place of $ R_0^+ $). With the multi-index notation, $ X^R = \ensemble{ X^r, r \in R } = \ensemble{ x_i x_j\inv , i < j } $. 

We will in fact be more concerned with the set of negative roots (with $ \varepsilon = \ensemble{-1} $)
\begin{align}\label{Eq:MinusARootAlphabet}%$
X^{\varepsilon R} := \ensemble{X^r, r \in \varepsilon R} = \ensemble{ x_j x_i\inv, 1 \leq i < j \leq n }
\end{align}

% == 1^k alphabet

A last abstract alphabet that we will use is $ 1^{\plusInOne \kappa} $ ; it is defined as $ (1, 1, \dots, 1) $ ($ \kappa $ times) if $ \kappa \in \Nn^* $ and extended to the case $ \kappa \in \Rr_+^* $ by setting $ p_m\crochet{1^{\plusInOne \kappa} } = \kappa $ for all $ m \geq 1 $, or equivalently, $ H\crochet{t 1^{\plusInOne \kappa} } = H\crochet{t}^\kappa $. Note that $ 1^{\plusInOne \kappa}1^{\plusInOne \alpha} = 1^{\plusInOne \kappa\alpha} $ as one can see in the case $ \alpha, \kappa \in \Nn^* $. This alphabet is understood as the action of constants in the $\lambda$-ring terms of \cite[(2.1), (2.2)]{LascouxSym} and \cite[(18)]{PODmomentsCombin}, but since we do not use the convention of \cite{LascouxSym} or \cite{PODmomentsCombin} that writes alphabets with bold letters, we cannot distinguish between constants and alphabets and prefer to use this equivalent formulation. As a result, the \textit{plethystic/$ \lambda $-ring square expansion} writes $ (a + b)(a + b) = \ensemble{a^2} + \ensemble{b^2} + 1^{\plusInOne 2}\ensemble{ab} =: a^2 + 1^{\plusInOne 2} ab + b^2 $, but the advantage is to keep the natural scaling property of the symmetric functions intact, contrary e.g. to \cite[(20)]{PODmomentsCombin}. We will only be concerned in this article with the case $ \kappa \in \Nn^* $, but natural applications use $ \kappa \in \Rr_+^* $.

% == Examples

\begin{example}\label{Ex:HfunctorSpecialise}
Here are some examples of classical functions that write as a $ H\crochet{\Ae} $ for a certain alphabet $ \Ae $:
\begin{enumerate}

\item We have
\begin{align*}%$
H\crochet{X^R} = \prod_{1 \leq i < j} \frac{1}{ 1 - x_i x_j\inv}, \qquad H\crochet{-X^R} = \prod_{1 \leq i < j} (1 - x_i x_j\inv)
\end{align*}

\item The Vandermonde determinant in the variables $ X := \ensemble{x_1, \dots, x_n} $ is defined by 
\begin{align*}%$
\Delta(X) := \prod_{1 \leq i < j \leq n}(x_i - x_j)
\end{align*}

A factorisation that uses $ \delta_n := (n - 1, n - 2, \dots, 1, 0) $ gives
\begin{align*}%$
\Delta(X) := \prod_{1 \leq i < j \leq n}x_i(1 - x_j/x_i) = X^{\delta_n} H\crochet{ - X^{ \varepsilon R } }
\end{align*}

\item The Cauchy determinant $ C(X, Y) := \det\prth{ \frac{1}{1 - x_i y_j} }_{1 \leq i, j \leq N} $ writes 
\begin{align*}%$
C(X, Y) = H\crochet{XY - X^{ \varepsilon R } - Y^{ \varepsilon R } }
\end{align*}

\medskip
\item We recall that $ H\crochet{ t 1^{\plusInOne \kappa} } = H\crochet{t}^\kappa $. This implies $ H\crochet{t X 1^{\plusInOne \kappa} } = H\crochet{ t X}^\kappa = \prod_{i}(1 - tx_i)^{-\kappa} $.

\medskip
\item The usual convention $ h_\lambda := \prod_{k \geq 1} h_{\lambda_k} $ and \eqref{FourierRepr:hn} gives the \textit{Fourier representation}
\begin{align}\label{FourierRepr:hLambda}%$
h_\lambda\crochet{X} = \crochet{U^\lambda} H\crochet{X U}
\end{align}

Note that $ U = \ensemble{u_1, \dots, u_\ell } $ for all $ \ell \geq \ell(\lambda) $.

\medskip
\item If $ A = \ensemble{a_1, \dots, a_n} $, one has
\begin{align}\label{Eq:PolynomialWithHn}%$
h_n\crochet{x - A} = \prod_{k = 1}^n (x - a_k)
\end{align}

Indeed, 
\begin{align*}%$
h_n\crochet{x - A} = \crochet{t^n}H\crochet{t(x-A)} = \crochet{t^n}\frac{1}{1 - tx}H\crochet{-tA} = x^n H\crochet{-x\inv A}
\end{align*}
using the Cauchy (residue) formula for a polynomial.

\medskip
\item Let $ \Eee $ be the alphabet defined by $ p_k\crochet{\Eee} = \Unens{k = 1} $. Then, using the fundamental formula $ H\crochet{ \Ae} = \exp\prth{ \sum_{k \geq 1} \frac{1}{k} p_k\crochet{\Ae} } $ \cite[I-2, (2.10) p. 23]{MacDo}, one gets $ H\crochet{z \Eee} = e^z $ ; this motivates the name \textit{exponential alphabet} for $ \Eee $ and implies $ h_n\crochet{\Eee} = \frac{1}{n!} $. 

\medskip
\item Let $ q \in \Cc $. The alphabet $ \frac{1}{1 - q} $ is defined by $ \sum_{k \geq 0} \ensemble{q^k} := \ensemble{ q^k, k \geq 0 } $. As a result, $ p_k\crochet{\frac{1}{1 - q} } = \frac{1}{1 - q^k} $ if $ \abs{q} < 1 $. For a possible use of such an alphabet, see \cite{HaimanMacDo}.

\medskip
\item See also \cite[I-5, ex. 10]{MacDo} for other specialisations of $ H $ with other alphabets.
\end{enumerate}
\end{example}

% ==

\begin{remark}\label{Rk:Hopf}
Some of the aspects of symmetric function theory presented in this \S~can also be recast in the language of \textit{(combinatorial) Hopf algebras}, see e.g. \cite[ch. 2.3]{MeliotReprSym} for an introduction with probabilistic applications, \cite[ch. 2]{GrinbergReiner} for combinatorial aspects or \cite[I-5, ex. 25 p. 91]{MacDo} for a summary. The relevant Hopf algebra here is the algebra $ \Lambdab $ of symmetric functions obtained by a projective limit in the category of \textit{graded} rings as in \cite[ch. I.2]{MacDo}. The map $ S : f\crochet{X} \mapsto f\crochet{-X} $ is the antipode of $ \Lambdab $, $ \Delta : f\crochet{T} \mapsto f\crochet{X + Y} $ is a possible coproduct/comultiplication%
\footnote{There are in fact two natural Hopf algebra structures on $ \Lambdab $ as one can also consider the comultiplication $ \delta : f\crochet{T} \mapsto f\crochet{XY} $ that corresponds to the multiplication $ * : (p_\lambda, p_\mu) \mapsto \delta_{\lambda, \mu} z_\lambda p_\lambda $ ; see e.g. \cite[\S~2]{SchraffThibon} or \cite[I-7, ex. 20 p. 128]{MacDo} for an exposition.},
 the classical multiplication of functions being the multiplication of $ \Lambdab $, and the co-unit is given by $ 1^* : h_k \mapsto 1 - \Unens{k \geq 1} $ (i.e. it gives the constant term of a symmetric function, namely $ 1^*(f[X]) = \crochet{X^0}f\crochet{X} $). 
\end{remark}

% ==
\subsection{Scalar products and unitary integrals}\label{Subsec:ProdScal}

The scalar product on the space of symmetric functions in a finite number of variables can be realised as the $ L^2 $ scalar product of the normalised Haar measure of the unitary group (see \cite[ch. VI-9, rk 2 pp. 369-370]{MacDo} or \cite{DiaconisShahshahani}). We will denote it by
\begin{align*}%$
\bracket{f, g}_N := \int_{\Ue_N} f(U) \overline{g(U)} dU
\end{align*}
where $ dU $ designates the Haar measure on the space of unitary matrices $ \Ue_N $. 

% ==

If the functions $ f, g : \Ue_N \to \Cc $ are invariant by conjugacy, their values on $ \Ue_N $ only depend on the eigenvalues of the matrices, hence define symmetric functions $ \widetilde{f},\widetilde{g} : \Uu^N \to \Cc $ that we will denote, with an abuse of notation, by $ f, g : \Uu^N \to \Cc $. For such symmetric functions in $ N $ variables $ Z := (z_1, \dots, z_N) $, the Weyl integration formula (see \cite[ch. VI]{MacDo} or \cite{BourgadeCondHaar} for a probabilistic proof) writes 
\begin{align}\label{Eq:WeylHaarRealisation}%$
\bracket{f, g}_N =  \frac{1}{N!} \crochet{ Z^0 } f(Z)  g(Z\inv)   \Delta(Z)  \Delta(Z\inv)  
\end{align}

Using example~\ref{Ex:HfunctorSpecialise}, one has
\begin{align*}%$
\Delta(Z)\Delta(Z\inv) & = Z^{\delta_N} H\crochet{-Z^{\varepsilon R} } Z^{-\delta_N} H\crochet{-Z^R } \\
               & = \prod_{i \neq j} (1 - z_i z_j\inv) \\
               & = H\crochet{-Z^R - Z^{\varepsilon R} } = H\crochet{-Z^{(1 + \varepsilon )R}}
\end{align*}
with $ (1 + \varepsilon ) R = R + \varepsilon R := \ensemble{s r, r \in R, s\in\ensemble{\pm 1} } $ in the plethystic/$ \lambda $-ring sense. One thus has
\begin{align}\label{Eq:WeylHaarRealisationBis}%$
\bracket{f, g}_N =  \frac{1}{N!} \crochet{ Z^0 } f(Z)  g(Z\inv)  H\crochet{-Z^{(1 + \varepsilon) R} } 
\end{align}

% == RKHS

The Cauchy kernel $ H\crochet{XY} $ for $ X = \ensemble{x_1, \dots, x_N} $ and $ Y = \ensemble{y_1, \dots, y_N} $ is the reproducing kernel of the Hilbert space of square integrable symmetric functions in $N$ variables (see \cite[I-4, ex. 9]{MacDo}, \cite[lem. 5 p. 10]{LascouxSym} or \cite[(2.6)]{HaimanMacDo}), namely, for all such symmetric functions $ f $ 
\begin{align}\label{Eq:ReproducingKernelProperty}%$
f\crochet{X} = \bracket{H\crochet{XY}, f\crochet{Y} }_N = \frac{1}{N!} \crochet{Y^0} f\crochet{Y\inv} H\crochet{XY-Y^{(1 + \varepsilon) R}}
\end{align}

Using the Cauchy identity \eqref{Eq:CauchyIdentity}, this amounts to the following expansion
\begin{align*}%$
f\crochet{X} = \sum_\lambda \bracket{f , s_\lambda }_N s_\lambda\crochet{X}
\end{align*}
from which we deduce the orthogonality of the Schur functions \cite[ch. I-4]{MacDo}
\begin{align}\label{Eq:OrthogonalitySchur}%$
\bracket{s_\lambda, s_\mu}_N = \delta_{\lambda, \mu} \Unens{ \ell(\lambda)\leq N }
\end{align}

The reproducing kernel property is an equivalent formulation of this last orthogonality, as remarked in \cite[I-4, ex. 9]{MacDo} or in \cite[lem. 5 p. 10]{LascouxSym}, and the Cauchy identity \eqref{Eq:CauchyIdentity} amounts to the tensorial decomposition of the kernel in the theory of self-adjoint operator of Hilbert spaces~:
\begin{align*}%$
H\crochet{XY} = \sum_\lambda s_\lambda \otimes s_\lambda (X, Y)
\end{align*}

% == Autoadjonction de la composition pléthystique pour le prodscal

The last (critical) property that we will use is the following: the application of plethystic composition $ \Ce_A : f\crochet{X} \mapsto f\crochet{X A} $ is self-adjoint for the scalar product $ \bracket{\cdot, \cdot}_N $ \cite[(2.7)]{HaimanMacDo}:
\begin{align}\label{Eq:PlethysticComposition}%$
\bracket{\Ce_A f\crochet{X}, g\crochet{X}}_N = \bracket{ f\crochet{X}, \Ce_A g\crochet{X}}_N \quad \Longleftrightarrow \quad \bracket{ f\crochet{A X}, g\crochet{X}}_N = \bracket{ f\crochet{X}, g\crochet{AX} }_N
\end{align}

In particular, using the alphabet $ \widehat{\omega} $ defined in \eqref{Eq:InvolutionOmega}, one sees that $ \omega = \Ce_{\widehat{\omega}} $ is self-adjoint. Since it is clearly an involution, it is thus an isometry.

% ==

\medskip

\begin{remark}
Without specifying the number of variables, one still has the previous theory with an ``abstract'' scalar product $ \bracket{\cdot, \cdot}_\infty $ on symmetric functions (see \cite[I-4 (4.5)]{MacDo} or \cite[\S~1, p. 10]{LascouxSym}). In particular, all the previous properties work in this abstract setting (and were designed originally for it). As noted in \cite{LascouxSym}, there are more subtleties with scalar products with finite alphabets. Note that the self-adjunction of $ \Ce_A $ is proven in \cite[(2.7)]{HaimanMacDo} for $ \bracket{\cdot, \cdot}_\infty $, but the proof using the Cauchy product easily extends to $ \bracket{\cdot, \cdot}_N $ (alternatively, one can use an expansion of $ f, g $ in terms of power functions).
\end{remark}

% ==

% ==
\subsection{Classical tricks}

The following tricks are often useful in the study of scalar products with symmetric functions :
\begin{enumerate}

\item \textbf{\underline{Symmetrisation:}} Define 
\begin{align*}%$
\sigma \cdot f(x_1, \dots, x_n) := f(x_{\sigma(1)}, \dots, x_{\sigma(n)}) 
\end{align*}

Since $ \sigma \cdot \Delta(X) = \varepsilon(\sigma) \Delta(X) $ where $ \varepsilon(\sigma) $ is the signature of $ \sigma\in \Sg_n $ (not to be mistaken with \eqref{Def:AlphabetEpsilon}), for all $F$ symmetric,
\begin{align*}%$
\sum_{\sigma \in \Sg_N} \sigma \cdot \prth{ F(X) \Delta(X) } = \Delta(X) \sum_{\sigma \in \Sg_N} \varepsilon(\sigma)\sigma \cdot  F(X)  
\end{align*}

% ==

\item \textbf{\underline{Scalar product symmetrisation:}} For all $ F $ symmetric in $N$ variables, one has
\begin{align}\label{Trick:ProdScalSym}%$
\crochet{X^0}F(X)H\crochet{-X^R} = \frac{1}{N!} \crochet{X^0} F(X)\Delta(X)\Delta(X\inv) = \frac{1}{N!} \crochet{X^0} F(X) H\crochet{-X^{(1 + \varepsilon)R} }
\end{align}

The last equality follows from the equivalence between \eqref{Eq:WeylHaarRealisation} and \eqref{Eq:WeylHaarRealisationBis}. The first one comes from
\begin{align*}%$
\crochet{X^0}F(X)H\crochet{-X^R} & =  \crochet{X^0} F(X) X^{-\delta_N} \Delta(X) \\
             & = \crochet{X^0} \frac{1}{N!} \sum_{\sigma \in \Sg_n} \sigma \cdot \prth{ F(X) X^{-\delta_N} \Delta(X)} \\
             & = \crochet{X^0} \frac{1}{N!} F(X) \Delta(X) \sum_{\sigma \in \Sg_n} \varepsilon(\sigma) \sigma \cdot \prth{  X^{-\delta_N} } \\
             & = \frac{1}{N!}\crochet{X^0}  F(X) \Delta(X) \Delta(X\inv)
\end{align*}
where we have used $ \Delta(X) = \det\prth{ x_i^{\delta_N(j)} }_{1 \leq i, j \leq N} = \sum_{\sigma\in \Sg_n}\varepsilon(\sigma)\sigma\cdot X^{\delta_N} $. Note that one can also use $\crochet{X^0}F(X)H\crochet{-X^{\varepsilon R} } $ in place of $ \crochet{X^0}F(X)H\crochet{-X^R} $.

$ $

\item \textbf{\underline{$H$-reproducivity:}} The reproducing kernel applied to itself gives $ \bracket{ H\crochet{XU}, H\crochet{UY}}_N = H\crochet{XY} $, namely, using the previous trick
\begin{align}\label{Trick:Hreproducivity}%$
H\crochet{XY} = \crochet{U^0}H\crochet{X U + Y U\inv - U^{\varepsilon R} } 
\end{align}

\end{enumerate}

$ $

\medskip
% ==
\section{Theory}\label{Sec:MainFormulas}

% ==
\subsection{The ubiquitous Schur function}

The following ratio of alternants \cite[I-3 (3-1)]{MacDo} can be taken as a definition of the Schur functions:
\begin{align}\label{AlternantRepr:Schur}%$
s_\lambda(X) := \frac{\det\prth{ x_i^{\lambda_j + n - j} }_{1 \leq i, j \leq n} }{ \Delta(X) }, \qquad X := \ensemble{x_1, \dots, x_n}
\end{align}

Their \textit{Bethe ansatz form} is given by
\begin{align}\label{BetheRepr:Schur}%$
s_\lambda(X) = \sum_{\sigma \in \Sg_n} \sigma \cdot \prth{ X^\lambda H\crochet{ X^{\varepsilon R} } }
\end{align}

Indeed, one has
\begin{align*}%$
s_\lambda(X) & :=  \frac{\det\prth{ x_i^{\lambda_j + n - j} }}{ \Delta(X) } = \frac{1}{\Delta(X)} \sum_{\sigma \in \Sg_n } \varepsilon(\sigma) \sigma \cdot \prth{ X^{ \lambda + \delta_n }  }, \qquad \delta_n := (n-1, n-2, \dots, 1) \\
               & = \sum_{\sigma \in \Sg_n }\sigma\cdot \prth{ X^{\lambda + \delta_n} \frac{1}{\Delta(X)} } = \sum_{\sigma \in \Sg_n }\sigma\cdot \prth{ X^\lambda   H\crochet{X^{\varepsilon R}} } 
\end{align*}
since $ \Delta(X) = X^{\delta_n} H\crochet{ -X^{\varepsilon R} } $ by example~\ref{Ex:HfunctorSpecialise}.

From the \textit{Bethe representation} \eqref{BetheRepr:Schur}, one can deduce the following \textit{Fourier representation} analogous to \eqref{FourierRepr:hLambda} for $ h_\lambda $
\begin{align}\label{FourierRepr:Schur}%$
\boxed{s_\lambda(X) =  \crochet{ U^\lambda } H\crochet{ XU - U^{\varepsilon R} }}, \qquad U := \ensemble{u_1, \dots, u_\ell}, \ \ell \geq \ell(\lambda)
\end{align}

Indeed, using the reproducing kernel property \eqref{Eq:ReproducingKernelProperty}, one has
\begin{align*}%$
s_\lambda(X) & = \sum_{\sigma \in \Sg_n} \sigma \cdot \prth{ X^\lambda H\crochet{ X^{\varepsilon R} } } \\
             & = \frac{1}{n!} \crochet{U^0} \sum_{\sigma \in \Sg_n} \sigma \cdot \prth{ U^{-\lambda} H\crochet{ U^R } } H\crochet{ XU - U^{(1 + \varepsilon)R} } \\
             & = \frac{1}{n!}  \sum_{\sigma \in \Sg_n} \crochet{U^0} U_\sigma^{-\lambda} H\crochet{ U_\sigma^R }  H\crochet{ XU - U^{(1 + \varepsilon)R} }, \qquad U_\sigma := \sigma\cdot U \\
             & = \frac{1}{n!}  \sum_{\sigma \in \Sg_n} \crochet{U^0} U_\sigma^{-\lambda} H\crochet{ U_\sigma^R }  H\crochet{ XU_\sigma - U_\sigma^{(1 + \varepsilon)R} }, \qquad \mbox{ by symmetry}\\
             & = \crochet{U^0} U^{-\lambda} H\crochet{ U^R + XU - U^{(1 + \varepsilon)R} } \\
             & = \crochet{U^\lambda} H\crochet{ XU - U^{\varepsilon R} }
\end{align*}

This form can easily be seen equivalent to the Jacobi-Trudi identity \cite[I-3 (3.4)]{MacDo}
\begin{align}\label{Eq:JacobiTrudiSchur}%$
s_\lambda(X) =  \det\prth{ h_{\lambda_i + \ell - j}\crochet{ X } }_{1 \leq i, j \leq \ell}, \qquad \forall \ell \geq \ell(\lambda)
\end{align}
as one can write
\begin{align*}%$
s_\lambda(X) & = \crochet{U^\lambda} H\crochet{ XU - U^{\varepsilon R} } \\
             & = \crochet{U^\lambda} H\crochet{ XU } H\crochet{ - U^{\varepsilon R} } = \crochet{U^\lambda} \prod_{i = 1}^\ell H\crochet{u_i X} U^{-\delta_\ell } \Delta(U) \\
             & = \crochet{U^{\lambda + \delta_\ell} } \prod_{i = 1}^\ell H\crochet{u_i X} \sum_{\sigma \in \Sg_\ell} \varepsilon(\sigma) \prod_{r = 1}^\ell u_r^{\ell - \sigma(r) } \\
             & = \sum_{\sigma \in \Sg_\ell} \varepsilon(\sigma) \prod_{ r = 1 }^\ell \crochet{u^0} u^{-\lambda_r -\ell + r} u^{ \ell - \sigma(r)} H\crochet{u X} \\
             & =  \sum_{\sigma \in \Sg_\ell} \varepsilon(\sigma) \prod_{r = 1}^\ell h_{ \lambda_r + \sigma(r) - r }\crochet{X} \\
             & =:  \det\prth{ h_{\lambda_i + \ell - j}\crochet{ X } }_{1 \leq i, j \leq \ell}
\end{align*}

Note that \eqref{FourierRepr:Schur} can be specialised to any abstract alphabet previously defined using the specialisations of $ H\crochet{X U} $. For instance, 
\begin{align*}%$
s_\lambda\crochet{(X - Y) 1^{\plusInOne \kappa} + \theta \Eee} = \crochet{U^\lambda} H\crochet{ ( (X - Y)1^{\plusInOne \kappa} + \theta \Eee) U - U^{\varepsilon R} } =  \crochet{U^\lambda} \frac{H\crochet{XU}^\kappa}{H\crochet{YU}^\kappa} e^{\theta p_1(U)} H\crochet{ - U^{\varepsilon R} }
\end{align*}

This is also the case of \eqref{Eq:JacobiTrudiSchur} written as a ``Toeplitz minor'' $ s_\lambda\crochet{\Ae} = \det\prth{ \crochet{ t^{\lambda_i + \ell - j} } \!\! H\crochet{ t \Ae } }_{\!1 \leq i, j \leq \ell}$.

% ==

\begin{remark}\label{Rk:HLtoSchur} 
The Schur function is the Hall-Littlewood polynomial $ Q_\lambda(X\vert t) $ at $ t = 0 $. The equality \eqref{FourierRepr:Schur} is one such form. The Fourier form of the Hall-Littlewood polynomial is given by \cite[III-2 (2.15)]{MacDo} (see also \cite[def. 2.5.1]{ZabrockiThesis})
\begin{align}\label{FourierRepr:HLbis}%$
Q_\lambda(X \vert t) = \crochet{ U^\lambda } H\crochet{ (1 - t)( XU -  U^{\varepsilon R} )  }
\end{align}
and their Bethe form is given by \cite[III-2 (2.11)]{MacDo}
\begin{align}\label{BetheRepr:HL}%$
Q_\lambda\prth{X \vert t} := \sum_{\sigma \in \Sg_n} \sigma \cdot \prth{ X^\lambda H\crochet{ (1-t) X^{\varepsilon R} }}
\end{align}

Note also the following generalisation of $ Q_\lambda(X \vert t) $ due to Jing \cite[proof prop. 1.4]{JingQhypergeomMacDo}~:
\begin{align}\label{FourierRepr:Jing}%$
J_\lambda(X \vert q, t) = \crochet{ U^\lambda } H\crochet{ \frac{1 - t}{1 - q}( XU -  U^{\varepsilon R} )  }
\end{align}
\end{remark}

% ==

\begin{remark}\label{Rk:RestrictedKernel}
Using the scalar product symmetrisation \eqref{Trick:ProdScalSym}, one can write \eqref{FourierRepr:Schur} as
\begin{align}\label{Eq:SchurWithRKHS}%$
s_\lambda\crochet{X} & = \crochet{U^\lambda} H\crochet{ X U - U^{\varepsilon R} } \notag \\
              & = \crochet{U^0} U^{-\lambda -\delta } H\crochet{XU} \Delta(U) \quad\mbox{by example \ref{Ex:HfunctorSpecialise}} \notag \\
              & = \frac{1}{\ell !} \crochet{U^0} \prth{ \sum_{\sigma \in \Sg_\ell} \varepsilon(\sigma) \sigma \cdot U^{-\lambda - \delta } } H\crochet{XU} \Delta(U) \notag \\
              & = \frac{1}{\ell !} \crochet{U^0} s_\lambda(U\inv) H\crochet{XU} \Delta(U) \Delta(U\inv) \notag \\
              & = \bracket{ H\crochet{X U }, s_\lambda\crochet{U} }_\ell, \qquad \forall \ell \geq \ell(\lambda)
\end{align}
i.e. \eqref{FourierRepr:Schur} is just a particular instance of the Reproducing Kernel property \eqref{Eq:ReproducingKernelProperty} (it was used at the beginning of its proof). In the case of a rectangular partition, using the scalar product symmetrisation trick \eqref{Trick:ProdScalSym} or the last RKHS formula, one then gets $ s_{N^k}(X) = \frac{1}{\ell !} \crochet{U^0} s_{N^k}(U\inv) H\crochet{ XU } \Delta(U) \Delta(U\inv) $ for all $ \ell \geq k $. In particular, for $ \ell = k $, the Fourier representation becomes
\begin{align}\label{FourierRepr:SchurRectangle}%$
\boxed{s_{N^k}\crochet{X} = \frac{1}{k !} \crochet{ U^N } H\crochet{ XU } \Delta(U) \Delta(U\inv) }, \qquad U := \ensemble{u_1, \dots, u_k} 
\end{align}

Using the Cauchy identity \eqref{Eq:CauchyIdentity} and the fact that the only valid partitions in the expansion that will not give a zero scalar product are the ones satisfying $ \lambda \vdash Nk $ by the orthogonality \eqref{Eq:OrthogonalitySchur}, one has 
\begin{align*}%$
s_{N^k}\crochet{X} = \frac{1}{k !} \crochet{ U^N } \sum_{\lambda \vdash Nk} s_\lambda(U) s_\lambda\crochet{X} \Delta(U) \Delta(U\inv)
\end{align*}
namely, using \eqref{Eq:PlethisticHnSchur}
\begin{align}\label{FourierRepr:SchurRectangleWithHn}%$
\boxed{s_{N^k}\crochet{X} = \frac{1}{k !} \crochet{ U^N } h_{Nk}\crochet{ XU } \Delta(U) \Delta(U\inv) }, \qquad U := \ensemble{u_1, \dots, u_k} 
\end{align}
as announced in \eqref{Eq:SchurWithHn}. These formulas are of course valid for any specialisation of $ H $, i.e. for any abstract alphabet $ X $.
\end{remark}

\medskip
% ==
\subsection{The Schur-\textit{CUE} connection}\label{Subsec:SchurAndCUE}

The representation theory of $ \Ue_N $ was famously pioneered in random matrix theory by Diaconis and Shashahani in their study of $ ( \trace(U_N^k) )_{k \geq 1} $ \cite{DiaconisShahshahani}. In a similar vein of \eqref{FuncZ:MidSecularCoeff}/\eqref{Def:SecularCoeff}, the traces of powers of $ U_N $ write as the power functions in the eigenvalues of $ U_N $, and their moments are thus given by scalar products of these functions, a study that can be performed by a suitable change of basis with the Schur functions \cite[proof of thm. 2]{DiaconisShahshahani}. This first use of the Schur functions in this field opened the path for the use of representation-theoretic methods in the study of $ U_N $, and we refer to \cite{DiaconisRMTSurvey} for an overview of some of its developments.

\medskip

Since $ Z_{U_N}(X) = \prod_{j = 1}^N (1 - X \lambda_{j, N}) = H\crochet{-X \lambdab} $ where $ \lambdab := (\lambda_{j, N})_{ 1 \leq j \leq N} $ are the eigenvalues of $ U_N $, it seems natural to express quantities involving $ Z_{U_N} $ in terms of Schur functions with \eqref{Eq:DualCauchyIdentity} for instance. The joint moments (or \textit{autocorrelations}) of $ Z_{U_N}(X) $ have been computed in numerous articles \cite{Day, BasorForrester, BumpGamburd, ConreyFarmerKeatingRubinsteinSnaith, ConreyFarmerZirnbauer, ConreyForresterSnaith} after several similar studies for other matrix ensembles. In the case of the $ CUE_N $, up to a multiplicative factor, they take the form of a Schur function of a rectangular diagram, a fact first remarked in \cite[(2.11)]{ConreyFarmerKeatingRubinsteinSnaith} (and by Zirnbauer and Nonnenmacher in a private communication to the authors of \cite{ConreyFarmerKeatingRubinsteinSnaith})
\begin{align}\label{Eq:SchurJointMoments}%$
\Esp{ \prod_{\ell = 1}^m  Z_{U_N}\prth{ x_\ell } \prod_{j = 1}^k \overline{Z_{U_N}\prth{ \overline{y_j\inv} } } } = \prod_{ \ell = 1 }^k y_\ell^{-N} \, s_{ N^k }\prth{ x_1, \dots, x_m, y_1, \dots, y_k }
\end{align}
with $ N^k = (N, \dots, N) $ ($k$ times) a rectangular partition.

We recall here the proof of Bump and Gamburd \cite{BumpGamburd} with a plethystic/$ \lambda $-ring perspective:

\begin{proof}[Proof of \eqref{Eq:SchurJointMoments}]
Let $ X := \ensemble{ x_1, \dots, x_k} $, $Y := \ensemble{ y_1, \dots, y_m } $ and $ U := \ensemble{u_1, \dots, u_N }  $. We have
\begin{align*}%$
\Esp{ \prod_{\ell = 1}^m  Z_{U_N}\prth{ x_\ell } \!\prod_{j = 1}^k \! \overline{Z_{U_N}\prth{ \overline{y_j\inv} } } \! }    & =     \int_{ \Ue_N } \prod_{ \ell = 1 }^m \det\prth{ I - x_\ell U } \prod_{j = 1}^k \det\prth{ I -  y_j\inv U\inv } dU \\
		 & = \int_{\Ue_N}  \prod_{ j = 1 }^k \det\prth{ -y_\ell\inv U\inv} \prod_{ \ell = 1 }^m   \det\prth{ I - z_\ell U } \prod_{j = 1}^k \det\prth{ I -  y_j U } dU \\
		 & = \prod_{j = 1}^k (- y_\ell)^{-N} \int_{\Ue_N} \det(U)^{-k} \prod_{z \in X + Y} \det\prth{ I - z U } dU \\
		 & = \prod_{j = 1}^k (- y_\ell)^{-N} \frac{1}{N!}  \crochet{U^0} U^{-k} H\crochet{-U(X + Y)}  \Delta(U) \Delta(U\inv)  \mbox{ by \eqref{Eq:WeylHaarRealisationBis} } \\ 
		 & = \prod_{j = 1}^k (- y_\ell)^{-N} \frac{1}{N!}  \crochet{U^k}  H\crochet{-U(X + Y) - U^{(1 + \varepsilon)R} }    \\
		 & = \prod_{j = 1}^k (- y_\ell)^{-N} s_{k^N}\crochet{- (X + Y) }
\end{align*}
using \eqref{Trick:ProdScalSym} followed by \eqref{FourierRepr:Schur} for $ \lambda = k^N $. 

We then conclude with \eqref{Eq:InvolutionOmega}, i.e. $ s_\lambda\crochet{-\Ae} = (-1)^{\abs{\lambda}} s_{\lambda'}\crochet{\Ae} $ and $ (k^N)' = N^k $.
\end{proof}

% ==
\medskip

The plethystic perspective on this theorem allows to generalise it to abstract specialisations of $ H\crochet{ -U \Ae } $ such as the ones given in example~\ref{Ex:HfunctorSpecialise}, since it appears in the core of the proof. Consider for instance the alphabet $ \Eee $ defined in example~\ref{Ex:HfunctorSpecialise} and let $ \theta \in \Rr $. One has
\begin{align*}%$
s_{N^k}\crochet{ X + \theta \Eee} & = (-1)^{kN} s_{k^N}\crochet{ -(X + \theta \Eee) } \quad\mbox{by \eqref{Eq:InvolutionOmega}} \\
                & = \frac{1}{N!}  \crochet{U^k}  H\crochet{-U(X + \theta \Eee) - U^{(1 + \varepsilon)R} } \\
                & = \frac{1}{N!}  \crochet{U^0} U^{-k}  H\crochet{ -U \theta \Eee} H\crochet{-UX }  H\crochet{ - U^{(1 + \varepsilon)R} } \\
                & = \int_{\Ue_N} \det(U)^{-k} e^{ - \theta \trace(U) } \prod_{x \in X } \det\prth{ I - x U }  dU
\end{align*}
since $ H\crochet{-U \theta \Eee} = \prod_{u \in U} H\crochet{ - \theta u \Eee} = \prod_{u \in U} e^{ \sum_{\ell \geq 1} \ell\inv p_\ell\,\crochet{- u \theta \Eee} } = e^{ -\theta \trace(U) } $.

The exponential alphabet will not be of any use in this article (see e.g. \cite[(2.11), \S~2.4.2, \&c.]{BorodinPetrov} for applications in integrable probabilities) but we will see the advantage of such a general framework in \S~\ref{Subsec:Ratios} with supersymmetric specialisations.

\medskip
% ==
\subsection{A plethystic-RKHS perspective on duality}\label{Subsec:DualityWithRKHSandPlethysm}

\textit{Duality} (or \textit{involution}) is a general mathematical concept that expresses a transformation exchanging two quantities which, once again applied gives back the original configuration. Originally, the \textit{projective duality} was deduced by Poncelet from the observation that exchanging the words ``points'' and ``lines'' in classical theorems in projective geometry implied new results \cite[\S~6.1.2]{AndreLMCI}. The projective point of view was then outclassed by the linear one, a projective space $ P $ of dimension $n$ being the space of lines $ P(V) $ of a vector space $V$ of dimension $ n + 1 $ and its projective dual being $ P^* \simeq P(V^*) $ where $ V^* $ is the space of linear forms on $V$. A last layer of generalisation started then from the category $ Vec $ of vector spaces to any category having a contravariant involution (the modern point of view). For instance, the inversion $ g \mapsto g\inv $ can be seen as an involution in any group (the group being seen as a category with one object, see \cite[\S~6.2.2]{AndreLMCI}). Since such a map is also the antipode of the underlying group algebra, it is natural to investigate the duality induced by the fundamental involution $ \omega $ which is related to the antipode of $ \Lambdab $ as noted in remark~\ref{Rk:Hopf}. 

In the Bump-Gamburd proof of \eqref{Eq:SchurJointMoments}, one sees that for $ \Ae := X + Y $
\begin{align*}%$
s_{N^k}\crochet{ \Ae }  = \int_{\Ue_N} \det(U)^{-k}   \prod_{a \in \Ae } \det\prth{ I - a U }  dU = \Esp{ \det(U_N)^{-k} \prod_{a \in \Ae} Z_{U_N}(a) }
\end{align*}  
whereas the RKHS formula \eqref{FourierRepr:Schur}/\eqref{Eq:SchurWithRKHS} writes for $ \ell = k $ (in which case $ s_{N^k}(U) = \det(U)^N $)
\begin{align*}%$
s_{N^k}\crochet{ \Ae }  & = \bracket{H\crochet{X\Ae}, s_{N^k}\crochet{X} }_k = \int_{\Ue_k} \det(U)^{-N}   \prod_{a \in \Ae } \det\prth{ I + a U }\inv  dU \\
                &  = \Esp{ \det(U_k)^{-N} \prod_{a \in \Ae} Z_{U_N}(-a)\inv }
\end{align*}

We thus have the formula
\begin{align}\label{Eq:DualityCUECharpol}%$
\Esp{ \det(U_N)^{-k} \prod_{j = 1}^\ell Z_{U_N}(-x_j) } = \Esp{  \det(U_k)^{-N} \prod_{j = 1}^\ell Z_{U_k}(x_j)\inv }  
\end{align}
that exchanges $ N $ and $k$ and that can be written in a more algebraic way as
\begin{align}\label{Eq:InvolutionOmegaOnSchur}%$
s_{k^N}\crochet{\widehat{\omega} X} = s_{N^k}\crochet{X},
\end{align}
a property that was already remarked in \S~\ref{Subsec:FuncSym} with \cite[I-3, (3.8)]{MacDo}.

The formula \eqref{FourierRepr:Schur} equivalent to \eqref{Eq:InvolutionOmegaOnSchur} is thus a duality formula in the categorical sense (more precisely in the category of graded algebras by \eqref{Eq:InvolutionOmegaOnSchur} since $ \Lambdab $ is such an algebra). Nevertheless, to pass from \eqref{Eq:InvolutionOmegaOnSchur} to \eqref{Eq:DualityCUECharpol}, one has applied the Riesz representation theorem for the evaluation%
\footnote{The canonical evaluation $ r_a : f \mapsto f(a) $ is continuous in an RKHS $ (H, \bracket{\cdot, \cdot}_H) $ and by the Riesz representation theorem, $ \exists !\, R_a \in H $ s.t. $ r_a = \bracket{R_a, \cdot}_H $. The Riesz representant $ R_a $ is said to be a reproducing kernel by the property $ R_a(b) = r_b(R_a) = \bracket{R_a, R_b}_H =: R(a, b) $ a particular case of which is given in \eqref{Trick:Hreproducivity}.}
in the $ L^2 $ space of the relevant group ($ \Ue_N $ or $ \Ue_k $) with the reproducing kernel structure. This is thus the conjunction of the $ \lambda $-ring structure and the RKHS structure that produces the duality.

\medskip

Such dualities \textit{\`a la Kontsevitch} \cite{Kontsevich1992} exchanging $ N $ and $k$ integrals are of obvious use in tackling asymptotics issues as it bypasses the usual pathology of probability theory (the growing number of integrals)~; as a result, it has been extensively studied in probability theory and mathematical physics, and in particular in random matrix models \cite[\&c.]{BrezinHikamiDuality, ForresterRainsFuchsian, ImamuraSasamotoGOE, Kontsevich1992, KillipRyckman, MorozovRMTsystint, TribeZaboronskiPfaffianCoalescing}. For instance, for $ GUE $ random matrices, one has \cite[(2.57)]{MorozovRMTsystint}
\begin{align}\label{Eq:DualityGUEKontsevich}%$ 
\int_{ \He_N } \det\prth{ \Lambda_k \otimes I_N - I_k \otimes H }  d\!\GUE_N(H)  = \int_{ \He_k } \crochet{ \det\prth{ H + \Lambda_k} }^N d\!\GUE_k(H)
\end{align}
where $ \He_k $ is the space of Hermitian matrices, $ \Lambda_k $ is a $ k \times k $ diagonal matrix, $ d\GUE_k(H) = e^{- \frac{1}{2} \trace(H H^*) } \frac{dH }{ (2\pi)^{k(k + 1)/4} }$ and $ dH $ is the Lebesgue measure on this space (note also the deformed $ GUE $ with external source in e.g. \cite[(5), (6)]{BrezinHikamiDuality}). 

The general concept of duality is nowadays coined to designate a general exchange between a set of parameters in an equality of expectations. In the case of Hermitian matrix models and their chiral generalisations (analogous to $ \beta $ circular matrix models), the breakthrough by Desrosiers \cite{DesrosiersDuality} sumarises all previous attempts to derive duality results. The general exchange of parameters is given in  \cite[(1.6)]{DesrosiersDuality} and shows that the $ GUE $ case is self-dual (since $ \beta $ is replaced by $ 4/\beta $ in general $ \beta $-ensembles). This is also the case of the $ CUE $, as shown by Matsumoto \cite{MatsumotoCbetaE}\footnote{Even if the paper was ``withdrawn by the author due to a crucial error'' for general $ \beta $, the case $ \beta = 2 $ remains valid and recovers the duality result of the present paper.}.

% ==

Apart from exact formulas resulting from ad-hoc manipulations, methods to prove duality have been so far of two types that one could call, in a probabilistic language, Gaussian and Markovian. Gaussian processes are the probabilistic counterpart of the theory of RKHS~: every (centered) Gaussian process $ (X_z)_z $ is uniquely defined (up to a modification) by its covariance $ C_Z : (u, v) \mapsto \Esp{X_u \overline{X_v}} $ which is a positive definite function, or equivalently by the associated kernel operator, and a famous theorem of Moore-Aronszajn associates in a unique way an RKHS to its kernel~; a Gaussian process is thus uniquely equivalent to an RKHS. Similarly to a Gaussian process, a Markovian process $ (X_t)_t $ is defined in a unique way by an operator $ \Le $ called the \textit{infinitesimal generator} (and by the domain on which it acts). Methods to prove duality by means of an operator are philosophically of the same type and only differ in practice from the nature of the considered operators~: reproducing kernel/covariance operators are integral operators whereas infinitesimal generators are often differential or integro-differential operators. In fact, since the semi-group $ (e^{t\Le})_{t \geq 0} $ of a differential operator $ \Le $ is an integral operator, RKHS methods can become equivalent to Markovian methods such as in \cite[(4.1)]{DesrosiersDuality} (it is nevertheless considered simpler to apply a differential operator than an integral one).

This operator method to find duality is well-described by Desrosiers in \cite[\S~4.1]{DesrosiersDuality}. It consists in finding a good function of two variables on which the semi-group/the covariance operator acts identically on each different variable. Finding such a duality function is the key point in the method, and although some general abstract construction are available in terms of the spectral decomposition of the operator, a ``practical'' function is often to be found with an ansatz, assuming a preferential form (see e.g. \cite[prop. 7]{DesrosiersDuality}). In the setting of Markov processes, it is implied by the intertwinning of two generators \cite{RogersPitman}.

Duality methods were successfully applied in numerous domains of probability theory and mathematical physics, which turns the task of reviewing their use into an impossibility~; we nevertheless refer the reader to \cite{AssiotisIntertwining, BorodinCorwinSasamoto, CarmonaPetitYor, DesrosiersLiu2011, DiaconisFill, GiardinaKurchanRedigVafayi, KarlinMcGregor, KuanUqAn, KuanAsep, KuanMarkovSchurWeyl, Liggett, RogersPitman, SchutzDuality, Spitzer} and references cited for a (very restricted) selection in domains ranging from coalescence-fragmentation to integrable probabilities.

\medskip
% ==
\subsection{CFKRS with RKHS}\label{Subsec:Theory:CFKRSwithRKHS}

The CFKRS formula \eqref{Eq:SchurCFKRS} also starts from \eqref{Eq:SchurJointMoments} with $ s_{N^m}\crochet{X + Y} $ given in \cite[(2.11)]{ConreyFarmerKeatingRubinsteinSnaith}~; here $ X := \ensemble{x_1, \dots, x_m} $ and $ Y := \ensemble{y_1, \dots, y_k} $ (hence $ k = n - m $ in \cite[(2.11)]{ConreyFarmerKeatingRubinsteinSnaith}). It then proceeds with manipulations that are ``non standard'' in the theory of symmetric functions, ending up with a sum over the set $ \Xi_m $ of $ {n \choose m} $ permutations that differentiates between $ X $ and $ Y $, by ordering the $ m $ first variables and the $ k $ last variables. With the notations of \S~\ref{Sec:Notations}, \cite[(2.16)]{ConreyFarmerKeatingRubinsteinSnaith} writes
\begin{align*}%$
s_{N^m } \crochet{X + Y} = \! \sum_{\sigma \in \Xi_m} \! \sigma \cdot \prth{ Y^N H\crochet{X Y\inv} } \quad \Longleftrightarrow \quad s_{N^m}\!\crochet{e^X + e^Y} = \!\sum_{\sigma \in \Xi_m} \! \sigma \cdot \prth{ e^{N p_1(Y)} H\crochet{e^X e^{\varepsilon Y} } }
\end{align*}

Now, applying the RKHS formula for symmetric functions with the Cauchy kernel in $ n = m + k $ variables given in \eqref{Eq:ReproducingKernelProperty}, one gets with $ U := \ensemble{u_1, \dots, u_n} = U_- + U_+ $ and $ U_- := \ensemble{u_1, \dots, u_n} $, 
\begin{align*}%$
s_{N^m}\!\crochet{e^X + e^Y} & = \bracket{  \sum_{\sigma \in \Xi_m} \! \sigma \cdot \prth{ e^{N p_1(U_+)} H\crochet{e^{U_-} e^{\varepsilon U_+} } } , H\crochet{ (X + Y) U }  }_{\!\!\! U \in \Ue_n} \\
                 & = \abs{\Xi_m} \bracket{   e^{N p_1(U_+)} H\crochet{e^{U_-} e^{\varepsilon U_+} }  , H\crochet{ (X + Y) U }  }_{\! U \in \Ue_n} \\
                 & = {n \choose m} \frac{1}{n!} \oint_{\Uu^n}  e^{N p_1(U_+)} H\crochet{e^{U_-} e^{\varepsilon U_+} } H\crochet{ (X + Y) U\inv } \abs{\Delta(U)}^2 \frac{d^*U}{U}
\end{align*}
which is exactly \cite[(2.17)]{ConreyFarmerKeatingRubinsteinSnaith} given in \eqref{Eq:SchurCFKRS} after easy manipulations on the Vandermonde determinant and the Cauchy product.

\medskip

The CFKRS formula \eqref{Eq:SchurCFKRS} is thus a particular evaluation formula in an RKHS. It perfectly fits into the framework of the theory presented here and only differs from \eqref{Eq:SchurWithRKHS} by the ``non standard'' manipulations that aim at separating $ X $ and $ Y $. Such a manipulation, although relevant from the point of view of the analogy with the $ \zeta $ function, has several drawbacks~: the number of integrals is $ n = k + m $ in place of $ m $ (the minimal required amount of possible integrals to apply the RKHS property), the separation of variables seems artificial with regards to the original symmetry, and one does not easily see the scaling property of the Schur function. More importantly, since the ``non standard manipulations'' are only valid for ``classical'' alphabets, one cannot extend the formalism for abstract alphabets (see e.g. the discrepancy with the supersymmetric case of \cite{ConreyFarmerZirnbauer} given in \cite[Thm 1.3]{BaileyBettinBlowerConreyProkhorovRubinsteinSnaith}), and due to the exponential term $ e^{N p_1(U_+)} $ that is not polynomial, one cannot replace $ H\crochet{ (X + Y) U\inv } $ by $ h_{Nn}\crochet{ (X + Y) U\inv } $. This last drawback is at the core of the additional manipulations described in \S~\ref{Subsec:Intro:ComparisonLiterature} when extracting any asymptotic behaviour.

\medskip

% ==
\subsection{Duality in $ \He_N $}\label{Subsec:Theory:DualityHermitian}

There are three ways to exchange integrals over $ \Ue_N $ and $ \He_N $~:
\begin{itemize}

\item the \textit{exponential map} $ H \in \He_N \mapsto e^{i H} \in \Ue_N $ gives a measure studied in \cite{FyodorovKhoruzhenkoSimm} in link with the fractional Brownian motion with Hurst index equal to $0$~; the Jacobian of the ``diagonalisation map'' was computed in \cite[(4)]{MorozovCUE} and gives an interesting measure proportional to $ \prod_{1 \leq i < j \leq N} \sinc(\tfrac{h_i - h_j}{2})^2 d\hb $ for the law of the eigenvalues,

\item the \textit{stereographic projection} or \textit{Cayley map} $ \Ce : U \in \Ue_N \mapsto (1 - U)(1 + U)\inv = -I_N + 2 (I_N - U)\inv \in \He_N $ (see e.g. \cite[ch. 3-1]{HuaHarmonic}) maps the $ CUE_N $ to the Cauchy Ensemble on $ \He_N $ whose measure is given by $ d\Pp_{\operatorname{Cauchy}, N}(H) := \det(I_N + H^2)^{-N} \frac{dH}{Z_N} $ where $ Z_N $ is an explicit constant, 

\item the \textit{Askey trick} (see e.g. \cite[(1-15)]{ForresterOleWarnaar}) is the following identity
\begin{align*}%$
\oint_{\Uu^N} Z^\zeta f(-Z) \frac{d^*Z}{Z} = \prth{ \frac{\sin(\pi\zeta)}{\pi} }^N \int_{\crochet{0, 1}^N } X^\zeta f(X) \frac{dX}{X}
\end{align*}
where $ Z := \ensemble{z_1, \dots, z_N} $, $ X := \ensemble{x_1, \dots, x_N} $, $f$ is a Laurent series and $ \zeta \notin \Zz $ is such that the RHS (hence the LHS) exists. The limitation $ \zeta \notin \Zz $ is critical to define the LHS, and one cannot express Fourier coefficients in such a way. It was for instance used in \cite[(1.29)]{ForresterWitteTau2} to relate the $ CUE $ with the $ JUE $.
\end{itemize}

These links were thoroughly exploited in numerous articles, and in particular the Cayley transform (see e.g. \cite{BourgadeNajnudelNikeghbali, Neretin} for a similar construction in both spaces, with similarities and differences). It was notably used by Winn in his study of the derivative of $ Z_{U_N} $ \cite[proof of prop. 3.4]{Winn} using the change of variables $ x_j = \cot(\theta_j/2) $ in the Weyl integration formula \eqref{Eq:WeylHaarRealisationBis}. Note indeed that\footnote{These manipulations are equivalent to \cite[between (4.7) and (4.9)]{Winn} in $ X $-coordinates.} (see e.g. \cite[\S~3.1]{SantilliTierz} or \cite[(1)]{MorozovCUE})
\begin{align*}%$
\Delta(e^{i\thetab}) & = \prod_{1 \leq k < \ell \leq N} (e^{i\theta_k} - e^{i\theta_\ell}) = \prod_{1 \leq k < \ell \leq N}  e^{i(\theta_k +\theta_\ell)/2} (e^{i(\theta_k - \theta_\ell)/2} - e^{i(\theta_\ell - \theta_k)/2} ) \\
              & = e^{i\frac{N - 1}{2} \sum_j \theta_j } \prod_{1 \leq k < \ell \leq N}  2i \sin(\theta_k/2 - \theta_\ell/2) \\
              & = (2i)^{N(N - 1)/2} e^{i\frac{N - 1}{2} \sum_j \theta_j } \prod_{1 \leq k < \ell \leq N} \prth{ \sin(\theta_k/2)\cos(\theta_\ell/2) - \sin(\theta_\ell/2)\cos(\theta_k/2) } \\
              & = (2i)^{N(N - 1)/2} e^{i\frac{N - 1}{2} \sum_j \theta_j } \prod_{1 \leq k < \ell \leq N} \sin(\theta_k/2) \sin(\theta_\ell/2) \prth{ \cot(\theta_k/2) - \cot(\theta_\ell/2)  } \\
              & = (2i)^{N(N - 1)/2} e^{i\frac{N - 1}{2} \sum_j \theta_j } \prod_k \sin(\theta_k/2)^{(N - 1)/2} \Delta(\cot(\thetab/2))
\end{align*}
and $ e^{i \mathrm{Arc}\cot(x)} = \frac{i + x}{\sqrt{1 + x^2}} = \frac{i + x}{\abs{i + x}} $. Since $ d \mathrm{Arc}\cot(x) = - \frac{dx}{1 + x^2}  $, we get for an integrable class function $f$,
\begin{align*}%$
\int_{\Ue_N} f(U) dU & = \frac{1}{N!} \int_{[-\pi, \pi]^N} f(e^{i\theta_1}, \dots, e^{i \theta_N}) \abs{\Delta( e^{i \thetab})}^2 \frac{d\thetab}{(2\pi)^N} \mbox{ by \eqref{Eq:WeylHaarRealisationBis}} \\
                   & = \frac{2^{N(N - 1)}}{(2\pi)^N \, N!} \int_{\Rr^N} f\prth{ \frac{(i + x_1)^2}{ 1 + x_1^2 }, \cdots, \frac{(i + x_N)^2}{ 1 + x_N^2 } } \Delta(\xb)^2 \prod_{j = 1}^N  \frac{1}{(1 + x_j^2)^{(N - 1)} }  \frac{dx_j}{1 + x_j^2 }   
\end{align*}
which gives the Cauchy measure proportional to $ \det(I + H^2)^{-N} $ after the usual Jacobian computation in $ \He_N $. Note that the characteristic polynomial in $ 1 $ corresponds to $ f(X) = \prod_j (1 - x_j) $, hence is particularly adapted to the transform. So is $ Z_{U_N}'/Z_{U_N} $ that writes in terms of $ \cot(\theta_i/2) $. It seems nevertheless more complicated to look at the multipoint setting with this method, or in any other point than $ 1 $, hence, the duality method that we propose in \S~\ref{Subsec:Derivatives} seems more general from this perspective. 

A duality for the Cauchy measure in the same vein as \eqref{Eq:DualityGUEKontsevich} for the $ GUE $ was given in \cite[(4.23)]{Winn} using the $ LUE $ instead and a result of Br\'ezin-Hikami \cite[(15)]{BrezinHikamiCharpol}. Such a result constitutes a duality formula that is proven without any symmetric function theory but which is equivalent to \eqref{Eq:DualityCUECharpol} and more precisely to \eqref{Eq:AutocorrelationsAsAQuotientOfDets} (the method is classical and used e.g. in \cite{AkemannVernizzi, ConreyForresterSnaith} ; it is equivalent to the computation of the correlation functions of a determinantal measure by successive integrations). This last formula was re-written as an integral in \cite[(21), (34)]{BrezinHikamiCharpol}, and one can guess with \eqref{AlternantRepr:Schur} that it hides an instance of the Schur function with (the numerator of) the alternant present in \cite[(21)]{BrezinHikamiCharpol}. This explains why this formula ``looks like'' \eqref{Eq:SchurCFKRS} as stated in \S~\ref{Subsec:Intro:ComparisonLiterature}.

$ $

% ==
\section{Applications}\label{Sec:Applications}

We now give new proofs of the results described in section~\ref{Subsec:AllProblems}. The notations and functions $ h_{c, \infty}^{(\kappa)} $ and $ \widetilde{h}_{c, \infty}^{(\kappa)}$ that we use are given in annex~\ref{Sec:ProbReprGegenbauer}.

\medskip
% ==
\subsection{The Keating-Snaith theorem}\label{Subsec:KS} 

The celebrated Keating-Snaith theorem \cite[(15), (16)]{KeatingSnaith} is enunciated in theorem~\ref{Thm:KSwithDuality} and in \eqref{Eq:KSmomentsCharpol}.
%
%\begin{theorem}[Keating and Snaith] For all $ k \geq 1 $, one has
%%
%%
%%
%\begin{align}\label{Eq:KeatingSnaithMomentsCUE}%$
%\frac{ \Esp{ \abs{Z_{U_N}(1) }^{2k} } }{ N^{k^2 } } \tendvers{N}{+\infty} \beta_k 
%\end{align}
%%
%%
%%
%where $ \beta_k $ is the random matrix factor defined in \eqref{Def:MatrixFactor}.
%%
%\end{theorem}
%
Although a particular case of several results defined in section~\ref{Subsec:AllProblems}, the importance of this theorem and the conjecture about the moments of $ \zeta(1/2 + i T U) $ for $ U \sim \Us([0, 1]) $ that followed justifies a particular treatment. The machinery to prove it based on the previous considerations will moreover be the prototype of all following proofs. For the reader's convenience, we recall here its statement~:

\begin{theorem}[Value of the characteristic polynomial in 1]\label{Thm:KSwithDualityBis} 
For all $ k \geq 1 $, one has 
\begin{align}\label{Eq:KSCUEwithDuality2}%$
\frac{ \Esp{ \abs{Z_{U_N}(1) }^{2k} } }{ N^{k^2 } } \tendvers{N}{+\infty} \widetilde{L}_1(k)
\end{align}
where
\begin{align}\label{EqPhi:KS2}%$
\begin{aligned}
\widetilde{L}_1(k) & := \frac{(2\pi)^{k(k - 1) } }{k!} \int_{\Rr^{k - 1} } \Phi_{\widetilde{L}_1}(0, x_2, \dots, x_k) \Delta(0, x_2, \dots, x_k)^2 dx_2 \dots dx_k \\
\Phi_{\widetilde{L}_1}(0, x_2, \dots, x_k) & :=  e^{ 2 i \pi (k^2 - 1) \sum_{ j = 2 }^k x_j } \, h_{k, \infty}^{(2k) }(0, x_2, \dots, x_k)
\end{aligned} 
\end{align} 
\end{theorem}

% ==

\begin{proof}
Using \eqref{Eq:SchurJointMoments}, we have
\begin{align*}%$
\Esp{ \abs{Z_{U_N}(1) }^{2k} } & := \int_{ \Ue_N } \abs{   \det(I - U)^k }^2 dU \\
            & = \int_{ \Ue_N }  \det(I - U)^k \det(I - U\inv)^k  dU \\
            & = (-1)^{kN} \int_{ \Ue_N } \det(U)^{-k} \det(I - U)^{2k} dU = (-1)^{kN} s_{k^N}\crochet{-1^{\plusInOne 2k} } \\
            &  = s_{N^k}\crochet{ 1^{\plusInOne 2k} }
\end{align*}

One then has with \eqref{FourierRepr:SchurRectangle} 
\begin{align*}%$
s_{N^k}\crochet{ 1^{\plusInOne 2k} } & = \frac{1}{k!} \oint_{ \Uu^k } U^{-N} H\crochet{ U  1^{\plusInOne 2k} }  \abs{\Delta(U) }^2 \frac{d^*U}{U}, \qquad U := \ensemble{ u_1, \dots, u_k } \\ 
             & = \frac{1}{k!} \oint_{ \Uu^k } U^{-N} h_{Nk}\crochet{ 1^{\plusInOne 2k} U}   \abs{\Delta(U) }^2 \frac{d^*U}{U} \qquad \mbox{ with \eqref{FourierRepr:SchurRectangleWithHn}. } \\
              & =: \frac{1}{k!} \oint_{ \Uu^k } U^{-N} h_{Nk}^{(2k)}( U )  \abs{\Delta(U) }^2 \frac{d^*U}{U} \\
              & =  \frac{1}{k!} \oint_{ \Uu^k } u_1^{ -Nk } \prod_{j = 2}^k \prth{ \frac{u_j}{u_1} }^{-N} \times u_1^{Nk} \, h_{Nk}^{(2k)} \prth{1, \frac{u_2}{u_1}, \dots, \frac{u_k}{u_1 } }  \abs{ \Delta  \prth{1, \frac{u_2}{u_1}, \dots, \frac{u_k}{u_1 } } }^2 \frac{d^*U}{U} \\
              & = \frac{1}{k!} \oint_{ \Uu^{k - 1} }  \prod_{j = 2}^k v_j^{-N} \times h_{Nk}^{(2k)} \prth{1, v_2 , \dots, v_k }  \abs{ \Delta  \prth{1, v_2, \dots, v_k } }^2 \frac{d^*V}{V}, \quad V := \ensemble{v_2, \dots, v_k}
\end{align*}

Here, we have used the trick of remark~\ref{Rk:MultivariateFourierWithHomogeneity}, i.e. we have set $ u_1 = e^{2 i \pi \phi} $, $ v_j := u_j/u_1 = e^{2 i \pi \theta_j} $ with $ \theta_j \in \crochet{ - \frac{1}{2}, \frac{1}{2} } $ for all $ j \in \intcrochet{2, k} $ and then integrated in $ u_1 $, making this variable disappear. Note that this all comes from the homogeneity of $ h_m $, i.e. $ h_m\crochet{\lambda \Ae} = \lambda^m h_m\crochet{\Ae} $ (and the homogeneity of $ \abs{\Delta(U)}^2 $) ; this homogeneity is of course valid for any abstract alphabet $ \Ae $ ; there is no such property for $ H\crochet{\lambda \Ae} $, even if the two integrals are equal, hence, the ``hyperplane concentration'' is more difficult to highlight with $ \oint_{ \Uu^k } U^{-N} H\crochet{ U  1^{\plusInOne 2k} }  \abs{\Delta(U) }^2 \frac{d^*U}{U} $ (see also remark~\ref{Rk:DegreeFreedomReprHn}).

Setting $ x_j := \theta_j N $, one now gets
\begin{align*}%$
\Esp{ \abs{Z_{U_N}(1) }^{2k} } & =  \frac{1}{k!} \int_{ \crochet{ - \frac{N}{2}, \frac{N}{2} }^{k - 1} } e^{ - 2 i \pi \sum_{j = 2}^k x_j } \, h^{(2k)}_{ k N }\prth{1, e^{2 i \pi x_2/N} , \dots, e^{2 i \pi x_k/N} }   \\
                & \hspace{+4.5cm} \times  \abs{\Delta\prth{ 1, e^{2 i \pi x_2/N} , \dots, e^{2 i \pi x_k/N} } }^2 \frac{dx_2 \dots dx_k }{N^{k - 1} }  \\
% ======================================                
                & \equivalent{N\to +\infty } \frac{ N^{ (2k^2 - 1) - k(k-1) - (k - 1)}   }{k!} \int_{ \Rr^{k - 1} } \!\! e^{ - 2 i \pi \sum_{j = 2}^k x_j } \, h^{(2k)}_{ k , \infty }\prth{0, x_2, \dots, x_k } e^{  i \pi 2k\cdot k \sum_{j = 2}^k x_j }   \\
              & \hspace{+6cm} \times  (2\pi)^{ k(k - 1) }   \abs{\Delta\prth{ 0, x_2, \dots, x_k } }^2  dx_2 \dots dx_k \\
              & = N^{k^2} \frac{ (2\pi)^{ k(k - 1) }}{k!} \int_{ \Rr^{k - 1} } \!\! e^{ 2 i \pi (k^2 - 1) \sum_{ j = 2 }^k x_j } \, h_{k, \infty}^{(2k) }(0, x_2, \dots, x_k) \Delta\prth{ 0, x_2, \dots, x_k }^2  \\
              & \hspace{+11.5cm} dx_2 \dots dx_k
\end{align*}

The only step to justify is the limit that gives the equivalent. As stated in the introduction, this is a local limit theorem, hence, we use the dominated convergence given by inequality \eqref{Ineq:SpeedOfConvergenceHinftyBis}. The limiting function is seen integrable using the first criteria \eqref{Ineq:IntegrabilityDomination} with $ (K, M, \kappa) = (k - 1, 1, 2k ) $ which is clearly fullfilled for all $ p \in \ensemble{0, 1, 2} $ and $ k \geq 2 $.
\end{proof}

\medskip
% ==
\subsection{Autocorrelations of the characteristic polynomial in the microscopic setting}\label{Subsec:Autocorrelations} 

In the same vein as the moments of the characteristic polynomial in $1$, one can ask for the joint moments in different points. In the particular case where the points are separated by a distance of order $ 1/N $ for a $ CUE(N) $ matrix, one gets a computation achieved by Killip and Ryckman in the more general case of a $ \beta $-ensemble \cite{KillipRyckman}. Their result uses an OPUC machinery and writes \cite[thm. 4.1]{KillipRyckman} for $ \beta = 2 $ and for all $ x_1, \dots, x_k, y_1, \dots, y_m \in \Rr^{m + k} $
\begin{align*}%$
\frac{1}{N^{km} } \Esp{ \prod_{j = 1}^k Z_{U_N}\prth{e^{ 2 i \pi x_j/N} } \prod_{\ell = 1}^m \overline{ Z_{U_N}\prth{e^{ 2 i \pi y_\ell/N} }}  } \tendvers{N}{+\infty} \Psi(x_1, \dots, x_k, y_1, \dots, y_m)
\end{align*}
where $ \Psi $ is given as the solution of a particular PDE. 

In the particular case of $ \beta = 2 $, we have in fact a more explicit result \cite{AkemannVernizzi, FyodorovStrahovCorrChiralGUE} with $ Z := X + Y \subset \Cc $ if $ X := \ensemble{x_j, 1 \leq j \leq k} $ and $ Y := \ensemble{y_\ell, 1\leq \ell \leq m} $, for $ k = m $ and with $ K_N(z_1, z_2) := \sum_{j = 1}^N (z_1 \overline{z_2})^j $:
\begin{align}\label{Eq:AutocorrelationsAsAQuotientOfDets}%$
\Esp{ \prod_{j = 1}^k Z_{U_N}\prth{e^{ 2 i \pi x_j/N} } \overline{ Z_{U_N}\prth{e^{ 2 i \pi y_j/N} }}  } & = \frac{\det\prth{ K_{N + k}\prth{ e^{2 i \pi \frac{x_j}{N}}, e^{2 i \pi \frac{\overline{y_\ell}}{N}} } }_{1 \leq j, \ell \leq k} }{ \Delta\prth{ e^{2 i \pi \frac{x_1}{N} } , \dots, e^{2 i \pi \frac{x_k}{N} } } \overline{ \Delta\prth{ e^{  2 i \pi \frac{y_1}{N} } , \dots, e^{2 i \pi \frac{y_k}{N} } } } } \\
              & \equivalent{N\to +\infty} \frac{N^{k^2}}{(2\pi)^{k(k - 1)} } \frac{\det\prth{ \sinc\prth{ 2\pi(x_j - \overline{y_\ell}) } }_{1 \leq j, \ell \leq k} }{ \Delta\prth{ x_1, \dots, x_k } \overline{ \Delta\prth{ y_1, \dots, y_k } } } \notag
\end{align}

% ==

We now show that this last convergence can be proven with the previous machinery:

\begin{theorem}[Microscopic autocorrelations of the characteristic polynomial]\label{Theorem:Autocorrelations}
We have, locally uniformly in $ X + Y \in \Cc^{k + m} $ for all $ k, m \geq 2 $
\begin{align*}%$
\frac{1}{N^{km} } \Esp{ \prod_{j = 1}^k Z_{U_N}\prth{e^{ 2 i \pi x_j/N} } \prod_{\ell = 1}^m \overline{ Z_{U_N}\prth{e^{- 2 i \pi y_\ell/N} }}  } \tendvers{N}{+\infty} \Ae(X, Y)
\end{align*}
with 
\begin{align}\label{EqPhi:Autocorrels}%$
\begin{aligned}
\Ae(X, Y) & := \frac{(2\pi)^{m(m - 1) } }{m!} \int_{ \Rr^{ m - 1 } } \Phi_{\Ae(X, Y)}(0, t_2, \dots, t_m) \Delta(0, t_2, \dots t_m)^2 dt_2 \dots dt_m \\
\Phi_{\Ae(X, Y)}(0, t_2, \dots, t_m) & = e^{-2 i \pi \sum_\ell \overline{ y_\ell} + i \pi \,  \prth{ 1 + \sum_{j = 2}^m t_j }\prth{ m (\sum_{r = 1}^k x_r - \sum_{\ell = 1}^m  \overline{y_\ell} ) -2 } } h_{m, \infty}\crochet{ (0 + T)\tensorsum(X + \overline{Y} ) }  
\end{aligned}
\end{align}

Here, we have used the $ \lambda $-ring/plethystic formalism to write $(0 + T)\tensorsum(X + \overline{Y} ) = X + \overline{Y} + T \tensorsum X + T \tensorsum Y$ which is an alphabet in $ m(k + m) $ variables and 
\begin{align*}%$
f\crochet{ (0 + T)\tensorsum(X + \overline{Y} ) } & = f(x_1, \dots, x_k, y_1, \dots, y_m, t_1 + x_1, \dots, t_1 + x_k, t_2 + x_1, \dots,  \\
                 & \hspace{+7cm} t_m + x_k, t_1 + y_1, \dots, t_m + y_m)
\end{align*}
\end{theorem}

% ==

\begin{proof}
We have with $ e^{2 i \pi \frac{Z}{N}} := \ensemble{e^{2 i \pi \frac{x_1}{N}}, \dots, e^{2 i \pi \frac{x_k}{N}}, e^{2 i \pi \frac{\overline{y_1}}{N}}, \dots, e^{2 i \pi \frac{\overline{y_m}}{N}}} $ %and $ U := (u_1, \dots, u_m) $
\begin{align*}%$
\Ae_N(X, Y) & := \Esp{ \prod_{j = 1}^k Z_{U_N}\prth{e^{ 2 i \pi x_j/N} } \prod_{\ell = 1}^m \overline{ Z_{U_N}\prth{e^{ -2 i \pi y_\ell/N} }}  } \\
                 & = e^{-2 i \pi \sum_\ell \overline{ y_\ell} } \, \Esp{ \det(U_N)^{-m} \prod_{j = 1}^{k + m} Z_{U_N}\prth{e^{ 2 i \pi z_j/N} }  } \\
                 & =  e^{-2 i \pi \sum_\ell \overline{ y_\ell} } \, s_{N^m }\crochet{ e^{2 i \pi \frac{Z}{N}} } \qquad \mbox{with \eqref{Eq:SchurJointMoments}}\\ 
                 & = e^{-2 i \pi \sum_\ell \overline{ y_\ell} } \, \crochet{U^N} H\crochet{ U e^{2 i \pi \frac{Z}{N}} - U^{\varepsilon R} }, \qquad U := \ensemble{u_1, \dots, u_m}, \mbox{ using \eqref{FourierRepr:Schur} } \\
                 & = \frac{e^{-2 i \pi \sum_\ell \overline{ y_\ell} }}{m!} r^{-mN} \oint_{(r\Uu)^m} U^{-N} H\crochet{ U e^{2 i \pi \frac{Z}{N}} } \abs{\Delta(U)}^2 \frac{d^* U}{U}, \quad r < \min_{t \in e^{2 i \pi Z/N} }\!\!\ensemble{\abs{t}\inv } \\
                 & = \frac{e^{-2 i \pi \sum_\ell \overline{ y_\ell} }}{m!} \oint_{\Uu^m} U^{-N} h_{Nm}\crochet{ U e^{2 i \pi \frac{Z}{N}} } \abs{\Delta(U)}^2 \frac{d^* U}{U} \qquad \mbox{ with \eqref{FourierRepr:SchurRectangleWithHn}. }
\end{align*}

The homogeneity of $ h_m $ gives for all alphabet $\Be$
\begin{align*}%$
h_{Nm}\crochet{ U \Be } = h_{Nm}\crochet{ u_1 \prth{1 + \frac{u_2}{u_1} + \dots + \frac{u_k}{u_1} }\Be} = u_1^{Nm}  h_{Nm}\crochet{  \prth{1 + \frac{u_2}{u_1} + \dots + \frac{u_k}{u_1} } \Be }
\end{align*}
thus, one can write
\begin{align*}%$
\Ae_N(X, Y) & = \frac{e^{-2 i \pi \sum_\ell \overline{ y_\ell} }}{m!} \oint_{\Uu^m} \prod_{j = 2}^m \prth{\frac{u_j}{u_1}}^{-N} h_{Nm}\crochet{    \prth{1 + \frac{u_2}{u_1} + \dots + \frac{u_k}{u_1} } e^{2 i \pi \frac{Z}{N}} } \\
                & \hspace{+5cm} \abs{\Delta\prth{1 , \frac{u_2}{u_1} , \dots , \frac{u_k}{u_1} }}^2 \frac{d^* U}{U}
\end{align*}

Set $ u_j/ u_1 = e^{2 i \pi t_j /N } $ for $ j \in \intcrochet{2, N} $ and $ t_j \in \crochet{-\frac{N}{2}, \frac{N}{2} } $ and $ u_1 = e^{2 i \pi \theta} $ with $ \theta\in \crochet{-\frac{1}{2}, \frac{1}{2}} $. Using remark~\ref{Rk:MultivariateFourierWithHomogeneity}, one has with $ T := \ensemble{t_j, j \in \intcrochet{2, m} } $ and $ e^{2 i \pi T/N } := \ensemble{e^{2 i \pi t_j/N} , j \in \intcrochet{2, m} } $  
\begin{align*}%$
\Ae_N(X, Y) & = \frac{e^{-2 i \pi \sum_\ell \overline{ y_\ell} }}{m!} \int_{ \crochet{ - \frac{N}{2}, \frac{N}{2} }^{m - 1} } e^{ -2 i \pi\sum_{j = 2}^m t_j } h_{Nm}\crochet{     e^{2 i \pi \frac{ (0 + T)\tensorsum Z}{N}}  }    \abs{\Delta\crochet{1 + e^{2 i \pi \frac{T}{N} } }}^2 \frac{d T}{N^{m - 1}}  \\
             & \equivalent{N \to +\infty} N^{m(m + k) - 1  - m(m - 1) - (m - 1) } e^{-2 i \pi \sum_\ell \overline{ y_\ell} } \,  \frac{ (2\pi)^{m(m - 1)} }{m!}  \\
             & \hspace{+2cm} \times \int_{\Rr^{m - 1} } e^{-2 i \pi \sum_{j = 2}^m t_j }  e^{ i \pi m \, \prth{ 1 + \sum_{j = 2}^m t_j }\prth{   \sum_{r = 1}^k   x_r - \sum_{\ell = 1}^m  \overline{y_\ell} } }  \\
             & \hspace{+3cm} \times h_{m, \infty}\crochet{ Z + T \tensorsum Z } \Delta(0, t_2, \dots, t_m)^2 dT \\
             & = N^{km } e^{-2 i \pi \sum_\ell \overline{ y_\ell} } \,  \frac{ (2\pi)^{m(m - 1)} }{m!} \int_{\Rr^{m - 1} } e^{ i \pi \,  \prth{ 1 + \sum_{j = 2}^m t_j }\prth{ -2 +   m \sum_{r = 1}^k   x_r - m \sum_{\ell = 1}^m  \overline{ y_\ell} } } \\
             & \hspace{+6.5cm} \times h_{m, \infty}\crochet{ Z + T \tensorsum Z} \Delta(0, t_2, \dots, t_m)^2 dT
\end{align*}

Here, we have used dominated convergence given by inequality \eqref{Ineq:DominationSupersymWithoutSpeed} with $ Y = \emptyset $. To show that the limiting function is integrable, we adapt the first criteria \eqref{Ineq:IntegrabilitySupersymDomination} with $ (K, L, M) = (m - 1, 0, 1) $. We have to adapt as the integration space is $ \Rr^{m - 1} $ but we have $ \Delta(0, T) = \prod_i t_i \times \Delta(T) $ and the number of variables of $ h_{m, \infty} $ is $ R = m(k + m)  $. This function will be integrable against $ \Delta(0, T)^2 $ if and only if $ R - m(m + 1) > m $, which amounts to $ k(m - 1) > 0 $. This is satisfied for all $ k, m \geq 1 $. Note nevertheless that $ m \geq 2 $ otherwise the integral over $ \Rr^{m - 1} $ is not defined. Note also that there is a symmetry between $k$ and $m$ as one has also $ \Ae_N(X, Y) = e^{-2 i \pi \sum_\ell x_\ell } \, \Esp{ \det(U_N)^k \overline{\prod_{j = 1}^{k + m} Z_{U_N}\prth{e^{ - 2 i \pi \overline{z_j}/N} } } } $. Hence, one can still have $ m = 1$ if $ k \geq 2 $ and in the same way, one can get another expression of $ \Ae(X, Y) $ as an integral over $ \Rr^{k - 1} $. Details are left to the reader.
\end{proof}

\medskip
% ==
\subsection{Ratios of the characteristic polynomial in the microscopic setting}\label{Subsec:Ratios} 

A crucial generalisation of the autocorrelation functional is given by the joint moments of ratios of the characteristic polynomial. This generalisation allows to perform much more computations and is equivalent to the knowledge of the whole eigenvalues point process, since one can get the resolvant $ \Esp{\trace\prth{ (I_N - x U_N)\inv }} = \Esp{ Z_{U_N}(x)\inv \frac{\partial Z_{U_N}(x) }{\partial x} } $ by logarithmic differentiation (which amounts to differentiate the numerator of the ratio and equate the variables) and ultimately the Dirac function by a limit on the imaginary axis (see e.g. the introduction of \cite{FyodorovStrahovCorrChiralGUE, BergereEynard} for the details). Generalising to joint resolvants, one can thus get correlation functions by means of ratios computations, which makes them, as Gian-Carlo Rota might have put it, ``nearly equi-primordial'' with the correlation functions.

Due to this fact, the literature on averages of ratios of characteristic polynomial for any type of ensembles grew dramatically during the last two decades (see e.g. the introduction of \cite{BorodinStrahov, FyodorovStrahovCorrGUE} for an account of numerous developments). In the restricted scope of this article, we are only interested in the case of the $ CUE $ and the relevant articles, in the chronological order, are given by \cite{Day, BasorForrester, MoensVanDerJeugt, BumpGamburd, ConreyFarmerKeatingRubinsteinSnaith, ConreyFarmerZirnbauer, ConreyForresterSnaith, RiedtmannMixedRatios, BaileyBettinBlowerConreyProkhorovRubinsteinSnaith, BasorBleherBuckinghamGravaIts2Keating}.

Here, we have included results that are directly linked to the computation by means of known connections with $ CUE $ averages (that historically appeared later), such as the ``Toeplitz connection'' \cite[\textit{fact five}]{DiaconisRMTSurvey} or the supersymmetric Schur function connection noticed in \cite{BumpGamburd, ConreyFarmerZirnbauer}. Thus the article by Moens and Van Der Jeugt \cite{MoensVanDerJeugt} that gives a determinantal formula for the supersymmetric specialisation of the Schur function gives in fact a formula for the expectation of ratios, and so is the Day formula \cite{Day} that gives the evaluation of a Toeplitz determinant with a rational symbol.

$ $

A general ratio is given by 
\begin{align*}%$
\Ree'_N(X, X', Y, Y') := \Esp{ \frac{ \prod_{j = 1}^{\ell_1} Z_{U_N}\prth{e^{ 2 i \pi x_j/N} }  }{ \prod_{r = 1}^{m_1} Z_{U_N}\prth{ e^{ 2 i \pi y_r/N} } } \overline{ \frac{ \prod_{j = 1}^{\ell_2} Z_{U_N}\prth{e^{ 2 i \pi x'_j/N} }  }{ \prod_{r = 1}^{m_2} Z_{U_N}\prth{ e^{ 2 i \pi y'_r/N} } }} }  
\end{align*}

Using the functional equation
\begin{align}\label{Eq:EqFuncCharpol}%$
Z_{U_N}(X) = \det(U) (-X)^N \overline{ Z_{U_N }(1/X)}
\end{align}
one can always transform such a product, up to a multiplicative factor, into the following one:
\begin{align*}%$
\Ree_N(X, Y) := (-1)^{kN} \Esp{ \det(U_N)^{-k} \prod_{j = 1}^\ell Z_{U_N}\prth{e^{ 2 i \pi x_j/N} } \prod_{r = 1}^m \frac{1}{ Z_{U_N}\prth{ e^{ 2 i \pi y_r/N} } }  }  
\end{align*}
with $ \ell = \ell_1 + \ell_2 $, $ m = m_1 + m_2 $ and $ k = \ell_2 - m_2 $. Note that for integrability reason (by expanding $ \abs{\Delta(U)}^2 $ for instance), one has to suppose $ \ell - m \geq 0 $.

We suppose that $ m > 0 $ (otherwise, we enter into the framework of the autocorrelations). We will also suppose $ k > 0 $ and $ N \geq m $ as our goal is to take $ N\to+\infty $ with fixed $ k, \ell, m $ (this is slightly different from the \textit{stable range} of \cite{ConreyFarmerZirnbauer} defined by $ N \geq k + m $).

\begin{lemma}[Ratios expectation as a supersymmetric Schur function]\label{Lemma:SuperSymRatios}
Define $ e^{i\widehat{X}/N} := \ensemble{ e^{2i \pi x_j/N}, j \in \intcrochet{1, \ell} } $, $ e^{i\widehat{Y}/N} := \ensemble{ e^{2i \pi y_j/N}, j \in \intcrochet{1, m} } $. Then, for all $ k \geq 1 $
\begin{align}\label{Eq:ExpectationRatiosAsSupersymSchur}%$
\Ree_N(X, Y) = s_{N^k}\crochet{e^{i\widehat{X}/N} - e^{i\widehat{Y}/N}} 
\end{align}
\end{lemma}

% ==

\begin{proof}
We adapt the proof given by Bump and Gamburd \cite[thm. 3]{BumpGamburd} with the $ \lambda$-ring/plethystic formalism and make it identical to the proof of \eqref{Eq:SchurJointMoments}. Write
\begin{align*}%$
\Ree_N(X, Y) & = (-1)^{kN} \int_{\Ue_N} \det(U)^{-k} \prod_{j = 1}^\ell \det\prth{I - e^{ 2 i \pi x_j/N} U } \prod_{r = 1}^m \frac{1}{ \det\prth{I - e^{ 2 i \pi y_r/N} U } } dU \\
             & = (-1)^{kN} \frac{1}{N!} \oint_{(r \Uu)^N } \prod_{j = 1}^N z_j^{-k} H\crochet{ -Z \prth{ e^{i\widehat{X}/N} -  e^{i\widehat{Y}/N} } } \abs{\Delta(Z)}^2 \frac{d^*Z}{Z}  \\
             & = (-1)^{kN} \frac{1}{N!} \oint_{(r \Uu)^N } \overline{s_{k^N}(Z) } H\crochet{ -Z \prth{ e^{i\widehat{X}/N} -  e^{i\widehat{Y}/N} } } \abs{\Delta(Z)}^2 \frac{d^*Z}{Z} \\
             & = (-1)^{kN} s_{k^N}\crochet{- \prth{ e^{i\widehat{X}/N} -  e^{i\widehat{Y}/N} } }\\
             & = s_{N^k}\crochet{ e^{i\widehat{X}/N} - e^{i\widehat{Y}/N} }
\end{align*}

Here, we have used the reproducing kernel property \eqref{Eq:ReproducingKernelProperty} valid for any specialisation, in particular, a supersymmetric one, and the involution property \eqref{Eq:InvolutionOmega}, with $ (k^N)' = N^k $. 
\end{proof}

% ==

\begin{remark}
The fact that $ \Ree_N(X, Y) $ is the supersymmetric specialisation of a Schur function (namely, a supersymmetric Schur function) was first remarked in \cite[prop. 1.3]{ConreyFarmerZirnbauer} where the formula \cite[I-3 ex. 24 (1) p. 60]{MacDo} was rederived (see also \cite[thm. 1.1]{ConreyFarmerZirnbauer}) using the formalism of Howe duality. It was then derived in the stable range $ N \geq k + m $ in \cite[thm. 3]{BumpGamburd} using Frobenius-Schur duality and symmetric functions (denoting by ``Littlewood-Schur symmetric functions'' the supersymmetric Schur functions). The formula \cite[(1.6)]{ConreyFarmerZirnbauer} was moreover rederived in the stable range in \cite[(2.3)]{ConreyForresterSnaith} using the Day formula \cite{Day} and with a method of Basor and Forrester \cite{BasorForrester} but without a reference to supersymmetric Schur functions.
\end{remark}

% ==

We are now ready to extend theorem \ref{Theorem:Autocorrelations} to ratios. Note that we now suppose $ k \geq 2 $ due to the particular form of the limit as an integral over $ \Rr^{k - 1} $.

\begin{theorem}[Ratios of characteristic polynomials in the microscopic setting]\label{Theorem:Ratios}
We have, locally uniformly in $ X + Y \in \Rr^{\ell + m} $ for all $ \ell, m \geq 1 $ and $ k \geq 2 $
\begin{align*}%$
\frac{1}{ N^{k(\ell - m - k)} } \Ree_N(X, Y) \tendvers{N}{+\infty} \Ree(X, Y)
\end{align*}
with 
\begin{align}\label{EqPhi:Ratios}%$
\begin{aligned}
\Ree(X, Y) & := \frac{(2\pi)^{k(k - 1) } }{k!} \int_{ \Rr^{ k - 1 } } \Phi_{\Ree(X, Y)}(0, t_2, \dots, t_k) \Delta(0, t_2, \dots t_k)^2 dt_2 \dots dt_k \\
\Phi_{\Ree(X, Y)}(0, T) & = e^{ - 2 i \pi \sum_{j = 2}^k t_j  +i \pi \sum_{a \in (0 + T) \tensorsum(Y - X) } a } \, h_{k, \infty}\crochet{ (0 + T) \tensorsum Y - (0 + T) \tensorsum X }  
\end{aligned}
\end{align}
where $ (X - Y)\oplus (0 + T) = X \oplus (0 + T) - Y \oplus (0 + T) $ and $ h_{k, \infty}\crochet{ A - B } $ is defined in \eqref{Def:SupersymHinfty}.
\end{theorem}

% ==

\begin{proof}
One then has with \eqref{FourierRepr:SchurRectangle} with $ U := \ensemble{ u_1, \dots, u_k } $
\begin{align*}%$
\Ree_N(X, Y)  & = s_{N^k}\crochet{ e^{i\widehat{X}/N} - e^{i\widehat{Y}/N} }  \\
              & = \frac{ 1 }{k!} \oint_{ (r \Uu)^k } U^{-N} H\crochet{ U \prth{ e^{i\widehat{X}/N} - e^{i\widehat{Y}/N} } } \abs{\Delta(U)}^2 \frac{d^* U}{U}, \quad r < \min_{t \in e^{i\widehat{Y}/N} + e^{i\widehat{X}/N}}\ensemble{ \abs{t}\inv } \\
              & = \frac{1}{k!} \oint_{ \Uu^k } U^{-N} h_{kN}\crochet{ U \prth{ e^{i\widehat{Y}/N} - e^{i\widehat{X}/N} } } \abs{\Delta(U)}^2 \frac{d^* U}{U} \quad\mbox{with \eqref{FourierRepr:SchurRectangleWithHn}}
\end{align*}

Using the homogeneity of $ h_{kN} $ and $ \abs{\Delta}^2 $, one has 
\begin{align*}%$
\Ree_N(X, Y)  & = \frac{1}{k!} \oint_{ \Uu^k } \prod_{j = 2}^k \prth{\frac{u_j}{u_1} }^{-N} h_{kN}\crochet{ \prth{1 + \frac{u_2}{u_1} + \cdots + \frac{u_k }{u_1}} \prth{ e^{i\widehat{Y}/N} - e^{i\widehat{X}/N} } } \\
              & \hspace{+8cm} \times \abs{ \Delta\prth{1 , u_2/u_1 , \cdots , u_k/u_1 } }^2  \frac{d^* U}{U}
\end{align*}

Set $ u_1 = e^{2 i \pi \theta} $ with $ \theta \in \crochet{-\frac{1}{2}, \frac{1}{2}} $ and $ u_j/u_1 = e^{2 i \pi t_j/N} $ for $ t_j \in \crochet{ - N/2, N/2} $ for all $j \in \intcrochet{2, k} $. Using remark~\ref{Rk:MultivariateFourierWithHomogeneity} and lemma \ref{Lemma:RescaledSupersym}, one gets
\begin{align*}%$
\Ree_N(X, Y)  & = \frac{1}{k!} \int_{ \crochet{ - \frac{N}{2}, \frac{N}{2} }^{k - 1} } e^{ - 2 i \pi \sum_{j = 2}^k t_j  } \\
              & \hspace{+2cm} \times h_{kN} \crochet{ (1 + e^{i \widehat{T}/N} )e^{i \widehat{Y}/N}  -  (1 + e^{i \widehat{T}/N} )e^{i \widehat{X}/N} } \abs{\Delta\crochet{ 1 + e^{i \widehat{T}/N}  } }^2 \frac{dT}{ N^{k - 1} }\\
% ===============================              
              & \equivalent{N\to+\infty}  \frac{ N^{ k(\ell - m) - 1 - k(k - 1) - (k - 1)} }{k!}   \\
              & \hspace{+1.5cm} \times \int_{ \Rr^{k - 1} } e^{ - 2 i \pi \sum_{j = 2}^k t_j  +i \pi \sum_{a \in (0 + T) \tensorsum(Y - X) } } h_{k, \infty}\crochet{ (0 + T) \tensorsum Y - (0 + T) \tensorsum X } \\
              & \hspace{+9.5cm} \times (2\pi)^{k(k - 1)} \abs{\Delta(0, T)}^2 dT \\
% =============================== 
              & = N^{k(\ell - m - k)} \, \frac{(2\pi)^{k(k - 1)}}{k!} \int_{ \Rr^{k - 1} } e^{ - 2 i \pi \sum_{j = 2}^k t_j  +i \pi \sum_{a \in (0 + T) \tensorsum(Y - X) } a } \\
              & \hspace{+6cm} \times h_{k, \infty}\crochet{ (0 + T) \tensorsum (Y - X) }  \Delta(0, T)^2 dT
\end{align*}

Note that $ \sum_{a \in \Ae} a = p_1\crochet{\Ae} $ for an abstract alphabet $ \Ae $ and that the supersymmetric power functions are defined in \eqref{Eq:SupersymPn}.

To pass to the limit in the integral, we have used the domination \eqref{Ineq:DominationSupersymWithoutSpeed}. The limiting function is integrable in application of lemma \ref{Lemma:IntegrabilitySupersym} and the first criteria of \eqref{Ineq:IntegrabilitySupersymDomination}, with the modification of remark \ref{Rk:IntegrabilityDominationSupersym}, namely, one uses the criteria \eqref{Ineq:IntegrabilitySupersymDominationBis} with $ (K, R, L, M) = (k, \ell, m, 1) $ which is satisfied for $ k, \ell, m \geq 1 $ with $ k(\ell - m) \geq 2 $. This is satisfied if $ \ell - m \geq 2 $ or $ \ell - m = 1 $ and $ k \geq 2 $. Note that one needs $ k \geq 2 $ for the integral over $ \Rr^{k - 1} $ to be defined.
\end{proof}

\medskip
% ==
\subsection{The mid-secular coefficients}\label{Subsec:Midcoeff}
%\c{C}
%
%
% ==
\subsubsection{Motivations}
The determinantal nature of the eigenvalues $ \prth{ e^{i \Theta_{k, N} } }_{1 \leq k \leq N} $ of $ U_N \sim CUE(N) $ lies at the core of several computations involving the characteristic polynomial $ Z_{U_N}(X) := \prod_{k = 1}^N (1 - X e^{i \Theta_{k, N} }) $. But before analysing its roots, the natural way to study a polynomial is to consider its Fourier coefficients, or \textit{secular coefficients}, i.e. its coefficients in the canonical basis $ (X^k)_{k \geq 0} $. In the case of $ U_N $, one defines the secular coefficients by
\begin{align}\label{Def:SecularCoeff}%$
\sc_k(U_N) := (-1)^k \crochet{z^k} Z_{U_N}(z)
\end{align}

This slight change of normalisation becomes natural if one writes
\begin{align*}%$
\sc_k(U_N) = \trace(\wedge^k U_N)
\end{align*}
or, in terms of symmetric functions of the eigenvalues, 
\begin{align*}%$
\sc_k(U_N) = e_k\prth{ e^{i\Theta_{1, N}}, \dots, e^{i\Theta_{N, N} } } := \sum_{1 \leq j_1 < j_2 < \cdots < j_k \leq N} e^{i\Theta_{j_1, N}}  \cdots  e^{i\Theta_{j_k, N} }
\end{align*}

A study of $ \prth{\sc_k(U_N)}_{0 \leq k \leq N} $ was first performed by Haake et al.~\cite{HaakeEtAl} in relation with classical and quantum chaotic dynamics. In substance, since secular coefficients are elementary symmetric functions in the eigenvalues, they contain the same equal amount of information as their spectral counterpart as pointed Diaconis and Gamburd \cite{DiaconisGamburd} who continued the study (see the quotation recalled in \S~\ref{Subsubsec:Midcoeff}). 

The asymptotic behaviour of $ (\sc_k(U_N))_{k \geq 1} $ when $ N\to +\infty $ and $k$ is independent of $N$ was investigated in \cite{DiaconisGamburd} using the change of basis between the elementary symmetric functions $ (e_k)_{k \geq 0} $ and the power functions $ (p_j)_{j \geq 0} $ \cite[I-2, (2.14')]{MacDo} and the celebrated Diaconis-Shashahani theorem on traces of powers $ (\tr{U_N^k})_{k \geq 0} $ \cite{DiaconisShahshahani}. The difference of behaviour between $ \sc_1(U_N) $ and $ \sc_N(U_N) = \det(U_N) $ led nevertheless Diaconis and Gamburd to the following interrogation (\cite{DiaconisGamburd}, see also \cite[§ 5]{DiaconisRMTSurvey}):

$ $
\vspace{-0.3cm}
\begin{quote}
It is natural to ask for a limiting distribution as $j$ grows with $N$. For example what is the limiting distribution of the $ \pe{N/2} $ secular coefficient~? On the one hand, the formula $ \sc_j(U_N) = \sum_{\lambda\vdash j} \frac{(-1)^{\ell(\lambda) }}{z_\lambda} \prod_{k = 1}^{\ell(\lambda)} \tr{ U_n^{ \lambda_k  } } $ suggests it is a complex sum of [weakly dependent] random variables, so perhaps normal. On the other hand, the formula $ \Esp{ \abs{ \sc_j(U_N) }^2 } = 1 $ holds for all $j$ making normality questionable.
\end{quote}
\vspace{-0.3cm}
$ $

Let $ \rho \in (0, 1) $ and recall that $ \overline{\rho} := 1 - \rho $. Before investigating the moments of $ \sc_{\pe{\rho N}}(U_N) $, we give a last bit of motivation for its study, the comparison with the i.i.d. case. Major \cite{Major} investigates the behaviour of $ e_{\pe{\rho N}}(Z_1, \dots, Z_N) $ when $ (Z_k)_{k \geq 1} $ is a sequence of i.i.d. random variables, in particular random variables uniformly distributed on the unit circle. Studying $ \sc_{\pe{\rho N}}(U_N) $ amounts thus to generalise from $ (Z_k)_{k \geq 1} $ to $ \prth{ e^{i \Theta_{k, N} } }_{1 \leq k \leq N} $, a determinantal point process of %``pseudo-sine'' 
kernel $ K_N(\theta, \alpha) := \Ee\big(Z_{U_N}(e^{i\theta}) \overline{ Z_{U_N}(e^{i\alpha}) } \big) = \sum_{k = 1}^N e^{ik(\theta - \alpha) } $. Major's result writes
\begin{align*}%$
\prth{ \log\abs{ e_{\pe{\rho N}}(Z_1, \dots, Z_N) } - \mu_n(\rho) } / \sigma_N(\rho) \cvlaw{N}{+\infty} \Ns_\Rr(0, 1)
\end{align*}
for a suitable rescaling $ \mu_N(\rho) $ and $ \sigma_N(\rho) $. The moments of this last random variable thus satisfy
\begin{align}\label{Eq:MajorLogNormalMoments}%$
\Esp{ \abs{ e_{\pe{\rho N}}(Z_1, \dots, Z_N) }^{2k} } \equivalent{N\to +\infty}   e^{ 2 k^2 \sigma_n(\rho)^2 + 2k \mu_N(\rho) + O_k(1) }
\end{align}

\medskip

% ==
\subsubsection{The moments of $ \sc_{\pe{\rho N}}(U_N) $}

We now compare the moments of $ \sc_{\pe{\rho N}}(U_N) $ with \eqref{Eq:MajorLogNormalMoments}.

% ==
\begin{theorem}[Absence of log-normality for $ \sc_{\pe{\rho N}}(U_N) $]\label{Theorem:MidCoeffs}
We have for $ \rho \in (0, 1) $ and $ k \geq 2 $
\begin{align*}%$
\frac{ \Esp{ \abs{ \sc_{\pe{\rho N}}(U_N) }^{2k} } }{N^{(k - 1)^2} }  \tendvers{N}{+\infty } \Se\Ce_\rho^{(k)}
\end{align*}
with 
\begin{align}\label{EqPhi:MidCoeff}%$
\begin{aligned}
\Se\Ce_\rho^{(k)} & := \frac{(2\pi)^{k(k - 1) } }{k!} \! \int_{ \Rr^{ k - 1 } } \! \Phi_{\Se\Ce_\rho^{(k)} }(0, x_2, \dots, x_k) \Delta(0, x_2, \dots x_k)^2 dx_2 \dots dx_k \\
\!\!\Phi_{\Se\Ce_\rho^{(k)} }(0, x_2, \dots, x_k) & = e^{ i \pi (k - 2) \sum_{j = 2}^k x_j } \, h_{\rho, \infty}(0, x_2, \dots, x_k)^k \,  h_{\overline{\rho}, \infty}(0, x_2, \dots, x_k)^k  
\end{aligned}
\end{align}
\end{theorem}

% ==

\begin{remark}
We thus see that $ \Esp{ \abs{ \sc_{\pe{\rho N}}(U_N) }^{2k} } = e^{  \log(N) (k - 1)^2 + O_{\rho, k}(1) } $, hence, is not of the form given in \eqref{Eq:MajorLogNormalMoments}. It thus proves that $ \log\abs{  \sc_{\pe{\rho N}}(U_N) } $ cannot be Gaussian at the limit after renormalisation. We will see another behaviour of the type $ \Esp{ X_N^k } \sim a_k \, e^{ \sigma_N^2 (k - 1)^2} $ (i.e. another random variable in the universality class of $ \sc_{\pe{\rho N} }(U_N) $) in \S~\ref{Subsec:TruncatedCharpol}. 
\end{remark}

\begin{remark}
One also has the following combinatorial result from \cite[thm. 1.4.9]{BarhoumiThesis}~:
\begin{align*}%$
\Esp{ \abs{\sc_m (U_N)  }^{2k} } = \#\ensemble{ T \in SST(N^k)\,:\, \forall \ell \leq k, c_\ell(T) = m, c_{k + \ell}(T) = N - m }
\end{align*}
where $ SST(\lambda) $ is the set of semi-standard Young tableaux (see footnote~\ref{Footnote:SST}).
\end{remark}

% ==

\begin{proof}
Using \eqref{Eq:SchurJointMoments}, we have with $ X := \ensemble{x_1, \dots, x_k} $ and $ Y := \ensemble{y_1, \dots, y_k} $
\begin{align*}%$
\Esp{ \abs{ \sc_{\pe{\rho N}}(U_N) }^{2k} } & := \Esp{ \crochet{X^{\pe{\rho N}} Y^{-\pe{\rho N}}}\prod_{j = 1}^k Z_{U_N}(x_j)\overline{Z_{U_n}(y_j\inv)}  } \\
            & = \crochet{X^{\pe{\rho N}} Y^{-\pe{\rho N}}} \int_{ \Ue_N } \prod_{j = 1}^k \det(I - x_j U_N ) \overline{\det(I - y_j\inv U_N ) } dU \\
            & = \crochet{X^{\pe{\rho N}} Y^{-\pe{\rho N}}} (-1)^{kN} \prod_{j = 1}^k y_j^{-N} \int_{ \Ue_N }  \prod_{j = 1}^k \det(I - x_j U_N ) \det(I - y_j U_N ) \\
            & \hspace{+10cm} \times \det(U)^{-k}dU \\
            & = \crochet{X^{\pe{\rho N}} Y^{N - \pe{\rho N}}} (-1)^{kN}  s_{k^N}\crochet{- (X + Y) } \\
            & = \crochet{X^{\pe{\rho N}} Y^{ \pe{ \overline{\rho} N}}} s_{N^k}\crochet{ X + Y }
\end{align*}
since $ N - \pe{\rho N} = \pe{N - \rho N} = \pe{\overline{\rho} N} $. We remark that under this form, it is clear that $ \Esp{ \abs{ \sc_{\pe{\rho N}}(U_N) }^{2k} } = \Esp{ \abs{ \sc_{\pe{\overline{\rho} N}}(U_N) }^{2k} } $, a symmetry that comes in full generality from the the functional equation \eqref{Eq:EqFuncCharpol} that implies $ \sc_k(U_N) = (-1)^{N - k} \det(U_N) \sc_{N - k}(U_N) $.

One then has with \eqref{FourierRepr:SchurRectangle} 
\begin{align*}%$
\Esp{ \abs{ \sc_{\pe{\rho N}}(U_N) }^{2k} } & = \crochet{X^{\pe{\rho N}} Y^{ \pe{ \overline{\rho} N}}} \frac{1}{k!} \oint_{ \Uu^k } U^{-N} H\crochet{ U  (X + Y) }  \abs{\Delta(U) }^2 \frac{d^*U}{U}, \qquad U := \ensemble{ u_1, \dots, u_k } \\ 
             & = \frac{1}{k!} \oint_{ \Uu^k } U^{-N} \crochet{X^{\pe{\rho N}} Y^{ \pe{ \overline{\rho} N}}} H\crochet{XU} H\crochet{YU} \abs{\Delta(U) }^2 \frac{d^*U}{U} \quad\mbox{(Fubini)} \\
             & = \frac{1}{k!} \oint_{ \Uu^k } U^{-N} h_{ \pe{\rho N} }\crochet{ U}^k h_{ \pe{ \overline{\rho} N} }\crochet{ U}^k   \abs{\Delta(U) }^2 \frac{d^*U}{U} \qquad\qquad \mbox{ with \eqref{FourierRepr:hLambda}, } \\
% ====================== 
              & =  \frac{1}{k!} \oint_{ \Uu^k } u_1^{ -Nk } \prod_{j = 2}^k \prth{ \frac{u_j}{u_1} }^{-N} \times u_1^{ \pe{\rho N} k} \, h_{ \pe{\rho N} } \prth{1, \frac{u_2}{u_1}, \dots, \frac{u_k}{u_1 } }^k \\
              & \hspace{ + 3cm} \times u_1^{ \pe{\overline{ \rho} N} k} \, h_{ \pe{ \overline{ \rho} N} } \prth{1, \frac{u_2}{u_1}, \dots, \frac{u_k}{u_1 } }^k  \abs{ \Delta  \prth{1, \frac{u_2}{u_1}, \dots, \frac{u_k}{u_1 } } }^2 \frac{d^*U}{U} \\
% ======================              
              & = \frac{1}{k!} \oint_{ \Uu^{k - 1} } \!\! V^{-N} \, h_{ \pe{\rho N} } \prth{1, v_2, \dots, v_k }^k \,  h_{ \pe{\overline{ \rho} N} } \prth{1, v_2, \dots, v_k }^k \\
              & \hspace{+5cm} \times \abs{ \Delta  \prth{1, v_2, \dots, v_k } }^2 \frac{d^*V}{V}, \quad V := \ensemble{v_2, \dots, v_k}
\end{align*}

Here, we have used the trick of remark~\ref{Rk:MultivariateFourierWithHomogeneity}, i.e. we have set $ u_1 = e^{2 i \pi \phi} $, $ v_j := u_j/u_1 = e^{2 i \pi \theta_j} $ with $ \theta_j \in \crochet{ - \frac{1}{2}, \frac{1}{2} } $ for all $ j \in \intcrochet{2, k} $ and then integrated in $ u_1 $, making this variable disappear (as the power of $ u_1 $ before integrating is $ -Nk + \pe{\rho N}k + \pe{\overline{\rho} N}k = 0 $).

Setting $ x_j := \theta_j N $ and using $ \widetilde{h}_{\rho, \infty} $ defined in \eqref{Def:hTildeInfty}, one now gets
\begin{align*}%$
\Esp{ \abs{Z_{U_N}(1) }^{2k} } & =  \frac{1}{k!} \int_{ \crochet{ - \frac{N}{2}, \frac{N}{2} }^{k - 1} } e^{ - 2 i \pi \sum_{j = 2}^k x_j } \, h_{ \pe{\rho N} } \prth{1, e^{2 i \pi x_2/N} , \dots, e^{2 i \pi x_k/N} }^k  \\
                & \hspace{+4.5cm} \times h_{ \pe{\overline{\rho} N} } \prth{1, e^{2 i \pi x_2/N} , \dots, e^{2 i \pi x_k/N} }^k \\
                & \hspace{+4.5cm} \times  \abs{\Delta\prth{ 1, e^{2 i \pi x_2/N} , \dots, e^{2 i \pi x_k/N} } }^2 \frac{dx_2 \dots dx_k }{N^{k - 1} }  \\
% ======================================                
                & \equivalent{N\to +\infty } \frac{ N^{ 2k(k-1) - k(k-1) - (k - 1)}   }{k!} \int_{ \Rr^{k - 1} } \!\! e^{ - 2 i \pi \sum_{j = 2}^k x_j } \, \widetilde{h}_{ \rho , \infty }\prth{0, x_2, \dots, x_k }^k  \\
                & \hspace{+6cm} \times \widetilde{h}_{ \overline{\rho} , \infty }\prth{0, x_2, \dots, x_k }^k  \\
                & \hspace{+6cm} \times (2\pi)^{ k(k - 1) } \abs{\Delta\prth{ 0, x_2, \dots, x_k } }^2 dx_2 \dots dx_k \\
% ======================================             
                & = N^{(k - 1)^2} \frac{ (2\pi)^{ k(k - 1) }}{k!} \!\! \int_{ \Rr^{k - 1} } \!\! e^{ i \pi (k - 2) \sum_{ j = 2 }^k x_j } \, h_{\rho, \infty} (0, x_2, \dots, x_k)^k \, h_{\overline{\rho}, \infty} (0, x_2, \dots, x_k)^k  \\
                & \hspace{+9cm} \Delta\prth{ 0, x_2, \dots, x_k }^2  dx_2 \dots dx_k
\end{align*}

Here, we have used dominated convergence given by inequality \eqref{Ineq:SpeedOfConvergenceHinftyBis}. The limiting function is seen integrable using the second criteria \eqref{Ineq:IntegrabilityDomination} with $ (K, M, M', \kappa, \kappa') = (k - 1, k, k, 1, 1 ) $ which is clearly fullfilled for all $ p \in \ensemble{0, 1, 2} $ and $ k \geq 2 $.
\end{proof}

% ==

\begin{remark}\label{Rk:Alternative:MidSecularCoefficient}
Since one has with \eqref{Eq:SchurWithHn}/\eqref{FourierRepr:SchurRectangleWithHn}
\begin{align*}%$
s_{N^k}\crochet{ X + Y } & =  \frac{1}{k!} \oint_{ \Uu^k } U^{-N} h_{Nk}\crochet{ U  (X + Y)}   \abs{\Delta(U) }^2 \frac{d^*U}{U}, \qquad  U := \ensemble{ u_1, \dots, u_k } \\ 
              & =  \frac{1}{k!} \oint_{ \Uu^k } u_1^{ -Nk } \prod_{j = 2}^k \prth{ \frac{u_j}{u_1} }^{-N} \times u_1^{Nk} \, h_{Nk} \crochet{ \prth{1 + \frac{u_2}{u_1} + \dots + \frac{u_k}{u_1 } } (X + Y) } \\
              & \hspace{+9cm} \abs{ \Delta  \prth{1, \frac{u_2}{u_1}, \dots, \frac{u_k}{u_1 } } }^2 \frac{d^*U}{U} \\
              & = \frac{1}{k!} \oint_{ \Uu^{k - 1} }  V^{-N}  h_{Nk} \crochet{  (1 + V) (X + Y) }  \abs{ \Delta  \crochet{1 + V } }^2 \frac{d^*V}{V}, \quad V := \ensemble{v_2, \dots, v_k},
\end{align*}
one can use Fubini to take the $ \crochet{X^{\pe{\rho N}} Y^{ \pe{ \overline{\rho} N} } } $-Fourier coefficient of this formula directly inside the integral, namely to write
\begin{align*}%$
\Esp{ \abs{ \sc_{\pe{\rho N}}(U_N) }^{2k} } & = \crochet{X^{\pe{\rho N}} Y^{ \pe{ \overline{\rho} N}}} s_{N^k}\crochet{ X + Y } \\
                 & = \frac{1}{k!} \oint_{ \Uu^{k - 1} \times \Uu^{2k}} \hspace{-0.3cm}  V^{-N}  X^{-\pe{\rho N}} Y^{- \pe{ \overline{\rho} N}} h_{Nk} \crochet{  (1 + V) (X + Y) }  \\
                 & \hspace{+7cm} \times \abs{ \Delta \crochet{1 + V } }^2 \frac{d^*V}{V} \frac{d^*X}{X}\frac{d^*Y}{Y} \\
                 & = \frac{1}{k!} \oint_{ \Uu^{k - 1} \times \Uu^{2k - 1}} \hspace{-0.3cm}   V^{-N}  T^{-\pe{\rho N}} W^{- \pe{ \overline{\rho} N}} h_{Nk} \crochet{  (1 + V) (1 + T + W) }  \\
                 & \hspace{+7cm} \times \abs{ \Delta \crochet{1 + V } }^2 \frac{d^*V}{V} \frac{d^*T}{T}\frac{d^*W}{W}
\end{align*}
using remark~\ref{Rk:MultivariateFourierWithHomogeneity} and the fact that $ \rho + \overline{\rho} = 1 $. A slight adaptation of the second criteria \eqref{Ineq:IntegrabilityDomination} with variables $ T + W \in \Uu^{2k - 1} $ that are not integrated against a squared Vandermonde determinant allows to use dominated convergence to get (with the conventions of theorem~\ref{Theorem:Autocorrelations})
\begin{align*}%$
\Esp{ \abs{ \sc_{\pe{\rho N}}(U_N) }^{2k} } & \equivalent{N \to +\infty} \,  \frac{1}{k!} \int_{ \Rr^{3k - 2} } \hspace{-0.3cm} e^{-2 i \pi ( \, \overline{\rho} \alpha_1 + \sum_{j = 2}^k ( \theta_j + \rho \varphi_j + \overline{\rho} \alpha_j ) \, ) } N^{ 2k^2 - 1} \widetilde{h}_{k, \infty} \crochet{  (0 + \thetab) \oplus (0 + \alphab + \varphib) }  \\
                 & \hspace{+6cm} \times N^{-k(k - 1)} \abs{ \Delta \crochet{ 0 + 2 i \pi \thetab } }^2 \frac{d\thetab}{N^{k - 1}}   \frac{d\alphab}{N^{k}}   \frac{d\varphib}{N^{k - 1}} \\
                 & = N^{2k^2 - 1 -k(k - 1) - (3k - 2) }  \frac{(2\pi)^{k(k - 1)} }{k!}  \\
                 & \hspace{+2cm} \times \int_{ \Rr^{3k - 2} } e^{-2 i \pi ( \, \overline{\rho} \alpha_1 + \sum_{j = 2}^k ( \theta_j + \rho \varphi_j + \overline{\rho} \alpha_j ) \, ) }  \widetilde{h}_{k, \infty} \crochet{  (0 + \thetab) \oplus (0 + \alphab + \varphib) }   \\
                 & \hspace{+9.5cm} \times \Delta \crochet{ 0 +  \thetab }^2 d\thetab\, d\alphab\, d\varphib \\
                 & =: N^{(k - 1)^2 } \frac{(2\pi)^{k(k - 1)} }{k!} \!\!\int_{ \Rr^{3k - 2} } e^{-2 i \pi ( \, \overline{\rho} \alpha_1 + \sum_{j = 2}^k ( \theta_j + \rho \varphi_j + \overline{\rho} \alpha_j ) \, ) } \\
                 & \hspace{+4cm} \times \widetilde{h}_{k, \infty} \crochet{  (0 + \thetab) \oplus (0 + \alphab + \varphib) } \Delta \crochet{ 0 +  \thetab }^2 d\thetab\, d\alphab\, d\varphib
\end{align*}

By Fubini, we can thus replace $ \Phi_{\Se\Ce_\rho^{(k)} } $ in \eqref{EqPhi:MidCoeff} by
\begin{align}\label{EqPhi:Alternative:MidSecularCoeff}%$
\begin{aligned}
\Phi^{(Alt)}_{\Se\Ce_\rho^{(k)} }(0, x_2, \dots, x_k) &  =  e^{ - i \pi k \sum_{j = 2}^k x_j }\int_{ \Rr^{2k - 1} } e^{-2 i \pi ( \, \overline{\rho} \alpha_1 + \sum_{j = 2}^k (  \rho \varphi_j + \overline{\rho} \alpha_j ) \, ) } \\
                 & \hspace{+4cm} \times \widetilde{h}_{k, \infty}\crochet{  (0 + \xb) \oplus (0 + \alphab + \varphib) }  d\alphab\, d\varphib
\end{aligned}
\end{align}

This last expression is in fact equal to the one given in \eqref{EqPhi:MidCoeff}, as one can see working directly on the functions instead of the integrals. The formula for fixed $N$ that allows to pass from one to another is a ``branching formula'' in the theory of symmetric functions. 
\end{remark}

% ==

\medskip

\medskip
% ==
\subsection{Back to the autocorrelations~: the randomisation paradigm}\label{Subsec:RandomisationParadigm}

As stated by Diaconis and Gamburd paraphrasing Rota, the mid-secular coefficients are indeed ``equiprimordial'' with the eigenvalues when studying the autocorrelations of $ Z_{U_N} $ as one can write the characteristic polynomial as the following randomisation of the sequence $ (\sc_k(U_N))_{k \geq 0} $
\begin{align}\label{RandomEq:Charpol}%$
Z_{U_N}(e^{s/N}) = \sum_{\ell = 0}^N \sc_\ell(U_N) e^{s\ell/N} = N \, \Esp{ \sc_{ V_N }(U_N) e^{s V_N /N} \big\vert U_N }
\end{align}
where $ V_N \sim \Us(\intcrochet{0, N}) $ is independent of $ U_N \sim CUE_N $ (note also that writing $ s = a + ib $, one could equivalently use a truncated geometric random variable of parameter $ e^{a/N} $ which corresponds to the exponential bias of the uniform random variable). Considering a sequence $ \Vb_{\! N} := ( V_N^{(j)} )_{j \geq 1} $ of i.i.d. such $ V_N $s, one thus gets with $ \Xb := \ensemble{x_1, \dots, x_{k + m}} $ 
\begin{align*}%$
\Ae_N(X, Y) & := \Esp{ \prod_{j = 1}^k Z_{U_N}\prth{e^{ 2 i \pi x_j/N} } \prod_{\ell = 1}^m \overline{ Z_{U_N}\prth{e^{ -2 i \pi y_\ell/N} }}  } \\
                 & = (-1)^{Nm} e^{-2 i \pi \sum_\ell \overline{ y_\ell} } \, \Esp{ \det(U_N)^{-m} \prod_{j = 1}^{k + m} Z_{U_N}\prth{e^{ 2 i \pi z_j/N} }  } \quad\mbox{by \eqref{Eq:EqFuncCharpol}} \\
                 & = (-1)^{Nm} N^{k + m} e^{-2 i \pi \sum_\ell \overline{ y_\ell} } \, \Esp{ \det(U_N)^{-m} \prod_{j = 1}^{k + m} \sc_{V_N^{(j)} }(U_N) e^{ 2 i \pi z_j V_N^{(j)}/N } } \\
                 & = N^{k + m} e^{-2 i \pi \sum_\ell \overline{ y_\ell} } \, \Esp{ e^{ 2 i \pi \sum_{j = 1}^{k + m} \! z_j V_N^{(j)}\!/N } \crochet{ \Xb^{ \Vb_{\! N} } } s_{N^m}\crochet{ \Xb } } \quad\mbox{by \eqref{Eq:SchurJointMoments}} \\
                 & = \frac{N^{k + m}}{ m!} e^{-2 i \pi \sum_\ell \overline{ y_\ell} } \, \Esp{ \! e^{ 2 i \pi \sum_{j = 1}^{k + m} \! z_j V_N^{(j)}\!/N  } \crochet{ \Xb^{ \Vb_{\! N} } } \!\! \oint_{\Uu^m}  U^{-N} H\crochet{U \Xb } \abs{\Delta(U)}^2 \frac{d^*U}{U} } \mbox{ by \eqref{FourierRepr:SchurRectangle}}  \\
                 & = \frac{N^{k + m}}{ m!} e^{-2 i \pi \sum_\ell \overline{ y_\ell} } \, \Esp{ e^{ 2 i \pi \sum_{j = 1}^{k + m} \! z_j V_N^{(j)}\!/N  }   \oint_{\Uu^m}  U^{-N} \prod_{j = 1}^{k + m} h_{V_N^{(j)} }(U) \abs{\Delta(U)}^2 \frac{d^*U}{U} }  \\
                 & = \frac{N^{k + m}}{ m!} e^{-2 i \pi \sum_\ell \overline{ y_\ell} } \, \Ee \Bigg( e^{ 2 i \pi \sum_{j = 1}^{k + m} \! z_j V_N^{(j)}\!/N  } \Unens{ \sum_{j = 1}^{k + m} V_N^{(j)} = (k + m) N }  \\
                 & \hspace{+5cm} \times \oint_{\Uu^{m - 1}}  W^{-N} \prod_{j = 1}^{k + m} h_{V_N^{(j)} }\crochet{1 + W} \abs{\Delta\crochet{1 + W}}^2 \frac{d^*W}{W} \Bigg) 
\end{align*}
using the trick of remark \ref{Rk:MultivariateFourierWithHomogeneity} and integrating out $ u_1 $. 

Using the coupling \eqref{Eq:CouplingUniform} for $ V_N \sim \Us(\intcrochet{0, N}) $ under the form $ V_N \eqlaw \pe{ (N + 1) V } $ with $ V \sim \Us([0, 1]) $, we get a product of $ h_{ [ (N + 1) V ^{(j)} ] }\crochet{1 + W} $ with an i.i.d. sequence $ (V^{(j)})_j $ of uniform random variables in $ [0, 1] $, and we can naturally use the rescaling of lemma~\ref{Lemma:RescaledGegenbauer} with lemma~\ref{Lemma:IntegrabilityDominationInC} for the integration in $ v_j \in [0, 1] $ corresponding to the uniform random variables. The only point of rescaling that needs a precision is the indicator of $ \sum_{j = 1}^{k + m} V_N^{(j)} = (k + m)N $. This is the idea of local CLT presented in \eqref{Eq:PhilosophyCLT}, as 
\begin{align*}%$
\sum_{j = 1}^{k + m} V_N^{(j)} \eqlaw \sum_{j = 1}^{k + m} \pe{N V^{(j)} } = \pe{ N  \sum_{j = 1}^{k + m} V^{(j) } } =: \pe{N S_{k + m} }
\end{align*}
and $ N\inv \pe{N S_{k + m} } \to S_{k + m} $ in law. One can adapt the computation of \eqref{Eq:PhilosophyCLT} in the case of $ \Esp{ X_N \Unens{ \pe{N S_{k + m} } = \pe{N x} } } $ if $ N^{-\alpha} X_N \to X_\infty $ in law with a good domination, in which case, one gets 
\begin{align*}%$
\Esp{ X_N \Unens{ \pe{N S_{k + m} } = \pe{N x} } } \approx \frac{N^\alpha}{N} \Esp{ X_\infty \delta_0\prth{ S_{k + m} - x } } = N^{\alpha - 1} \Esp{ X_\infty \big\vert S_{k + m} = x }  f_{S_{k + m}}(x)
\end{align*}
where $ f_{S_{k + m}} $ is the Lebesgue density of $ S_{k + m} $, i.e. $ f_{S_{k + m}}(x) = \Esp{ \delta_0(S_{k + m} - x) } $. Here again, the only step to justify is the $ \approx $ given in the proof of lemma~\ref{Lemma:RescaledGegenbauer}.

Note that one could write equivalently
\begin{align*}%$
\Esp{ X_N \Unens{ \pe{N S_{k + m} } = \pe{N x} } }  = \Prob{ \pe{N S_{k + m} } = \pe{N x} } \times \Esp{ X_N \Big\vert  \pe{N S_{k + m} } = \pe{N x} } 
\end{align*}
and only deal with the local CLT in the probability since $ \Esp{ N^{-\alpha} X_N  \vert  \pe{N S_{k + m} } = \pe{N x} } \to \Esp{ X_\infty \vert S_{k + m} = x } $, such a conditioning with a set of measure 0 being exactly an instance of the hyperplane concentration phenomenon. This is exactly how the Dirac mass should be understood~: in the sense of a conditioning, i.e. as the restriction of the integral over $ [0, 1]^{m + k} $ to a hyperplane given by $ \sum_{j = 1}^{m + k} v_j = x $ (by setting $ v_1 := x - \sum_{j = 2}^{m + k} v_j $ for instance) and not as a function to be dominated (which is impossible since $ \delta_0 $ is not a function). 

Setting $ w_j = e^{2 i\pi t_j/N} $ for $ t_j \in [-N/2, N/2] $, one thus gets
\begin{align*}%$
\Ae_N(X, Y) & = \frac{N^{k + m}}{ m!} e^{-2 i \pi \sum_\ell \overline{ y_\ell} } \, \Ee \Bigg( \! e^{ 2 i \pi \sum_{j = 1}^{k + m} \! z_j V_N^{(j)}\!/N  } \Unens{ \sum_{j = 1}^{k + m} V_N^{(j)} = (k + m)N }  \\
                 & \hspace{+2cm} \times \int_{\crochet{-\frac{N}{2}, \frac{N}{2} }^{m - 1}} \!\!  e^{- 2 i \pi \sum_{\ell = 2}^m t_\ell } \!\! \prod_{j = 1}^{k + m} \! h_{[N V^{(j)} ] }\!\crochet{1 + e^{2 i \pi \tb/N}} \abs{\Delta\crochet{1 + e^{2 i \pi \tb/N}}\! }^2 \!\frac{d\tb}{N^{m - 1}} \! \Bigg) \\
% ==================================================================   
                 & \equivalent{N\to +\infty}  N^{k + m -1 + (k + m)(m - 1) -m(m - 1) - (m - 1)} \frac{(2\pi)^{m(m - 1)}}{ m!}     \\
                 & \hspace{+1cm} \times e^{-2 i \pi \sum_\ell \overline{ y_\ell} } \, \Ee \Bigg( \! e^{ 2 i \pi \sum_{j = 1}^{k + m} \! z_j V ^{(j)}  } \delta_0\prth{\sum_{j = 1}^{k + m} V^{(j)} - (k + m) \! } \\
                 & \hspace{+6cm} \int_{\Rr^{m - 1}} \!\!  e^{- 2 i \pi \sum_{\ell = 2}^m t_\ell } \!\! \prod_{j = 1}^{k + m} \! h_{ V^{(j)}\! , \infty } \prth{0, \tb} \Delta(0, \tb)^2 d\tb   \! \Bigg) \\
% ==================================================================  
                 & = N^{km} \frac{(2\pi)^{m(m - 1)}}{ m!} f_{S_{m + k\!}} (k + m) \!\! \int_{\Rr^{m - 1}} \!\!  e^{- 2 i \pi \! \sum_{\ell = 2}^m t_\ell } \,  \Ee \Bigg( \! e^{ 2 i \pi \,\prth{ \sum_{j = 1}^{k} \! x_j V ^{(j)} + \sum_{\ell = 1}^{m} \! \overline{ y_\ell} (1 - V^{(\ell + k)} )  } } \\
                 & \hspace{+5.5cm} \times  \prod_{j = 1}^{k + m} \! \widetilde{h}_{ V^{(j)}\! , \infty } \prth{0, \tb} \Bigg\vert \sum_{j = 1}^{k + m} V^{(j)} = k + m \! \Bigg) \Delta(0, \tb)^2 d\tb    
\end{align*}
where we have applied dominated convergence as in theorem~\ref{Theorem:Autocorrelations} with the additional integrability in the uniform random variables coming from lemma~\ref{Lemma:IntegrabilityDominationInC} (recall that $ \widetilde{h}_{c, \infty} $ is defined in \eqref{Def:hTildeInfty} and differs from $ h_{c, \infty} $ by an exponential factor). 

One can perform a last bit of massaging to the previous formula by giving the exact value of the constant $ f_{S_{m + k\!}} (k + m) $. Indeed, the sum $ S_n $ of $n$ independent uniform random variables on $ [0, 1] $ has the Bates/Irwin-Hall distribution whose Lebesgue-density is easily given (by Fourier inversion) by 
\begin{align}\label{Eq:BatesIrwinHall}%$
f_{S_n}(x) = \frac{1}{2(n - 1)!} \sum_{k = 0}^n (-1)^k {n \choose k} (x - k)_+^{n - 1} = \frac{1}{2(n - 1)!} \sum_{k = 0}^{\pe{x}} (-1)^k {n \choose k} (x - k)^{n - 1} 
\end{align}
where $ x_+ := x\Unens{x \geq 0} $. As a result, $ f_{S_n}(0) = \frac{1}{2(n - 1)!} $. And since $ 1 - U \eqlaw U $ if $ U $ is uniform in $ [0, 1] $, one gets $ \{ \sum_{j = 1}^{k + m} V^{(j)} = k + m \} \eqlaw \{ \sum_{j = 1}^{k + m} V^{(j)} = 0 \} $. 

In the end, we have then proven the following~:

% ==

\begin{theorem}[Microscopic autocorrelations of the characteristic polynomial with randomisation]\label{Theorem:Autocorrelations:WithRandomisation}
We have, locally uniformly in $ X + Y \in \Cc^{k + m} $ for all $ k, m \geq 2 $
\begin{align*}%$
\frac{1}{N^{km} } \Esp{ \prod_{j = 1}^k Z_{U_N}\prth{e^{ 2 i \pi x_j/N} } \prod_{\ell = 1}^m \overline{ Z_{U_N}\prth{e^{- 2 i \pi y_\ell/N} }}  } \tendvers{N}{+\infty} \Ae(X, Y)
\end{align*}
with 
\begin{align}\label{EqPhi:Random:Autocorrels}%$
\begin{aligned}
\Ae(X, Y) & := \frac{(2\pi)^{m(m - 1) } }{m!} \int_{ \Rr^{ m - 1 } } \Phi^{(Rand)}_{\!\Ae(X, Y)}(0, \tb) \Delta(0, \tb)^2 d\tb \\
\Phi^{(Rand)}_{\!\Ae(X, Y)}(0, \tb) & = \frac{e^{- 2 i \pi \! \sum_{\ell = 2}^m t_\ell } }{2(k + m - 1)!}  \,  \Ee \Bigg( \! e^{ 2 i \pi \,\prth{ \sum_{j = 1}^{k} \! x_j V ^{(j)} + \sum_{\ell = 1}^{m} \! \overline{ y_\ell} (1 - V^{(\ell + k)} )  } } \\
                 & \hspace{+5.5cm} \times  \prod_{j = 1}^{k + m} \! \widetilde{h}_{ V^{(j)}\! , \infty } \prth{0, \tb} \Bigg\vert \sum_{j = 1}^{k + m} V^{(j)} = k + m \! \Bigg)
\end{aligned}
\end{align}
\end{theorem}

% ==

From now on, we will alternate the presentations between the two possible ways of rescaling a functional, namely with and without the randomisation paradigm, sometimes leaving to the interested reader the exercise of adapting one proof. We will see that this last paradigm is even simpler than the ``formula paradigm'' and leads to several interesting relations between the limiting functionals.

% ==

\begin{remark}
By \eqref{Def:hTildeInfty} and \eqref{Eq:Beta_ProbRepr}, one has
\begin{align*}%$
\widetilde{h}^{(\kappa)}_{ c , \infty }\prth{x_1, \dots, x_k} = \frac{c^{k\kappa - 1}}{\Gamma(k\kappa)}   \, \Esp{ e^{ 2 i \pi c \sum_{j = 1}^k x_j \betab_{\kappa, 1}^{(j)}   } \Bigg\vert \sum_{j = 1}^k \betab_{\kappa, 1}^{(j)}  = 1 }
\end{align*}
and as a result, for $ \kappa = 1 $, 
\begin{align*}%$
\widetilde{h}_{ c, \infty }\prth{x_1, \dots, x_k} = \frac{1}{(k - 1)!} \, \Esp{ e^{ 2 i \pi c \sum_{j = 1}^k x_j U^{(j)} } \Bigg\vert \sum_{j = 1}^k U^{(j)}  = 1 }
\end{align*}
where $ \prth{ U^{(j)} }_{j \geq 1} $ is a sequence of i.i.d. uniform random variables in $ [0, 1] $. Using the convention \eqref{Convention:Bar} for $ V^{(j)} $ (but not for $ y_j $), combining with \eqref{EqPhi:Random:Autocorrels} and using a double sequence of i.i.d. such random variables (all present sequences being independent), one gets
\begin{align}\label{EqPhi:Random:AutocorrelsAsFourier}%$
\begin{aligned}
\Phi^{(Rand)}_{\!\Ae(X, Y)}(0, \tb) = \, & \frac{e^{- 2 i \pi \! \sum_{\ell = 2}^m t_\ell } }{2(k + m - 1)! ( (m - 1)! )^{k + m}} \\
                 & \times \Ee \Bigg( \! e^{ 2 i \pi \,\prth{ \sum_{\ell' = 1}^{k} \! x_{\ell'} V ^{(\ell'	)} + \sum_{\ell = 1}^{m} \! \overline{ y_\ell}  \overline{V^{(\ell + k)}}      +  \sum_{r = 2}^m t_r \,\prth{ \sum_{j = 1}^{k + m} V^{(j)} U^{(j, r)}  }  } } \\
                 & \hspace{+3.6cm} \Bigg\vert \sum_{\ell = 1}^{k + m} V^{(\ell)} = k + m, \sum_{r = 2}^m  U^{(j, r)} = 1, \, \forall j \leq k + m \! \Bigg)
\end{aligned}
\end{align}

This shows that, up to a constant, $ \Phi^{(Rand)}_{\!\Ae(X, Y)}(0, \tb) $ is a joint Fourier transform in $ (X, Y, \tb) $ built out of uniform random variables conditioned on their sums (this conditioning being not only equivalent to the hyperplane concentration phenomenon but also the only condition to have a non-tensorial function). Note that using \eqref{Eq:NegBinRepresentation_h_n}, one can have this result for fixed $N$.
\end{remark}

\medskip
% ==
\subsection{The truncated characteristic polynomial}\label{Subsec:TruncatedCharpol} 

% ==
\subsubsection{Motivations}

Denote by $ \Pe $ the set of ordered prime numbers and $ v_p(n) $ the $p$-adic valuation of $ n \in \Nn^*$, i.e. the exponent of $p$ in its prime decomposition $ n = \prod_{p \in \Pe} p^{ v_p(n) } $. Let $ (Z_p)_{p \in \Pe} $ be a sequence of i.i.d random variables uniformly distributed on the unit circle $ \Uu $. 
%
%
%
%We recall that the Keating-Snaith moments conjecture is the following assertion
%%
%%
%%
%\begin{align}\label{Eq:MomentsConjecture}%$
%\Esp{ \abs{\zeta\prth{\frac{1}{2} + i T U } }^{2k} } \equivalent{T\to+\infty} a_k \beta_k \, e^{ k^2 \log\log T  }
%\end{align}
%%
%%
%%
%where $ U $ is a random variable uniformly distributed on $ \crochet{0, 1} $, $a_k$ is the arithmetic factor defined by 
%%
%%
%%
%\begin{align}\label{Def:ArithmeticFactor}%$
%a_k = e^{ C k^2 } \prod_{p \in \Pe} e^{ - \frac{k^2}{p } } \Esp{ \abs{ 1 - p^{-1/2} Z_p  }^{-2k} }, \qquad C := \sum_{p \in \Pe} \log(1 - p\inv) + p\inv 
%\end{align}
%%
%%
%%
%and $ \beta_k $ is the random matrix factor, conjectured to be universal, and given by
%%
%%
%%
%\begin{align}\label{Def:MatrixFactor}%$
%\beta_k = \prod_{\ell = 0}^{k - 1} \frac{\ell !}{ (\ell + k )!} = \frac{G(1 + k)^2}{G(1 + 2k)}
%\end{align}
%%
%%
%%
%with $G$ the Barnes $G$-function or double Gamma function (see e.g. \cite{KeatingSnaith}).
%
%
We recall that for $ X \geq 1 $ and $ \sigma \in \Rr_+ $ the \textit{Random Multiplicative Function} defined in \ref{Def:RandomMultiplicativeFunction} is
\begin{align*}%$%$
\Ze_{X, \sigma} := \sum_{k = 1}^X \frac{1}{k^\sigma} \prod_{p \in \Pe } Z_p^{v_p(k)}
\end{align*}
and that by the Bohr-Jessen theorem (see \cite{BohrJessen}) or the Bohr correspondence (see e.g. \cite[ch. 3.2]{KowalskiRandonnee} or \cite[introduction]{HeapLindqvist}), this random variable occurs as the following limit in distribution 
\begin{align*}%$%$
\zeta_X\prth{\sigma + i T U} := \sum_{k = 1}^X \frac{1}{ k^{\sigma + i T U} } \cvlaw{T}{+\infty }  \Ze_{X, \sigma}
\end{align*}

%Instead of considering the whole Zeta function, Conrey and Gamburd \cite{ConreyGamburd} consider the following truncated version
%%
%%
%%
%\begin{align*}%$%$
%\zeta_X(s) := \sum_{n = 1}^X \frac{1}{n^s}
%\end{align*}
%%
%%
%%
%and ask the question of computing the equivalent, when $ X \to \infty $ of the limit 
%%
%%
%%
%\begin{align*}%$%$
%\lim_{T \to +\infty} \Esp{ \abs{\zeta_X\prth{\frac{1}{2} + i T U } }^{2k} }
%\end{align*}

Conrey and Gamburd \cite{ConreyGamburd} ask for an equivalent of $ \Esp{ \vert \Ze_{X, 1/2 } \vert^{2k} } $ when $ X \to +\infty $. Their theorem \cite[thm. 1]{ConreyGamburd} writes
\begin{align}\label{Eq:ConreyGamburdNT}%$
\Esp{ \abs{ \Ze_{X, 1/2 } }^{2k} } \equivalent{X\to+\infty} A_k \gamma_k \, e^{ k^2 \log\log X }
\end{align}
with a universality factor $ \gamma_k  $ different from $ M_k $ defined in \eqref{Def:MatrixFactor}, a fundamental difference with the moments conjecture \eqref{Eq:MomentsConjecture}. This comes from the truncation, too strong to allow the random matrix factor $ M_k $ to survive after the first limit. The right analogue of the characteristic polynomial in this setting is the \textit{truncated characteristic polynomial} $ Z_{U_N, \ell} $ defined by
\begin{align}\label{Def:TruncatedCharPol1}%$
Z_{U_N, \ell}(z) := \sum_{j = 0}^\ell \sc_j(U_N) (-z)^j
\end{align}

The factor $ \gamma_k  $ then occurs as the limit \eqref{Eq:ConreyGamburdCV}, i.e. $ \ell^{-k^2} \Ee \big( \abs{ Z_{U_N, \ell}(1) }^{2k} \big) \to \gamma_k $ when $ \ell \to +\infty $ for $ \ell \geq k N $ \cite[(34)]{ConreyGamburd}.
%
%
%
%\begin{align*}%$%$
%\frac{ \Esp{ \abs{ Z_{U_N, \ell}(1) }^{2k} } }{\ell^{k^2} } \tendvers{\ell}{+\infty} \gamma_k  
%\end{align*}

The random variables in the universality class of the mid-coefficient are the $ \Ze_{X, \sigma} $ for $ \sigma \in [0, 1/2) $. The behaviour is radically different from the case $ \sigma = \frac{1}{2} $ investigated by Conrey and Gamburd. The articles \cite{ConreyGamburd, HarperNikeghbaliRadziwill, HeapLindqvist} give the following results, for all $ k \in \Nn^* $
\begin{align}\label{Eq:BehaviourRandomMultiplicativeFunction}%$
\begin{aligned} 
\sigma = \frac{1}{2} :  &\qquad \Esp{ \abs{ \Ze_{X, \sigma} }^{2k} } \equivalent{X\to+\infty} A_k \gamma_k \, (\log X)^{k^2} \\
\sigma < \frac{1}{2} :  &\qquad \Esp{ \abs{ \Ze_{X, \sigma} }^{2k} } \equivalent{X\to+\infty} A_k \gamma_k \, \frac{\Gamma(2k - 1)}{ \Gamma(k)^2 } \frac{X^{k(1 - 2\sigma)} }{(1 - 2\sigma)^{2k - 1 } } (\log X)^{ (k - 1)^2 }
\end{aligned}
\end{align}

One sees that setting $ N = \log(X) $, a behaviour in $ N^{(k - 1)^2} $ holds for the moments of $ X^{ \sigma - \frac{1}{2} } \Ze_{X, \sigma} $ for $ \sigma < \frac{1}{2} $.

The results of \eqref{Eq:BehaviourRandomMultiplicativeFunction} for $ \sigma < 1/2 $ were moreover proven for their $ CUE $ analogue by Heap and Lindqvist \cite{HeapLindqvist}. Sumarising the results of \cite{ConreyGamburd, HeapLindqvist}, one gets for $ \lambda \in \Rr_+ $ and $ \ell = \pe{N/k} $
\begin{align}\label{Eq:BehaviourTruncatedCharPol}%$
\begin{aligned} 
\lambda = 1 :  &\qquad \Esp{ \abs{ Z_{U_N, \ell }(\lambda) }^{2k} } \equivalent{\ell \to+\infty} \gamma_k \, \ell^{k^2} \\
\lambda > 1 :  &\qquad \Esp{ \abs{ Z_{U_N, \ell }(\lambda) }^{2k} } \equivalent{\ell \to+\infty} \gamma_k \, \frac{\Gamma(2k - 1)}{ \Gamma(k)^2 } F_k(\lambda) \frac{\lambda^{2k \ell } }{(1 - \lambda^{-2})^{2k - 1 } } \ell^{ (k - 1)^2 }
\end{aligned}
\end{align}
where 
\begin{align*}%$
F_k(\lambda) := {}_2 F_1\prth{ \begin{matrix} 1 - k, 1 - k \\ 2 - 2k \end{matrix} \bigg\vert 1 - \lambda^{-2} }
\end{align*}

The proof given in \cite{ConreyGamburd, HeapLindqvist} uses the polytope method. The goal of this \S\, is to give a new proof with the previous machinery. 

Note that \eqref{Eq:BehaviourRandomMultiplicativeFunction} yields the natural question:

\begin{question}\label{Question:NTphase}
Is there a phase transition at $ \sigma = \frac{1}{2} - o(X) $ for $ \Ze_{X, \sigma} $ that links the Gaussian universality class $ \sigma = 1/2 $ and this new universality class $ \sigma < 1/2 $ ?
\end{question}

The same question can of course be asked for the $ CUE $ in view of \eqref{Eq:BehaviourTruncatedCharPol}.

\medskip
% ==
\subsubsection{The Conrey-Gamburd theorem}\label{Subsubsec:ConreyGamburd}

We first rederive the convergence \eqref{Eq:ConreyGamburdCV}. Since we can take $ N = k \ell $, i.e. $ \ell = \pe{N k\inv } $, we consider the case of a general truncation at $ T = \pe{\rho N} $ for any $ \rho \in (0, 1) $. We also set $ \overline{T} := N - T = \pe{\overline{\rho} N} $.  

\begin{theorem}[Truncated characteristic polynomial at $ \pe{\rho N} $ in $1$]\label{Theorem:TruncatedCharpolIn1}
For all $ k \geq 1 $, one has
\begin{align*}%$
\frac{ \Esp{ \abs{ Z_{U_N, \pe{\rho N} }(1) }^{2k} } }{ N^{k^2} }  \tendvers{N}{+\infty } \Ze\Te_\rho^{(k)}
\end{align*}
with 
\begin{align}\label{EqPhi:TruncatedCharpolIn1}%$
\begin{aligned}
\Ze\Te_\rho^{(k)} & := \frac{(2\pi)^{k(k - 1) } }{k!} \! \int_{ \Rr^k } \!\! \Phi_{\!\Ze\Te_\rho^{(k)} }(0, x_2, \dots, x_{k + 1}) \Delta(0, x_2, \dots x_{k + 1})^2 d\xb\\
\!\!\Phi_{\!\Ze\Te_\rho^{(k)} }(0, x_2, \dots, x_{k + 1}) & = e^{ i \pi (\rho k - 2) \sum_{j = 2}^{k + 1} x_j } \, h_{\rho, \infty}(0, x_2, \dots, x_{k + 1})^k \\
 & \hspace{+2.3cm} \times \prth{ \vphantom{a^{a^a}} \widetilde{h}_{ 1, \infty}(0, x_2, \dots, x_{k + 1}) - \widetilde{h}_{\overline{\rho}, \infty}(0, x_2, \dots, x_{k + 1}) }^{\! k}
\end{aligned}
\end{align}
where $ \widetilde{h}_{\rho, \infty} $ is defined in \eqref{Def:hTildeInfty}.
\end{theorem}

% ==

\begin{remark}\label{Rk:ScZt}
One thus has $ \gamma_k = k^{k^2} \Ze\Te_{1/k}^{(k)} $. Note also that the integral is on $ \Rr^k $ and not on $ \Rr^{k - 1} $. 
\end{remark}

% ==

\begin{proof}
We have
\begin{align*}%$
Z_{U_N, T}(x) = \sum_{m = 0}^T x^m \crochet{s^m} Z_{U_N, T}(s) = \crochet{s^0} \prth{ \sum_{m = 0}^T (x s\inv)^m } Z_{U_N, T}(s)
\end{align*}

For $ X := \ensemble{x_1, \dots, x_k} $, define 
\begin{align}\label{Def:PT}%$
\begin{aligned}
P_T(x) & := \sum_{m = 0}^T x^m = \frac{1 - x^{T + 1} }{1 - x} = H\crochet{x - x^{T + 1}} = h_T\crochet{1 + X} \\
P_T\crochet{X } & := \prod_{j = 1}^k P_T(x_j)
\end{aligned}
\end{align}

Note that in one variable, $ P_\infty : (x, s) \mapsto \frac{1}{1 - x s\inv} $ is the Cauchy reproducing kernel on $ L^2(\Uu, \frac{d^*z}{z}) $, and in several variables, this is the reproducing kernel of the Hilbert space of square integrable \textit{tensorial functions} (i.e. functions of the form $ f(x_1, \dots, x_n) = f_1(x_1) \cdots $ $f_n(x_n) $). The operator $ \Pb_{\! T} $ of kernel $ (x, s) \mapsto P_T(x s\inv) $ has the same property on the space of tensorial polynomials of degree at most $T$. In case of a general Laurent polynomial $ f(z) = \sum_{k \in \Zz} f_k z^k $, it is the operator of projection on the vector space generated by $ 1, z, \dots z^T $, i.e. $ \Pb_{\! T} f(z) = \sum_{k = 0}^T f_k z^k $.

Using \eqref{Eq:SchurJointMoments}, we have for all $T \in \Nn^* $ with $ X := \ensemble{x_1, \dots, x_k} $ and $ Y := \ensemble{y_1, \dots, y_k} $
\begin{align*}%$
\Esp{ \abs{  Z_{U_N, T}(1)  }^{2k} } & := \Esp{ \abs{ \, \crochet{x^0} P_T(x\inv) Z_{U_N }(x)   }^{2k} } \\
            & = \crochet{X^0 Y^0 } P_T\crochet{X\inv + Y\inv } \int_{ \Ue_N } \prod_{j = 1}^k \det(I - x_j U_N ) \overline{\det(I - y_j U_N ) } dU \\
            & = \crochet{X^0 Y^0 } P_T\crochet{X\inv + Y\inv } (-1)^{kN} \prod_{j = 1}^k y_j^{ N} \int_{ \Ue_N }  \prod_{j = 1}^k \det(I - x_j U_N ) \\
            & \hspace{+8cm} \times \det(I - y_j\inv U_N ) \det(U)^{-k}dU \\
            & = \crochet{X^0 Y^{-N} } P_T\crochet{X\inv + Y\inv }  s_{N^k}\crochet{ X + Y\inv }  
\end{align*}

Using \eqref{FourierRepr:SchurRectangle} with $ U := \ensemble{ u_1, \dots, u_k } $ and changing $ Y\inv $ to $ Y $, one gets
\begin{align*}%$
\Esp{ \abs{ Z_{U_N, T}(1)  }^{2k} } & = \crochet{X^0 Y^N } P_T\crochet{X\inv + Y }  \frac{1}{k!} \oint_{ \Uu^k } U^{-N} H\crochet{ U  (X + Y ) }  \abs{\Delta(U) }^2 \frac{d^*U}{U} \\ 
             & = \frac{1}{k!} \oint_{ \Uu^k } U^{-N} \crochet{ X^0 } P_T\crochet{X\inv} H\crochet{XU} \crochet{ Y^N } P_T\crochet{Y } H\crochet{Y U } \abs{\Delta(U) }^2 \, \frac{d^*U}{U} \\
             & = \frac{1}{k!} \oint_{ \Uu^k } U^{-N} \prth{ \vphantom{a^{a^{a^a}}} \crochet{ x^0 } P_T\crochet{x\inv} H\crochet{xU} }^{\! k} \prth{  \vphantom{a^{a^{a^a}}} \crochet{ y^N } P_T\crochet{y } H\crochet{y U } }^{\! k} \abs{\Delta(U) }^2 \, \frac{d^*U}{U}
\end{align*}

By the projection property of $ \Pb_{\! T} $ and \eqref{FourierRepr:hLambda}, one has  
\begin{align*}%$
\crochet{ x^0 } P_T\crochet{x\inv} H\crochet{xU} & = \sum_{\ell = 0}^T h_\ell\crochet{U } \\
\crochet{ y^N } P_T\crochet{y } H\crochet{yU } & = \sum_{\ell = 0}^T  \crochet{ y^N } y^{\ell}   H\crochet{y U}  = \sum_{\ell = 0}^T  h_{N -\ell}\crochet{U} = \sum_{\ell = N - T}^N h_\ell\crochet{U}
\end{align*}

A classical formula is given by
\begin{align}\label{Eq:FiniteSumHomogeneous}%$
h_T\crochet{1 + U} = \sum_{m = 0}^T h_m(U)  
\end{align}

It is easily proven using lemma \ref{Lemma:ShallowCrucialLemma}~:
\begin{align*}%$
h_T\crochet{1 + U} & = \crochet{s^T} \frac{1}{1 - s} H\crochet{s U} = \crochet{s^T} \frac{1 - s^{T + 1} }{1 - s} H\crochet{s U } = \sum_{m = 0}^T \crochet{s^T} s^m  H\crochet{s U}   = \sum_{m = 0}^T h_{T - m}(U)
\end{align*}

As a result, one has
\begin{align*}%$
\crochet{ x^0 } P_T\crochet{x\inv} H\crochet{xU} & = h_T\crochet{1 + U } \\
\crochet{ y^N } P_T\crochet{y\inv } H\crochet{yU } & = h_{N }\crochet{1+ U} - h_{\overline{T} - 1 }\crochet{1 + U}
\end{align*}
and
\begin{align*}%$
\Esp{ \abs{ Z_{U_N, T}(1)  }^{2k} }  & = \frac{1}{k!} \oint_{ \Uu^k } U^{-N} h_T\crochet{1 + U }^k \prth{\vphantom{a^{a^{a}}} h_{N }\crochet{1+ U} - h_{\overline{T} - 1 }\crochet{1 + U} }^k  \abs{\Delta(U) }^2 \frac{d^*U}{U}  
\end{align*}

Setting $ u_j = e^{2 i \pi  x_j/N  } $ for all $ j \in \intcrochet{1, k} $ with $ x_j \in \crochet{ - \frac{N}{2}, \frac{N}{2} } $, one gets
\begin{align*}%$
\Esp{ \abs{ Z_{U_N, \pe{\rho N} } (1) }^{2k} } & = \frac{1}{k!} \int_{ \crochet{ - \frac{N}{2}, \frac{N}{2} }^k }   h_{\pe{\rho N} }(1 , e^{2 i \pi x_1/N}, \dots, e^{ 2i \pi x_k/N} )^k \\
               & \qquad \times \prth{ h_{N }(1 , e^{2 i \pi x_1/N}, \dots, e^{ 2i \pi x_k/N}) - h_{\pe{\overline{\rho} N} - 1 }(1 , e^{2 i \pi x_1/N}, \dots, e^{ 2i \pi x_k/N} )  }^k    \\
               & \qquad  \times  e^{ - 2 i \pi \sum_{j = 1}^k x_j } \abs{\Delta(1,  e^{2 i \pi (x_2 - x_1)/N}, \dots, e^{2 i \pi (x_k - x_1)/N} ) }^2  \prod_{j = 1}^k \frac{dx_j}{N} \\
% ============================================               
               & \equivalent{N\to +\infty} N^{2k^2 - k(k - 1) - k} \frac{(2\pi)^{k(k - 1)}}{k!} \int_{\Rr^k} e^{i \pi \rho k \sum_{j = 1}^k x_j } h_{\rho, \infty}(0, x_1, \dots, x_k)^k \\
               & \hspace{+4cm} \times \Big( \widetilde{h}_{1 , \infty}(0, x_1, \dots, x_k)  - \widetilde{h}_{\overline{\rho}, \infty}(0, x_1, \dots, x_k) \Big)^k \\
               & \hspace{+4cm} \times  e^{ - 2 i \pi \sum_{j = 1}^k x_j }  \abs{\Delta(0, x_2 - x_1 , \dots, x_k - x_1 ) }^2  \prod_{j = 1}^k dx_j 
\end{align*}

Here, we have used dominated convergence under the form of inequality \eqref{Ineq:SpeedOfConvergenceHinftyBis} with an integrable limiting function due to the second criteria \eqref{Ineq:IntegrabilityDomination} with $ (K, M, M', \kappa, \kappa') = (k - 1, k, k, 1, 1) $ which is clearly fullfilled for all $ p \in \ensemble{0, 1, 2} $ and $ k \geq 2 $ (one can expand the difference of functions using Newton's binomial formula for a strict application of the criteria).

We conclude by remarking that $ \Delta(0, x_2 - x_1 , \dots, x_k - x_1 ) = \Delta(x_1, \dots, x_k) $.
\end{proof}

% ==

\begin{remark}\label{Rk:Alternative:KS}
Taking $ \rho = 1 $ gives $ T = N $ and the previous computations can still be done. It gives an alternative expression to $ \widetilde{L}_1(k) $ given in \eqref{EqPhi:KS} as an integral over $ \Rr^k $ in place of $ \Rr^{k - 1} $ with a function $ h_{1, \infty}(0, x_1, \dots, x_k)^{2k} $. The equivalence of expression comes from formuli of symmetric functions for fixed $N$ and their rescaling.
\end{remark}

% ==

\begin{remark}\label{Rk:Alternative:Truncation}
One can also use an alternative ending to the previous proof by using \eqref{FourierRepr:SchurRectangleWithHn} instead of \eqref{FourierRepr:SchurRectangle}~:
\begin{align*}%$
\Esp{ \abs{ Z_{U_N, T}(1)  }^{2k} } & = \crochet{X^0 Y^N } P_T\crochet{X\inv + Y }  \frac{1}{k!} \oint_{ \Uu^k } U^{-N} h_{Nk}\crochet{ U  (X + Y ) }  \abs{\Delta(U) }^2 \frac{d^*U}{U} \\ 
             & =  \frac{1}{k!} \oint_{ \Uu^k } U^{-N} \prth{ \vphantom{a^{a^a}} \crochet{X^0 Y^N } P_T\crochet{X\inv + Y }h_{Nk}\crochet{ U  (X + Y ) } }  \abs{\Delta(U) }^2 \frac{d^*U}{U} \\
             & = \frac{1}{k!} \oint_{ \Uu^{k - 1} } \!\! V^{-N} \prth{ \vphantom{a^{a^a}} \crochet{X^0 Y^N } \! P_T\crochet{X\inv + Y } \! h_{Nk}\crochet{ (1 + V) (X + Y ) } } \! \abs{\Delta\crochet{1 + V} }^2 \frac{d^*V}{V} 
\end{align*}
where we have used the trick of remark~\ref{Rk:MultivariateFourierWithHomogeneity}.

Using $ v_j = e^{2 i \pi \nu_j/N} $, $ x_j = e^{2 i \pi \alpha_j/N} $, $ y_j = e^{2 i \pi \beta_j / N } $ and $ P_T(x) = h_T\crochet{1 + x} $, one sees that the integrand will rescale with a product of functions $ h_{\rho, \infty}(0, s_j) $ with $ s_j \in \ensemble{\alpha_j, \beta_j} $ and $ h_{k, \infty}\crochet{(0 + \nub) \oplus (\alphab + \betab) } $. Nevertheless, the critical step is always to have a domination to pass to the limit. The approach that led to lemma~\ref{Lemma:IntegrabilityDomination} and \eqref{Ineq:SpeedOfConvergenceHinfty} can be adapted for such a purpose, but as there is already a proof with the current domination, we leave these details to the interested reader. In the end, we find the following alternative expression to \eqref{EqPhi:TruncatedCharpolIn1} on $ \Rr^{k - 1} $
\begin{align}\label{EqPhi:TruncatedCharpolIn1Bis}%$
\begin{aligned}
\Ze\Te_\rho^{(k)} & := \frac{(2\pi)^{k(k - 1) } }{k!} \! \int_{ \Rr^{k - 1} } \!\! \widetilde{\Phi}_{\!\Ze\Te_\rho^{(k)} }(0, \xb) \Delta(0, \xb)^2 d\xb\\
\!\!\widetilde{\Phi}_{\!\Ze\Te_\rho^{(k)} }(0, \xb) & = e^{ -2 i \pi  \sum_{j = 2}^k x_j } \!\!  \int_{\Rr^{2k} } e^{-2 i \pi \sum_{j = 1}^k \beta_j} \prod_{j = 1}^k \widetilde{h}_{\rho, \infty}(0, \alpha_j)  \widetilde{h}_{\rho, \infty}(0, \beta_j) \\
           & \hspace{+6cm} \times \widetilde{h}_{k, \infty}\crochet{(0 + \xb) \oplus (\alphab + \betab) } d\alphab d\betab 
\end{aligned}
\end{align}
\end{remark}

\medskip

% ==
\subsubsection{The Heap-Lindqvist theorem}\label{Subsubsec:HeapLindqvist}

We now consider the case of $ Z_{U_N, \pe{\rho N} } \prth{ \lambda } $ with $ \lambda > 1 $~; this is the analogue of $ \sigma < \frac{1}{2} $ in \eqref{Eq:BehaviourRandomMultiplicativeFunction} since $ \Ze_{X, \sigma} $ converges in $ L^2(\Pp) $ and in law for $ \sigma > \frac{1}{2} $ and $ Z_{U_N, \pe{\rho N} } \prth{ \lambda } $ converges in law to the exponential of a Gaussian analytic function for $ \lambda < 1 $, using the Diaconis-Shahshahani theorem on traces of powers of $ U_N $ \cite{DiaconisShahshahani}. The Heap-Lindqvist theorem \cite[thm. 3]{HeapLindqvist} writes for $ \rho = k\inv $
\begin{align}\label{Eq:HeapLindqvistTruncated}%$
\frac{\Esp{ \abs{ Z_{U_N ,\, \pe{\rho N} }(\lambda) }^{2k} }}{ \lambda^{2k \pe{\rho N} } (\rho N)^{ (k - 1)^2 } } \tendvers{N}{ + \infty }  \gamma_k \times \frac{(2k - 2) ! }{ (k - 1)!^2 } \frac{ F_k(\lambda)  }{(1 - \lambda^{-2} )^{2k - 1} } 
\end{align} 
with $ F_k(\lambda) = {}_2 F_1\prth{ \begin{matrix} 1 - k, 1 - k \\ 2 - 2k \end{matrix} \bigg\vert 1 - \lambda^{-2} } $.

We see that such a behaviour in $ N^{(k - 1)^2} $ puts $ \lambda^{ -\pe{\rho N} } Z_{U_N ,\, \pe{\rho N} }(\lambda) $ in the universality class of $ \sc_{\pe{\rho N}}(U_N) $. The limiting factor is not $ \Se\Ce_{\rho}^{(k)} $ defined in \eqref{EqPhi:MidCoeff}, though. This shows that the simple approximation $ Z_{U_N, T}(\lambda) \approx (-\lambda)^T \sc_T(U_N) $ is not valid~: the terms of order $ T - 1, T-2, \dots, T-o(n) $ still interact at the level of moments. We now give a new proof of \eqref{Eq:HeapLindqvistTruncated} with the previous machinery that formalises this idea, i.e. with the randomisation paradigm of \S~\ref{Subsec:RandomisationParadigm}. The randomisation will indeed be particularly suited to understand how a mixture of mid-secular coefficients ends up with an additional factor at the limit~: 

% ==
\begin{theorem}[Truncated characteristic polynomial at $ \pe{\rho N} $ in $ \lambda > 1 $]\label{Theorem:TruncatedCharpolHeapLindqvist}
For all $ k \geq 2 $, one has
\begin{align}\label{Eq:LimTruncatedCharpolHeapLindqvist}%$
\frac{ \Esp{ \abs{ Z_{U_N ,\, \pe{\rho N} }(\lambda) }^{2k} } }{ \lambda^{ 2k \rho N } N^{(k - 1)^2} }  \tendvers{N}{+\infty }  {}_2 F_1\prth{ \begin{matrix} k, k \\ 1 \end{matrix} \bigg\vert \lambda^{-2} }  \times \Se\Ce_{\rho}^{(k)} 
\end{align}
where $ \Se\Ce_{\rho}^{(k)} $ is defined in \eqref{EqPhi:MidCoeff}.
\end{theorem}

% ==

\begin{proof}
Set $ T := \pe{\rho N} $. The beginning of the proof is similar to the previous one, except that we are using $ P_T(\lambda x\inv) $ in place of $ P_T(x\inv) $ for the tensorial Cauchy reproducing kernel. Adapting the steps, we easily arrive at
\begin{align*}%$
\Esp{ \abs{ Z_{U_N, T}(\lambda)  }^{2k} }  & = \frac{1}{k!} \oint_{ \Uu^k } U^{-N} \prth{ \vphantom{a^{a^{a^a}}} \crochet{ x^0 } \! P_T\crochet{ \lambda x\inv} H\crochet{xU} }^{\! k} \prth{  \vphantom{a^{a^{a^a}}} \crochet{ y^N } \! P_T\crochet{\lambda y } H\crochet{y U } }^{\! k} \abs{\Delta(U) }^2  \frac{d^*U}{U}
\end{align*}

% First formula
%
%
%
%\begin{align*}%$
%\crochet{ x^0 } \! P_T\crochet{\lambda x\inv} H\crochet{xU} & = \sum_{\ell = 0}^T \lambda^\ell h_\ell\crochet{U } = \sum_{\ell = 0}^T h_\ell\crochet{\lambda U} = h_T\crochet{1 + \lambda U} = \lambda^T h_T\crochet{\lambda\inv + U}
%\end{align*}
%%
%%
%%
%and
%%
%%
%%
%\begin{align*}%$
%\crochet{ y^N } \! P_T\crochet{ \lambda y } H\crochet{yU } & = \sum_{\ell = 0}^T \lambda^\ell \crochet{ y^0 } y^{-N + \ell} H\crochet{y U}  =  \sum_{\ell = 0}^T \lambda^\ell h_{N - \ell}\crochet{ U} = \lambda^N \sum_{\ell = 0}^T h_{N - \ell}\crochet{\lambda\inv U} \\
%                 & = \lambda^N \sum_{j = \overline{T} }^N h_j\crochet{\lambda\inv U} = h_{N }\crochet{ \lambda +  U} - \lambda^{T + 1} h_{\overline{T} - 1 }\crochet{\lambda + U}  
%\end{align*}

Note the following randomisation formula analogous to \eqref{RandomEq:Charpol}~:
\begin{align}\label{RandomEq:TruncatedCharpol}%$
Z_{U_N, T}(\lambda) = \prth{ \sum_{\ell = 0}^T \lambda^\ell } \Esp{ \sc_{ G_T(\lambda) }(U_N) \vert U_N },  \qquad G_T(\lambda) \sim \Geom_\lambda\prth{ \intcrochet{0, T} } 
\end{align}
where $ G_T(\lambda) \sim \Geom_\lambda\prth{ \intcrochet{0, T} } $ is a random variable with truncated geometric random variable, i.e. $ \Prob{ G_T(\lambda) = k } = \lambda^k H\crochet{ \lambda - \lambda^{T + 1} } \Unens{0 \leq k \leq T} $.

In the same way, the projection property of $ \Pb_{\! T} $, \eqref{FourierRepr:hLambda} and \eqref{Eq:FiniteSumHomogeneous} yield the following randomisation formula (after duality)
\begin{align*}%$
\crochet{ x^0 } \! P_T\crochet{\lambda x\inv} H\crochet{xU} & = \sum_{\ell = 0}^T \lambda^\ell h_\ell\crochet{U } =  \prth{ \sum_{\ell = 0}^T \lambda^\ell } \Esp{ h_{G_T(\lambda) }\crochet{U} }
\end{align*}

Note that replacing $ k $ by $ T - k $ in the definition of $ G_T(\lambda) $ implies the following equality in law~:
\begin{align}\label{EqLaw:GeomTrunc}%$
G_T(\lambda) \eqlaw T - G_T(\lambda\inv)
\end{align}

Note also that if $ q < 1 $, $ G_T(q) \to G_\infty(q) $ in law, where $ G_\infty(q) $ has the Geometric distribution given by $ \Prob{G_\infty(q) = k} = (1 - q) q^k $ for all $ k \geq 0 $.

\medskip

In the same vein, we get
\begin{align*}%$
\crochet{ y^N } \! P_T\crochet{ \lambda y } H\crochet{yU } & = \sum_{\ell = 0}^T \lambda^\ell \crochet{ y^0 } y^{-N + \ell} H\crochet{y U}  =  \sum_{\ell = 0}^T \lambda^\ell h_{N - \ell}\crochet{ U}  \\
                 & = H\crochet{ \lambda - \lambda^{T + 1} } \Esp{ h_{N - \widetilde{G}_T(\lambda)}\crochet{ U} }, \qquad \widetilde{G}_T(\lambda) \sim \Geom_\lambda\prth{ \intcrochet{0, T} }
\end{align*}

We moreover take $ (U_N, G_T(\lambda), G_T'(\lambda)) $ independent. Using \eqref{EqLaw:GeomTrunc}, we then arrive at
\begin{align*}%$
\Esp{ \abs{ Z_{U_N, T}(\lambda)  }^{2k} }  & = \frac{ 1 }{k!} H\!\crochet{ \lambda \!-\! \lambda^{T +1} }^{\! 2k}  \oint_{ \Uu^k } U^{-N} \prth{ \Esp{ h_{ \vphantom{\overline{T}} T - G_T(\lambda\inv) }\crochet{U} } }^k \\
                & \hspace{+6cm} \times \prth{ \Esp{ h_{\overline{T} + \widetilde{G}_T(\lambda\inv) }\crochet{U} } }^k  \abs{\Delta(U) }^2 \frac{d^*U}{U}  %\\
% =============================================================
%                & = \frac{\lambda^{k T} }{k!} \oint_{ \Uu^k } V^{-N} u_1^{-kN} \times u_1^{T k} h_T\crochet{ \lambda u_1\inv + 1 + V }^k  \\
%                & \hspace{+2cm} \prth{\vphantom{a^{a^{a}}}  u_1^N h_N \crochet{ \lambda u_1\inv + 1 + V} - \lambda^{T + 1} u_1^{\overline{T} - 1} h_{\overline{T} - 1 }\crochet{\lambda u_1\inv + 1 + V}  }^k  \\
%                & \hspace{+3cm} \abs{\Delta\crochet{1 + V} }^2 \frac{d^*V}{V} \frac{d^*u_1}{u_1} 
\end{align*}

Introducing two independent sequences of i.i.d. random variables $ (G_T^{(m)}(\lambda\inv))_{m \geq 1} $ and $ (\widetilde{G}_T^{(m)}(\lambda\inv))_{m \geq 1} $ %and setting $ S_k(T, \lambda) := \sum_{m = 1}^k G_T^{(m)}(\lambda\inv) $, $ \widetilde{S}_k(T, \lambda) := \sum_{m = 1}^k \widetilde{G}_T^{(m)}(\lambda\inv) $ 
yields
\begin{align*}%$
\Esp{ \abs{ Z_{U_N, T}(\lambda)  }^{2k} }  & = \frac{ H\!\crochet{ \lambda \!-\! \lambda^{T +1} }^{\! 2k} }{k!}   \Ee\Bigg(\! \oint_{ \Uu^k }\!\! U^{-N}\!\! \prod_{m = 1}^k \!  h_{ \vphantom{\overline{T}} T - G_T^{(m)}(\lambda\inv) }\crochet{U} h_{\overline{T} + \widetilde{G}_T^{(m)}(\lambda\inv) }\crochet{U}  \! \abs{ \Delta(U)  }^2 \frac{d^*U}{U} \Bigg) \\
% =======================================                
                & = \frac{ 1 }{k!} H\!\crochet{ \lambda \!-\! \lambda^{T +1} }^{\! 2k}  \Ee\Bigg( \oint_{ \Uu^k } V^{-N} u_1^{-kN} \prod_{m = 1}^k u_1^{T - G_T^{(m)}(\lambda\inv) } h_{ \vphantom{\overline{T}} T - G_T^{(m)}(\lambda\inv) }\crochet{1 + V}    \\
                & \hspace{+5cm} \times \prod_{m = 1}^k u_1^{\overline{T} + \widetilde{G}_T^{(m)}(\lambda\inv) } h_{ \overline{T} + \widetilde{G}_T^{(m)}(\lambda\inv) }\crochet{1 + V}  \\
                & \hspace{+5cm} \times \abs{\Delta\crochet{1 + V} }^2 \frac{d^*V}{V} \frac{d^*u_1}{u_1} \Bigg) \\
% =======================================                
                & = \frac{ 1 }{k!} H\!\crochet{ \lambda \!-\! \lambda^{T +1} }^{\! 2k}  \Ee\Bigg( \Unens{ \sum_{m = 1}^k G_T^{(m)}(\lambda\inv) = \sum_{m = 1}^k \widetilde{G}_T^{(m)}(\lambda\inv) } \\
                & \hspace{+2cm} \times \oint_{ \Uu^{k - 1} } V^{-N}  \prod_{m = 1}^k   h_{ \vphantom{\overline{T}} T - G_T^{(m)}(\lambda\inv) }\crochet{1 + V} h_{ \overline{T} + \widetilde{G}_T^{(m)}(\lambda\inv) }\crochet{1 + V}  \\ 
                & \hspace{+5cm} \times \abs{\Delta\crochet{1 + V} }^2 \frac{d^*V}{V}   \Bigg)
% =======================================  
%                & = \frac{\lambda^{k T} }{k!} \oint_{ \Uu^k } V^{-N} u_1^{-kN} \times u_1^{T k} h_T\crochet{ \lambda u_1\inv + 1 + V }^k  \\
%                & \hspace{+2cm} \prth{\vphantom{a^{a^{a}}}  u_1^N h_N \crochet{ \lambda u_1\inv + 1 + V} - \lambda^{T + 1} u_1^{\overline{T} - 1} h_{\overline{T} - 1 }\crochet{\lambda u_1\inv + 1 + V}  }^k  \\
%                & \hspace{+3cm} \abs{\Delta\crochet{1 + V} }^2 \frac{d^*V}{V} \frac{d^*u_1}{u_1} 
\end{align*}
where we have used the trick of remark~\ref{Rk:MultivariateFourierWithHomogeneity} to separate $ u_1 $ from $ V $. We note that this last formula also holds for $ \lambda = 1 $ if one uses $ G_T(1) $ which is uniform on $ \intcrochet{0, T} $ (as a result, the factor $ H\!\crochet{ \lambda  \! - \! \lambda^{ T + 1} } $ becomes $T + 1$). Of course, it also holds for $ \lambda < 1 $ but in this case, one should use \eqref{EqLaw:GeomTrunc} and $ G_T(\lambda) $ instead.

An easy adaptation of lemma~\ref{Lemma:BetaProbRepr} and \eqref{Ineq:SpeedOfConvergenceHinfty} gives that $ h_{\pe{\rho N} + o(N) }\crochet{ e^{X/N} } \sim N^{k - 1} \widetilde{h}_{\rho, \infty}(X) $ if $ X = \ensemble{x_1, \dots, x_k} $, and $ G_T(\lambda\inv) \to G_\infty(\lambda\inv) $ in law, hence is $ O_\Pp(1) $. Setting $ u_j = e^{2 i \pi  x_j/N  } $ for all $ j \in \intcrochet{1, k - 1} $ with $ x_j \in \crochet{ - \frac{N}{2}, \frac{N}{2} } $, one then gets
\begin{align*}%$
\Esp{ \abs{ Z_{U_N, \pe{\rho N} } (\lambda) }^{2k} } & = \frac{ 1 }{k!} H\!\crochet{ \lambda \!-\! \lambda^{T +1} }^{\! 2k}  \Ee\Bigg( \Unens{ \sum_{m = 1}^k G_T^{(m)}(\lambda\inv) = \sum_{m = 1}^k \widetilde{G}_T^{(m)}(\lambda\inv) } \\
                & \hspace{+2cm} \times \int_{ \crochet{ - \frac{N}{2}, \frac{N}{2} }^{k - 1} } e^{-2 i \pi \sum_{j = 2}^k x_j}  \prod_{m = 1}^k   h_{ \vphantom{\overline{T}} T - G_T^{(m)}(\lambda\inv) }\crochet{1 + e^{2 i \pi X/N} }   \\ 
                & \hspace{+3cm} \times  \prod_{m = 1}^k h_{ \overline{T} + \widetilde{G}_T^{(m)}(\lambda\inv) }\crochet{1 + e^{2 i \pi X/N} }  \abs{\Delta\crochet{1 + e^{2 i \pi X/N}} }^2 \frac{dX}{N^{k - 1}}   \Bigg) \\
% ============================================               
               & \equivalent{N\to +\infty} \frac{ \lambda^{2k T} }{k!} \frac{\lambda^{2k}}{(\lambda - 1)^{2k} } N^{2k^2 - k(k - 1) - k}  (2\pi)^{k(k - 1) }  \\
               & \qquad \times \Prob{  \sum_{m = 1}^k G_T^{(m)}(\lambda\inv) = \sum_{m = 1}^k \widetilde{G}_T^{(m)}(\lambda\inv) } \\
               & \qquad \times \int_{\Rr^{k - 1} } e^{i \pi (k - 2) \sum_{j = 1}^k x_j } h_{\rho, \infty}(0, \xb)^k h_{\overline{\rho}, \infty}(0, \xb)^k \Delta(0, \xb)^2 \, d\xb \\ 
% ============================================               
               & = N^{ (k - 1)^2 } \lambda^{2k T} \times  \frac{ 1 }{(1 - \lambda\inv)^{2k} } \, \Prob{  \sum_{m = 1}^k G_T^{(m)}(\lambda\inv) = \sum_{m = 1}^k \widetilde{G}_T^{(m)}(\lambda\inv) } \! \times \Se\Ce_\rho^{(k)}  
\end{align*}

Here, we have used dominated convergence given by inequality \eqref{Ineq:DominationSupersymWithoutSpeed} with $ Y = \emptyset $. As in the proof of theorem~\ref{Theorem:Autocorrelations}, to show that the limiting function is integrable we adapt the first criteria \eqref{Ineq:IntegrabilitySupersymDomination} with $ (K, L, M) = (k - 1, 0, 2k) $. More precisely, since the integration space is $ \Rr^{k - 1} $, we use the criteria \eqref{Ineq:IntegrabilitySupersymDominationBis} which is satisfied if $ (k - 1)(2k(k - 2) - k + 1) > 2k $. This is equivalent to $ (k - 1)(2k^2 - 3k - 3) > 2 $, which is clearly satisfied for all $ k \geq 2 $.

\medskip

It remains to compute the given probability and show that it writes as a $ {}_2 F_1 $. We have
\begin{align*}%$
\Prob{  \sum_{m = 1}^k G_T^{(m)}(\lambda\inv) = \sum_{m = 1}^k \widetilde{G}_T^{(m)}(\lambda\inv) }  & = \crochet{ t^0 } \Esp{ t^{ \sum_{m = 1}^k G_T^{(m)}(\lambda\inv) - \sum_{m = 1}^k \widetilde{G}_T^{(m)}(\lambda\inv) } } \\
              & = \crochet{ t^0 } \Esp{ t^{ G_T (\lambda\inv) } }^k \Esp{ t^{ -G_T (\lambda\inv) } }^k \\
              & =  \crochet{ t^0 } H\crochet{ (t - 1) \lambda\inv 1^{\plusInOne k} } H\crochet{ (t\inv - 1) \lambda\inv 1^{\plusInOne k} } \\
              & = H\crochet{-\lambda\inv 1^{\plusInOne k} }^2 \crochet{ t^0 } \sum_{\ell_1, \ell_2 \geq 0} t^{\ell_1 - \ell_2} \lambda^{- (\ell_1 + \ell_2) } h_{\ell_1}\crochet{ 1^{\plusInOne k} } h_{\ell_2}\crochet{ 1^{\plusInOne k} } \\
              & = \prth{ 1 - \lambda\inv }^{ 2k} \sum_{\ell \geq 0} \lambda^{-2\ell} h_\ell\crochet{1^{\plusInOne k} }^2
\end{align*}

Now, $ h_\ell\crochet{1^{\plusInOne \theta} } = \crochet{t^\ell} H\crochet{t}^\theta = \crochet{t^\ell}\sum_{m \geq 0} t^m \frac{\theta^{\uparrow m}}{m! } = \frac{\theta^{\uparrow \ell}}{\ell !} $ with $ \theta^{\uparrow \ell} := \theta(\theta + 1) \cdots (\theta + \ell - 1) $. This implies
\begin{align*}%$
\prth{ 1 - \lambda\inv }^{ -2k} \Prob{  \sum_{m = 1}^k G_T^{(m)}(\lambda\inv) = \sum_{m = 1}^k \widetilde{G}_T^{(m)}(\lambda\inv) } & = \sum_{\ell \geq 0} \lambda^{-2\ell} \prth{ \frac{k^{\uparrow \ell}}{\ell !} }^2  = \sum_{\ell \geq 0} \frac{k^{\uparrow \ell} \times k^{\uparrow \ell}}{1^{\uparrow \ell}} \frac{\lambda^{-2\ell} }{\ell !}  \\
              & =: {}_2 F_1\prth{ \begin{matrix} k, k \\ 1 \end{matrix} \bigg\vert \lambda^{-2} }
\end{align*}
since by definition $ {}_2 F_1 \prth{ \begin{smallmatrix} a, b \\ c \end{smallmatrix} \vert z } := \sum_{\ell \geq 0} \frac{a^{\uparrow \ell}   b^{\uparrow \ell}}{c^{\uparrow \ell}} \frac{z^\ell }{\ell !} $.  
\end{proof}

% ==

\begin{remark}\label{Rk:TruncatedMomentsWithRandomisation}
The formula
\begin{align}\label{Eq:RandomisationForTruncatedMoments}%$
\begin{aligned}
\Esp{ \abs{ Z_{U_N, \pe{\rho N} } (\lambda) }^{2k} }  & = \frac{ H\!\crochet{ \lambda \!-\! \lambda^{T +1} }^{\! 2k}  }{k!}  \, \Ee\Bigg( \Unens{ \sum_{m = 1}^k G_T^{(m)}(\lambda ) = \sum_{m = 1}^k \widetilde{G}_T^{(m)}(\lambda ) } \oint_{ \Uu^{k - 1} } V^{-N} \\
                & \hspace{+1cm} \times \!\!   \prod_{m = 1}^k \!  h_{  G_T^{(m)}(\lambda ) }\crochet{1 + V} h_{ N - \widetilde{G}_T^{(m)}(\lambda ) }\crochet{1 + V}  \abs{\Delta\crochet{1 + V} }^2 \frac{d^*V}{V}   \Bigg)
\end{aligned}
\end{align}
is valid for all $ \lambda \in \Rr_+^* $. In the case $ \lambda = 1 $, one uses the fact that $ G_T(1) \sim \Us(\intcrochet{0, T}) $. This is the very discrepancy of behaviour of these truncated Geometric random variables that implies the discrepancy of behaviour of the whole functional, as $ G_T(1)/T \to U \sim \Us([0, 1]) $ (in law), $ G_T(q) \to G_\infty(q) $ if $ q < 1 $ and $ G_T(\lambda)/T \eqlaw 1 - G_T(1/\lambda)/T \to 1 $ if $ \lambda > 1 $. Adapting the previous proof for $ \lambda = 1 $ shows that 
\begin{align*}%$
\Esp{ \abs{ Z_{U_N, \pe{\rho N} } (1) }^{2k} } & = \frac{ (T + 1)^{ 2k} }{k!}   \Ee\Bigg( \Unens{ \sum_{m = 1}^k G_T^{(m)}(1) = \sum_{m = 1}^k \widetilde{G}_T^{(m)}(1) } \\
                & \hspace{+2cm} \times \int_{ \crochet{ - \frac{N}{2}, \frac{N}{2} }^{k - 1} } e^{-2 i \pi \sum_{j = 2}^k x_j}  \prod_{m = 1}^k   h_{ \vphantom{\overline{T}} T - G_T^{(m)}(1) }\crochet{1 + e^{2 i \pi X/N} }   \\ 
                & \hspace{+3.5cm} \times  \prod_{m = 1}^k h_{ \overline{T} + \widetilde{G}_T^{(m)}(1) }\crochet{1 + V}  \abs{\Delta\crochet{1 + e^{2 i \pi X/N}} }^2 \frac{dX}{N^{k - 1}}   \Bigg)   
\end{align*}

To pursue, we need the following local CLT : writing $ U_T := \pe{(T + 1) U} $ for $ U \sim \Us([0, 1]) $, noticing that $ T - U_T \eqlaw U_T $ (as in the case of the truncated Geometric random variable) and introducing a sequence $ (U_k)_k $ of i.i.d. uniform such random variables, one gets
\begin{align*}%$
\Prob{  \sum_{m = 1}^k G_T^{(m)}(1) = \sum_{m = 1}^k \widetilde{G}_T^{(m)}(1) } & =  \Prob{  \sum_{m = 1}^{2k} \pe{(T + 1) U_m} = k T } \\
                 & = \crochet{u^{kT} } \Esp{ u^{U_T} }^{2k} =  \crochet{ u^{kN} } \prth{ \frac{1}{T + 1} H\crochet{u - u^{T + 1} }  }^{2k}  \\
                 & = \frac{1}{(T + 1)^{2k} } \int_{-T}^T e^{-2 i \pi k \theta}  \prth{ \frac{1 - e^{2 i \pi \theta (T+1)/T } }{ 1 - e^{2 i \pi \theta T } } }^{2k} \frac{d\theta}{T} \\
                 & \equivalent{N \to + \infty} \, \frac{1}{\rho N} \, h_{k\rho, \infty}^{(2k)} (0) = \frac{1}{\rho N} \int_\Rr e^{ -2 i \pi k \theta } \, \Esp{ e^{2 i \pi \theta U} }^{2k} d\theta \\
                 &  = \frac{1}{\rho N} \int_\Rr e^{ -2 i \pi k \theta } \, \Esp{ e^{2 i \pi \theta \sum_{m = 1}^{2k} U_m} } d\theta \\
                 & = \frac{1}{\rho N} f_{V_k}(k), \qquad V_k := \sum_{m = 1}^{2k} U_m
\end{align*}

The only step to justify is the equivalent given in the proof of lemma~\ref{Lemma:RescaledGegenbauer}. Here, $ f_{V_k} $ is the Lebesgue density of $ V_k $, i.e. $ f_{V_k}(x) = \Esp{ \delta_0(V_k - x) } $. 

Nevertheless, since $ T\inv G_T(1) \to U $ (in law), this limiting random variable is not deterministic and one needs to keep the expectation. One thus gets
\begin{align*}%$
\Esp{ \abs{ Z_{U_N, \pe{\rho N} } (1) }^{2k} } &  \equivalent{N\to +\infty} \frac{ (\rho N)^{2k } }{k!} N^{2k^2 - k(k - 1) - k}  (2\pi)^{ k(k - 1) } \frac{1}{\rho N}  \\ 
               & \qquad \times \int_{\Rr^{k - 1} }  \Esp{ \delta_0\prth{ \sum_{\ell = 1}^{2k} U_\ell  - k } \prod_{\ell = 1}^k \widetilde{h}_{\rho - U_\ell, \infty}(0, \xb)  \widetilde{h}_{\overline{\rho} + U_{k + \ell }, \infty}(0, \xb)  } \\
               & \hspace{+3cm} e^{i \pi (k - 2) \sum_{j = 1}^k x_j } \Delta(0, \xb)^2 \, d\xb                 
\end{align*}

We see that we recover the scaling in $ N^{k^2} $ and that the discrepancy with the case $ \lambda > 1 $ comes from the difference of renormalisation between the truncated geometric law and the uniform one. This moreover gives an alternative integral on $ \Rr^{k - 1} $ for $ \Ze\Te_\rho^{(k)} $, i.e.
\begin{align}\label{EqPhi:Random:TruncatedCharpolIn1}%$
\begin{aligned}
\Ze\Te_\rho^{(k)} &\! := \frac{(2\pi)^{k(k - 1) } }{k!} \! \int_{ \Rr^{k - 1} } \!\! \Phi^{(Rand)}_{\!\Ze\Te_\rho^{(k)} }(0, \xb) \Delta(0,\xb)^2 d\xb\\
\!\!\!\!\!\!\Phi^{(Rand)}_{\!\Ze\Te_\rho^{(k)} }\!(0, \xb) &\! = e^{ i \pi ( k - 2) \! \sum_{j = 2}^k x_j } \rho^{2k -\! 1}\! f_{V_k}\!(k) \,  \Ee\Bigg( \! \prod_{\ell = 1}^k \! \widetilde{h}_{\rho - U_\ell, \infty} (0, \xb) \widetilde{h}_{\overline{\rho} + U_{k + \ell }, \infty}(0, \xb) \Bigg\vert \! \sum_{\ell = 1}^{2k} U_\ell = k \! \Bigg)  
\end{aligned}
\end{align}

The function $ \Phi^{(Rand)}_{\!\Ze\Te_\rho^{(k)} } $ uses nevertheless $ 2k - 1 $ random variables, hence $ 2k - 1 $ additional integrals on $ [0, 1] $, and showing directly the equivalence between \eqref{EqPhi:TruncatedCharpolIn1} and \eqref{EqPhi:Random:TruncatedCharpolIn1} does not seem an easy task. Last, note that the expression of $ f_{V_k}(k) $ given by \eqref{Eq:BatesIrwinHall} does not simplify and that one can express $ \Phi^{(Rand)}_{\!\Ze\Te_\rho^{(k)} } $ as a Fourier transform in the vein of \eqref{EqPhi:Random:AutocorrelsAsFourier}.
\end{remark}

% ==

\begin{remark}
In view of \eqref{Eq:HeapLindqvistTruncated}, \eqref{Eq:LimTruncatedCharpolHeapLindqvist} and remark~\ref{Rk:ScZt}, one has for $ \rho = k\inv $ 
\begin{align*}%$
{}_2 F_1\prth{ \begin{matrix} k, k \\ 1 \end{matrix} \bigg\vert \lambda^{-2} \! }     \, \Se\Ce_{\rho}^{(k)}  = \Ze\Te_{\rho}^{(k)}  \,   \frac{(2k - 2) ! }{ (k - 1)!^2 }  (1 - \lambda^{-2} )^{1 - 2k} \, {}_2 F_1\prth{ \begin{matrix} 1 - k, 1 - k \\ 2 - 2k \end{matrix} \bigg\vert 1 - \lambda^{-2} }      
\end{align*}

Now, the hypergeometric equation having one solution given by $ {}_2 F_1\prth{ \begin{smallmatrix} a, b \\ c \end{smallmatrix} \big\vert z  } $ is 
\begin{align*}%$
z(1 - z) F''(z) + [c - (a + b + 1)z] \, F'(z) - ab \, F(z) = 0     
\end{align*}
and setting\footnote{This is a simple case of Kummer's group of 24 transformations (isomorphic to $ \Sg_4 $) that one can treat by hand, see \cite[\S~8 p. 52, (7)]{KummerHypergeom}.} $ z = 1 - x $ shows that there are two solutions to the so transformed equation, namely $ {}_2 F_1\prth{ \begin{smallmatrix} a, b \\ 1 + a + b - c \end{smallmatrix} \big\vert 1 - x  } $ and $ (1 - x)^{c - a - b} {}_2 F_1\prth{ \begin{smallmatrix} c - a, c - b \\ 1 - a - b + c \end{smallmatrix} \big\vert 1 - x  } $. The value in 0 for the function implies in our case that only the second solution is relevant, hence that, for $ x = \lambda^{-2} $, 
\begin{align*}%$
{}_2 F_1\prth{ \begin{matrix} k, k \\ 1 \end{matrix} \bigg\vert x \! } =  (1 - x )^{1 - 2k} \, {}_2 F_1\prth{ \begin{matrix} 1 - k, 1 - k \\ 2 - 2k \end{matrix} \bigg\vert 1 - x }   
\end{align*}
which shows that (at least for $ \rho = 1/k $) there exists a proportionality relation between $ \Se\Ce_{\rho}^{(k)}  $ and $ \Ze\Te_{\rho}^{(k)} $~:
\begin{align}\label{Eq:LinkScZt}%$
\Se\Ce_{\rho}^{(k)}  = \Ze\Te_{\rho}^{(k)}  \,   \frac{(2k - 2) ! }{ (k - 1)!^2 }         
\end{align}

This would be interesting to know if such a link is persistent for all $ \rho \in (0, 1) $. The polytope method seems particularly adapted to do so (but we do not pursue here).
\end{remark}

\medskip

% ==
\subsubsection{The microscopic setting}\label{Subsubsec:TruncatedMicro}

This part is a variation on \S~\ref{Subsubsec:ConreyGamburd} ; details are thus omitted.

\begin{theorem}[Truncated characteristic polynomial at $ \pe{\rho N} $ in $ e^{s/N} $]\label{Theorem:TruncatedCharpolMicro}
For all $ k \geq 2 $ and $ s \in \Rr $, one has
\begin{align*}%$
\frac{ \Esp{ \abs{ Z_{U_N ,\, \pe{\rho N} }( e^{s/N} ) }^{2k} } }{ N^{\color{black} k^2 \color{black} } }  \tendvers{N}{+\infty } \Ze\Te_{\rho, s}^{(k)}
\end{align*}
where 
\begin{align}\label{EqPhi:TruncatedCharpolMicro}%$
\begin{aligned}
\Ze\Te_{\rho, s}^{(k)} := & \ \frac{(2\pi)^{k(k - 1) } }{k!} \! \int_{ \Rr^k } \!\! \Phi_{\!\Ze\Te_{\rho, s}^{(k)} }(0, x_2, \dots, x_{k + 1}) \Delta(0, x_2, \dots x_{k + 1})^2 d\xb\\
\!\!\Phi_{\!\Ze\Te_{\rho, s}^{(k)} }(0, x_2, \dots, x_{k + 1}) = & \  e^{2ks} e^{ i \pi (\rho k - 2) \sum_{j = 2}^{k + 1} x_j } \, h_{\rho, \infty}(0, x_2 \!+\! s, \dots, x_{k + 1} \!+\! s)^k \\
         &  \ \, \prth{ \vphantom{a^{a^a}} \widetilde{h}_{ 1, \infty}(0, x_2 \!-\! s, \dots, x_{k + 1} \!-\! s)\! - \! \widetilde{h}_{\overline{\rho}, \infty}(0, x_2 \!-\! s, \dots, x_{k + 1} \!-\! s) }^{\! k}
\end{aligned}
\end{align}
\end{theorem}

% ==

\begin{proof}
The proof is a slight modification of the proof of theorem~\ref{Theorem:TruncatedCharpolIn1}. We start from the formula
\begin{align*}%$
\Esp{ \abs{ Z_{U_N, T}(\lambda)  }^{2k} }  & = \frac{1}{k!} \oint_{ \Uu^k } U^{-N} \prth{ \vphantom{a^{a^{a^a}}} \crochet{ x^0 } \! P_T\crochet{ \lambda x\inv} H\crochet{xU} }^{\! k} \prth{  \vphantom{a^{a^{a^a}}} \crochet{ y^N } \! P_T\crochet{\lambda y } H\crochet{y U } }^{\! k} \abs{\Delta(U) }^2  \frac{d^*U}{U}
\end{align*}

Now,
\begin{align*}%$
\crochet{ x^0 } \! P_T\crochet{\lambda x\inv} H\crochet{xU} & = \sum_{\ell = 0}^T \lambda^\ell h_\ell\crochet{U } = \sum_{\ell = 0}^T h_\ell\crochet{\lambda U} = h_T\crochet{1 + \lambda U}  
\end{align*}
and
\begin{align*}%$
\crochet{ y^N } \! P_T\crochet{ \lambda y } H\crochet{yU } & = \sum_{\ell = 0}^T \lambda^\ell \crochet{ y^0 } y^{-N + \ell} H\crochet{y U}  =  \sum_{\ell = 0}^T \lambda^\ell h_{N - \ell}\crochet{ U} = \lambda^N \sum_{\ell = 0}^T h_{N - \ell}\crochet{\lambda\inv U} \\
                 & = \lambda^N \sum_{j = \overline{T} }^N h_j\crochet{\lambda\inv U} = \lambda^N \prth{  h_{N }\crochet{ 1 + \lambda\inv U} - h_{\overline{T} - 1 }\crochet{1 + \lambda\inv U}  }
\end{align*}

Setting $ \lambda = e^{s/N} $, an easy adaptation of the end of the proof of theorem~\ref{Theorem:TruncatedCharpolIn1} allows then to conclude.
\end{proof}

% ==

\begin{remark}
The previous proof is an adaptation of the proof of theorem~\ref{Theorem:TruncatedCharpolIn1}. An alternative proof with the randomisation paradigm of \S~\ref{Subsec:RandomisationParadigm} can also be given by adapting remark~\ref{Rk:TruncatedMomentsWithRandomisation} and the randomisation formula \eqref{Eq:RandomisationForTruncatedMoments}. This last formula can be rescaled similarly to the uniform case with the limit in law $ T\inv G_T(e^{-s/T}) \to X_s $ if $ s > 0 $ (in the contrary case, one uses the duality formula \eqref{EqLaw:GeomTrunc} and the limit law is then $ 1 - X_s $). The law of $ X_s $ is given by an exponential bias of the uniform distribution, i.e. $ \Esp{ f(X_s) } = \Esp{ e^{-s U} f(U) }/ \Esp{ e^{-s U}} $ with $ U \sim \Us([0, 1]) $ (or equivalently an exponential random variable $ s\inv \ee $ conditionned to be less than $1$). It gives an alternative expression to \eqref{EqPhi:TruncatedCharpolMicro} in the same vein as \eqref{EqPhi:Random:Autocorrels}/\eqref{EqPhi:Random:TruncatedCharpolIn1} with $ U_i $s replaced by i.i.d. copies of $ X_s $. Details are left to the reader.
\end{remark}

% ==

\begin{remark}
One can easily adapt the previous proof to show the following ``multipoint'' version of the previous result~:
\begin{align*}%$
\frac{ \Esp{  \prod_{\ell = 1}^k Z_{U_N ,\, \pe{\rho_\ell N} }( e^{s_\ell/N} ) \overline{ Z_{U_N ,\, \pe{\rho_{\ell + k} N} }( e^{w_\ell/N} ) } } }{ N^{ k^2 } }  \tendvers{N}{+\infty } \Ze\Te_{\rho_1, \dots, \rho_{2k}}^{(k)}(\sbb, \wb)
\end{align*}

The limiting function is a slight modification of \eqref{EqPhi:TruncatedCharpolMicro} where the powers $k$ are replaced by a product ; it is to be compared with \eqref{EqPhi:Autocorrels}. One can also compute it with the randomisation paradigm of \S~\ref{Subsec:RandomisationParadigm}. Details are left to the reader.
\end{remark}

%\medskip
\newpage
% ==
\subsection{The Birkhoff polytope}\label{Subsec:Birkhoff}

% ==
\subsubsection{The Beck-Pixton approach to the Birkhoff polytope}\label{Subsubsec:BirkhoffBeckPixton}

Let $ D \subset \Rr^n $ be a lattice and $ P $ a polytope with the property that all its vertices are points of $L$. For all $ t \in \Nn $, the \textit{Ehrhart polynomial} of $P$ is defined by 
\begin{align}\label{Def:EhrhartPol}%$
L(t, P) := \#(D \cap tP)
\end{align}
where $ tP $ is the $t$-dilation of $P$. $ L(t, P) $ is thus the number of lattice points contained in $tP$. Ehrhart \cite{Ehrhart62} showed that this is a polynomial in $t$ with rational coefficients. Using the residue theorem, Beck \cite{BeckResidue} gave an expression of $ L(t, P) $ and re-proved this fact. Beck and Pixton finally applied the residue formula to the case of the Birkhoff polytope \cite{BeckPixton}. We now review this approach. 

A convex lattice polytope $P$ located in the nonnegative orthant can be described by an intersection of a finite number of half-spaces of the form
\begin{align*}%$
P := \ensemble{ \xb \in \Rr_+^n \ / \ \Ab \xb \leq \bb }
\end{align*}
where $ \Ab $ is an $ m \times n $ matrix with integer coefficients and $ \bb \in \Zz^n $. More rarely, this last intersection can also be described by an equality. Denote by $ (\ab_1, \dots, \ab_n) $ the colums of $ \Ab $. Then, the Beck formula writes \cite[(11) \& thm. 8]{BeckResidue} 
\begin{align}\label{Eq:BeckFormula}%$
L(t, P) = \oint_{ (r\Uu)^m } \prod_{ \ell = 1 }^n \frac{1}{1 - \zb^{ \ab_\ell } } \prod_{ j = 1 }^m \frac{ z_j^{-t b_j} }{1 - z_j } \frac{d^* \zb}{\zb}, \qquad r < 1
\end{align}
or, equivalently
\begin{align}\label{Eq:BeckFormulaBis}%$
L(t, P) = \crochet{ \zb^{t \bb} } H\crochet{ \zb } \prod_{ \ell = 1 }^n \frac{1}{1 - \zb^{ \ab_\ell } } 
\end{align}

\begin{definition}[Birkhoff and sub-Birkhoff polytopes]\label{Def:BirkhoffPolytope}
The \textit{Birkhoff polytope} $ \Be_k $ is the polytope of stochastic matrices, namely matrices $ (M_{i, j})_{1 \leq i, j \leq k} $ satisfying $ M_{i, j} \geq 0 $, for all $i, j$, $ \sum_{i = 1}^k M_{i, j} = 1 $ for all $i$ and $ \sum_{j = 1}^k M_{i, j} = 1 $ for all $j$. This is a polytope of $ \Rr^{ (k - 1)^2 } $ described by equalities. 

The \textit{sub-Birkhoff polytope} $ \Se_k $ is the polytope of sub-stochastic matrices $ (M_{i, j})_{1 \leq i, j \leq k} $ i.e. matrices satisfying $ M_{i, j} \geq 0 $, for all $i, j$, $ \sum_{i = 1}^k M_{i, j} \leq 1 $ for all $i$ and $ \sum_{j = 1}^k M_{i, j} \leq 1 $ for all $j$ (see e.g. \cite{DiaconisGamburd, ConreyGamburd}). This is a polytope of $ \Rr^{k^2 } $.

In each case, the corresponding matrix $ \Ab $ of size $ 2k \times k^2 $ is given by \cite{BeckPixton, ConreyGamburd}
\begin{align*}%$
\Ab = \prth{  \begin{array}{cccccccccccccc} 
               1 & \cdots & 1                                                  \\ 
                 &        &    & 1 & \cdots & 1                                \\  
                 &        &    &   &        &    &   \ddots &                  \\ 
                 &        &    &   &        &    &          &  1 & \cdots & 1  \\ 
               1 &        &    & 1 &        &    &          &  1               \\ 
	             & \ddots &    &   & \ddots &    &  \cdots  &    & \ddots      \\ 
		         &        &  1 &   &        & 1  &          &    &        & 1 
              \end{array}    }
\end{align*}

\end{definition}

As a consequence, using two sets of variables $ \zb = (x_1, \dots, x_k, y_1, \dots, y_k) =: (X, Y) $ the Beck formula \eqref{Eq:BeckFormulaBis} becomes (see also \cite[(35)]{ConreyGamburd}) 
\begin{align*}%$
L(t, \Se_k) = \oint_{ (r\Uu)^{2k} } \prod_{ j, \ell = 1 }^k \frac{1}{1 - x_j y_\ell } \prod_{ j = 1 }^k \frac{ x_j^{-t }  y_j^{-t } }{ (1 - x_j)(1 - y_j) } \frac{d^* X}{X} \frac{d^* Y}{Y}, \qquad r < 1
\end{align*}
namely
\begin{align}\label{Eq:BeckFormulaSubBirkhoff}%$
\boxed{L(t, \Se_k) = \crochet{ X^t Y^t } H\crochet{ X + Y + XY }  }
\end{align}

In the case of a polytope defined by an equality, the Beck formula has to be modified accordingly (see \cite{BeckPixton} ; see also \cite[thm. 2.13, ch. 6.2]{BeckRobins}) and amounts to forget the term $ H\crochet{\zb} $ in \eqref{Eq:BeckFormulaBis}. In the particular case of the Birkhoff polytope, this gives the following Beck-Pixton formula:
\begin{align}\label{Eq:BeckFormulaBirkhoff}%$
\boxed{L(t, \Be_k) = \crochet{ X^t Y^t } H\crochet{ XY }  }
\end{align}

\medskip
% ==

A fundamental result of Ehrhart concerns the dominant coefficient of the Ehrhart polynomial. It is the relative volume of the underlying polytope (see e.g. \cite{BeckResidue, BeckPixton, ConreyGamburd}). Writing this coefficient as a limit of the rescaled polynomial gives\footnote{One has to divide by the relative volume of the fundamental domain of the sublattice of $ \Zz^{k^2} $ in the affine space spanned by $ \Be_k $ which is $ k^{k - 1} $, see \cite{BeckPixton}.}
\begin{align*}%$
\frac{ L(t, \Be_k) }{ t^{(k - 1)^2} } \tendvers{t }{+\infty} \frac{\vol(\Be_k)}{ k^{k - 1} }, \qquad \frac{ L(t, \Se_k) }{ t^{k^2} } \tendvers{t }{+\infty} \frac{\vol(\Se_k)}{ k^{k - 1} }
\end{align*}

Conrey and Gamburd \cite[(32)]{ConreyGamburd} showed using the results of Diaconis-Gamburd \cite{DiaconisGamburd} that
\begin{align*}%$
\Esp{ \abs{Z_{U_N, \ell}(1) }^{2k} } = L(\ell, \Se_k) \quad \forall N \geq \ell k
\end{align*}
hence, that
\begin{align}\label{Eq:ConreyGamburdLimit}%$
\frac{ \Esp{ \abs{Z_{U_{\ell k}, \ell }(1) }^{2k} } }{ \ell^{k^2} } \tendvers{\ell }{+\infty} \frac{\vol(\Se_k)}{ k^{k - 1} }
\end{align}

\begin{remark}
Similarly, Heap and Lindqvist \cite[Prop. 3]{HeapLindqvist} showed using the same results of Diaconis-Gamburd \cite{DiaconisGamburd} that for all $ N \geq \ell k $ 
\begin{align*}%$
\Esp{ \abs{Z_{U_N, \ell}(\lambda) }^{2k} } =   \crochet{ X^t Y^t } H\crochet{ X + Y + \lambda^2 XY }
\end{align*}
although they did not link their result with any polytopial functional. In view of the Beck formulas \eqref{Eq:BeckFormula}/\eqref{Eq:BeckFormulaBis}, it does not enter into the framework of a classical polytope. It would be interesting to find a polytopial interpretation of such a quantity.
\end{remark}

We already gave an expression of the truncated characteristic polynomial in the proof of theorem \ref{Theorem:TruncatedCharpolIn1}. Nevertheless, we give here yet another expression starting directly from the Beck formula \eqref{Eq:BeckFormulaSubBirkhoff}.

% ==
\subsubsection{Another expression for $ L(\ell, \Be_k) $ and $ L(\ell, \Se_k) $}

\begin{lemma}[The Ehrhart polynomial of $ \Be_k $ and $ \Se_k $ with duality]\label{Lemma:EhrhartBirkhoff}
One has with $ U := (u_1,\dots, u_k) $ and for all $ t \in \Nn $
\begin{align}\label{Eq:MomentsMacroTruncatedWithDuality}%$
\begin{aligned}
L(t, \Se_k) & =  \frac{1}{k!} \oint_{ \Uu^k } \abs{ h_t\prth{1, u_1, \dots, u_k} }^{2k} \abs{ \Delta(u_1, \dots, u_k) }^2 \frac{d^*U}{U} \\
L(t, \Be_k) & =  \frac{1}{k!} \oint_{ \Uu^k } \abs{ h_t\prth{ u_1, \dots, u_k} }^{2k} \abs{ \Delta(u_1, \dots, u_k) }^2 \frac{d^*U}{U}
\end{aligned}
\end{align}
\end{lemma}

\begin{proof}
The reproducing kernel property \eqref{Eq:ReproducingKernelProperty} of the Cauchy kernel applied to itself gives $ \bracket{H\crochet{X U}, H\crochet{Y U} }_U = H\crochet{XY}  $, i.e. the trick~\eqref{Trick:Hreproducivity} $ H\crochet{XY} = \crochet{U^0}H\crochet{ XU + YU\inv - U^{\varepsilon R} } $. One can thus write \eqref{Eq:BeckFormulaSubBirkhoff} as
\begin{align*}%$
L(t, \Se_k) & = \crochet{ X^t Y^t } H\crochet{ X + Y + XY }   \\
            & = \crochet{ U^0 X^t Y^t } H\crochet{ X + Y + XU + YU\inv - U^{\varepsilon R} }  \\
            & = \crochet{ U^0 X^t Y^t } H\crochet{ X(U + 1) + (U\inv + 1)Y - U^{\varepsilon R} } \\
            & = \crochet{ U^0 } \crochet{ X^t } H\crochet{ (U + 1)X } \crochet{ Y^t } H\crochet{ (U\inv + 1)Y }  H\crochet{ - U^{\varepsilon R} }
\end{align*}

Now, $ h_t\crochet{U + 1}^k = \crochet{ X^t } H\crochet{ (U + 1)X } $ and using the fact that $ U\inv = \overline{U} $ if $ u_i \in \Uu $, one has
\begin{align*}%$
L(t, \Se_k) & = \crochet{ U^0 } h_t\crochet{U + 1}^k h_t\crochet{U\inv + 1}^k  H\crochet{ - U^{\varepsilon R} } \\
            & = \oint_{\Uu^k } h_t\crochet{U + 1}^k \overline{h_t\crochet{U + 1}^k } H\crochet{ - U^{\varepsilon R} } \frac{d^*U}{U} \\
            & = \frac{1}{k!} \oint_{ \Uu^k } \abs{ h_t\crochet{1 + U } }^{2k} \abs{ \Delta(U) }^2 \frac{d^*U}{ U}
\end{align*}
using the trick \eqref{Trick:ProdScalSym}. 

The case of $ \Be_k $ is treated in the same way. 
\end{proof}

\begin{remark}\label{Rk:BirkhoffWithH}
Let $ X := (x_1, \dots, x_k)  $. One can also write with \eqref{FourierRepr:hLambda}
\begin{align*}%$
L(t, \Be_k) = \crochet{ X^t Y^t } H\crochet{ XY }  = \crochet{ X^t y_1^t \dots y_k^t } \prod_{j = 1}^k H\crochet{y_j X}  = \crochet{ X^t } h_t\crochet{ X }^k
\end{align*}

The function $ h_t $ has several possible forms. Using the fact that $ h_n = s_{(n)} $ and formula \eqref{AlternantRepr:Schur}, one gets by expanding the determinant on the first column ($ j = 1 $) and using the formula for the Vandermonde determinant
\begin{align*}%$
h_t(x_1, \dots, x_k) & = \frac{ \det\prth{ x_i^{ t\delta_{1, j} + k -  j } }_{1 \leq i, j \leq k} }{ \det\prth{ x_i^{ k -  j } }_{1 \leq i, j \leq k} } = \sum_{i = 1}^k x_i^{t + k - 1} \frac{ \det\prth{ x_m^{  k -  j } }_{1 \leq j \leq k - 1, m \in \intcrochet{1, k}\setminus\ensemble{i} } }{ \det\prth{ x_i^{ k -  j } }_{1 \leq i, j \leq k} } \\
                & = \sum_{i = 1}^k x_i^{t + k - 1} \prod_{ m \in \intcrochet{1, k}\setminus\ensemble{i} } \frac{1}{ x_i - x_m }
\end{align*}

Alternatively, one can also write a partial fraction decomposition~:
\begin{align*}%$
h_t(x_1, \dots, x_k) & = \crochet{s^t} H\crochet{ s X} = \crochet{s^t} \prod_{j = 1}^t \frac{1}{1 - s x_i} = \crochet{s^t} \sum_{i = 1}^k \frac{1}{1 - s x_i } \prod_{ m \in \intcrochet{1, k}\setminus\ensemble{i} } \frac{1}{ 1 - x_m/x_i } \\
                & = \sum_{i = 1}^k x_i^{t + k - 1} \prod_{ m \in \intcrochet{1, k}\setminus\ensemble{i} } \frac{1}{ x_i - x_m }
\end{align*}

We thus see that the formula 
\begin{align}\label{Eq:BeckPixtonBirkhoffBis}%$
L(t, \Be_k) = \crochet{ X^t } h_t\crochet{ X }^k
\end{align}
is exactly \cite[thm. 2]{BeckPixton}. The other form of $h_t$ given by \cite[ch. I-2]{MacDo} is also a particular case of the expression of a Schur function in terms of monomials
\footnote{It writes $ s_\lambda(X) = \sum_{T \in SST(\lambda)} X^T $ where $ SST(\lambda) $ is the set of \textit{semi-standard (Young) tableaux} of shape $ \lambda $, see e.g. \cite[ch. I-5, (5.12)]{MacDo}.\label{Footnote:SST}} 
\begin{align*}%$
h_t(x_1, \dots, x_k) = \sum_{1 \leq i_1 \leq \cdots \leq i_t \leq k} x_{i_1}\cdots x_{i_t} = \sum_{ m_1 + \cdots + m_k = t } x_1^{m_1} \cdots x_k^{m_k}
\end{align*}
and is used in \cite{BeckPixton} to show that $ L(t, \Be_k) $ is a polynomial in $t$. 

In the same vein, We have 
\begin{align}\label{Eq:BeckPixtonSubBirkhoffBis}%$
L(t, \Se_k) = \crochet{ X^t } h_t\crochet{ 1 + X }^{k + 1}
\end{align}

Indeed, one has
\begin{align*}%$
L(t, \Se_k) & = \crochet{ X^t Y^t } H\crochet{X + Y + XY} =  \crochet{ X^t y_1^t \dots y_k^t } H\crochet{X} \prod_{i = 1}^k H\crochet{y_i (X + 1)} \\
             & = \crochet{ X^t } \sum_{m \geq 0} h_m\crochet{X} h_t\crochet{1 + X}^k \\
             & = \crochet{ X^t } \sum_{m = 0}^t h_m\crochet{X} h_t\crochet{1 + X}^k
\end{align*}
since $ \crochet{ X^t } h_m\crochet{X} = 0 $ if $ m > t $. We conclude with \eqref{Eq:FiniteSumHomogeneous}. 
\end{remark}

Using \eqref{Eq:MomentsMacroTruncatedWithDuality}, one can give an alternative proof of \eqref{Eq:ConreyGamburdLimit} and a new expression of $ \vol(\Be_k) $:

\begin{theorem}[Another expression of $ \vol(\Be_k) $ and $ \vol(\Se_k) $]\label{Theorem:VolumeBirkoff}
One has %with $ \xb := (x_1, \dots, x_k) $
\begin{align}\label{EqPhi:VolumeSubBirkoffPolytopeRMT}%$
\begin{aligned}
\frac{\vol(\Se_k)}{k^{k - 1}}  & =   \frac{(2\pi)^{k(k - 1) } }{k!} \int_{ \Rr^k } \Phi_{\Se_k}(x_1, \dots, x_k) \Delta(x_1, \dots, x_k)^2 dx_1 \dots dx_k \\
\Phi_{\Se_k}(x_1, \dots, x_k)  & := \abs{ h_{1, \infty}(0, x_1, \dots, x_k) }^{2k}
\end{aligned}
\end{align}
and
\begin{align}\label{EqPhi:VolumeBirkoffPolytopeRMT}%$
\begin{aligned}
\frac{\vol(\Be_k)}{k^{k - 1}}  & =   \frac{(2\pi)^{k(k - 1) } }{k!} \int_{ \Rr^{k - 1} } \Phi_{\Be_k}(0, x_2, \dots, x_k) \Delta(0, x_2, \dots, x_k)^2 dx_2 \dots dx_k \\
\Phi_{\Be_k}(0, x_2, \dots, x_k)  & := \abs{ h_{1, \infty}(0, x_2, \dots, x_k) }^{2k}
\end{aligned}
\end{align}
\end{theorem}

% ==

\begin{proof}
Recall that
\begin{align*}%$
\ell^{k(k - 1) }  \abs{ \Delta(e^{2 i \pi x_1/\ell}, \dots, e^{2 i \pi x_k/\ell}) }^2 & = \ell^{k(k - 1) }  \abs{ \Delta(1, e^{2 i \pi (x_2 - x_1)/\ell}, \dots, e^{2 i \pi (x_k - x_1)/\ell}) }^2 \\
              & \equivalent{\ell\to+\infty} (2 \pi)^{k(k - 1) } \prod_{j = 2}^k \abs{x_j - x_1}^2  \abs{ \Delta( x_2 - x_1, \dots, x_k - x_1) }^2 \\
              & = (2 \pi)^{k(k - 1) } \prod_{j = 2}^k \abs{x_j - x_1}^2  \abs{ \Delta( x_2 , \dots, x_k ) }^2 \\
              & = (2 \pi)^{k(k - 1) } \abs{ \Delta( x_1 , \dots, x_k ) }^2
\end{align*}

We first prove \eqref{EqPhi:VolumeSubBirkoffPolytopeRMT}. By lemma \ref{Lemma:RescaledGegenbauer}, we have
\begin{align*}%$
\frac{1}{\ell^k } h_\ell(1, e^{2 i \pi x_1/\ell}, \dots, e^{2 i \pi x_k/\ell})  \tendvers{\ell}{+\infty} e^{   i \pi \sum_{j = 1}^k x_j  } h_{1, \infty}(0, x_1, \dots, x_k)
\end{align*}
hence, using $ \thetab := (\theta_1, \dots, \theta_k) $ and $ \xb := (x_1, \dots, x_k) $, this implies that 
\begin{align*}%$
L(\ell, \Se_k) & = \frac{1}{k!} \oint_{ \Uu^k } \abs{ h_\ell\crochet{1 + U } }^{2k} \abs{ \Delta(U) }^2 \frac{d^*U}{ U} \\
               & = \frac{1}{k!}\int_{ \crochet{-\frac{1}{2}, \frac{1}{2} }^k } \abs{h_\ell(1,e^{2 i \pi \theta_1 }, \dots, e^{2 i \pi \theta_k }) }^{2k} \abs{\Delta(e^{2 i \pi \theta_1 }, \dots, e^{2 i \pi \theta_k })}^2 d\thetab \\
               & = \frac{1}{k!}\int_{ \crochet{-\frac{\ell}{2}, \frac{\ell}{2} }^k } \abs{h_\ell(1,e^{2 i \pi x_1/\ell}, \dots, e^{2 i \pi x_k/\ell}) }^{2k} \abs{\Delta(e^{2 i \pi x_1/\ell}, \dots, e^{2 i \pi x_k/\ell})}^2 \frac{d\xb}{\ell^k} \\
               & \hspace{-0.2cm} \equivalent{\ell \to +\infty } \hspace{+0.2cm} \ell^{ 2 k^2 - k(k - 1) - k  } \frac{(2\pi)^{k(k - 1) }}{k!} \int_{\Rr^k} \abs{ h_{1, \infty}(0, x_1, \dots, x_k )}^{2k} \abs{ \Delta( x_1 , \dots, x_k ) }^2 d\xb
\end{align*}

Here, we have used dominated convergence given by inequality \eqref{Ineq:SpeedOfConvergenceHinftyBis}. The limiting function is integrable using the first criteria \eqref{Ineq:IntegrabilityDomination} with $ (K, M, \kappa) = (k, 2k, 1) $ which is clearly fullfilled for all $ p \in \ensemble{0, 1, 2} $ and $ k \geq 2 $.

% ==
\medskip

We now prove \eqref{EqPhi:VolumeBirkoffPolytopeRMT}. We again use lemma \ref{Lemma:RescaledGegenbauer}, but with one less variable. Setting $ u_j/u_1 = e^{2 i \pi x_j/\ell} $ for all $ j \in \intcrochet{2, k} $ and $ u_1 = e^{2 i \pi \varphi} $ gives
\begin{align*}%$
L(\ell, \Be_k) & = \frac{1}{k!} \oint_{ \Uu^k } \abs{ h_\ell\crochet{ U } }^{2k} \abs{ \Delta(U) }^2 \frac{d^*U}{ U} \\
               & = \frac{1}{k!} \int_{ -\frac{1}{2} }^{ \frac{1}{2} } d\varphi \int_{ \crochet{-\frac{1}{2}, \frac{1}{2} }^{k - 1} } \abs{h_\ell(1,e^{2 i \pi x_2/\ell}, \dots, e^{2 i \pi x_k/\ell}) }^{2k} \abs{\Delta(1, e^{2 i \pi x_2/\ell}, \dots, e^{2 i \pi x_k/\ell})}^2 \frac{d\xb}{\ell^{k - 1}} \\ 
               & \hspace{-0.2cm} \equivalent{\ell \to +\infty } \hspace{+0.2cm} \ell^{ 2 k(k - 1) - k(k - 1) - (k - 1)  } \frac{(2\pi)^{k(k - 1) }}{k!} \int_{\Rr^{k -1} } \abs{ h_{1, \infty}(0, x_2, \dots, x_k )}^{2k} \abs{ \Delta( 0, x_2 , \dots, x_k ) }^2 d\xb \\
               & = \ell^{ (k - 1)^2 } \frac{(2\pi)^{k(k - 1) }}{k!} \int_{\Rr^{k - 1} } \!\!\Phi_{\Be_k}(0, \xb) \Delta( 0, \xb )^2 d\xb
\end{align*}
hence the result using dominated convergence and the first criteria \eqref{Ineq:IntegrabilityDomination} with $ (K, M, \kappa) = (k - 1, 2k, 1) $.
\end{proof}

% ==

\begin{remark}
This description of $ \vol(\Be_k) $ or $ \vol(\Se_k) $ are \textit{a priori} different from the ones given in \cite{HeapLindqvist}, \cite{ConreyGamburd} or \cite{HarperNikeghbaliRadziwill} since they involve $k$ and $(k - 1)$-fold real integrals, as opposed to $ (2k - 1)$-fold complex integrals. A direct proof of the equivalence between these representations and \eqref{EqPhi:VolumeSubBirkoffPolytopeRMT}/\eqref{EqPhi:VolumeBirkoffPolytopeRMT} is still to be found.
\end{remark}

% ==

\begin{remark}\label{Rk:SubBirkhoffWithUniforms}
Another expression can be given directly using lemma \ref{Lemma:ShallowCrucialLemma} starting from equation \eqref{Eq:BeckFormulaSubBirkhoff}:
\begin{align*}%$
L(\ell, \Se_k) & = \crochet{ X^\ell Y^\ell } H\crochet{ X + Y + XY }   \\
               & = \crochet{ X^\ell Y^\ell } H\crochet{ X - X^{\ell + 1} + Y - Y^{\ell + 1} + XY - X^{\ell + 1} Y^{\ell + 1} }  \\
               & = \int_{ \crochet{ - \frac{\ell}{2}, \frac{\ell}{2} }^{2k} } e^{ -2 i \pi \sum_{j = 1}^k (x_j + y_j) } \prod_{j, m = 1}^k e^{ i \pi (x_j + y_m)  }\frac{ \sin( \pi \frac{x_j + y_m}{\ell} (\ell + 1) )  }{\sin( \pi \frac{x_j + y_m}{\ell}  ) } \\
               & \hspace{+4.3cm} \times \prod_{j = 1}^k e^{ i \pi x_j    } \frac{ \sin( \pi \frac{x_j  }{\ell} (\ell + 1) )  }{\sin( \pi \frac{x_j  }{\ell}  ) } e^{ i \pi y_j    } \frac{ \sin( \pi \frac{y_j  }{\ell} (\ell + 1) )  }{\sin( \pi \frac{y_j  }{\ell}  ) } \frac{dx_j dy_j}{\ell^2 } \\
               & \equivalent{\ell\to +\infty} \ell^{ k^2 + 2k - 2k } \int_{ \Rr^{2k} } e^{  i \pi (k - 1) \sum_{j = 1}^k (x_j + y_j) } \prod_{j, m = 1}^k  \sinc( \pi (x_j + y_m) ) \\
               & \hspace{+6.5cm} \times \prod_{j = 1}^k  \sinc( \pi x_j ) \sinc( \pi y_j ) \, dx_j dy_j 
\end{align*}
which implies
\begin{align}\label{Eq:VolumeSubBirkoffPolytopeSinc}%$  
\frac{\vol(\Se_k)}{k^{k - 1} }  =  \int_{ \Rr^{2k} } \!e^{  i \pi (k - 1) \sum_{j = 1}^k (x_j + y_j) } \! \prod_{j, m = 1}^k  \!  \sinc( \pi (x_j + y_m) )\prod_{j = 1}^k  \sinc( \pi x_j ) \sinc( \pi y_j ) \, dx_j dy_j 
\end{align}

Using the Fubini theorem, one can rewrite \eqref{Eq:VolumeSubBirkoffPolytopeSinc} as
\begin{align}\label{EqPhi:VolumeSubBirkoffPolytopeBeckPixton}%$
\frac{\vol(\Se_k)}{k^{k - 1}} =  \int_{ \Rr^k } e^{i \pi (k - 1) \sum_{j = 1}^k x_j }  h_{1, \infty}(0, x_1, \dots, x_k)^k dx_1 \dots dx_k  
\end{align}

This expression can also be obtained by a rescaling of \eqref{Eq:BeckPixtonBirkhoffBis}.

Using \eqref{Eq:Beta_ProbRepr} with $ \kappa = 1 $, i.e. with $ V $ uniform in $ \crochet{-1, 1} $, this representation is equivalent to
\begin{align*}%$  
\frac{\vol(\Se_k)}{k^{k - 1} } = 2^{2k} \Esp{ \prod_{j = 1}^k \delta_0\prth{ k - 1 - V_{j, k + 1} - \sum_{i = 1}^k V_{j, i} } \delta_0\prth{ k - 1 - V_{ k + 1, j} - \sum_{i = 1}^k V_{ i, j} } }
\end{align*}
where $ (V_{i, j})_{1 \leq i, j \leq k + 1} $ is a matrix of i.i.d. random variables uniformly distributed in $ \crochet{-1, 1} $.

Let $ W $ be an absolutely continuous random variable. Denote by $ f_W(x) = \Esp{\delta_0(W - x) } $ its Lebesgue density. Then, setting $ W_j = V_{j, k + 1} + \sum_{i = 1}^k V_{j, i} $ if $ j \leq k $ and $ W_j = V_{ k + 1, j} + \sum_{i = 1}^k V_{ i, j} $ if $ j \geq k + 1 $, we thus have
\begin{align*}%$  
\frac{\vol(\Se_k)}{k^{k - 1} } = 2^{2k} f_{ (W_1, \dots, W_{2k}) } (k - 1, \dots, k - 1)
\end{align*}

Note that such Lebesgue densities are multidimensional generalisations of the rectangle function $ \Esp{ \delta_0(x - V) } = \frac{1}{2} \Unens{-1 \leq x \leq 1} $, hence, $f_{ (W_1, \dots, W_{2k}) } (x_1, \dots, x_{2k}) $ is a piecewise multivariate polynomial density and taking it in $ \xb := (k - 1)\Un_{2k} $ gives thus a piecewise polynomial in $k$ in the same vein as \cite[cor. 1.7]{KeatingR3G} or \cite[(1.3), (1.4)]{BettinConrey}. Piecewise polynomiality is thus a general phenomenon due to the probabilistic representation of $ h_{c, \infty} $ and will also hold in the more general case of $ h_{c, \infty}^{(\kappa)} $ where uniform random variables are replaced by Beta($ \kappa, 1 $) ones if $ \kappa\in \Nn^* $ since one has $ \betab_{\kappa, 1} \eqlaw U^{1/\kappa} $ with $ U \sim \Us([0, 1]) $, hence $ \Esp{ \delta_0(x - U^{1/\kappa}) } = \kappa x^{\kappa - 1} \Unens{0 \leq x \leq 1} $.

The same can be said about the Birkhoff polytope $ \Be_k $ (details are left to the reader).
\end{remark}

% ==
\subsubsection{Extension to transportation polytopes}

Let $ \lambda, \mu \vdash N $ with $ \ell(\lambda) = m $ and $  \ell(\mu) = n $. A transportation polytope $ \Te_{\lambda, \mu} $ is a generalisation of the Birkhoff polytope where the sum of the lines is given by $ \lambda $ and the sum of the colums by $ \mu $, namely 
\begin{align*}%$
\Te_{\lambda, \mu} := \ensemble{ (M_{i, j})_{(i, j)\in\intcrochet{1, m} \times \intcrochet{1, n} } \in \Rr_+^{mn} \, \bigg/ \, \begin{array}{l} \sum_j M_{i, j} = \lambda_i \ \forall i \in\intcrochet{1, m}  \\ \sum_i M_{i, j} = \mu_j \ \forall j \in\intcrochet{1, n} \end{array}   }
\end{align*}

These polytopes are the set of solutions to a transportation problem in linear programming (see \cite{BeckPixton} and references cited). The Birkhoff polytope is $ \Be_k = \Te_{1^k, 1^k} $. Note that  it is always possible to assume $ \lambda_1 \geq \lambda_2 \geq \dots $, etc. by multiplying the matrix on the left and on the right with permutation matrices (which amounts to change the order of the coordinates). 

$ \Te_{\lambda, \mu} $ is a convex polytope defined by equalities of dimension $ (m - 1)(n - 1) $ in $ \Rr^{mn} $. We are interested in the behaviour of the Erhahrt polynomial $ L(\ell, \Te_{\lambda, \mu} ) :=  \# ( \ell \Te_{\lambda, \mu} \cap \Zz^{mn}) $ when $ \ell\to+\infty $ and in particular to $ \vol(\Te_{\lambda, \mu}) $. Beck and Pixton give the following formula \cite[\S 5]{BeckPixton}
\begin{align*}%$
L(\ell, \Te_{\lambda, \mu} ) = \crochet{X^{\ell \lambda} Y^{\ell \mu} } H\crochet{ XY }
\end{align*}

They transform it into \cite[thm. 5]{BeckPixton}
\begin{align*}%$
L(\ell, \Te_{\lambda, \mu} ) = \crochet{ X^{\ell \lambda} } h_{\ell \mu}\crochet{X}
\end{align*}

This formula can be proven with the same manipulations as in remark \ref{Rk:BirkhoffWithH}. In the same vein as lemma \ref{Lemma:EhrhartBirkhoff} one has
\begin{align}\label{Eq:EhrhartTransportationPolytope}%$
L(\ell, \Te_{\lambda, \mu} ) = \frac{1}{k!} \int_{\Uu^k} h_{\ell \lambda}(u_1, \dots, u_k) \overline{h_{\ell \mu}(u_1, \dots, u_k) } \abs{\Delta(U)}^2 \frac{d^*U}{U}, \quad k := \max\ensemble{m, n} 
\end{align}

This formula is in fact well-known. It is given by Diaconis and Gangolli in \cite[\S 4, thm.]{DiaconisGangolli}. The manipulations of lemma \ref{Lemma:EhrhartBirkhoff} amount thus to prove the Beck-Pixton formula directly from this symmetric function formula using the reproducing kernel property \eqref{Eq:ReproducingKernelProperty}/\eqref{Trick:Hreproducivity}. 

In \cite[\S 7]{DiaconisGangolli}, the question is asked about the behaviour of $ L(1, \Te_{\lambda, \mu} ) $ when $ \abs{\lambda} = N \to +\infty $ and $ \ell(\lambda), \ell(\mu) $ are fixed ; this is equivalent to the behaviour of $ L(\ell, \Te_{\lambda, \mu} ) $ when $ \ell \to +\infty $ and $ \lambda, \mu $ are fixed if one supposes that the partition is well-balanced in the sense that $ \lambda_i = \pe{N c_i} $ for all $i$ (and the same for $ \mu $). By the Ehrhart theorem, we know that 
\begin{align*}%$
\frac{L(\ell, \Te_{\lambda, \mu} )}{ \ell^{(m - 1)(n - 1)} } \tendvers{\ell}{+\infty} \covol( \Te_{\lambda, \mu}, \Rr^{mn})
\end{align*}

Diaconis and Efron (see \cite[(7.2)]{DiaconisGangolli} and reference cited) suggest without proof\footnote{We modify their statement which is an approximation with slightly different quantities equivalent to ours when $ \ell \to +\infty $. } that in such a case
\begin{align*}%$
\covol( \Te_{\lambda, \mu}, \Rr^{mn}) \approx \prth{ \prod_{j = 1}^m \lambda_j }^{ n - 1 } \prth{ \prod_{j = 1}^n \mu_j }^{ K - 1 } \frac{ \Gamma(mn) }{ \Gamma(n)^{m - 1} \Gamma(K + n) }, \qquad K := \frac{n + 1}{n \sum_j \lambda_j^2 } - \frac{1}{n}
\end{align*}
%
%
%
%where
%%
%%
%%
%\begin{align*}%$
%K := \frac{n + 1}{n \sum_j \lambda_j^2 } - \frac{1}{n}
%\end{align*}

We now give the value of $ \covol( \Te_{\lambda, \mu}, \Rr^{mn}) $ as an integral over $ \Rr^{\max\ensemble{m, n} - 1} $:

\begin{lemma}[Relative volume of $ \Te_{\lambda, \mu} $ in $ \Rr^{mn} $]
We have, with $ U := (u_1, \dots, u_k) $
\begin{align}\label{EqPhi:CovolumeTransportationPolytope}%$
\begin{aligned}
\covol( \Te_{\lambda, \mu}, \Rr^{mn}) & = \frac{(2\pi)^{k(k - 1)} }{k!} \int_{\Rr^{k - 1} } \Phi_{\Te_{\lambda, \mu}}(0, \xb) \Delta(0, \xb)^2 d\xb \\
\Phi_{\Te_{\lambda, \mu}}(0, \xb) & = h_{\lambda, \infty}(0, \xb) \overline{h_{\mu, \infty}(0, \xb) }
\end{aligned}
\end{align}
where $ h_{\lambda, \infty} := \prod_{j \geq 1} h_{\lambda_j, \infty} $ and $ k := \max\ensemble{m, n} $.
\end{lemma}

% ==

\begin{proof} 
Since $ \abs{\lambda} = \abs{\mu} $ and $ h_\lambda(s X) = s^{\abs{\lambda}} h_\lambda(X) $, we have
\begin{align*}%$
L(\ell, \Te_{\lambda, \mu} )  & = \frac{1}{k!} \int_{\Uu^k} h_\lambda\prth{1, \frac{u_2}{u_1}, \dots, \frac{u_k}{u_1} } \overline{ h_\mu\prth{ 1, \frac{u_2}{u_1}, \dots, \frac{u_k}{u_1} } } \abs{\Delta \prth{1, \frac{u_2}{u_1}, \dots, \frac{u_k}{u_1} } }^2 \frac{d^*U}{U} \\
                & = \frac{1}{k!} \int_{\Uu^{k - 1} } h_\lambda\crochet{1 + V} \overline{ h_\mu\crochet{1 + V} } \abs{\Delta \crochet{1 + V} }^2 \frac{d^*V}{V}
\end{align*}
using the trick of remark \ref{Rk:MultivariateFourierWithHomogeneity} and integrating out $ u_1 $

Setting $ v_j  = e^{2 i \pi x_j /\ell } $ for $ j \in \intcrochet{2, \ell} $ with $ x_j \in \crochet{-\frac{\ell}{2}, \frac{\ell}{2} } $, we get
\begin{align*}%$
L(\ell, \Te_{\lambda, \mu} )  & = \frac{1}{k!}  \int_{ \crochet{-\frac{\ell }{2}, \frac{\ell }{2} }^{k -1} } h_\lambda\prth{1, e^{2 i \pi \frac{x_2}{\ell} } , \dots, e^{2 i \pi \frac{x_k}{\ell} } } \overline{ h_\mu\prth{ 1, e^{2 i \pi \frac{x_2}{\ell} }, \dots, e^{2 i \pi \frac{x_k}{\ell} } } } \\
                      & \hspace{+4cm} \abs{\Delta \prth{1, e^{2 i \pi \frac{x_2}{\ell} }, \dots, e^{2 i \pi \frac{x_k}{\ell} } } }^2 \frac{d\xb}{\ell^{k - 1} } \\
                     & \equivalent{\ell\to+\infty} \ell^{ (k - 1) n + (k - 1) m - k(k - 1) - (k - 1)  } \frac{(2\pi)^{k(k - 1) } }{k!} \int_{ \Rr^{k - 1} } e^{ i \pi \abs{\lambda} \sum_{j = 2}^k x_j } h_{\lambda, \infty}(0, \xb) \\
                     & \hspace{+8cm} e^{ -i \pi \abs{\mu} \sum_{j = 2}^k x_j } \overline{h_{\mu, \infty}(0, \xb) } \Delta(0, \xb)^2 d\xb
\end{align*}

We have moreover 
\begin{align*}%$
(k - 1) n + (k - 1) m - k(k - 1) - (k - 1) & = (k - 1)(m + n - k - 1)  \\
             & = (\max\ensemble{m, n} - 1)(\min\ensemble{m, n} - 1) \\
             & = (m - 1)(n - 1)
\end{align*}
which gives the right scaling in view of Ehrhart's theorem.

We conclude using dominated convergence. The domination is given by \eqref{Ineq:SpeedOfConvergenceHinftyBis} and the limiting function is integrable using a slight modification of the second criteria \eqref{Ineq:IntegrabilityDomination} with $ (K, M, M', \kappa, \kappa') = (k - 1, m, n, 1, 1) $ which is clearly fullfilled for all $ p \in \ensemble{0, 1, 2} $ and $ m, n \geq 2 $. Here, using a product of $ h_{\lambda_i, \infty} $ or a power $ h_{c, \infty}^M $ gives the same result.
\end{proof}

\medskip
% ==
\subsection{Iterated derivatives of the characteristic polynomial}\label{Subsec:Derivatives} 

\subsubsection{Motivations}\label{Subsubsec:Applications:Derivatives:Motivations}

The problem of the derivatives of $ Z_{U_N} $ in $1$ has a long history. The fourth moment of the random variable $ \zeta'\prth{\frac{1}{2} + i T U} $ with $ U $ uniform in $ \crochet{0, 1} $ was already investigated by Conrey in \cite{ConreyFourth}. After the seminal work \cite{KeatingSnaith} and the Keating-Snaith paradigm that used $ Z_{U_N}(1) $ as a toy-model for $ \zeta\prth{\frac{1}{2} + i T U} $, it became natural to investigate the moments of $ Z'_{U_N}(1) $ where 
\begin{align*}%$
Z'_{U_N}(z) = \frac{d}{dz} Z_{U_N}(z) = \sum_{k = 1}^N k \sc_k(U_N) z^{k - 1}
\end{align*}

The asymptotics of the joint moments of $ (Z_{U_N}(1), Z'_{U_N}(1) ) $ were performed with several methods, see \cite{ConreyRubinsteinSnaith, POD, PODfpsac, HughesKeatingOConnell2, HughesThesis, MezzadriDerivatives, SnaithDerivatives, Winn, RiedtmannMixedRatios} (see also \cite{ChhaibiNajnudelNikeghbali} where a related quantity is computed, starting from points at the microscopic scale). For instance, \cite[thm. 1]{ConreyRubinsteinSnaith} gives
\begin{align*}%$
\frac{\Esp{ \abs{ Z_{U_N}'(1) }^{ 2k } }}{ N^{k^2 + 2k} } \tendvers{N}{+\infty} (-1)^{ \frac{k(k + 1) }{2 } }  \prth{ 1 + \frac{d}{dx} }^k \prth{ \frac{d}{dx} }^k \prth{ \frac{e^{- \frac{x}{2} } }{x^{ \frac{k^2}{2} } } \det\prth{ I_{i + j - 1}(2\sqrt{x}  }_{1 \leq i, j \leq k} }\bigg\vert_{x = 0}
\end{align*}
where $ I_k(x) := \crochet{t^k} e^{ x \frac{t + t\inv}{2} } $ is the Bessel function of order $k$. This last expression was moreover proven to be equivalent to an expression involving Painlev\'e III' and V equations in \cite{ForresterWittePainleve3p5}.

Another step in this study was performed by Winn \cite{Winn} and Dehaye \cite{POD} who give the limiting value of 
\begin{align}\label{Def:JointMomentsDerivative_0_1}%$
\De_N(k, h) := \Esp{ \abs{Z_{U_N}(1)}^{2(k - h)} \abs{Z'_{U_N}(1)}^{2h} }
\end{align}
for half integers $ h $ ; more precisely, we have with $ \De_\infty(k, h) $ given in \cite[(6.14)]{Winn}/\cite[(2)]{POD}
\begin{align}\label{Eq:CvJointMomentsDerivative_0_1}%$
\frac{\De_N(k, h)}{N^{k^2 + 2h} } \tendvers{N}{+\infty} \De_\infty(k, h)
\end{align}

These studies culminate in the recent work of Assiotis-Keating-Warren \cite{AssiotisKeatingWarren} who change probability in \eqref{Def:JointMomentsDerivative_0_1}. Consider indeed the probability measure
\begin{align*}%$
\Pp_{k, N}(d\thetab) := \frac{\abs{Z_{U_N}(1)}^{2k}  }{ \Esp{\abs{Z_{U_N}(1)}^{2k}  } } \bullet \Pp_{CUE_N}(d\thetab) = \frac{1}{ \Esp{\abs{Z_{U_N}(1)}^{2k}  } } \prod_{j = 1}^N \abs{1 - e^{2 i \pi \theta_j} }^{2k} \abs{\Delta(e^{2 i \pi \thetab}) }^2 \frac{d\thetab}{N!}
\end{align*}
for $ \thetab \in [0, 1]^N $. Since the weight $ \abs{Z_{U_N}(1)}^{2k} $ is tensorial in the eigenvalues, the measure such defined is still determinantal (i.e. its eigenvalues form a determinantal point process). For $ k = 1 $, one can write the weight as $ \abs{\Delta(1, e^{2 i \pi \thetab}) }^2 $ which shows that one has conditioned $ U_{N + 1} \sim CUE_{N + 1} $ on having an eigenvalue in $ 1 $. For $ k \geq 1 $, one has done the same with imposing $ k $ eigenvalues equal to $1$ to $ U_{N + k} \sim CUE_{N + k} $, see e.g. \cite{BourgadeCondHaar} ; for $ k \notin \Nn $, this is still possible to define this ensemble without the interpretation as a conditioning, and one obtains the \textit{circular Jacobi Ensemble} \cite{BourgadeNikeghbaliRouault, Neretin}. 

Under this new measure, and setting $ U_{N, k } \sim \Pp_{N, k} $, one has\footnote{One needs to suppose $ k \notin \Nn $ otherwise $ Z_{U_{N, k}} $ has $ \Pp_{N, k}$-almost surely a zeroe of order $k$ in $1$, hence $ Z_{U_{N, k}}'(1) = \cdots = (\frac{d}{dx})^{k - 1} Z_{U_{N, k}}(x)\vert_{x = 1} = 0 $ a.s.}  
\begin{align*}%$
\De_N(k, h) = \Esp{\abs{Z_{U_N}(1)}^{2k} }   \Espr{k, N}{ \abs{ \frac{Z'_{U_N}(1)}{Z_{U_N}(1)} }^{2h} } = \Esp{\abs{Z_{U_N}(1)}^{2k} }   \Esp{ \abs{ \trace\prth{ (I - U_{N, k})\inv } }^{2h} }
\end{align*}

Since \eqref{Eq:KSCUEwithDuality} gives $ N^{-k^2} \Esp{\abs{Z_{U_N}(1)}^{2k} }  \to \widetilde{L}_1(k) $, proving \eqref{Eq:CvJointMomentsDerivative_0_1} amounts to show that
\begin{align*}%$
\Espr{k, N}{ \abs{ \frac{1}{N} \trace\prth{ (I - U_N)\inv } }^{2h} } \tendvers{N}{+\infty } \Ree(k, h) := \frac{\De_\infty(k, h)}{\De_\infty(k, 0)}
\end{align*}

The statistics $ \trace(f(U_N)) := \sum_{k = 1}^N f(\lambda_k) $ (where $ (\lambda_k)_k $ designates the spectrum of $ U_N $) are called \textit{linear statistics} of the eigenvalues of $ U_N $ and there is a whole literature designed to their study. In particular, one can directly study their Fourier-Laplace transform in lieu of their moments\footnote{Convergence in law is not equivalent to convergence in moments, and one needs some additional uniform integrability to conclude.}, as it writes as a Fredholm determinant. In this particular case, the change of variables $ h_i := (1 - \lambda_i)\inv $ amounts to the Cayley transform of \S~\ref{Subsec:Theory:DualityHermitian} and maps the circular Jacobi Ensemble to the Hua-Pickrell measure on $ \He_N $ ; one can then directly study the trace of such a random hermitian matrix. Note nevertheless that $ \frac{1}{N} \trace\prth{ (I - U_N)\inv } $ is the logarithmic derivative of $ Z_{U_N}(e^{x/N})/ Z_{U_N}(1) $ in $ x = 0 $ and this last ratio is known to converge in law in the topology of locally uniform convergence to a limiting analytic function in the case $ k = 0 $ by \cite{ChhaibiNajnudelNikeghbali}. A generalisation of this last work to the case $ k \neq 0 $ would thus give an interesting comparison to \cite{AssiotisKeatingWarren}.

\medskip
% ==
\subsubsection{The framework considered in this article}

We will only work in the limited framework of $ k, h \in \Nn $, with the additional gain to consider iterated derivatives. Obtaining such a generalisation was posed in the conclusion of Dehaye's work \cite[\S 10]{POD}. The method employed by Dehaye uses representation theory and symmetric functions in the same vein as this article, but it is nevertheless very different in nature since it proceeds by expansion of the integrand in the basis of Schur functions, ending up with combinatorial expressions (see \S~\ref{Subsec:Ultimate:Expansions} for a description~; see also the recent work \cite{RiedtmannMixedRatios} that treats in a similar way moments of a functional mixing ratios and first order derivatives).

\medskip
% ==
\subsubsection{Expressing $ \De_\infty(k ; h_1, \dots, h_m)  $}

Define $ k^{\downarrow m} := k(k - 1) \dots (k - m + 1) $ and $ \hb := (h_1, \dots, h_m) \in \Nn^m $. Then, setting 
\begin{align*}%$
\partial^m Z_{U_N}(z) & := \prth{\frac{d}{dz} }^m Z_{U_N}(z) = \sum_{k = m }^N k^{\downarrow m} \sc_k(U_N) z^{k - m } \\
C(\hb) & := \sum_{r \geq 1} r h_r \\
\abs{\hb} & := \sum_{r \geq 1} h_r
\end{align*}
we consider for $ k \geq \abs{\hb} $
\begin{align}\label{Def:DerivativeFunctional}%$
\De_N(k ; \hb) := \Esp{ \abs{Z_{U_N}(1)}^{2(k - \abs{\hb} )} \prod_{r = 1}^m \abs{\partial^r Z_{U_N}(1)}^{2h_r} }
\end{align}

\begin{theorem}[Iterated derivatives of the characteristic polynomials in $1$]\label{Theorem:Derivatives}
Suppose $ k \geq \abs{\hb} $ and $ k \geq 2 $. Then, the following convergence holds 
\begin{align}\label{Eq:ConvergenceDerivatives}%$
\frac{ \De_N(k ; \hb) }{ N^{k^2 + 2 C(\hb) } } \tendvers{N}{+\infty } \De_\infty(k ; \hb) 
\end{align}
with
\begin{align}\label{EqPhi:Derivatives}%$
\begin{aligned}
\De_\infty(k ; \hb)  & = \frac{(2\pi)^{ k(k - 1) }  }{k! } \int_{ \Rr^{k - 1} } \Phi_{ \De_\infty(k ; \hb) }(0, X) \Delta\prth{ 0, X}^2  d X \\
\Phi_{ \De_\infty(k ; \hb) }(0, x_2, \dots, x_k) & :=  \prth{\frac{-1}{4\pi^2} }^{ C(\hb) }  \prod_{r = 1}^m (r!)^{2 h_r} \ e^{ - 2 i \pi \sum_{j = 2}^k x_j }   \\
               & \hspace{+1cm} \times \crochet{Z_1^1 \cdots Z_r^r  (Z'_1)^1 \cdots (Z_r')^r } e^{ 2 i \pi \sum_{z' \in Z'}   z'  +  i \pi k \sum_{a \in X \tensorsum ( Z + Z') } a } \\
               & \hspace{+5cm} h_{k, \infty}\crochet{ (0 + X) \tensorsum \prth{ Z + Z' + 0^{ \plusInOne 2 h_0 } } }
\end{aligned}
\end{align}
where $ Z_r := (z_{\ell, r})_{ 1 \leq \ell \leq h_r  } $ for all $ r \in \intcrochet{1, m} $ and $ Z := (Z_1, \dots, Z_m) $.
\end{theorem}

% ==

\begin{proof}
Recall the Taylor formula for a polynomial $ f $:
\begin{align}\label{Eq:TaylorFormula}%$
\partial^r f(1) = r! \crochet{z^r } f(1 + z) = A^r r! \crochet{z^r } f\prth{ 1 + \frac{z}{A} }\quad \forall A \neq 0
\end{align}

The main idea of this proof is to write this formula for $ A = N $ and to justify the replacement $ f(1 + z/N) \approx f(e^{z/N}) $ that then allows to use the previous machinery in the microscopic setting\footnote{See \cite[(42)]{PODfpsac} for a similar idea with a different continuation (binomial formula)}.

Set $ h_0 := k - \abs{\hb} $. We have 
\begin{align*}%$
\De_N(k ; \hb) & = \prod_{r = 1}^m \prod_{\ell = 1}^{h_r} \frac{\partial^r}{(\partial t_{\ell, r})^r }  \frac{\partial^r}{(\partial t'_{\ell, r})^r } \Esp{ \! \abs{Z_{U_N}(1) }^{2 h_0 }\! \prod_{r = 1}^m \prod_{\ell = 1}^{h_r} \! Z_{U_N}\prth{ t_{\ell, r} } \overline{ Z_{U_N}\prth{ t'_{\ell, r} } }  }\!\!\bigg\vert_{  T = T' = 1^{\plusInOne \abs{\hb} } } \\
              & =  \prod_{r = 1}^m (r!)^{2 h_r} \crochet{T_1^1 \cdots T_m^m }\crochet{(T'_1)^1 \cdots (T_m')^m }  \Esp{ \! \abs{Z_{U_N}(1) }^{2 h_0 }\! \prod_{r = 1}^m \prod_{\ell = 1}^{h_r}  Z_{U_N}( 1 + t_{\ell, r} )  \overline{ Z_{U_N}( 1 + t'_{\ell, r} ) }  } 
\end{align*}

Now, setting $ x_{\ell, r} := 1 + t_{\ell, r} $ and $ x'_{\ell, r} := 1 + t'_{\ell, r} $, we have by \eqref{Eq:InvolutionOmegaOnSchur} or \eqref{Eq:SchurJointMoments}
\begin{align*}%$
\Qe_N(X, X') & := \frac{ 1}{N! } \int_{\Uu^N } H\crochet{ - U (1^{\plusInOne h_0} +  X   ) - U\inv (1^{\plusInOne h_0} + X'  )  }   \abs{\Delta(U)}^2 \frac{d^*U}{ U} \\
             & = (-1)^{ N k } \prod_{x' \in X'} (x')^N \frac{ 1 }{N! } \int_{\Uu^N } U^{-k } H\crochet{ - U (1^{\plusInOne 2 h_0} +  X + X'  )  }   \abs{\Delta(U)}^2 \frac{d^*U}{ U}  \\
             & = (-1)^{ N k } \prod_{x' \in X'} (x')^N \,  s_{N^k}\crochet{ 1^{\plusInOne 2 h_0} +  X + X' } \\
             & = (-1)^{ N k } \prod_{x' \in X'} (x')^N \, \int_{\Ue_k } \det(U)^{-N} \prod_{a \in 1^{\plusInOne 2 h_0} +  X + X' } \det(I_k + a U)\inv dU \\
             & = \prod_{x' \in X'} (x')^N \, \int_{\Ue_k } \det(U)^{-N} \prod_{a \in 1^{\plusInOne 2 h_0} + X + X' } \det(I_k - a U)\inv dU
\end{align*}
where the last equality comes from the change of variable $ U' = -U $ and the invariance of the Haar measure by mutiplication by the unitary matrix $ -I_k $.

Using \eqref{Eq:TaylorFormula} with $ A = N/i $, one gets with $ T = iZ/N $ and $ T' = iZ'/N $
\begin{align*}%$
\De_N(k ; \hb) & = (-iN)^{2C(\hb)} \prod_{r = 1}^m (r!)^{2 h_r} \crochet{Z_1^1 \cdots Z_m^m  (Z'_1)^1 \cdots (Z_m')^m }  \prod_{z' \in Z'} \prth{ 1 + i\frac{z'}{N} }^N \\
               & \hspace{+3cm} \times  \Esp{ \! \det(U_k)^{-N} Z_{U_k}(1)^{-2 h_0 } \! \prod_{r = 1}^m \prod_{\ell = 1}^{h_r} Z_{U_k}\prth{ 1 + i \frac{z_{\ell, r} }{N} }\inv \overline{ Z_{U_k}\prth{ 1 + i \frac{z'_{\ell, r}}{N} }\inv } } 
\end{align*}

If a meromorphic function $f$ on $ \Cc $ has all its singularities on the unit circle, then for $ z \in \Dd_\rho := \ensemble{z \in \Cc : \abs{z} < \rho} $ with $ \rho < 1 $, $f$ is holomorphic and in particular
\begin{align*}%$
\abs{ f\prth{1 +  i\frac{z}{N} } - f\prth{e^{iz/N} } } \leq  \frac{ \rho^{2k} }{N^2} \sup_{z \in \Dd_\rho} \abs{  f'(z) }
\end{align*}

Since all the zeroes of $ Z_{U_k} $ are on the unit circle, $ 1/Z_{U_k} $ is holomorphic on $ \Dd_\rho $ for all $ \rho < 1 $. Moreover, $ \abs{  Z_{U_k}'(z)/Z_{U_k}(z) }  $ is almost surely bounded by a deterministic constant on $\Dd_\rho$ (one can take $ k (1 - \rho)\inv $). One thus has, almost surely with a deterministic $ O_{\rho, k}(\cdot) $ for all real $ z' \in (-\rho, \rho) $
\begin{align*}%$
Z_{U_k}\prth{ 1 + i\frac{z'_{\ell, r}}{N} }\inv = Z_{U_k}\prth{ e^{ i\frac{z'_{\ell, r} }{N} }   }\inv(1 + O_{\rho, k}(N^{-2}) )
\end{align*}

Replacing $ z $ by $ 2\pi z $, we get
\begin{align*}%$
\De_N(k ; \hb) & = \prth{ \frac{N}{2i \pi } }^{2C(\hb)} \prod_{r = 1}^m (r!)^{2 h_r} \crochet{Z_1^1 \cdots Z_m^m  (Z'_1)^1 \cdots (Z_m')^m } \prod_{z' \in Z'} \prth{ 1 + 2i \pi \frac{z'}{N} }^N \\
               & \hspace{+2cm} \times  \Esp{ \!\det(U_k)^{-N} Z_{U_k}(1)^{-2 h_0 } \! \prod_{r = 1}^m \prod_{\ell = 1}^{h_r} Z_{U_k}\prth{ e^{ 2i \pi  \frac{z_{\ell, r} }{N} } }\inv \overline{ Z_{U_k}\prth{ e^{ 2i \pi \frac{z'_{\ell, r} }{N} } }\inv } } \\
               & \hspace{+10cm} \times (1 + O_{\rho, k}(N^{ - 2}) )
\end{align*}

Now, one has for all $ 2 \pi z  \in (-\rho, \rho) $ with a uniform $ O(\cdot) $
\begin{align*}%$
\prth{ 1 + 2i \pi \frac{z }{N} }^N = e^{2i \pi z + O(\rho/N) }
\end{align*}
and, using lemma \ref{Lemma:RescaledSupersym}, for all $ 2\pi z_{\ell, r}, 2\pi z'_{\ell, r} \in (-\rho, \rho) $ and with $ U := (u_1, \dots, u_k) $
\begin{align*}%$
\Psi_k(Z, Z') & :=  \Esp{ \! \det(U_k)^{-N} Z_{U_k}(1)^{-2 h_0 } \! \prod_{r = 1}^m \prod_{\ell = 1}^{h_r} Z_{U_k}\prth{ e^{ 2i\pi \frac{z_{\ell, r} }{N} } }\inv \overline{ Z_{U_k}\prth{ e^{ 2i\pi \frac{z'_{\ell, r} }{N} } }\inv } } \\
              & =  \frac{1}{k!} \oint_{ \Uu^k } U^{-N} h_{kN}\crochet{U (1^{ \plusInOne 2 h_0 } + e^{2i\pi Z/N} + e^{2i\pi Z'/N}) } \ \abs{\Delta(U) }^2 \frac{d^*U}{U}  \qquad\mbox{with \eqref{FourierRepr:SchurRectangleWithHn}} \\
              & = \frac{1}{k!} \oint_{ \Uu^{k - 1} }\!\! V^{-N} h_{kN}\crochet{(1 + V) (1^{ \plusInOne 2 h_0 } + e^{2i\pi Z/N} + e^{2i\pi Z'/N}) } \ \abs{\Delta\crochet{1 + V} }^2 \frac{d^*V}{V}  
\end{align*}  
using remark \ref{Rk:MultivariateFourierWithHomogeneity}. Then, setting $ v_j  = e^{2 i \pi x_j/N} $ for $ j \in \intcrochet{2, k} $ and $ x_j \in \crochet{ - \frac{N}{2},  \frac{N}{2} } $, we get using lemma \ref{Lemma:RescaledGegenbauer}
\begin{align*}%$
\Psi_k(Z, Z') & :=  \frac{1}{k!} \int_{  \crochet{ - \frac{N}{2},  \frac{N}{2} }^{k - 1} } e^{ - 2 i \pi \sum_{j = 2}^k x_j} h_{kN}\crochet{ (1 + e^{2 i \pi X/N} ) (  e^{2i\pi Z/N} + e^{ 2i\pi Z'/N} + 1^{ \plusInOne 2   h_0   }) } \\
              & \hspace{+9cm} \times \abs{\Delta \crochet{1 + e^{2 i \pi X/N} } }^2 \frac{dX}{N^{k - 1}}  \\
% ================================              
              & \equivalent{N\to+\infty} N^{ 2k^2 - 1 - k(k - 1) - (k - 1) } \frac{ (2\pi)^{k(k - 1)} }{k!} \!\! \int_{  \Rr^{k - 1} } e^{ - 2 i \pi \sum_{j = 2}^k x_j} \\
              & \hspace{+6cm} \times h_{k, \infty}\crochet{ (0 + X) \! \tensorsum \! \prth{ Z + Z' + 0^{ \plusInOne 2  h_0 } } } \! \Delta(0, X)^2 dX
\end{align*}

Indeed, the number of variables of $ h_{kN} $ is $ k (2h_0 + \abs{\hb} + \abs{\hb}) = 2k^2 $. To pass to the limit, we have used dominated convergence given by inequality \eqref{Ineq:SpeedOfConvergenceHinftyBis}, the integrability of the limiting function coming from the first criteria \eqref{Ineq:IntegrabilityDomination} (with $ \kappa = 1 $). The number of variables of $ h_{k, \infty}\crochet{ (0 + X) \tensorsum \prth{  Z + Z' + 0^{ \plusInOne 2   h_0   } } } $ being $ 2k^2 $, this function will be integrable against $ \Delta(0, X)^2 \norm{X}_p^p $ if and only if $ 2k^2 - k(k - 1) - p > k - 1 $, which amounts to $ k^2 > p - 1 $. This is satisfied for all $ k \geq 2 $ \footnote{Note that $ k = 1 $ gives the case $p = 1$ that corresponds to the simple integrability without the speed of convergence, but the condition $k \geq 2$ is needed for the integral on $ \Rr^{k - 1} $ to be defined}.

We thus have
\begin{align*}%$
\frac{\De_N(k ; \hb) }{ N^{k^2 + 2C(\hb)} } & \equivalent{N\to+\infty}  \prth{\frac{-1}{4\pi^2} }^{ C(\hb) } \frac{ (2\pi)^{k(k - 1)} }{k!} \prod_{r = 1}^m (r!)^{2 h_r} \crochet{Z_1^1 \cdots Z_m^m  (Z'_1)^1 \cdots (Z_m')^m } e^{ 2 i \pi \sum_{z' \in Z'}   z'  } \\
               & \hspace{+1cm}   \int_{  \Rr^{k - 1} }\!\! e^{ - 2 i \pi \sum_{j = 2}^k x_j + i \pi k \sum_{a \in X \tensorsum ( Z + Z') } a } h_{k, \infty}\!\crochet{ (0 + X) \! \tensorsum \hspace{-0.05cm} \prth{ Z + Z' + 0^{ \plusInOne 2 h_0 } \! }\! } \! \Delta(0, X)^2 dX
\end{align*}

This result already gives a limit, but we want to put it into the framework of \eqref{Def:FormPhi}. It thus remains to see if we can exchange $ \crochet{Z Z'} $ and $ \int_{\Rr^{k - 1} } $. Note that one cannot use the classical formula \eqref{Eq:FourierCoeff} for the Fourier coefficient as an integral over a complex domain since $ z_{i, j}, z'_{i, j} $ have been supposed real to allow to use lemma \ref{Lemma:RescaledGegenbauer} ; one should instead understand $ \crochet{Z Z'} $ with the Taylor formula \eqref{Eq:TaylorFormula} and look if the interchange of the integral and the derivative is allowed. For this, we use a domination of the integrand. A simple induction shows that
\begin{align*}%$
\frac{1}{r!} \sinc^{(r)}(x) := \frac{1}{r!}\prth{\frac{d}{dx} }^r \sinc (x) =  (-1)^r \frac{\sin(x) }{x^r} + \sum_{j = 1}^{r - 1} \frac{ (-x)^{-j } }{ (r - 1 - j)! } \sin\prth{x + (j + \Unens{r \in 2\Nn}) \frac{\pi}{2} }
\end{align*}
moreover, this function is continuous in $0$ using the Taylor series of $ \sinc $, namely $ \sinc(x) = \sum_{k \geq 0} (-x^2)^k /(2k + 1)! $. Due to the decay of this derivative (hence its integrability), one can differentiate under the integral sign in \eqref{Def:HKappaCInfty}
\begin{align*}%$
h_{k, \infty, h_r, \ell} & := \frac{\partial^{h_r} }{\partial z_{\ell, r}^{h_r} } h_{k, \infty}\crochet{ (0 + X) \tensorsum \prth{ Z + Z'  + 0^{ \plusInOne 2   h_0 } } }\Big\vert_{ z_{\ell, r} = 0 }  \\
                & = k^{k^2} \int_\Rr e^{ i (k - 2) \pi \theta} \hspace{-1.4cm} \prod_{ a \in (0 + X) \tensorsum \, \prth{ Z - z_{\ell, r} + Z' + 0^{ \plusInOne 2 h_0 } } } \hspace{-1.4cm} \sinc(\pi(\theta + a)) \frac{\partial^{h_r} }{\partial z_{\ell, r}^{h_r} } \prod_{ b \in (0 + X) } \sinc(\pi(\theta + b ) + z_{\ell, r}  )\big\vert_{z_{\ell, r} = 0 } d\theta
\end{align*}

Differentiating the product and using the previous formula for $ \sinc^{(r)} $, we see that the function thus obtained is integrable in $ \theta $ and moreover in $ X $ against $ \Delta(0, X)^2 $ on $ \Rr^{k - 1} $. We can moreover reproduce the reasonning for the other derivatives, which gives the exchange. 

In the same vein, we can deal with the case of $ \crochet{Z'} $ with $ e^{ \sum_{z' \in Z'}  i z'  } h_{k, \infty} $. In the end, we get the desired result. 
\end{proof}

% ==

\begin{remark}\label{Rk:Alternative:DerivativeOther}
One could have used a reproducing kernel formula in the vein the the truncation operator $ \Pb_{\! T} $ of kernel given by \eqref{Def:PT}. Indeed, the Cauchy (reproducing) formula writes on the unit circle if $ f(z) = \sum_{k \geq 0} f_k z^k $
\begin{align*}%$
f(z) = \oint_{r \Uu} \frac{f(t)}{1 - zt\inv} \frac{d^*t}{t} = \Pb_{\!\infty} f(z), \qquad \forall r < \mathrm{Rad}(f)
\end{align*}
and any linear operator (truncation, derivatives, etc.) is thus a kernel operator by acting on the kernel of $ \Pb_{\!\infty} $. Note also that if $ f $ is a polynomial of degree $N$, one can replace $ \Pb_{\!\infty} $ by $ \Pb_{\! N + 1} $ (or $ \Pb_{\! N + 1 + r} $ for all $ r \geq 0 $). For instance, one has
\begin{align*}%$
\partial^r Z_{U_N}(1) & = \oint_{\Uu} Z_{U_N}(t) \prth{ \frac{\partial}{\partial z} }^{\!\! r} \!\! \frac{1}{1 - zt\inv}\bigg\vert_{z = 1} \frac{d^*t}{t} \\
                & = \oint_{\Uu} Z_{U_N}(t) \prth{ \frac{\partial}{\partial z} }^{\!\! r} \!\! \frac{1 - (z t\inv)^{N + 1} }{1 - zt\inv}\bigg\vert_{z = 1} \frac{d^*t}{t} = \oint_{\Uu} Z_{U_N}(t) \prth{ \sum_{\ell = r}^{N} \ell^{\downarrow r} t^{-\ell}  } \frac{d^*t}{t} \\
                & =: \oint_{\Uu} K_{r, N}(1, t\inv) Z_{U_N}(t) \frac{d^*t}{t} \\
                & = \int_{ \crochet{-\frac{N}{2}, \frac{N}{2} } } K_{r, N}\prth{ 1, e^{-2 i \pi \theta/N} } Z_{U_N}\prth{e^{2 i \pi \theta/N} } \frac{d\theta}{N}
\end{align*}

This gives an alternative way to perform the computations as 
\begin{align*}%$
K_{r, N}\prth{ 1, e^{-2 i \pi \theta/N} } & = \sum_{\ell = r}^{N} \ell^{\downarrow r} e^{-2 i \pi \theta \ell/N } \\
                & = \prth{ \sum_{\ell = r}^{N} \ell^{\downarrow r} } \Esp{ e^{ - 2i \pi D_{N, r}/N } }, \, \Prob{ D_{N, r} = \ell } := \frac{ \ell^{\downarrow r} }{\sum_{j = r}^{N} j^{\downarrow r}} \Unens{r \leq \ell \leq N} \quad\mbox{(see \eqref{Def:DnR})} \\
                & \equivalent{N\to +\infty} \ N^{r + 1} \int_0^1 s^r e^{-2 i \pi s \theta } ds \\
                & = N^{r + 1} (r + 1) \Esp{ e^{- 2i\pi \theta \betab_{r + 1 } } }
\end{align*}

Note that $ C_{N, r} := \sum_{\ell = r}^{N} \ell^{\downarrow r} \sim N^{r + 1} $, hence that the last computation proves the convergence in distribution $ D_{N, r}/ N \to \betab_{r + 1 } $. We will see the appearance of the random variables $ \betab_{r + 1} $ and $ D_{N, r} $ in the next \S. 
\end{remark}

% ==

\begin{remark}\label{Rk:Derivative:RKHS}
Note that the Cauchy reproducing formula gives a ``true'' representation of the evaluation $ \mathrm{Ev}_z : f \mapsto f(z) $ in $ L^2(r\Uu) $ that can be differentiated as in remark~\ref{Rk:Alternative:DerivativeOther}. Such an integral representation does not exist in $ L^2(\Rr) $ (for the Lebesgue measure), the right analogue on the real line being the distributional pairing $ \bracket{f, \delta_z} $ with a Dirac distribution (stated differently~: $ L^2(\Rr) $ is not an RKHS). The way to overcome this difficulty is usually to consider an $ L^2 $ approximation $ \delta_z^{(\varepsilon)} $ and to pass to the limit $ \varepsilon \to 0 $ after proving a uniform integrability. The advantage of this last method lies in the choice of the approximation of unity $ \delta_z^{(\varepsilon)} $ that can allow exact computations with a clever choice. This is the point of view adopted by Winn \cite{Winn} that uses the probability measure of a Cauchy random variable with parameter $ \varepsilon $ to fit with the Cauchy ensemble described in \S~\ref{Subsec:Theory:DualityHermitian}. The method described in remark~\ref{Rk:Alternative:DerivativeOther} seems to use instead the natural RKHS structure on the circle, but it could be replaced by any other RKHS such as $ L^2(\Dd) $ (unit disc) with its kernel (in this case, the Bergman kernel). Here again, a clever choice is to be made.
\end{remark}

% ==

\begin{conjecture}
In application of the Keating-Snaith philosophy, one can refine conjecture \eqref{Eq:MomentsConjecture} and \cite[conj. 6.1]{HughesThesis} into the following conjecture about the behaviour of the random vector of iterated derivatives $ \prth{ \zeta^{(j)}\prth{\frac{1}{2} + i T U} }_{j \in \intcrochet{0, m} } $ %for $ U $ a random variable uniformly distributed in $ \crochet{0, 1} $:
\begin{align*}%$
\Esp{ \abs{ \zeta\prth{\frac{1}{2} + i T U} }^{2(k - \abs{\hb} ) } \prod_{j = 1}^m \abs{\zeta^{(j)}\prth{\frac{1}{2} + i T U} }^{2 h_j} } \equivalent{T\to+\infty} a_k \beta_k(\hb) e^{ (k^2 + 2 C(\hb)) \log\log T }
\end{align*}
where $ a_k $ is given by \eqref{Def:MatrixFactor} and $ \beta_k(\hb) $ is given by \eqref{EqPhi:Derivatives}.
\end{conjecture}

Note the conjectured limiting covariance structure of $ \prth{ \frac{ \log \abs{ \zeta^{(j)}( \frac{1}{2} + i T U ) }^2 }{\sqrt{\log\log(T) / 2 }} }_{j \in \intcrochet{0, m} } =: (R_j)_j $: if one can extend analytically in $ \hb $ the limiting factor, one would have
\begin{align*}%$
\Esp{ e^{ \sum_{j = 0}^m h_j R_j } } \equivalent{T\to+\infty} \exp\prth{ \frac{1}{2} \sum_{j = 0}^m (j + \delta_{j, 0})^2 h_j^2 + \sum_{0 \leq i < j \leq m} (i + \delta_{0, i}) j h_i h_j + \sum_{j = 0}^m j h_j }
\end{align*}
namely $ \Esp{ (R_j - j)(R_i - i) } \equivalent{T\to+\infty} (i + \delta_{0, i}) (j + \delta_{0, j}) $. Equivalently, the random variables $ X_j := (R_j - j)/(j + \delta_{0, j} ) $ converge in law to the same Gaussian. In the case of the $ CUE $, such a behaviour is reminiscent of the limiting covariance structure of the vector $ \prth{ \frac{ \log \abs{ Z_{U_N}(e^{i x_j/N}) }^2 }{\sqrt{\log(N)/ 2 }} }_{j \in \intcrochet{1, m} } $ for $ (x_j)_j \in \Rr^m $ (see \cite[cor. 1.8]{ChhaibiNajnudelNikeghbali} ; see also \cite[thm. 1.4]{BourgadeMesoscopic} for a related result). 

% ==

\begin{remark}
With the expression \eqref{Eq:HcWithResidues}, one can write with $ \Ce := Z + Z' + 0^{ \plusInOne 2 h_0 } $
\begin{align*}%$
\widehat{h}_{c, \infty}(T) & = \crochet{Z_1^1 \cdots Z_r^r  (Z'_1)^1 \cdots (Z_r')^r } e^{ \sum_{z' \in Z'}  i z'  }   e^{ i \pi k \sum_{a \in X  \tensorsum  \Ce } a } h_{k, \infty}\crochet{ (0 + X) \tensorsum \Ce   }  \\
              & =  (- 2 i \pi)^{ k^2 } \crochet{Z_1^1 \cdots Z_r^r  (Z'_1)^1 \cdots (Z_r')^r } e^{ \sum_{z' \in Z'}  i z'  } \crochet{ s^{ k^2 } } e^{ 2 i \pi c s\inv } H\crochet{s ((0+X)\tensorsum \Ce ) } \\
              & = (- 2 i \pi)^{ k^2 } \crochet{ s^{ k^2 } } e^{ 2 i \pi c s\inv } \crochet{Z_1^1 \cdots Z_r^r (Z'_1)^1 \cdots (Z_r')^r } e^{ \sum_{z' \in Z'} i z' }  H\crochet{s ( (0+X)\tensorsum \Ce )}
\end{align*}

This expression can lead to a signed combinatorial sum using $ H\crochet{a} = \frac{1}{1 - a} = \int_{\Rr_+} e^{at} e^{-t }dt $ for $ a \in \Dd_1\!\setminus\!\ensemble{1} $.
\end{remark}

% ==
\medskip

\subsubsection{Proof with randomisation}

The previous proof uses a derivative in place of the usual argument with integration. The goal here is to use an alternative probabilistic representation of $ \partial^r Z_{U_N} $ to avoid this point, i.e. the randomisation paradigm of \S~\ref{Subsec:RandomisationParadigm}. To avoid running out of variety, we will work in the microscopic setting, namely in points of the form $ e^{s/N} $ for $ s \in \Rr $.

We now define for $ k \geq \abs{\hb} $
\begin{align}\label{Def:DerivativeFunctionalMicro}%$
\De_N(k ; \hb ; \sbb) := \Esp{ \abs{Z_{U_N}(e^{s_0/N})}^{2(k - \abs{\hb} )} \prod_{r = 1}^m \abs{\partial^r Z_{U_N}(e^{s_r/N})}^{2h_r} }
\end{align}

\begin{theorem}[Iterated derivatives of the characteristic polynomials in the microscopic setting]\label{Theorem:DerivativesMicro}
Suppose $ k \geq C(\hb) $ and $ k \geq 2 $. Then, the following convergence holds 
\begin{align}\label{Eq:ConvergenceDerivativesMicro}%$
\frac{ \De_N(k ; \hb ; \sbb) }{ N^{k^2 + 2 C(\hb) } } \tendvers{N}{+\infty } \De_\infty(k ; \hb ; s) 
\end{align}
with
\begin{align}\label{EqPhi:Random:DerivativesMicro}%$
\begin{aligned}
\De_\infty(k ; \hb ; s)  & = \frac{(2\pi)^{ k(k - 1) }  }{k! } \int_{ \Rr^{k - 1} } \Phi^{(Rand)}_{ \De_\infty(k ; \hb ; s) }(0, \xb) \Delta\prth{ 0, \xb}^2  d \xb \\
\Phi^{(Rand)}_{ \De_\infty(k ; \hb ; s) }(0, x_2, \dots, x_k) & := \prod_{r = 1}^m \frac{1}{(r + 1)^{2 h_r}}   e^{s_0 (k - \abs{\hb})}   f_{m, \hb}(k) \\ 
                  & \qquad \times  \Ee\Bigg( \! e^{ \sum_{r = 0}^m s_r S_r(\infty) } \int_{\Rr^{k - 1} }  \!\! e^{- 2 i \pi \sum_{j = 2}^k x_j}  \!\! \prod_{d \in \Db_\infty} \!\! \widetilde{h}_{ d, \infty }\prth{0, \xb} \\
                  & \hspace{+2.5cm}  \prod_{\widetilde{d} \in \widetilde{\Db}_\infty} \!\! \widetilde{h}_{ \widetilde{d}, \infty }\prth{0, \xb}  \Delta(0, \xb)^2 d\xb \Bigg\vert  \sum_{r = 0}^m  S_r(\infty) = k  \Bigg)  
\end{aligned}
\end{align}
where $ \Db_\infty := (D_{\infty, r}^{(j_r)})_{0 \leq r \leq m, 1 \leq j_r \leq  h_r} $ and $ \widetilde{\Db}_\infty :=  (\widetilde{D}_{\infty, r}^{(j_r)})_{0 \leq r \leq m, 1 \leq j_r \leq  h_r} $ are two independent copies of independent random variables of law $ D_{\infty, r}^{(j_r)} \eqlaw U^{1/(r + 1)} $ with $ U \sim \Us([0, 1]) $ (i.e. $ D_{\infty, r}^{(j)} $ are beta-distributed), $ S_r(\infty) := \sum_{j = 1}^{h_r}( D_{\infty, r}^{(j ) } +  \widetilde{D}_{\infty, r}^{(j ) } ) $ and $f_{m, \hb}(k) $ is the Lebesgue density of the random variable $ Z :\eqlaw \sum_{r = 0}^m  S_r(\infty) $ (recall also that $ h_0 := k - \abs{\hb} $).
\end{theorem}

% ==

\begin{proof}
Similarly to \eqref{RandomEq:Charpol} and \eqref{RandomEq:TruncatedCharpol}, one can write for the derivative
\begin{align}\label{RandomEq:Derivative}%$
\partial^r Z_{U_N}(e^{s/N}) = \sum_{\ell = r  }^{N } \ell^{\downarrow r} \sc_\ell(U_N) e^{s\ell/N} =: C_{N, r} \Esp{ \sc_{D_{N, r} }(U_N) e^{s D_{N, r} /N} \big\vert U_N }
\end{align}
where $ \ell^{\downarrow 0} = 1 $ and 
\begin{align}\label{Def:DnR}%$
\Prob{D_{N, r} = \ell} = \frac{\ell^{\downarrow r }  }{C_{N, r} } \Unens{r  \leq \ell \leq N}, \qquad C_{N, r} := \sum_{\ell = r  }^{N } \ell^{\downarrow r}
\end{align}

Note also that for $ x \in \Rr $
\begin{align*}%$
\overline{Z_{U_N}(x) } & = \det(I_N - x U_N\inv) = (-x)^N \det(U_N) \det(I_N - x\inv U_N) \\
                 & = (-1)^N \det(U_N) \sum_{\ell = 0}^N \sc_{N - \ell}(U_N) x^\ell
\end{align*}
hence that
\begin{align*}%$
\overline{\partial_r Z_{U_N}(x) } & = (-1)^N \det(U_N) \sum_{\ell = r + 1}^N \ell^{\downarrow r} \sc_{N - \ell}(U_N) x^{\ell } \\
                 & = (-1)^N \det(U_N) C_{N, r} \Esp{  \sc_{N - D_{N, r}}(U_N) x^{D_{N, r} } \big\vert U_N }
\end{align*}

As a result, introducing two independent triangular i.i.d. sequences $ (D_{N, r}^{(j_r)})_{0 \leq r \leq m, 1 \leq j_r \leq  h_r} $ and $ (\widetilde{D}_{N, r}^{(j_r)})_{0 \leq r \leq m, 1 \leq j_r \leq  h_r} $, one gets
\begin{align*}%$
\De_N(k ; \hb ; \sbb) & = (-1)^{N C(\hb)} \Ee \Bigg( \!\! \abs{Z_{U_N}(e^{s_0/N})}^{2(k - \abs{\hb})}   \prod_{r = 1}^m \prod_{j_r = 1}^{h_r} \sc_{ D_{N, r}^{(j_r)} }(U_N) e^{ s_r D_{N, r}^{(j_r)}/N }  \\
                  & \hspace{+4cm} \times \det(U_N)^{ \abs{\hb} } \prod_{r = 1}^m \prod_{j_r = 1}^{h_r} \sc_{ N - \widetilde{D}_{N, r}^{(j_r )} }(U_N) e^{ s_r \widetilde{D}_{N, r}^{(j_r )}/N } \!\Bigg)
\end{align*}

Recall that $ \abs{Z_{U_N}(e^{s_0 / N})}^{2(k - \abs{\hb} )} = (-e^{s_0 / N})^{N(k - \abs{\hb} )} \det(U_N)^{ k -\abs{\hb} }  Z_{U_N}(e^{s_0/N})^{2(k - \abs{\hb}) }  $ and that $ Z_{U_N}(e^{s_0/N}) = \Esp{ e^{s_0V_N/N} \sc_{V_N}(U_N) \vert U_N } = \Esp{ e^{s_0 (N - V_N)/N} \sc_{N - V_N}(U_N) \vert U_N } $ since $ V_N \eqlaw N - V_N := D_{N, 0} \sim \Us(\intcrochet{0, N}) $. Define $ h_0 := k - \abs{\hb} $ and 
\begin{align*}%$
\Xb & := \prth{ x_{j_r, r} }_{0 \leq r \leq m, 1 \leq j_r \leq h_r }\\
\Yb & := \prth{ y_{j_r, r} }_{0 \leq r \leq m, 1 \leq j_r \leq h_r } \\
\Xb^{\Db} \Yb^{\widetilde{\Db} } & := \prod_{r = 0}^m \prod_{j_r = 1}^{h_r} x_{j_r, r}^{D_{N, r}^{(j_r) } }  y_{j_r, r}^{\widetilde{D}_{N, r}^{(j_r) } }  \\
S_r(N) & := \frac{1}{N} \sum_{j_r = 1}^{h_r} ( D_{N, r}^{(j_r )} + \widetilde{D}_{N, r}^{(j_r )} ) \\
\widetilde{\De}_N(k ; \hb ; \sbb) & :=   \De_N(k ; \hb ; \sbb) \times \prod_{r = 0}^m C_{N, r}^{-2 h_r}
\end{align*}

Then, conditioning on $ (\Db, \widetilde{\Db}) $ (i.e. integrating only on $ U_N $), one gets
\begin{align*}%$
\widetilde{\De}_N(k ; \hb ; \sbb) & = (-1)^{Nk} e^{s_0 h_0} \Esp{ \! e^{ \sum_{r = 0}^m s_r \sum_{j_r = 1}^{h_r} ( D_{N, r}^{(j_r )} + \widetilde{D}_{N, r}^{(j_r )} )/N  } \! \det(U_N)^{-k}   \crochet{ \Xb^{\Db} \Yb^{ -(N -\widetilde{\Db}) } }  \!\!\!\!\! \prod_{ t \in \Xb + \Yb\inv } \!\!\!\!\!\!  Z_{U_N}(t) \!\! } \\
                  & = e^{s_0 h_0} \Esp{ \! e^{ \sum_{r = 0}^m s_r S_r(N)  }   \crochet{ \Xb^{\Db} \Yb^{ \widetilde{\Db} } }  s_{N^k}\crochet{\Xb + \Yb} } \qquad\mbox{ by \eqref{Eq:SchurJointMoments}} \\
                  & = \frac{e^{s_0 h_0}}{k!} \Esp{ \! e^{ \sum_{r = 0}^m s_r S_r(N)  }   \crochet{ \Xb^{\Db} \Yb^{ \widetilde{\Db} } } \oint_{\Uu^k} U^{-N} H\crochet{U(\Xb + \Yb)} \abs{\Delta(U)}^2 \frac{d^*U}{U} } \mbox{ by \eqref{FourierRepr:SchurRectangle}} \\
                  & = \frac{e^{s_0 h_0}}{k!} \Esp{ \! e^{ \sum_{r = 0}^m s_r S_r(N) } \oint_{\Uu^k } U^{-N} \prod_{d \in \Db} h_{ d }\crochet{U}  \prod_{\widetilde{d} \in \widetilde{\Db}} h_{ \widetilde{d} }\crochet{U}  \abs{\Delta(U)}^2 \frac{d^*U}{U} }  \\
                  & = \frac{e^{s_0 h_0}}{k!} \Ee\Bigg( \! e^{ \sum_{r = 0}^m s_r S_r(N) }  \Unens{ \sum_{r = 0}^m  N S_r(N)  = kN } \\
                  & \hspace{+4cm} \oint_{\Uu^{k - 1} } V^{-N} \prod_{d \in \Db} h_{ d }\crochet{1 + V}  \prod_{\widetilde{d} \in \widetilde{\Db}} h_{ \widetilde{d} }\crochet{1 + V}  \abs{\Delta\crochet{1 + V} }^2 \frac{d^*V}{V} \Bigg)
\end{align*}
using the trick of remark \ref{Rk:MultivariateFourierWithHomogeneity} and integrating out $ u_1 $.

To conclude, we proceed as in \S~\ref{Subsec:RandomisationParadigm}/remark~\ref{Rk:TruncatedMomentsWithRandomisation}, i.e. we use a local limit theorem for $ \sum_{r = 1}^m N S_r(N) $ where the limit is justified in lemma~\ref{Lemma:RescaledGegenbauer}, and the (easily proven) convergence in law $ D_{N, r}/N \to \betab_{r + 1} \eqlaw U^{1/(r + 1)} $, where $ U $ is uniform on $ [0, 1] $, i.e. $ \betab_{r + 1} $ is Beta distributed of parameter $(r + 2, 1)$ (which means $ \Prob{ \betab_{r + 1} \in dx } = (r + 1) x^r dx $). 

Using two i.i.d. sequences $ \Db_\infty := (D_{\infty, r}^{(j_r)})_{0 \leq r \leq m, 1 \leq j_r \leq  h_r} $ and $ \widetilde{\Db}_\infty :=  (\widetilde{D}_{\infty, r}^{(j_r)})_{0 \leq r \leq m, 1 \leq j_r \leq  h_r} $ for such limits, using $ S_r(\infty) $ for the limit in law of $ S_r(N) $, and setting $ v_j = e^{2 i\pi x_j/N} $ for $ x_j \in [-N/2, N/2] $, one gets
\begin{align*}%$
\widetilde{\De}_N(k ; \hb ; \sbb) & = \frac{e^{s_0 h_0}}{k!} \Ee\Bigg( \! e^{ \sum_{r = 0}^m s_r S_r(N) }  \Unens{ \sum_{r = 0}^m  N S_r(N)  = kN } \int_{\crochet{-\frac{N}{2}, \frac{N}{2}}^{k - 1} } e^{- 2 i \pi \sum_{j = 2}^k x_j}  \\
                  & \hspace{+2.3cm}  \times \prod_{d \in \Db} h_{ d }\crochet{1 + e^{2 i \pi X/N}} \prod_{\widetilde{d} \in \widetilde{\Db}} h_{ \widetilde{d} }\crochet{1 + e^{2 i \pi X/N}}  \abs{\Delta\crochet{1 + e^{2 i \pi X/N}} }^2 \frac{dX}{N^{k - 1}} \Bigg) \\
% ==================================================
                  & \equivalent{N \to +\infty } N^{- 1 + 2(k - 1) A -k(k - 1) - (k - 1) } \frac{(2\pi)^{k(k - 1)} }{k!} e^{s_0 h_0} \Ee\Bigg( \! e^{ \sum_{r = 0}^m s_r S_r(\infty) }    \delta_0\prth{ \sum_{r = 0}^m  S_r(\infty)  - k }  \\ 
                  & \hspace{+4cm}   \int_{\Rr^{k - 1} }  \!\! e^{- 2 i \pi \sum_{j = 2}^k x_j}  \!\! \prod_{d \in \Db_\infty} \!\! \widetilde{h}_{ d, \infty }\prth{0, \xb}  \!\!  \! \prod_{\widetilde{d} \in \widetilde{\Db}_\infty} \!\! \widetilde{h}_{ \widetilde{d}, \infty }\prth{0, \xb}  \Delta(0, \xb)^2 d\xb   \Bigg)  
\end{align*}

Here, $ A $ is the cardinal of the random set $ \Db_\infty $, i.e. $ A = \sum_{r = 0}^m \sum_{j = 1}^{h_r} 1 = \sum_{r = 0}^m h_r = \abs{\hb} + h_0 = k $. One moreover easily finds that $ C_{N, r}  \sim \frac{N^{r + 1}}{r + 1} $, hence
\begin{align*}%$
\prod_{r = 0}^m C_{N, r}^{ 2 h_r} \equivalent{N\to +\infty} N^{ 2 \sum_{r = 0}^m (r + 1) h_r } \prod_{r = 0}^m \frac{1}{(r + 1)^{2 h_r}}  = N^{ 2 C(\hb) + 2 (\abs{\hb} + h_0) } \prod_{r = 0}^m \frac{1}{(r + 1)^{2 h_r}}
\end{align*}
which implies that 
\begin{align*}%$
\De_N(k ; \hb ; \sbb)  & \equivalent{N \to +\infty } N^{k^2 + 2 C(\hb) } \prod_{r = 0}^m \frac{1}{(r + 1)^{2 h_r}} \frac{(2\pi)^{k(k - 1)} }{k!} e^{s_0 h_0}   f_{m, \hb}(k) \\ 
                  & \qquad \times  \Ee\Bigg( \! e^{ \sum_{r = 0}^m s_r S_r(\infty) } \int_{\Rr^{k - 1} }  \!\! e^{- 2 i \pi \sum_{j = 2}^k x_j}  \!\! \prod_{d \in \Db_\infty} \!\! \widetilde{h}_{ d, \infty }\prth{0, \xb} \\
                  & \hspace{+3cm}  \prod_{\widetilde{d} \in \widetilde{\Db}_\infty} \!\! \widetilde{h}_{ \widetilde{d}, \infty }\prth{0, \xb}  \Delta(0, \xb)^2 d\xb \Bigg\vert  \sum_{r = 0}^m  S_r(\infty) = k  \Bigg)  
\end{align*}
where $ f_{m, \hb} $ is the Lebesgue density of the random variable $ \sum_{r = 0}^m S_r(\infty) $.

To pass to the limit, we have used dominated convergence given by inequality \eqref{Ineq:DominationSupersymWithoutSpeed} with $ Y = \emptyset $. As in the proof of theorems~\ref{Theorem:Autocorrelations} and~\ref{Theorem:TruncatedCharpolHeapLindqvist}, to show that the limiting function is integrable we use the criteria \eqref{Ineq:IntegrabilitySupersymDominationBis} with $ (K, L, M) = (k - 1, 0, 2k) $ ; it is satisfied if $ (k - 1)(2k(k - 2) - k + 1) > 2k $, i.e. $ (k - 1)(2k^2 - 3k - 3) > 2 $, which is true for all $ k \geq 2 $. The integrability in $ d, \widetilde{d} $ is the consequence of lemma~\ref{Lemma:IntegrabilityDominationInC}.
\end{proof}

% ==

\medskip
\begin{remark}\label{Rk:Alternative:Derivative}
In the same vein as before, one can use an alternative ending to the proof using \eqref{FourierRepr:SchurRectangleWithHn} in place of \eqref{FourierRepr:SchurRectangle}~:
\begin{align*}%$
\widetilde{\De}_N(k ; \hb ; \sbb) & = e^{s_0 h_0} \Esp{ \! e^{ \sum_{r = 0}^m s_r S_r(N)  }   \crochet{ \Xb^{\Db} \Yb^{ \widetilde{\Db} } }  s_{N^k}\crochet{\Xb + \Yb} } \\
                  & = \frac{e^{s_0 h_0}}{k!} \Esp{ \! e^{ \sum_{r = 0}^m s_r S_r(N)  }   \crochet{ \Xb^{\Db} \Yb^{ \widetilde{\Db} } } \oint_{\Uu^k} U^{-N} h_{Nk}\crochet{U(\Xb + \Yb)} \abs{\Delta(U)}^2 \frac{d^*U}{U} } \\%\quad\mbox{ by \eqref{FourierRepr:SchurRectangleWithHn}} \\
                  & = \frac{e^{s_0 h_0}}{k!} \\
                  & \ \ \times \Esp{ \! e^{ \sum_{r = 0}^m s_r S_r(N)  }   \crochet{ \Tb^{\Db'} \Yb^{ \widetilde{\Db} } } \!\! \oint_{\Uu^{k - 1} }\!\!\!\! V^{-N} h_{Nk}\crochet{(1 \!+\! V)(1 \!+\! \Tb \!+\! \Yb)} \abs{\Delta\crochet{1 + V}}^2 \frac{d^*V}{V} } 
\end{align*}
using the trick of remark \ref{Rk:MultivariateFourierWithHomogeneity}. Here, we have used $ \Db' := \Db\!\setminus\!\{ D_{N, 0}^{(1)} \} $, and we have integrated directly on $ u_1 $, getting the indicator $ \Unens{ (2k - 1) D_{N, 0}^{(1)} = 0 } = \Unens{ D_{N, 0}^{(1)} = 0 } $. On this last set, the variable $ D_{N, 0}^{(1)} $ in $ S_0(N) $ vanishes and one gets $ S_0'(N) $ by removing it so. Taking the expectation shows a factorisation by $ \Prob{ D_{N, 0}^{(1)} = 0 } = \frac{1}{N + 1} $ since $ D_{N, 0}^{(1)} \sim \Us(\intcrochet{0, N}) $ by \eqref{Def:DnR}. Hence, this value remains in the product of constants $ C_{N, r} $, leaving intact the previous expression. 

The number of variables of $ h_{Nk} $ is $ 2k^2 $ hence a rescaling in $ N^{2k \abs{\hb} - 1 -k(k - 1) - (k - 1)} $ ; one finally finds $ N^{k^2 + 2 C(\hb) } $ after taking the product of constants $ C_{N, r} $ into account. 

In the end, we can replace $ \Phi^{(Rand)}_{\De_\infty(k ; \hb ; \sbb)}(0, \xb) $ in \eqref{EqPhi:Random:DerivativesMicro} by 
\begin{align}\label{EqPhi:Alternative:DerivativesMicro}%$
\begin{aligned}
\hspace{-0.5cm}  \Phi^{(Alt)}_{\De_\infty(k ; \hb ; \sbb)}\!(0, \xb) \! &  := e^{ - 2 i \pi \sum_{j = 2}^k x_j } \!\! \int_{\Rr^{2k - 1}} \hspace{-0.3cm}  \widetilde{h}_{  k , \infty } \crochet{ (0 + \xb) \tensorsum (0 + \tb + \yb)   }   \\
                  & \hspace{+2.9cm} \times \Esp{\! e^{- 2i \pi \, \prth{ \sum_{r = 0}^m s_r S'_r(\infty) + \sum_{r, j_r} (t_{j_r, r} \widehat{D}_{\infty, r}^{(j_r) } + y_{j_r, r} \widetilde{D}_{\infty, r}^{(j_r) } ) } } }   \!    d\tb \, d\yb           
\end{aligned}
\end{align}
where $ S'_r(\infty) $ is equal to $ S_r(\infty) $ for $ r \geq 1 $ (and has $ D_{N, 0}^{(1)} = 0 $ for $ r = 0 $).
\end{remark}

% ==

\begin{remark}
We can make a parallel with the work of Assiotis-Keating-Warren \cite{AssiotisKeatingWarren} concerning the appearance of the power $ k^2 $ and the power $ 2 C(\hb) $ in the order of $ \widetilde{\De}_N(k ; \hb ; \sbb) $~: the power $ k^2 $ comes from the usual rescaling of the functional with $ h_{[cN]} $ whereas the power $ 2 C(\hb) $ comes from the (renormalisation constant of the) randomisation. This is somewhat similar to the change of probability with $ \abs{Z_{U_N}(1)}^{2k} $ that gives the power $ k^2 $ using the limit \eqref{EqPhi:KS}. Here, instead of a change of probability, we have used a randomisation which is one of the (other) main fundamental probabilistic concepts.
\end{remark}

\medskip
%\newpage

% ==
\subsection{Sums of divisor functions in $ \Ff_q\crochet{X} $}\label{Subsec:KR3G}

\subsubsection{Motivations} 

A successful heuristic to tackle problems in analytic number theory is to find analogues over function fields, namely on $ \Ff_q\crochet{X} $ for $ q $ a power of a prime number, $ \Ff_q $ being the finite field with $q$ elements. The many successes of this philosophy culminates with the work of Grothendieck expressing the zeta function of varieties over finite fields as the characteristic polynomial of the Frobenius map acting on a particular space of cohomology, and the proof of the Riemann hypothesis for this zeta function by Deligne. Several other results have seen their analogues over function fields resolved, allowing to create a conjecture in the number field case by analogy ; for instance, a conjecture of Montgomery-Odlyzko concerning spacings of zeroes of $L$-functions had its analogue proven by Katz and Sarnak \cite{KatzSarnak} on $ \Ff_q\crochet{X} $, and so was the Chowla and twin primes conjecture \cite{SawinShusterman}.

In the recent paper \cite{KeatingR3G}, the authors consider the function field analogue of a classical problem in number theory\footnote{This problem is also related to the moments of $ \zeta $ as $ \zeta(s)^k = \sum_{n \geq 1} d_k(n) n^{-s} $, see \cite{BettinConrey} for further motivations.}, the centered moments of $ d_k(U_n + H) - d_k(U_n) $ where $ U_n $ is a random variable uniformly distributed in $ \intcrochet{n + 1, 2n} $, $ H \in \intcrochet{1, n} $ and $ d_k $ is the $k$-th divisor function defined by
\begin{align*}%$
d_k(N) := \sum_{\ell_1, \dots, \ell_k \geq 1} \Unens{ \ell_1 \cdots \ell_k = N }
\end{align*}

Set $ \Delta_k(U_n) := d_k(U_n) - \Esp{ d_k(U_N) } $ and $ \Delta_k(U_n, H) := \Delta_k(U_n + H) - \Delta_k(U_n) $. One is interested in computing the moments $ \Esp{ \Delta_k(U_n, H)^{2r} } $, and in particular the case $r = 1$ (variance) when $ H = n^\delta $ for $ \delta \in (0, 1 - k\inv) $. Based on a function field analogue, using an expansion over irreducible characters of $ \Ff_q\crochet{X} $ and Katz' equidistribution theorem \cite{KatzSarnak} that allows to express the limit of an average over irreducible characters as a functional of the characteristic polynomial of a Haar-distributed random unitary matrix, the authors of \cite{KeatingR3G} form the following conjecture:

\begin{conjecture}[Keating, Rodgers, Roditty-Gershon and Rudnick] When $ n \to +\infty $ and $ \delta \in (0, 1 - k\inv) $
\begin{align*}%$
\Esp{ \Delta_k(U_n, n^\delta)^2 } \sim a_k \Pe_k(\delta) n^\delta (\log n)^{k^2 - 1}
\end{align*}
with $ a_k $ the arithmetic factor given in \eqref{Def:ArithmeticFactor} and $ \Pe_k(\delta) = (1 - \delta)^{k^2 - 1} \Ie_{ (1 - \delta)\inv }(k) $.
\end{conjecture}

The factor $ \Ie_c(k) $ is analoguous to $ \gamma_k $ previously defined. It is a random matrix factor coming from
\begin{align*}%$
I_k(m, N) := \Esp{ \abs{ \sum_{ \substack{ 1 \leq j_1, \dots, j_k \leq N \\ j_1 + \cdots + j_k = m } } \sc_{j_1}(U_N) \dots \sc_{j_k}(U_N) }^2 }
\end{align*}
or, equivalently, expressing $ \Unens{j_1 + \cdots + j_k = m} $ as a Fourier coefficient:
\begin{align}\label{Def:KR3Gfunctional}%$
I_k(m, N) := \int_{ \Ue_N } \abs{ \, \crochet{x^m} \det(I - x U)^k }^2 dU  
\end{align}

The following theorem is given in \cite[thm. 1.5]{KeatingR3G}:

\begin{theorem}[Keating, Rodgers, Roditty-Gershon and Rudnick]\label{Thm:KR3G}
We have for all $ c \in \crochet{0, k} $
\begin{align}\label{Eq:ConvergenceKR3G}%$
\frac{ I_k(\pe{ c N} , N)  }{ N^{k^2 - 1} } \tendvers{N}{+\infty } \Ie_c(k)  
\end{align}
where, using $ G(1 + k) := 1\cdot 2!\cdot 3!\cdots (k - 1)! $,
\begin{align*}%$
\Ie_c(k) := \frac{1}{k! G(1 + k)^2 } \int_{ [0, 1]^k } \delta_0\prth{ \sum_{j = 1}^k u_i - c}  \Delta(u_1, \dots, u_k)^2 du_1\dots du_k
\end{align*}
\end{theorem}

% ==

\begin{remark}\label{Rk:GorodetskyRodgers}
A generalisation of this result is proven in \cite{GorodetskyRodgers} with a determinantal method, hence also holds for $ k \notin \Nn $. The functional investigated in \cite{GorodetskyRodgers} writes for $ \alpha \in \Rr_+^* $ (and in particular $ \alpha = \frac{1}{2} $)
\begin{align*}%$
\widetilde{I}_\alpha (m , N) := \Esp{ \abs{ \, \crochet{x^m} Z_{U_N}(x)^\alpha Z_{U'_{N - 1} }(x)^\alpha }^2  } 
\end{align*}
where $ U_N \sim CUE_N $ is independent of $ U'_{N - 1} \sim CUE_{N - 1} $. Note also that the Toeplitz connection (i.e. the Andr\'eieff-Heine-Szeg\"o/Cauchy-Binet formula) allows to write such a functional with a Toeplitz determinant with symbol of the type $ (1 - t)^\alpha $ for $ \alpha = \frac{1}{2} $ and is present in particular cases of spin-spin correlations of the Ising model \cite[(31), (25), (94)]{DeiftItsKrasovsky}.
\end{remark}

% ==
\subsubsection{An alternative proof}

We now give another proof of theorem \ref{Thm:KR3G} with the previous machinery, leading to another expression of $ \Ie_c(k) $. 

\begin{theorem}[Keating, Rodgers, Roditty-Gershon and Rudnick rederived]\label{Thm:KR3G:withFormula}
The convergence \eqref{Eq:ConvergenceKR3G} holds, namely
\begin{align}\label{Eq:ConvergenceKR3Gbis}%$
\frac{ I_k(\pe{ c N} , N)  }{ N^{k^2 - 1} } \tendvers{N}{+\infty } \Je_c(k)  
\end{align}
with
\begin{align}\label{EqPhi:KR3G}%$
\begin{aligned}
\Je_c(k) & = \frac{(2\pi)^{ k(k - 1) }  }{k! } \int_{ \Rr^{k - 1} } \Phi_{\Je_c}(0, x_2, \dots, x_k) \Delta\prth{ 0, x_2, \dots, x_k }^2  dx_2 \dots dx_k \\
\Phi_{\Je_c}(0, x_2, \dots, x_k) & := e^{ - 2 i \pi \sum_{j = 2}^k x_j }  h^{(k)}_{  c , \infty }\prth{0, x_2, \dots, x_k } h^{(k)}_{ k - c, \infty  }\prth{0, x_2, \dots, x_k }             
\end{aligned}
\end{align}
\end{theorem}

% ==

\begin{proof} 
Expressing the sum as a Fourier coefficient, we have
\begin{align*}%$
I_k(m, N) & := \int_{ \Ue_N } \abs{ \, \crochet{x^m} \det(I - x U)^k }^2 dU \\
            & = \crochet{x^m y^m } \int_{ \Ue_N }  \det(I - x U)^k \det(I - y U\inv)^k  dU \\
            & = \crochet{x^m t^{kN - m} } \int_{ \Ue_N } \det(U)^{-k} \det(I - x U)^k \det(I - tU)^k  dU \times (-1)^{kN} \\
            & = \crochet{x^m t^{kN - m} } \int_{ \Ue_k } \det(U)^{-N} \det(I - x U)^{-k} \det(I - tU )^{-k}  dU  \times (-1)^{kN} \\
            & = \crochet{x^m t^{kN - m} } s_{N^k}\crochet{ (x + t) 1^k } \qquad\mbox{by \eqref{Eq:SchurJointMoments}.}
\end{align*}

Note that under this form, it is clear that $ I_k(m, N) = I_k(kN - m, N) $ as remarked in \cite[lemma 4.1]{KeatingR3G}. Note also that one could set $ x t\inv = y $ or $ x\inv t = \widetilde{y} $ to obtain $ I_k(m, N) = \crochet{y^{m - kN} } s_{N^k}\crochet{ (y + 1) 1^k } = \crochet{\widetilde{y}^{m } } s_{N^k}\crochet{ (\widetilde{y} + 1) 1^k } $ ; this last form does not exhibit the previous symmetry, though.

One then has with \eqref{FourierRepr:SchurRectangle} and $ U := (u_1, \dots, u_k) $
\begin{align*}%$
I_k(m, N) & = \crochet{x^m t^{kN - m} } \frac{1}{k!} \int_{ \Uu^k } U^{-N} H\crochet{(x + t) U}^k \abs{\Delta(U) }^2 \frac{d^*U}{U} \\
             & = \crochet{x^m t^{kN - m} } \frac{1}{k!} \int_{ \Uu^k } U^{-N} H\crochet{(x + t) 1^{\plusInOne k} U}   \abs{\Delta(U) }^2 \frac{d^*U}{U} \\
             & = \frac{1}{k!} \int_{ \Uu^k } U^{-N} h_m\crochet{ 1^{\plusInOne k} U} h_{kN - m}\crochet{ 1^{\plusInOne k} U}  \abs{\Delta(U) }^2 \frac{d^*U}{U} 
\end{align*}

Setting $ c = k \rho $ with $ \rho \in (0, 1) $ so that $ c \in (0, k) $ and $ k N - \pe{ c N} = \pe{k \overline{\rho} N} $, one gets
\begin{align*}%$
I_k(\, \crochet{k \rho N}, N) & = \frac{1}{k!} \int_{ \Uu^k } U^{-N} h^{(k)}_{ \pe{k \rho N} }(U) h^{(k)}_{\pe{k \overline{\rho} N } }(U)   \abs{\Delta(U) }^2 \frac{d^*U}{U} \\
              & = \frac{1}{k!} \int_{ \Uu^{k - 1} } V^{-N} h^{(k)}_{ \pe{k \rho N} }\crochet{1 + V }  h^{(k)}_{\pe{k \overline{\rho} N } }\crochet{1 + V} \abs{\Delta\crochet{ 1 +  V } }^2 \frac{d^*V}{V}
\end{align*}
using the trick of remark~\ref{Rk:MultivariateFourierWithHomogeneity} and integrating out $ u_1 $. 

Now, set $ v_j = e^{2 i \pi x_j/N } $ for all $ j \in \intcrochet{2, k} $, $ x_j \in [-N/2, N/2] $. This yields
\begin{align*}%$
I_k(\, \crochet{k \rho N}, N)  & = \frac{1}{k!} \int_{ \crochet{ - \frac{N}{2}, \frac{N}{2} }^{k - 1} } e^{ - 2 i \pi \sum_{j = 2}^k x_j } h^{(k)}_{ \pe{k \rho N} }\crochet{1 + e^{2 i \pi \xb/N} }  h^{(k)}_{\pe{k \overline{\rho} N } }\crochet{1 + e^{2 i \pi \xb/N} } \\
              & \hspace{+9cm}  \abs{\Delta\crochet{1 + e^{2 i \pi \xb/N} } }^2 \frac{d\xb }{N^{k - 1} }   \\
              & \equivalent{N\to +\infty } N^{2( k^2 - 1 ) - k(k-1) - (k - 1)}  \frac{ (2\pi)^{ k(k - 1) }  }{k!} \!\! \int_{ \Rr^{k - 1} }  e^{ - 2 i \pi \sum_{j = 2}^k x_j } \widetilde{h}^{(k)}_{  k \rho, \infty }\prth{0, \xb } \widetilde{h}^{(k)}_{ k \overline{\rho}, \infty  }\prth{0, \xb }  \\
              & \hspace{+11.8cm}        \Delta\prth{ 0, \xb }^2  d \xb \\
              & =    N^{ k^2 - 1  }  \Je_c(k)
\end{align*}

To pass to the limit, we have used the domination given by inequality \eqref{Ineq:SpeedOfConvergenceHinftyBis}. The limiting function is integrable using the second criteria \eqref{Ineq:IntegrabilityDomination} with $ (K, M, M', \kappa, \kappa') = (k - 1, 1, 1, k, k) $ which is clearly fullfilled for all $ p \in \ensemble{0, 1, 2} $ and $ k \geq 2 $.
\end{proof}

% ==

\begin{remark}
There are other expressions for $ \Ie_c(k) $ (hence $ \Je_c(k) $) in the literature. For instance, \cite{KeatingR3G} also gives a combinatorial expression and \cite{BasorGeRubinstein} gives an expression as the (inverse) Fourier transform of a Hankel determinant of size $ k \times k $ satisfying a Painlev\'e V equation. The Hankel connection is just an artefact of the appearance of $ \Delta(U)^2 $ in the expression of $ \Ie_c(k) $ (see \S~\ref{Subsec:Ultimate:Hankel}) added to the Fourier representation of the Dirac mass as the inverse Fourier transform of $ 1 $. The fact that such a Hankel determinant satisfies a Painlev\'e V equation is deep and related with integrable systems \cite{BasorGeRubinstein, ForresterWittePainleve3p5}.
\end{remark}

% ==

\begin{remark}\label{Rk:Alternative:KR3G}
As in remarks~\ref{Rk:Alternative:MidSecularCoefficient}, \ref{Rk:Alternative:Truncation}, \ref{Rk:Alternative:Derivative}, or \ref{Rk:Alternative:KS}, one can use \eqref{FourierRepr:SchurRectangleWithHn} instead of \eqref{FourierRepr:SchurRectangle}~:
\begin{align*}%$
I_k(\, \pe{k \rho N}, N) & = \crochet{x^{\pe{k \rho N}} t^{\pe{k \overline{\rho} N}} }    \frac{1}{k!} \oint_{ \Uu^k } U^{-N} h_{Nk}\crochet{ U  ( x + t ) 1^{\plusInOne k} }  \abs{\Delta(U) }^2 \frac{d^*U}{U} \\ 
             & =   \crochet{ y^{\pe{k \rho N}} } \frac{1}{k!} \oint_{ \Uu^k } U^{-N} h_{Nk}\crochet{ U  ( 1 + y ) 1^{\plusInOne k} }  \abs{\Delta(U) }^2 \frac{d^*U}{U} \\
             & =  \crochet{ y^{\pe{k \rho N}} } \frac{1}{k!} \oint_{ \Uu^{k - 1} } V^{-N} h_{Nk}\crochet{ (1 + V)  ( 1 + y ) 1^{\plusInOne k} }  \abs{\Delta\crochet{1 + V} }^2 \frac{d^*V}{V}
\end{align*}
where we have used the trick of remark~\ref{Rk:MultivariateFourierWithHomogeneity} and the ``de-symmetrisation'' $ y = x t\inv $, using the scaling of $ h_{Nk} $. Since we have exactly the same functional as in remarks~\ref{Rk:Alternative:MidSecularCoefficient}, \ref{Rk:Alternative:Truncation} and \ref{Rk:Alternative:Derivative} (with less variables), we conclude in the same way that one can replace $ \Phi_{\Je_c} $ in \eqref{EqPhi:KR3G} by
\begin{align}\label{EqPhi:Alternative:KR3G}%$
%\begin{aligned}
\Phi^{(Alt)}_{\Je_c}(0, x_2, \dots, x_k) := e^{ - 2 i \pi \sum_{j = 2}^k x_j } \int_\Rr e^{- 2i \pi c y} \, \widetilde{h}_{  k , \infty } \crochet{ (0 + y) \tensorsum (0 + \xb) \tensorsum  0^k  }  dy           
%\end{aligned}
\end{align}
\end{remark}

\medskip
% ==
\subsubsection{Proof by randomisation}\label{Subsubsec:RandomProof:KR3G}

Using the randomisation paradigm of \S~\ref{Subsec:RandomisationParadigm} and in particular \eqref{RandomEq:Charpol}, one can transform \eqref{Def:KR3Gfunctional} into
\begin{align*}%$
I_k(m, N) & := \Esp{ \abs{ \, \crochet{x^m} Z_{U_N}(x)^k }^2 }  \\
              & = \Esp{ \, \crochet{x^m y^m} \det(I - x U_N)^k \det(I - y U_N\inv)^k  } \\
              & = N^{2k} \crochet{x^m y^m} (-1)^{Nk} y^{ Nk } \, \Esp{ \det(U_N)^{-k} x^{ \sum_{j = 1}^{k} V_N^{(j)} } y^{ -\sum_{j = k + 1}^{2k} V_N^{(j)} } \prod_{j = 1}^{2k} \sc_{V_N^{(j)} }(U_N) } \\
              & = (-1)^{Nk} N^{2k} \, \Esp{ \!  \Unens{ \sum_{j = 1}^k V_N^{(j)} = m , \ \sum_{j = k + 1}^{2k} V_N^{(j)} = Nk - m } \det(U_N)^{-k} \prod_{j = 1}^{2k} \sc_{V_N^{(j)} }(U_N) }
\end{align*}
where $ (V_N^{(j)})_{j \geq 1} $ is a sequence of i.i.d. uniform random variables in $ \intcrochet{0, N} $. Defining the event $ \Se_{N, m} := \ensemble{  \sum_{j = 1}^k V_N^{(j)} = m , \  \sum_{j = k + 1}^{2k} V_N^{(j)} = Nk - m } $, defining the random partition $ \Vb_{\! N} \vdash Nk $ of length $ 2k $ by the ordering of the vector $ (V_N^{(j)})_{1 \leq j \leq 2k} $ and pursuing as in the case of the autocorrelations in \S~\ref{Subsec:RandomisationParadigm}, one gets with $ \Xb := \ensemble{x_1, \dots, x_{2k}} $~:
\begin{align*}%$
I_k(m, N) & = N^{2k} \, \Esp{ \Un_{\Se_{N, m}}     \crochet{ \Xb^{ \Vb_{\! N} } } s_{N^k}\crochet{ \Xb } } \quad\mbox{by \eqref{Eq:SchurJointMoments}} \\
                 & = \frac{N^{2k}}{ k!} \,    \Esp{ \Un_{\Se_{N, m}}  \crochet{ \Xb^{ \Vb_{\! N} } } \!\! \oint_{\Uu^k}  U^{-N} H\crochet{U \Xb } \abs{\Delta(U)}^2 \frac{d^*U}{U} } \mbox{ by \eqref{FourierRepr:SchurRectangle}}  \\
                 & = \frac{N^{2k}}{ k!} \, \Esp{ \Un_{\Se_{N, m}} \oint_{\Uu^k}  U^{-N}  h_{  \Vb_{\! N} }(U) \abs{\Delta(U)}^2 \frac{d^*U}{U} } \mbox{ by \eqref{FourierRepr:hLambda}} \\
                 & = \frac{N^{2k}}{ k!} \, \Ee \Bigg( \Un_{\Se_{N, m}} \Unens{ \sum_{j = 1}^{2k} V_N^{(j)} = kN }    \oint_{\Uu^{k - 1}}  W^{-N} h_{  \Vb_{\! N} }\crochet{1 + W} \abs{\Delta\crochet{1 + W}}^2 \frac{d^*W}{W} \Bigg) 
\end{align*}
using the trick of remark \ref{Rk:MultivariateFourierWithHomogeneity} and integrating out $ u_1 $.

We remark that the event $ \ensemble{ \sum_{j = 1}^{2k} V_N^{(j)} = kN } $ is contained in $ \Se_{N, m} $ hence is superfluous. As a result, one gets 
\begin{align*}%$ 
I_k(m, N) = \frac{N^{2k}}{ k!} \  \Prob{\Se_{N, m}} \, \Esp{   \oint_{\Uu^{k - 1}}  W^{-N} h_{ \Vb_{\! N} }\crochet{1 + W} \abs{\Delta\crochet{1 + W}}^2 \frac{d^*W}{W} \Bigg\vert \Se_{N, m} \! }
\end{align*}

For $ m = \pe{c N} $, using the local CLT approach of \S~\ref{Subsec:RandomisationParadigm}, the uniform coupling \eqref{Eq:CouplingUniform} and setting $ w_j := e^{2 i \pi x_j/N} $, one gets with $ \Vb_{\! N} = \pe{N \Vb} $ (where $ \Vb $ is the vector of ordered uniform random variables from the coupling) and  $ \Se^{(\infty)}_{k, k - c} := \ensemble{ \sum_{j = 1}^k V^{(j)} = k, \sum_{j = 1}^k V^{(j + k)} = k - c } $~: 
\begin{align*}%$ 
I_k(\pe{c N}, N) & = \frac{N^{2k}}{ k!} \  \Prob{\Se_{N, \pe{c N}}} \, \\
                 & \qquad \Esp{   \int_{\crochet{-\frac{N}{2}, \frac{N}{2} }^{k - 1}}  e^{-\sum_{j = 2}^k x_j } h_{ [N \Vb]  }\crochet{1 + e^{2 i \pi \xb/N}} \abs{\Delta\crochet{1 + e^{2 i \pi \xb/N}}}^2 \frac{d\xb}{N^{k - 1}} \Bigg\vert \Se_{N, \pe{c N}} \! } \\
                 & \equivalent{N \to +\infty} \frac{N^{2k}}{ k!} \ \frac{f_{S_k}(c) f_{S_k}(k - c) }{N^2} \times  (2\pi)^{k(k - 1) }   N^{ 2k(k - 1) -k(k - 1) -(k - 1) } \\
                 & \hspace{+3cm}   \int_{\Rr^{k - 1}}  e^{-2 i \pi \sum_{j = 2}^k x_j } \, \Esp{  \widetilde{h}_{ \Vb\!,\, \infty  }(0, \xb)  \Big\vert \Se^{(\infty)}_{k, k - c} \! } \Delta(0, \xb)^2 d\xb \\
                 & = N^{k^2 - 1} \frac{(2\pi)^{k(k - 1) }}{k!} f_{S_k}(c) f_{S_k}(k - c) \int_{\Rr^{k - 1}}  e^{-2 i \pi \sum_{j = 2}^k x_j } \, \Esp{  \widetilde{h}_{ \Vb\!,\, \infty  }(0, \xb)  \Big\vert \Se^{(\infty)}_{k, k - c} \! } \\
                 & \hspace{+12cm}\Delta(0, \xb)^2 d\xb
\end{align*}
where we have applied dominated convergence as in theorem~\ref{Theorem:Autocorrelations} with the additional integrability in the uniform random variables coming from lemma~\ref{Lemma:IntegrabilityDominationInC}, with $ \widetilde{h}_{c, \infty} $ defined in \eqref{Def:hTildeInfty} and with $ f_{S_k}(x) $ the Lebesgue density of $ S_k := \sum_{j = 1}^k V^{(j)} $ computed in \eqref{Eq:BatesIrwinHall}. The local CLT is indeed applied twice as, by independence and equality in law
\begin{align*}%$ 
\Prob{\Se_{N, \pe{c N}}} = \Prob{ [N S_k] = N k  } \Prob{ [N S_k] = N k - [c N] }
\end{align*}

We have thus proven the following~:
\begin{theorem}[Keating, Rodgers, Roditty-Gershon and Rudnick rederived with randomisation]\label{Theorem:KR3G:WithRandomisation}
The convergence \eqref{Eq:ConvergenceKR3G}/\eqref{Eq:ConvergenceKR3Gbis} holds, with 
\begin{align}\label{EqPhi:Random:KR3G}%$
\begin{aligned}
\Je_c(k) & = \frac{(2\pi)^{ k(k - 1) } }{k! } \int_{ \Rr^{k - 1} } \Phi^{(Rand)}_{\Je_c}(0, \xb) \Delta\prth{ 0, \xb }^2  d\xb \\
\Phi^{(Rand)}_{\Je_c}(0, \xb) & :=   f_{S_k}(c) f_{S_k}(k - c) \,  e^{-2 i \pi \sum_{j = 2}^k x_k } \, \Esp{  \widetilde{h}_{ \Vb\!,\, \infty  }(0, \xb)  \Big\vert \Se^{(\infty)}_{k, k - c} \! }  
\end{aligned}
\end{align}
\end{theorem}

% ==

\medskip
% ==
\subsection{The moments of moments}\label{Subsec:MoMo}

% ==
\subsubsection{Motivations}

Techniques of \textit{detropicalisation} of a maximum\footnote{That is, replacing a max which is an $ L^\infty $ norm by an $ L^\beta $ norm for $\beta$ big.} have proven interesting to tackle questions of extreme value fluctuations for some particular correlated random variables. The canonical example of such an approach is the Sherrington-Kirckpatrick model of spin glasses and its descendents (Random Energy Model, $p$-spins model, etc. ; see e.g. \cite{PanchenkoBook}) but it was also successfully applied in other fields such as random matrix theory when Fyodorov and Keating \cite{FyodorovKeating} studied the maximum of $ \abs{Z_{U_N}} $ on $ \Uu $ as an attempt to model $ \max_{s \in [0, 1]} \abs{\zeta(\frac{1}{2} + i sT U) } $ with $ U $ uniform on $ [0, 1] $ (see also \cite{FyodorovHiaryKeating}).

The resulting conjecture for this particular problem writes (with two independent Gumbel-distributed random variables)
\begin{align}\label{Conj:FyodorovKeatingCUE}%$
\max_{z \in \Uu} \log\abs{ Z_{U_N}(z) } - \log(N) + \frac{3}{4} \log\log(N) \cvlaw{N}{+\infty} \Gumbel(1) + \Gumbel'(1)
\end{align}
and showed interesting connections with log-correlated Gaussian fields (see \cite{MadauleGaussien} for a general structure theorem satisfied by these fields). The first part of the conjecture, i.e. the law of large number for $ \max_{z \in \Uu}\log\abs{ Z_{U_N}(z) } $ when $ N\to+\infty $ was addressed in \cite{ArguinBeliusBourgade, ChhaibiMadauleNajnudel, PaquetteZeitouni}, the case of the fluctuations being still opened. See also \cite{ArguinBeliusBourgadeRadziwillSoundararajan, ArguinBourgadeRadziwill1, HarperZetaFK, NajnudelExtremeZeta} for advances on the case of $ \zeta $.

A key tool to arrive at conjecture \eqref{Conj:FyodorovKeatingCUE} relied on the \textit{moments of moments} defined in \eqref{FuncZ:MoMo}, whose expression we recall:
\begin{align}\label{FuncZ:MoMoBis}%$
\Mom(N\vert k, \beta) := \Esp{ \prth{ \oint_\Uu \abs{ Z_{U_N}(z) }^{2\beta} \frac{d^*z}{z} }^{\!\! k}\, } 
\end{align}

The study of $ \Mom(N\vert k, \beta) $ for $ k, \beta \in \Nn^* $ (hence $ k\beta^2 > 1 $) in the $ CUE $ case was performed in \cite{AssiotisKeating, BaileyKeating} with three different methods (the Schur formula \eqref{Eq:SchurCFKRS}, the Young tableau combinatorial method and the polytope method, but for the Gelfand-Tsetlin polytope in place of the Birkhoff one) while the $ COE $ and $ CSE $ cases were studied in \cite{AssiotisBaileyKeating}. As recalled in \S~\ref{Subsubsec:MoMo}, this is the difficult regime in terms of Toeplitz determinant analysis. The result of \cite{AssiotisKeating, BaileyKeating} writes:

% ==

\begin{theorem}[Bailey-Keating, Assiotis-Keating]\label{Thm:AssiotisBaileyKeating}
Let $ k, \beta \in \Nn^* $. Then, when $ N \to +\infty $,
\begin{align*}%$
\Mom(N\vert k, \beta) \sim c_+(k, \beta) N^{ (k\beta)^2 + 1 - k }
\end{align*}
\end{theorem}

See also \cite{Fahs} for a result without an expression of $ c_+(k, \beta) $. The expression of $ c_+(k, \beta) $ given in \cite[lem. 3.6]{BaileyKeating} is of the type $ \sum_{0 \leq \ell_1, \dots, \ell_{k - 1} \leq 2\beta } c_{k, \beta ; \ellb} ( (k - 1)\beta - \abs{\ellb} )^{ b_{k, \beta ; \ellb} - {k \choose 2} } P_{k, \beta}(\ellb) $ where $ \ellb := (\ell_1, \dots, \ell_{k - 1}) $ and $ \abs{\ellb} := \sum_{j = 1}^{k - 1} \ell_j $ ; here, $ b_{k, \beta ; \ellb} $ and $ c_{k, \beta ; \ellb} $ are constants, and $ P_{k, \beta}(\ellb) $ is of the type \eqref{Eq:SchurCFKRS} after rescaling (the sum coming from manipulations of the type described in \S~\ref{Subsec:Intro:ComparisonLiterature}). 

\medskip
% ==
\subsubsection{An alternative proof}

We now give another proof of theorem~\ref{Thm:AssiotisBaileyKeating} with the previous machinery, giving thus a new expression of $ c_+(k, \beta) $.

\begin{theorem}[Bailey-Keating \& Assiotis-Keating rederived]\label{Theorem:MoMo}
One has
\begin{align}\label{Eq:ConvergenceMoMo}%$
\frac{\Mom(N\vert k, \beta)}{ N^{ (k\beta)^2 + 1 - k } } \tendvers{n}{+\infty} \MoMf_+(k, \beta) 
\end{align}
with 
\begin{align}\label{EqPhi:MoMo}%$
\begin{aligned}
\MoMf_+(k, \beta) & = \frac{(2\pi)^{ k\beta(k\beta - 1) }  }{(k\beta)! } \int_{ \Rr^{{k\beta} - 1} } \!\!\! \Phi_{\MoMf_+(k, \beta)}(0, \xb) \Delta\prth{ 0, \xb }^2  dx_2 \dots dx_{k\beta} \\
\hspace{-0.3cm} \Phi_{\MoMf_+(k, \beta)\!}(0, x_2, \dots, x_{k\beta}) & :=  e^{ 2 i \pi (k\beta^2 - 1) \sum_{\ell = 2}^{k\beta  }  x_\ell } \,\,  h_{\beta, \infty}^{(2\beta)}( 0, \xb )^k             
\end{aligned}
\end{align}
\end{theorem}

% ==

\begin{remark}\label{Rk:KSandMoMo}
Note that for $ k = 1 $, one recovers $ \Phi_{\widetilde{L}_1(\beta)} $ given in \eqref{EqPhi:KS} since $ \Mom(N\vert 1, \beta) = \Esp{ \oint_\Uu \abs{Z_{U_N}(z) }^{2\beta} \frac{d^*z}{z} } = \oint_\Uu \Esp{  \abs{Z_{U_N}(z) }^{2\beta}  } \frac{d^*z}{z} = \Esp{ \abs{Z_{U_N}(1) }^{2\beta} } $ by invariance of the Haar measure. 
\end{remark}

% ==

\begin{proof}
One has
\begin{align*}%$
\Mom(N\vert k, \beta) & := \int_{ \Ue_N } \prth{ \vphantom{a^{a^{a^a}}} \! \crochet{x^0} \det(I - x U)^\beta \det(I - x\inv U\inv)^\beta }^k dU \\
            & = (-1)^{Nk\beta} \int_{ \Ue_N } \prth{ \vphantom{a^{a^{a^a}}} \! \det(U)^{-\beta} \crochet{x^0} x^{-N\beta} \det(I - x U)^{2\beta} }^k dU \\
            & = (-1)^{Nk\beta} \int_{ \Ue_N } \det(U)^{-k\beta}   \crochet{X^{N\beta}} H\crochet{ - U X 1^{\plusInOne 2\beta} } dU, \qquad X := \ensemble{x_1, \dots, x_k} \\
            & = (-1)^{Nk\beta} \crochet{X^{N\beta}} \int_{ \Ue_N } \det(U)^{-k\beta}   H\crochet{ - U X 1^{\plusInOne 2\beta} } dU \\
            & = \crochet{X^{N\beta}} s_{N^{k\beta} } \crochet{ X 1^{\plusInOne 2\beta} } \mbox{ by \eqref{Eq:SchurJointMoments}.}
\end{align*}

This last formula was the starting point of the computations in \cite{AssiotisBaileyKeating, AssiotisKeating, BaileyKeating}\footnote{Note also the slight (critical) typo in \cite[Prop. 2.1]{BaileyKeating} where $ \beta $ has to be replaced by $ 2\beta $, corrected in the proof of \cite[Prop. 2.6]{AssiotisKeating}.}. We now give an alternate end to this result using \eqref{FourierRepr:SchurRectangle} with $ U := \ensemble{u_1, \dots, u_{k\beta} } $~:
\begin{align*}%$
\Mom(N\vert k, \beta) & := \crochet{X^{N\beta}} \frac{1}{(k\beta)!} \oint_{\Uu^{k\beta} } U^{-N} H\crochet{U X 1^{\plusInOne 2\beta} } \abs{\Delta(U)}^2 \frac{d^*U}{U} \\
               & = \frac{1}{(k\beta)!} \oint_{\Uu^{k\beta} } U^{-N} h_{N\beta}\crochet{U 1^{\plusInOne 2\beta} }^k \abs{\Delta(U)}^2 \frac{d^*U}{U} \\
               & = \frac{1}{(k\beta)!} \oint_{\Uu^{k\beta} } U^{-N} u_1^{Nk\beta^2} h_{N\beta}\crochet{ ( 1 + V) 1^{\plusInOne 2\beta} }^k \abs{\Delta(U)}^2 \frac{d^*U}{U}
\end{align*}
where $ V := \ensemble{ u_2/u_1, \dots, u_{k\beta}/u_1 } $. Since $ \abs{\Delta(U)}^2 = \abs{\Delta\crochet{1 + V}}^2 $, one gets after integrating out $u_1$ as in remark \ref{Rk:MultivariateFourierWithHomogeneity} and setting $ v_j := e^{2 i \pi x_j/N} $ and $ x_j \in \crochet{-\frac{N}{2}, \frac{N}{2}} $ 
\begin{align*}%$
\Mom(N\vert k, \beta) & := \frac{1}{(k\beta)!} \oint_{\Uu^{k\beta - 1} } V^{-N}  h_{N\beta}\crochet{ ( 1 + V) 1^{\plusInOne 2\beta} }^k \abs{\Delta\crochet{1 + V}}^2 \frac{d^*V}{V}  \\
               & = \frac{1}{(k\beta)!} \int_{ \crochet{-\frac{N}{2}, \frac{N}{2}}^{k\beta - 1} } e^{ -2 i \pi \sum_{\ell = 2}^{k\beta }  x_\ell } \, h_{N\beta}^{(2\beta)}\crochet{ 1 + e^{2 i \pi \xb/N}  }^k \abs{\Delta\crochet{1 + e^{2 i \pi \xb/N} }}^2 \frac{d\xb}{N^{k\beta - 1}} \\
% =========================================               
               & \!\!\!\! \equivalent{N\to +\infty} \frac{1}{(k\beta)!} \int_{ \Rr^{k\beta - 1} } \!\! e^{ -2 i \pi \sum_{\ell = 2}^{k\beta  }  x_\ell }  \prth{ N^{2\beta( k\beta ) - 1} \widetilde{h}_{\beta, \infty}^{(2\beta)}\prth{ 0, \xb} }^k \! N^{-k\beta(k\beta - 1) } \\
               & \hspace{+9cm} \times \abs{\Delta\prth{0, 2 i \pi \xb } }^2 \! \frac{d\xb}{N^{k\beta - 1}} \\
% =========================================  
               & = N^{ k(2\beta(k\beta) - 1) - k\beta(k\beta - 1)  -k\beta + 1 } \, \frac{ (2\pi)^{k\beta(k\beta - 1) } }{ (k\beta)! }\\
               & \hspace{+5cm}\times \int_{ \Rr^{k\beta - 1} } \!\! e^{ -2 i \pi \sum_{\ell = 2}^{k\beta }  x_\ell } \, \widetilde{h}_{\beta, \infty}^{(2\beta)}( 0, \xb )^k \Delta(0, \xb)^2 d\xb \\
% =========================================  
               & = N^{k^2\beta^2 - k + 1 }  \, \frac{ (2\pi)^{k\beta(k\beta - 1) } }{ (k\beta)! } \!\! \int_{ \Rr^{k\beta - 1} } \!\! e^{ 2 i \pi (k\beta^2 - 1) \sum_{\ell = 2}^{k\beta }  x_\ell } \, h_{\beta, \infty}^{(2\beta)}( 0, \xb )^k \Delta(0, \xb)^2 d\xb
\end{align*}

Here, we have used dominated convergence under the form of inequality \eqref{Ineq:SpeedOfConvergenceHinftyBis} with an integrable limiting function due to the first criteria \eqref{Ineq:IntegrabilityDomination} with $ (K, M, \kappa) = (k\beta, k, 2\beta) $ which is clearly fullfilled for all $ p \in \ensemble{0, 1, 2} $ and $ k \geq 2 $ with $ \beta \geq 1 $. Remark that this condition writes in terms of the exponent, i.e. $ (k\beta)^2 - k + 1 \geq p + 1 $.
\end{proof}

% ==

\begin{remark}\label{Rk:Alternative:MoMo}
One can also use directly the expression \eqref{Eq:SchurWithHn}/\eqref{FourierRepr:SchurRectangleWithHn} to get
\begin{align*}%$
\Mom(N\vert k, \beta) & := \crochet{X^{N\beta}} s_{N^{k\beta} } \crochet{ X 1^{\plusInOne 2\beta} }, \qquad X := \ensemble{x_1, \dots, x_k} \\
               & =  \crochet{X^{N\beta}} \frac{1}{(k\beta)! } \oint_{\Uu^{k\beta - 1} } V^{-N} h_{Nk\beta}\crochet{(1 + V) X 1^{\plusInOne 2 \beta} } \abs{\Delta\crochet{1 + V}}^2 \frac{d^*V}{V} \\
               & = \frac{1}{(k\beta)! } \oint_{\Uu^{k - 1} } \!\! W^{-N\beta} \! \oint_{\Uu^{k\beta - 1} } \!\! V^{-N} h_{Nk\beta}\crochet{(1 + V) (1 + W) 1^{\plusInOne 2 \beta} } \abs{\Delta\crochet{1 + V}}^2 \frac{d^*V}{V}\frac{d^*W}{W}
\end{align*}

A slight adaptation of the first criteria \eqref{Ineq:IntegrabilityDomination} with variables $ W $ that are not integrated against a squared Vandermonde determinant allows to use dominated convergence to get (with the conventions of theorem~\ref{Theorem:Autocorrelations})
\begin{align*}%$
\Mom(N\vert k, \beta) & \equivalent{N \to +\infty} \int_{\Rr^{k(\beta + 1) - 2} } e^{ -2 i \pi  \left( \beta \sum_{j = 2}^k \theta_j + \sum_{\ell = 2}^{k\beta} \varphi_\ell \right) } \,  N^{ k\cdot k\beta \cdot 2\beta -1 } \, \widetilde{h}_{k\beta, \infty}^{(2\beta)}\crochet{ (0 + \thetab)\oplus(0 + \varphib) } \\
               &   \hspace{+6.5cm} \times N^{- k\beta(k\beta - 1) } \abs{\Delta\crochet{ 0 + 2 i \pi \varphib } }^2 \, \frac{d\thetab}{ N^{ k - 1 } } \, \frac{d\varphib}{ N^{k\beta - 1} }              \\
               & = N^{ (2k^2 \beta^2  - 1) - k\beta(k\beta - 1)  - k + 1 - k\beta + 1 } \, \frac{ (2\pi)^{k\beta(k\beta - 1) } }{ (k\beta)! } \!\! \int_{\Rr^{k(\beta + 1) - 2} } e^{ -2 i \pi  \left( \beta \sum_{j = 2}^k \theta_j + \sum_{\ell = 2}^{k\beta} \varphi_\ell \right) } \\
               & \hspace{ +7cm} \times \widetilde{h}_{k\beta, \infty}^{(2\beta)}\crochet{ (0 + \thetab)\oplus(0 + \varphib) } \Delta(0, \thetab)^2 d\thetab d\varphib \\
               & = N^{k^2\beta^2 - k + 1 }  \, \frac{ (2\pi)^{k\beta(k\beta - 1) } }{ (k\beta)! } \!\! \int_{\Rr^{k(\beta + 1) - 2} } e^{ -2 i \pi  \left( \beta \sum_{j = 2}^k \theta_j + \sum_{\ell = 2}^{k\beta} \varphi_\ell \right) } \\
               & \hspace{+6.5cm} \times \widetilde{h}_{k\beta, \infty}^{(2\beta)}\crochet{ (0 + \thetab)\oplus(0 + \varphib) } \Delta(0, \thetab)^2 d\thetab d\varphib
\end{align*}

We thus also have, with Fubini and setting $ \xb := \ensemble{x_2, \dots, x_{k\beta }} $
\begin{align}\label{EqPhi:Alternative:MoMo}%$
\!\!\!\Phi^{(Alt)}_{\MoMf_+(k, \beta)\!}(0, \xb) = e^{ -2 i \pi  ( k\beta^2 - 1 + \beta) \sum_{\ell = 2}^{k\beta} x_j } \!\! \int_{\Rr^{k - 1} } \!\!\! e^{ -2 i \pi  \sum_{\ell = 2}^{k\beta} \varphi_\ell }  \, \widetilde{h}_{k\beta, \infty}^{(2\beta)}\crochet{ (0 + \xb)\oplus(0 + \varphib) }   d\varphib
\end{align}

The expression \eqref{EqPhi:MoMo} seems more ergonomic than this last one as it involves less integrals~; nevertheless, one also has to take into account the amount of integrals used to define $ h_{c, \infty}^{(\kappa)} $, and the power $k$ in \eqref{EqPhi:MoMo} shows then an equal computational complexity.
\end{remark}

% ==
\medskip
\subsubsection{Proof by randomisation}

We now use the randomisation paradigm of \S~\ref{Subsec:RandomisationParadigm} to give an alternative expression to \eqref{EqPhi:MoMo}/\eqref{EqPhi:Alternative:MoMo}. 
\begin{theorem}[Bailey-Keating \& Assiotis-Keating rederived with randomisation]\label{Theorem:MoMo:WithRandomisation}
The convergence \eqref{Eq:ConvergenceMoMo} holds with ($ \xb := \ensemble{x_2, \dots, x_{k\beta }} $)
\begin{align}\label{EqPhi:Random:MoMo}%$
\begin{aligned}
\MoMf_+(k, \beta) & = \frac{(2\pi)^{ k\beta(k\beta - 1) }  }{(k\beta)! } \int_{ \Rr^{{k\beta} - 1} } \!\!\! \Phi^{(Rand)}_{\MoMf_+(k, \beta)}(0, \xb) \Delta\prth{ 0, \xb }^2  d\xb \\
\hspace{-0.3cm} \Phi^{(Rand)}_{\MoMf_+(k, \beta)\!}(0, \xb) & := \frac{1}{ (2(2\beta)!)^k }   \,  e^{-2 i \pi \sum_{j = 2}^{k\beta} x_j } \, \Esp{  \widetilde{h}_{ \Vb\!,\, \infty  }(0, \xb)  \Big\vert \Me_{k, \beta}^{(\infty)}  }  
\end{aligned}
\end{align}
where $ \Vb := (V^{(j, \ell) })_{1 \leq j \leq 2k, 1 \leq \ell \leq \beta} $ is a sequence of i.i.d. uniform random variables in $ [0, 1] $ and $ \Me_{k, \beta}^{(\infty)} := \bigcup_{j = 1}^k \ensemble{ \sum_{\ell = 1}^\beta ( V^{(j, \ell) } + 1 - V^{(j+ k, \ell)} ) = 0 } $.
\end{theorem}

% ==

\begin{proof}
Using \eqref{RandomEq:Charpol} and $ X := \ensemble{x_1, \dots, x_j} $, one has
\begin{align*}%$
\Mom(N\vert k, \beta) & := \Esp{ \prth{ \, \crochet{x^0} \abs{Z_{U_N}(x)}^{2\beta} }^{\! k} } \\
               & = (-1)^{Nk\beta} \, \Esp{ \det(U_N)^{-k\beta} \prth{ \, \crochet{x^0} Z_{U_N}(x)^{ \beta} x^{N\beta} Z_{U_N}(x\inv)^{ \beta} }^{\! k} } \mbox{ with \eqref{Eq:EqFuncCharpol}} \\
              & = (-1)^{Nk\beta} \, \Esp{ \det(U_N)^{-k\beta} \crochet{ X^0 } \prod_{j = 1}^k  Z_{U_N}(x_j )^{ \beta}  x_j^{ N\beta} Z_{U_N}(x_j\inv)^{ \beta}  } \\
              & = N^{2k\beta } (-1)^{Nk\beta} \, \Esp{ \det(U_N)^{-k\beta} \crochet{X^0 } \prod_{j = 1}^k \prod_{\ell = 1}^\beta x_j^{ V_N^{(j, \ell)} } x_j^{ N - V_N^{ (j + k, \ell)} }  \sc_{V_N^{(j, \ell)} }(U_N) \sc_{V_N^{(j + k, \ell)} }(U_N) } \\
              & = N^{2k\beta } (-1)^{Nk\beta} \, \Esp{ \!  \Unens{ \forall j \leq k, \, \sum_{\ell = 1}^\beta (V_N^{(j, \ell)} + N - V_N^{(j + k, \ell)} ) = 0 } \det(U_N)^{-k\beta} \prod_{j = 1}^{2k}\prod_{\ell = 1}^\beta \sc_{V_N^{(j, \ell)} }(U_N) }
\end{align*}
where $ (V_N^{(j, \ell)})_{j \geq 1, \ell\geq 1} $ is a sequence of i.i.d. uniform random variables in $ \intcrochet{0, N} $. Define the event 
\begin{align*}%$
\Me_{k, \beta}^{(N)} := \ensemble{  \forall j \in \intcrochet{1, k}, \, \sum_{\ell = 1}^\beta (V_N^{(j, \ell)} + N - V_N^{(j + k, \ell)} ) = 0 }
\end{align*}

Defining moreover the random partition $ \Vb_{\! N} \vdash Nk\beta $ of length $ 2k\beta $ by the ordering of the vector $ (V_N^{(j, \ell)})_{1 \leq j \leq 2k, 1 \leq \ell \leq \beta} $ and pursuing as in \S~\ref{Subsec:RandomisationParadigm} or \S~\ref{Subsubsec:RandomProof:KR3G}, one gets with $ \Xb := \ensemble{x_1, \dots, x_{2k\beta}} $~:
\begin{align*}%$
\Mom(N\vert k, \beta) & = N^{2k\beta} \, \Esp{ \Un_{\Me_{k, \beta}^{(N)} }     \crochet{ \Xb^{ \Vb_{\! N} } } s_{N^{k\beta}}\crochet{ \Xb } } \quad\mbox{by \eqref{Eq:SchurJointMoments}} \\
                 & = \frac{N^{2k\beta}}{ (k\beta)! } \,    \Esp{ \Un_{\Me_{k, \beta}^{(N)} } \crochet{ \Xb^{ \Vb_{\! N} } } \!\! \oint_{\Uu^{k\beta}}  U^{-N} H\crochet{U \Xb } \abs{\Delta(U)}^2 \frac{d^*U}{U} } \mbox{ by \eqref{FourierRepr:SchurRectangle}}  \\
                 & = \frac{N^{2k\beta}}{ (k\beta)!} \,    \Esp{ \Un_{\Me_{k, \beta}^{(N)} } \oint_{\Uu^{k\beta}}  U^{-N}  h_{  \Vb_{\! N} }(U) \abs{\Delta(U)}^2 \frac{d^*U}{U} } \mbox{ by \eqref{FourierRepr:hLambda}} \\
                 & = \frac{N^{2k\beta}}{ (k\beta)!} \, \Ee \Bigg( \! \Un_{\Me_{k, \beta}^{(N)} }  \Unens{ \sum_{j = 1}^{2k}\!\sum_{\ell = 1}^\beta \!\! V_N^{(j, \ell)} = N k\beta } \!    \oint_{\Uu^{k\beta - 1}} \!\!\!\! W^{-N} h_{  \Vb_{\! N} }\crochet{1 + W} \! \abs{\Delta\crochet{1 + W}}^2 \frac{d^*W}{W} \! \Bigg) 
\end{align*}
using the trick of remark \ref{Rk:MultivariateFourierWithHomogeneity} and integrating out $ u_1 $. 

As for\footnote{As remarked in \cite[(16)]{BaileyKeating}, $ \Mom(N\vert 2, \beta) $ is in fact equal to $ I_{2\beta}(\beta N, N) $, hence the similarity.} $ \Se_{N, m} $ in \S~\ref{Subsubsec:RandomProof:KR3G}, the additional event $ \{ \Vb_{\! N} \vdash k\beta N \} = \{ \sum_{j = 1}^{2k}\sum_{\ell = 1}^\beta V_N^{(j, \ell)} = k\beta N  \} $ is superfluous since it is contained in $ \Me_{k, \beta}^{(N)} $. We thus have
\begin{align*}%$ 
\Mom(N\vert k, \beta) = \frac{N^{2k\beta}}{ (k\beta)!} \, \Prob{ \Me_{k, \beta}^{(N)} }  \Esp{  \oint_{\Uu^{k\beta - 1}}  W^{-N} h_{  \Vb_{\! N} }\crochet{1 + W} \abs{\Delta\crochet{1 + W}}^2 \frac{d^*W}{W} \bigg\vert \Me_{k, \beta}^{(N)} } 
\end{align*}

Remark moreover that by independence $\Prob{ \Me_{k, \beta}^{(N)} } = \Prob{ S_{2\beta, N} = 0 }^k $ with $ S_{2\beta, N} := \sum_{\ell = 1}^{2\beta} V_N^{(1, \ell) } \eqlaw \sum_{\ell = 1}^{\beta} (V_N^{(1, \ell) } + N - V_N^{(k + 1, \ell) }) $.

Using the local CLT approach of \S~\ref{Subsec:RandomisationParadigm}, the uniform coupling \eqref{Eq:CouplingUniform}, setting $ \Vb_{\! N} = \pe{N \Vb} $ (where $ \Vb $ is the vector of ordered uniform random variables $ V^{(j, \ell)} $ from the coupling) and $  \Me_{k, \beta}^{(\infty)} := \ensemble{ \forall j \in \intcrochet{1, k}, \, \sum_{\ell = 1}^\beta (V^{(j, \ell)} + 1 - V^{(j + k, \ell)} ) = 0 } $ and finally setting $ w_j := e^{2 i \pi x_j/N} $, one gets
\begin{align*}%$ 
\Mom(N\vert k, \beta) & = \frac{N^{2k\beta}}{ (k\beta)!} \  \Prob{ S_{2\beta, N} = 0 }^k \\
                 & \qquad \Esp{   \int_{\crochet{-\frac{N}{2}, \frac{N}{2} }^{k\beta - 1}} \! e^{-\sum_{j = 2}^{k\beta} x_j } h_{ [N \Vb]  }\crochet{1 + e^{2 i \pi \xb/N}} \! \abs{\Delta\crochet{1 + e^{2 i \pi \xb/N}}}^2 \! \frac{d\xb}{N^{k\beta - 1}} \Bigg\vert \Me_{k, \beta}^{(N)} \! } \\
% ==========================================================================================
                 & \equivalent{N \to +\infty} \frac{N^{2k\beta}}{ (k\beta)!} \ \frac{ f_{S_{2\beta}}(0)^k }{N^{k }}  \times  (2\pi)^{k\beta(k\beta - 1) }   N^{ 2k\beta(k\beta - 1) -k\beta(k\beta - 1) -(k\beta - 1) } \\
                 & \hspace{+3cm}   \int_{\Rr^{k - 1}}  e^{-2 i \pi \sum_{j = 2}^{k\beta} x_j } \, \Esp{  \widetilde{h}_{ \Vb\!,\, \infty  }(0, \xb)  \Big\vert \Me_{k, \beta}^{(\infty)} } \Delta(0, \xb)^2 d\xb \\
% ==========================================================================================
                 & = N^{(k\beta)^2 + 1 - k} \frac{(2\pi)^{k\beta(k\beta - 1) }}{(k\beta)! 2^k ((2\beta)!)^k }  \int_{\Rr^{k\beta - 1}}  e^{-2 i \pi \sum_{j = 2}^{k\beta} x_j } \, \Esp{  \widetilde{h}_{ \Vb\!,\, \infty  }(0, \xb)  \Big\vert \Me_{k, \beta}^{(\infty)}  } \\
                 & \hspace{+11.5cm}\Delta(0, \xb)^2 d\xb
\end{align*}
where we have applied dominated convergence as in theorem~\ref{Theorem:Autocorrelations} with the additional integrability in the uniform random variables coming from lemma~\ref{Lemma:IntegrabilityDominationInC}, with $ \widetilde{h}_{c, \infty} $ defined in \eqref{Def:hTildeInfty} and with $ f_{S_{2\beta}}(x) $ the Lebesgue density of $ S_{2\beta} := \sum_{j = 1}^{2\beta} V^{(j)} $ computed in \eqref{Eq:BatesIrwinHall} (we have also used the fact that $ 1 - V \eqlaw V $ if $V$ is uniform in $ [0, 1] $). Last, we have used the value $ f_{S_{2\beta}}(0) = \frac{1}{2 (2\beta) !} $.
\end{proof}

% ==

\begin{remark}
Note the difference of expressions between
\begin{align*}%$ 
\Mom(N\vert k, \beta) = N^{2k\beta} \, \Ee ( \Un_{\Me_{k, \beta}^{(N)} }  \crochet{ \Xb^{ \Vb_{\! N} } } s_{N^{k\beta}}\crochet{ \Xb } ), \qquad \Xb := \ensemble{ x_1, \dots, x_{2k\beta} }
\end{align*}
and
\begin{align*}%$ 
\Mom(N\vert k, \beta) = [X^{N\beta}] s_{N^{k\beta} } \crochet{  X 1^{\plusInOne 2\beta} } , \qquad X := \ensemble{x_1, \dots, x_k}
\end{align*}
obtained at the beginning of the proof of theorem~\ref{Theorem:MoMo}~: one has transferred the complexity of the additional alphabet $ \Ae = 1^{\plusInOne 2\beta} $ in $ s_{N^\kappa}\crochet{X \Ae} $ to the Fourier coefficient, obtaining thus a ``classical'' alphabet for the Schur function, but with more variables.
\end{remark}

$ $

% ==
\section{Ultimate remarks}\label{Sec:Ultimate}

% ==
\subsection{The Hankel form of the limit}\label{Subsec:Ultimate:Hankel}

\subsubsection{Motivations}

There has recently been a certain activity \cite{BaileyBettinBlowerConreyProkhorovRubinsteinSnaith, BasorBleherBuckinghamGravaIts2Keating} to express the autocorrelations \eqref{Eq:SchurJointMoments} or the joint autocorrelations with their first order derivatives \eqref{Def:DerivativeFunctional} (case $ m = 1 $, $ \hb = h \in \Nn^* $) in terms of Painlev\'e functions. The precise connection with integrable systems always comes from a Hankel/Toeplitz\footnote{The Hankel determinant given in \cite[(3-23)]{BasorBleherBuckinghamGravaIts2Keating} is size-dependent and can be transformed into a Toeplitz determinant by a change of indices. More generally, $ T_k(f) := \det(f_{i - j})_{i, j \leq k} = \det( f_{k - (i + j) } )_{i, j \leq k}  $ which is a Hankel determinant.} determinant representation, with Laguerre polynomials \cite[(4-3), (4-4), (3-23)]{BasorBleherBuckinghamGravaIts2Keating} or Bessel functions \cite[(34)]{BaileyBettinBlowerConreyProkhorovRubinsteinSnaith}. Other apparitions of a Toeplitz/Hankel determinant include \cite{AssiotisKeating, BasorGeRubinstein}. In particular, \cite[\S~4.3, last sentence]{AssiotisKeating} asks the question of the presence of a Hankel determinant in $ \MoMf_+(3, \beta) $.

The goal of this \S~ is to show that these Toeplitz-Hankel forms follow naturally from the theory presented in \S~\ref{Sec:MainFormulas}.

% ==
\subsubsection{An example~: the Keating-Snaith theorem}

A Hankel representation of $ \widetilde{L}_1(k) $ can be derived from \eqref{EqPhi:KS}. Define
\begin{align*}%$
f_{k, y}(x) := e^{  2 i \pi (k^2 - 1) x  } x^2 \sinc(\pi k (x + y) )^{2k}
\end{align*}

Using the integral representation \eqref{Def:HKappaCInfty} of $ h^{(2k)}_{  k , \infty } $, one has
\begin{align*}%$
\widetilde{L}_1(k) & = \frac{ (2\pi )^{ k(k - 1) } k^{2 k^2} }{k!} \int_{\Rr^k }   e^{ i \pi k (2k^2 - 2) y } \sinc(\pi k y)^{2k}   \prod_{j = 2}^k \sinc(\pi k (x_j + y) )^{2k} \\
              & \hspace{+6cm} e^{ 2 i \pi (k^2 - 1) \sum_{j = 2}^k x_j } \prod_{j = 2}^k x_j^2 \, \Delta\prth{ x_2, \dots, x_k }^2  dy \, dx_2 \dots dx_k \\
              & = \frac{ (2\pi )^{ k(k - 1) } k^{2 k^2} }{k!}  \int_\Rr  e^{ 2 i \pi k (k^2 - 1) y } \sinc(\pi k y)^{2k} H_{k - 1}(f_{k, y}) dy
\end{align*}
where the Hankel determinant (with size-dependent symbol) $ H_{k - 1}(f_{k, y}) $ is defined by
\begin{align*}%$
H_{k - 1}(f_{k, y}) := \det\prth{ \int_\Rr x^{i + j - 2} f_{k, y}(x) dx }_{1 \leq i, j \leq k - 1}
\end{align*}

For all $ r \in \intcrochet{2, 2k - 2} $, we have
\begin{align*}%$
\int_\Rr x^{r - 2} f_{k, y}(x) dx & = \int_\Rr x^r e^{  2 i \pi (k^2 - 1) x  } \sinc(\pi k (x + y) )^{2k}  dx \\
               & = \frac{e^{ - 2 i \pi (k^2 - 1) y} }{ k^{r + 1} } \int_\Rr (t-ky)^r e^{  2 i \pi t(k^2 - 1)/k  } \sinc(\pi t )^{2k}  dt 
\end{align*}
giving thus
\begin{align*}%$
H_{k - 1}(f_{k, y}) & = \det\prth{ \frac{e^{ - 2 i \pi (k^2 - 1) y} }{ k^{\ell + j + 1} } \int_\Rr (t-ky)^{\ell + j} e^{ 2 i \pi t(k^2 - 1)/k  } \sinc(\pi t )^{2k}  dt  }_{1 \leq \ell, j \leq k - 1} \\
               & =: e^{  -2 i \pi (k - 1)(k^2 - 1) y} k^{-k^2} \det\prth{   \int_\Rr P_{\ell + j}(t) e^{  2 i \pi t(k^2 - 1)/k  } \sinc(\pi t )^{2k}  dt  }_{1 \leq \ell, j \leq k - 1}
\end{align*}
with $ P_j(t) := (t - ky)^j $. This last identity is in fact valid for any monic polynomial $ P_r $ with $ \deg(P_r) = r $ by (classicaly) taking linear combinations of the lines or the colums of the determinant. The determinant is thus independent of $ y $ and setting
\begin{align*}%$
g_k(x) := e^{  2 i \pi x(k^2 - 1)/k  } \sinc(\pi x )^{2k} x^2
\end{align*}
one gets 
\begin{align*}%$
\widetilde{L}_1(k) = c_k  \   H_{k - 1}(g_k)
\end{align*}
with $ c_k := \frac{ (2\pi )^{ k(k - 1) } k^{ k^2} }{k!} \int_\Rr  e^{  2 i \pi (k^2 - 1) y} \sinc(\pi k y)^{2k} dy = \frac{ (2\pi )^{ k(k - 1) } k^{ k^2} }{k \cdot k!} \int_\Rr  e^{  2 i \pi x (k^2 - 1)/k } \sinc(\pi x)^{2k} dx $. Note that $ c_k \in \Rr $ since
\begin{align*}%$
\int_\Rr e^{  2 i \pi x a  } \sinc(\pi x )^{2k}  dx & = \Esp{ \int_\Rr e^{  2 i \pi x a + 2 i \pi x W_{2k} }   dx }  ,  \quad W_{2k} := \sum_{j = 1}^{2k} V_j \eqlaw - W_{2k} \\
               & = \Esp{ \delta_0\prth{ a - W_{2k}  } }  = f_{W_{2k}}(a) 
\end{align*}
where $ (V_j)_{j\geq 1} $ is a sequence of i.i.d. random variable with $ V_j \sim \Us\prth{\, \crochet{-1, 1} } $ as in remark~\ref{Rk:SubBirkhoffWithUniforms}. Note also that although the symbol is complex, one has in fact a real determinant:
\begin{align*}%$
\int_\Rr \!  x^{m} g_k(x) dx & = \int_\Rr x^{m + 2} e^{  2 i \pi x(k^2 - 1)/k  } \sinc(\pi x )^{2k}  dx \\
               & = \prth{ \frac{1}{ 2i \pi } \frac{d}{dy} }^{m + 2} \int_\Rr e^{  2 i \pi x y } \sinc(\pi x )^{2k}  dx \bigg\vert_{ y = \frac{k^2 - 1}{k} } \\
               & = \prth{ \frac{1}{ 2i \pi } \frac{d}{dy} }^{m + 2} \Esp{ \int_\Rr e^{  2 i \pi x(k^2 - 1)/k + 2 i \pi x W_{2k} }   dx }\bigg\vert_{ y = \frac{k^2 - 1}{k} } ,  \quad W_{2k} := \sum_{j = 1}^{2k} V_j \\
               & = \prth{ \frac{1}{ 2i \pi } \frac{d}{dy} }^{m + 2} \Esp{ \delta_0\prth{ y - W_{2_k}  } }\bigg\vert_{ y = \frac{k^2 - 1}{k} } \\
               & = \prth{ \frac{1}{ 2i \pi } \frac{d}{dy} }^{m + 2} f_{W_{2k}}(y)\big\vert_{ y = \frac{k^2 - 1}{k} } =: \prth{ \frac{1}{ 2i \pi }   }^{m + 2} f_{W_{2k}}^{(m + 2)}\prth{  \tfrac{k^2 - 1}{k} }
\end{align*}

The term $ (2i \pi)^m \equiv (2i \pi)^{ \ell + j} $ can moreover be factored out of the Hankel determinant (in this form, it becomes a Wronskian, and more precisely, a Turanian) ; note that the total term is $ \prod_{\ell, j = 1}^{k - 1} (2i \pi)^{ \ell + j} = (2 i \pi)^{k(k - 1)} \in \Rr $. This gives another form of the matrix factor.

% ==
\subsubsection{A general theory}  

Formulas \eqref{FourierRepr:SchurRectangle} and \eqref{FourierRepr:SchurRectangleWithHn} can be naturally turned into a Toeplitz determinant with the well-known ``Toeplitz connection'' \cite[\textit{fact five}]{DiaconisRMTSurvey}, i.e. the Andr\'eieff-Heine-Szeg\"o/continuous Cauchy-Binet formula that writes
\begin{align*}%$
\frac{1}{N!} \oint_{\Uu^N} \abs{\Delta(u_1, \dots, u_N)}^2 \prod_{j = 1}^N f(z_j) \frac{d^* z_j}{z_j} = \det\prth{ \oint_\Uu z^{i - j} f(z) \frac{d^*z}{z} }_{1 \leq i, j \leq N} =: \det \Tt_N(f)
\end{align*}

As a result, \eqref{FourierRepr:SchurRectangle} writes\footnote{This is also a particular case of the Jacobi-Trudi formula \eqref{Eq:JacobiTrudiSchur}. This last formula can be proven using the continuous or the usual (discrete) Cauchy-Binet formula starting from \eqref{AlternantRepr:Schur}.}
\begin{align*}%$
s_{N^k}\crochet{\Ae} = \det\prth{ \oint_\Uu z^{i - j - N} H\crochet{z \Ae} \frac{d^*z}{z} }_{1 \leq i, j \leq k} =: \det \Tt_k(\psi_{N, \Ae}), \qquad \psi_{N, \Ae} (z) := z^{ - N} H\crochet{z \Ae}
\end{align*}
and \eqref{FourierRepr:SchurRectangleWithHn} writes
\begin{align*}%$
s_{N^k}\crochet{\Ae} = \oint_\Uu t^{-kN} \det\prth{ \oint_\Uu z^{i - j - N} H\crochet{z t \Ae} \frac{d^*z}{z} }_{1 \leq i, j \leq k} \frac{d^*t}{t} = \oint_\Uu t^{-kN} \det \Tt_k(\psi_{N, t \Ae}) \frac{d^*t}{t}
\end{align*}

Note that the scaling property of $ h_N $ and the trick of remark~\ref{Rk:MultivariateFourierWithHomogeneity} imply that
\begin{align*}%$
s_{N^k}\crochet{\Ae} & = \frac{1}{k!} \oint_{\Uu^{k - 1} } V^{-N} h_{Nk}\crochet{ (1 + V) \Ae } \abs{\Delta\crochet{1 + V} }^2 \frac{d^*V}{V} \\
                     & = \frac{1}{k!} \oint_{\Uu^{k - 1} } V^{-N} \crochet{t^{Nk} } H\crochet{ t(1 + V) \Ae } \abs{\Delta(V) }^2 \prod_{j = 1}^k \abs{1 - v_j}^2 \frac{d^*v_j}{v_j} \\ 
                     & =  \oint_\Uu t^{-kN} H\crochet{t\Ae} \det \Tt_{k - 1}(z \mapsto \abs{1 - z}^2 \psi_{N, t \Ae}(z)) \frac{d^*t}{t}
\end{align*}
which uses a Toeplitz determinant of size $ k - 1 $ this time.

The rescaling of these determinants with $ h_{k, \infty} $ also writes as a Hankel determinant in a natural way~: note indeed that for i.i.d. random variables $ \Zb(\Ae) := (Z_j(\Ae))_{1 \leq j \leq k} $, one has 
\begin{align*}%$
D_k & := \int_{\Rr^{k - 1} } \Esp{ e^{ -2 i \pi \sum_{j = 2}^k x_j Z_j(\Ae) } } \prod_{j = 2}^k x_j^2 \, \Delta(\xb)^2 d\xb \\
                & = \sum_{\sigma, \tau \in \Sg_{k - 1}} \varepsilon(\sigma\tau) \int_{\Rr^{k - 1} }  \Esp{ e^{ -2 i \pi \sum_{j = 2}^k x_j Z_j(\Ae) } } \prod_{j = 2}^k x_j^{\sigma(j) + \tau(j)} d\xb \\
                & = \Esp{ \sum_{\sigma, \tau \in \Sg_{k - 1}} \varepsilon(\sigma\tau) \prod_{j = 2}^k \int_\Rr  e^{ -2 i \pi  x Z_j(\Ae) }   x^{\sigma(j) + \tau(j)} dx } \\
                & = \Esp{ \sum_{\sigma, \tau \in \Sg_{k - 1}} \varepsilon(\sigma\tau) \prod_{j = 2}^k \delta_0^{(\sigma(j) + \tau(j))}  (   Z_j(\Ae) ) }, \qquad \delta_0^{(m)} := \prth{ \frac{d}{dx} }^m \delta_0 \\
                & = \sum_{\sigma, \tau \in \Sg_{k - 1}} \varepsilon(\sigma\tau) \prod_{j = 2}^k \prth{ \frac{\partial}{\partial x_j} }^{\sigma(j) + \tau(j) } f_{ \Zb(\Ae) }(\xb)\Big\vert_{\xb = 0} \\
                & =: \Delta(0, \nabla_{\!\xb})^2 f_{ \Zb(\Ae) }(\xb)\Big\vert_{\xb = 0}, \qquad \Delta(0, \nabla_{\!\xb}) := \Delta\prth{0,  \frac{\partial}{\partial x_2}, \dots, \frac{\partial}{\partial x_k}}
\end{align*}
where we have used the conventions of \cite[(2.13)]{ConreyRubinsteinSnaith} for the differential operator. Note also the alternative ending to the previous computation (setting $ \alpha := \sigma\inv \tau $ and using the fact that $ \varepsilon(\sigma\inv\tau) = \varepsilon(\sigma)\inv\varepsilon(\tau) = \varepsilon(\sigma)\varepsilon(\tau) $ since $ \varepsilon $ is a group morphism with values in $ \ensemble{\pm 1} $)
\begin{align*}%$
D_k & = \sum_{\sigma, \tau \in \Sg_{k - 1}} \varepsilon(\sigma\inv\tau) \prod_{j = 2}^k \prth{ \frac{\partial}{\partial x_j} }^{\sigma(j) + \tau(j) } f_{ \Zb(\Ae) }(\xb)\Big\vert_{\xb = 0} \\
                & = (k - 1)! \sum_{\alpha \in \Sg_{k - 1} } \varepsilon(\alpha) \prod_{j = 2}^k \prth{ \frac{\partial}{\partial x_j} }^{\alpha(j) + j } f_{ \Zb(\Ae) }(\xb)\Big\vert_{\xb = 0} \\
                & = (k - 1)! \det\prth{ f_{Z_1(\Ae)}^{(i + j)}(0) }_{1 \leq i, j \leq k - 1}
\end{align*}
that allows to write the result as a Hankel determinant/Wronskian. Alternatively, one could have used \cite[(3.9)]{ConreyRubinsteinSnaith}. The operator $ \Delta(\nabla_\xb) $ and/or its square is very natural in this context due to the previous manipulations. It seems to have appeared first in \cite[(39) p. 442]{GelfandNaimarkUnimodular} and was used a lot in the context of the SoV method (see \S~\ref{Subsubsec:Ultimate:OtherApproachesSchur:SoV}). 

In the case of a conditioning of an i.i.d sequence by a sum (which is the case for $ h_{c, \infty}^{(\kappa)} $), one gets 
\begin{align*}%$
\widetilde{D}_k & := \int_{\Rr^{k - 1} } \Esp{ e^{ -2 i \pi \sum_{j = 2}^k x_j Z_j(\Ae) } \Bigg\vert  \sum_{j = 1}^k Z_j(\Ae) = c } \prod_{j = 2}^k x_j^2 \, \Delta(\xb)^2 d\xb \\
% ================================================================= 
                & = \Esp{ \int_{\Rr^{k - 1} } e^{ -2 i \pi \sum_{j = 2}^k x_j Z_j(\Ae) } \prod_{j = 2}^k x_j^2 \, \Delta(\xb)^2 d\xb \Bigg\vert  \sum_{j = 1}^k Z_j(\Ae) = c} \\
% ================================================================= 
                & =  \Delta(0, \nabla_{\! \xb})^2 f_{ \Zb(\Ae) \vert \sum_j Z_j(\Ae) = c}(\xb)\Big\vert_{\xb = 0}  
\end{align*}
and using the Fourier representation of $ \delta_0 $ in the same way as \cite[(2.1)]{BasorGeRubinstein}, one gets with $ \varphi_c := f_{\sum_j Z_j(\Ae) }(c) $ (Lebesgue-density in $c$)
\begin{align*}%$
\widetilde{D}_k & := \Delta(0, \nabla_{\! \xb})^2 \int_\Rr \Esp{ e^{-2 i \pi t (c - \sum_{j = 1}^k Z_j(\Ae) ) } \delta_0(\Zb(\Ae) - \xb)  } \frac{dt}{\varphi_c}  \Big\vert_{\xb = 0} \\
% ================================================================= 
                & =  \int_\Rr e^{-2 i \pi c t } \Delta(0, \nabla_{\! \xb})^2 \Esp{ e^{ 2 i \pi t \sum_{j = 1}^k x_j  } \delta_0(\Zb(\Ae) - \xb)  } \frac{dt}{\varphi_c}  \Big\vert_{\xb = 0} \\
% ================================================================= 
                & = \int_\Rr e^{-2 i \pi c t } \Delta(0, \nabla_{\! \xb})^2 \widetilde{f}_{ Z_1(\Ae), t}^{\otimes k - 1}(\xb) \Big\vert_{\xb = 0}  \frac{dt}{\varphi_c}, \qquad \widetilde{f}_{ Z_1(\Ae), t}(x) :=   e^{ 2 i \pi t x } f_{ Z_1(\Ae) }(x) %\\
% =================================================================                
%                & = \frac{(k - 1)!}{\varphi_c} \int_\Rr e^{-2 i \pi c t } \det\prth{ \widetilde{f}_{ Z_1(\Ae), t}^{(\ell + j) }(0) }_{\!\! 1 \leq \ell, j \leq k - 1} dt
\end{align*}
namely, defining 
\begin{align*}%$
\boldsymbol{e}_t(x) & := e^{2 i \pi tx} \\
\widetilde{\Hh}_k(f) & := \prth{ \prth{\tfrac{d}{dx}}^{\! i + j} f(x)\big\vert_{x = 0} }_{\! 1 \leq i, j \leq k} 
\end{align*}
one gets
\begin{align}\label{Eq:HankelGeneral}%$
\widetilde{D}_k  = \frac{(k - 1)!}{\varphi_c} \int_\Rr e^{-2 i \pi c t } \det\prth{ \widetilde{\Hh}_{k - 1}(\boldsymbol{e}_t f_{ Z_1(\Ae)}) } dt
\end{align}
which is similar to \cite[(2.4)]{BasorGeRubinstein} in the form (nevertheless, the determinants do not have the same size, and the derivatives are not $ f^{(i + j - 2)} $ but $ f^{(i + j)} $ in this last formula, due to the term $ \Delta(0, x)^2 $). Note that the second line of computation uses the formula $ f(x) \delta_c(dx) = f(c) \delta_c(dx) $ \cite[Ex. 2 (V, 3 ; 2) p. 121]{SchwartzDistributions} (with $ x_1 = 0 $). Note also that $ \gamma_k(c) $ defined in \cite[(2.5)]{BasorGeRubinstein} uses $ c' := c - k/2 $ and $ Z_1(\Ae) \eqlaw \Us([-1/2, 1/2]) $ whereas \cite[(2.4)]{BasorGeRubinstein} uses $c$ and $ Z_1(\Ae) \eqlaw \Us([0, 1]) $~; this corresponds to $ \kappa = 1 $ for $ \Ae = 1^{\kappa} $. The link with piecewise continuous polynomials is explained in remark~\ref{Rk:SubBirkhoffWithUniforms}.

% ==
\subsubsection{Application~: the ``moments of moments''}

We leave the task to the interested reader to write all the previous functionals in terms of Hankel/Toeplitz determinants to focus now on the last sentence of \cite{AssiotisKeating} where Assiotis and Keating ask the question of finding a Hankel determinant in $ \MoMf_+(3, \beta) $. 

To find such a determinant for $ k \geq 2 $, we can use equivalently \eqref{EqPhi:MoMo} or \eqref{EqPhi:Random:MoMo}. The involved random variables will not be the same~: on the one hand, \eqref{EqPhi:MoMo} gives with \eqref{Eq:Beta_ProbRepr} 
\begin{align*}%$
\Phi_{\MoMf_+(k, \beta)\!}(0, \xb) & :=  e^{ 2 i \pi A \sum_{j = 2}^{k\beta }  x_j } \,\,  h_{\beta, \infty}^{(2\beta)}( 0, \xb )^k, \qquad  A := k\beta^2 - 1  \\
          & =   \frac{\beta^{2k^2\beta - k} }{ ((2k\beta - 1)!)^k  } \, e^{ 2 i \pi A \sum_{j = 2}^{k\beta }  x_j } \\
          &  \hspace{+3cm} \times \Esp{ e^{ 2 i \pi \beta \sum_{j = 2}^{k\beta} x_j \sum_{r = 1}^k ( \bbbeta_{2\beta, 1}^{(j, r)} - \beta)  } \Bigg\vert \sum_{j = 1}^{k\beta} \bbbeta_{2\beta, 1}^{(j, r)}  = 1, \, \forall r \in \intcrochet{1, k} } \\
          & =   \frac{\beta^{2k^2\beta - k} }{ ((2k\beta - 1)!)^k  } \, e^{ -2 i \pi \sum_{j = 2}^{k\beta }  x_j } \\
          &  \hspace{+3cm} \times \Esp{ e^{ 2 i \pi \beta \sum_{j = 2}^{k\beta} x_j \sum_{r = 1}^k \bbbeta_{2\beta, 1}^{(j, r)}    } \Bigg\vert \sum_{j = 1}^{k\beta} \bbbeta_{2\beta, 1}^{(j, r)}  = 1, \, \forall r \in \intcrochet{1, k} }
\end{align*}
where $ \xb := \ensemble{x_2, \dots, x_{k\beta}} $ and $ (\bbbeta_{2\beta, 1}^{(j, r)})_{1 \leq j \leq k\beta, 1 \leq r \leq k} $ is a sequence of i.i.d. Beta-distributed random variables (note the typographical difference between $ \beta $ and $ \bbbeta $).

On the other hand, \eqref{EqPhi:Random:MoMo} gives with \eqref{Def:hTildeInfty}
\begin{align*}%$
\Phi^{(Rand)}_{\MoMf_+(k, \beta)\!}(0, \xb) & := \frac{1}{ (2(2\beta)!)^k }   \,  e^{-2 i \pi \sum_{j = 2}^{k\beta} x_j } \, \Esp{  \widetilde{h}_{ \Vb\!,\, \infty  }(0, \xb)  \Big\vert \Me_{k, \beta}^{(\infty)}  }  \\
               & =  \frac{1}{ (2(2\beta)!)^k }   \,  e^{-2 i \pi \sum_{j = 2}^{k\beta} x_j } \,\\
               & \hspace{+1cm} \times \Esp{ \prod_{j = 1}^{2k} \prod_{\ell = 1}^\beta \frac{(V^{(j, \ell)})^{k - 1} }{(k -1)! } \times e^{ 2 i \pi  \sum_{r = 2}^{k\beta} x_r \sum_{j = 1}^{2k} \sum_{\ell = 1}^\beta V^{(j, \ell) }  ( \bbbeta_{1, 1}^{(r, j, \ell)} - \beta ) }  \Bigg\vert \Me_{k, \beta}^{(\infty)}  }
\end{align*}

A power bias of a uniform random variable is a Beta random variable, since for all bounded $f$
\begin{align*}%$
\frac{\Esp{U^{k - 1} f(U) }}{\Esp{U^{k - 1} }} & = k \Esp{U^{k - 1} f(U) } = k \int_0^1 f(u) u^{k - 1} du = \int_0^1 f(v^{1/k}) dv = \Esp{ f( U^{1/k} ) } \\
              & =: \Esp{ f(\bbbeta_{k, 1}) }
\end{align*}

Since $ \bbbeta_{1, 1} \eqlaw U \sim \Us([0, 1]) $, one gets by biasing (also inside the conditioning event) 
\begin{align*}%$
\Phi^{(Rand)}_{\MoMf_+(k, \beta)\!}(0, \xb) &  =  \frac{ e^{-2 i \pi \sum_{j = 2}^{k\beta} x_j }  }{ (2(2\beta)!)^k (k!)^{2k\beta} }   \,   \, \Esp{  e^{ 2 i \pi  \sum_{j = 2}^{k\beta} x_j \sum_{m = 1}^{2k} \sum_{\ell = 1}^\beta \bbbeta_{k, 1}^{(m, \ell) }  ( U^{(j, m, \ell)} - \beta ) }  \Big\vert \Me_{k, \beta}^{(k, \infty)}  }
\end{align*}
and the conditioning event is then
\begin{align*}%$
\Me_{k, \beta}^{(k, \infty)} := \ensemble{ \forall j \in \intcrochet{1, k}, \, \sum_{\ell = 1}^\beta (\bbbeta_{k, 1}^{(j, \ell)} + 1 - \bbbeta_{k, 1}^{(j + k, \ell)} ) = 0 }
\end{align*}

In both cases, one has the joint Fourier transform of particular random variables
\begin{align*}%$
Z_j(\Ae) \in \ensemble{ \beta \sum_{r = 1}^k  \bbbeta_{2\beta, 1}^{(j, r)} , \ \ -1 + \sum_{m = 1}^{2k} \sum_{\ell = 1}^\beta \bbbeta_{k, 1}^{(m, \ell) }  ( U^{(j, m, \ell)} - \beta ) }, \qquad  \forall j \in \intcrochet{ 1, k\beta }
\end{align*}
with a conditioning event of the ``linear type'' ($k$ sums equal to $ c \in \ensemble{0, 1} $) but possibly with different random variables $ Y_j(\Ae) $. It enters into the previous framework with a slight modification (due to the discrepancy between $ Z_j(\Ae) $ and $ Y_j(\Ae) $). 

The case of $ \Phi_{\MoMf_+(k, \beta) } $ gives with $ Z_j := -1 + \beta \sum_{r = 1}^k  \bbbeta_{2\beta, 1}^{(j, r)} $ and $ Y_r := \beta \,\sum_{j = 1}^{k\beta} \bbbeta_{2\beta, 1}^{(j, r)} $ 
\begin{align*}%$
\MoMf_+(k, \beta) & = c_k \int_{\Rr^{k - 1} }  \Esp{ e^{ 2 i \pi \sum_{ j = 2 }^{k\beta} x_j Z_j }  \Bigg\vert Y_r = \beta, \  \forall r \in \intcrochet{1, k} } \Delta(0, \xb)^2 d\xb \\
              & = c'_k \int_{\Rr^{k - 1}\times \Rr^k } \Esp{  e^{ 2 i \pi \sum_{r = 1}^k t_r (Y_r - \beta ) }   \times   e^{ 2 i \pi \sum_{ j = 2 }^{k\beta} x_j Z_j }     }  \Delta(0, \xb)^2 d\xb\, d\tb \\
% =========================================
              & = c'_k \int_{\Rr^{k - 1} \times \Rr^k } e^{ -2 i \pi \beta \sum_{r = 1}^k \! t_r } \Esp{  e^{ 2 i \pi \sum_{r = 1}^k \beta t_r   \sum_{j = 1}^{k\beta}  \bbbeta_{2\beta, 1}^{(j, r)}    }   \times   e^{ 2 i \pi \sum_{ j = 2 }^{k\beta} x_j (-1 + \beta \sum_{r = 1}^k  \bbbeta_{2\beta, 1}^{(j, r)} )  }     } \\
              & \hspace{+7cm} \Delta(0, \xb)^2 d\xb\, d\tb\\
% =========================================
              & = c'_k \int_{\Rr^{k - 1}\times \Rr^k } e^{ -2 i \pi (\beta \sum_{r = 1}^k \! t_r + \sum_{j = 1}^{k\beta}  x_j )} \Esp{  e^{ 2 i \pi \beta \sum_{j = 1}^{k\beta} \sum_{r = 1}^k ( t_r + x_j)   \bbbeta_{2\beta, 1}^{(j, r)} } }\!\!\Big\vert_{x_1 = 0}      \Delta(0, \xb)^2 d\xb\, d\tb\\
% =========================================
              & =: c'_k \int_{\Rr^{k - 1}\times \Rr^k } e^{ -2 i \pi (\beta \sum_{r = 1}^k \! t_r + \sum_{j = 1}^{k\beta}  x_j )}  \prod_{j, r} f_{2\beta, 1}(x_j + t_r) \Big\vert_{x_1 = 0}      \Delta(0, \xb)^2 d\xb\, d\tb %\\
% =========================================
%              & = c'_k \int_{ \Rr^k }  \det(\widetilde{\Hh}_{k - 1}(\psi_\tb)) \prod_{r = 1}^k e^{ -2 i \pi \beta   t_r } f_{2\beta, 1}( t_r) \, dt_r
\end{align*}
namely, applying the previous \S, 
\begin{align}\label{Eq:HankelMoMo}%$
\MoMf_+(k, \beta) & =   c'_k \int_{ \Rr^k }  \det(\widetilde{\Hh}_{k - 1}(\psi_\tb)) \prod_{r = 1}^k e^{ -2 i \pi \beta   t_r } f_{2\beta, 1}( t_r) \, dt_r
\end{align}
where $ c'_k $ is a constant that comprises the Lebesgue density $ f_{Y_1}(\beta)^k $, where 
\begin{align*}%$
\psi_\tb(x) := \boldsymbol{e}_{\beta}(x) \prod_{r = 1}^k f_{2\beta, 1}(x + t_r) 
\end{align*}
and $ f_{2\beta, 1}(x) := \Esp{ \delta_0( \beta \bbbeta_{2\beta, 1} - x ) } = \beta\inv f_{\bbbeta_{2\beta, 1}}(\beta\inv x) =  2 (2\beta)^{1 - 2\beta} x^{2\beta - 1} \Unens{0 \leq x \leq 1} $.

We leave to the interested reader the task of expressing the second probabilistic representation in a similar way.

\newpage
% ==
\subsection{Expansions}\label{Subsec:Ultimate:Expansions}

\subsubsection{Motivations}

Dehaye's binomial method \cite{PODfpsac} is a transfert of technology from the world of (asymptotic) representation theory \cite{OkounkovOlshanski1, OkounkovOlshanski2, OkounkovOlshanskiJack, OkounkovOlshanskiBC} to the study of $ Z_{U_N} $. The original binomial expansion can be found in \cite[I-3, ex. 10 p. 47]{MacDo} and was originally due to Lascoux \cite{LascouxChern}. It is an application of the Cauchy-Binet formula starting from the binomial expansion in \eqref{AlternantRepr:Schur}~:
\begin{align*}%$
s_\lambda(1 + x_1, \dots, 1 + x_N) & = \frac{\det\prth{ (1 + x_i)^{\lambda_j + N - j} }_{\!  1 \leq i, j \leq N} }{\Delta(1 + x_1, \dots, 1 + x_N) } =  \frac{\det\prth{ (1 + x_i)^{\lambda_j + N - j} }_{\!  1 \leq i, j \leq N} }{\Delta( x_1, \dots, x_N) } \\
                 & = \frac{\det\prth{ \sum_{\ell = 0}^{\lambda_j + N - j} {\lambda_j + N - j \choose \ell} x_i^\ell }_{\!\! 1 \leq i, j \leq N} }{\Delta(  x_1, \dots,  x_N) } \qquad \mbox{(Newton's binomial formula)} \\
                 & = \frac{1}{\Delta(  x_1, \dots,  x_N) }\sum_{0 \leq \ell_N < \cdots < \ell_1 \leq N + \lambda_1 }\!\!\! \det\prth{   {\lambda_j + N - j \choose \ell_r}  }_{\!\! 1 \leq r, j \leq N} \!\!\! \det\prth{  x_i^{\ell_m} }_{\!\! 1 \leq i, m \leq N}
\end{align*}

Here, we have used the Cauchy-Binet formula that writes for matrices $ A \sim N \times \infty $ and $ B \sim \infty \times N $
\begin{align*}%$
\det(AB)_{N\times N} := \det\prth{ \sum_{\ell \geq 0 } A_{i, \ell} B_{\ell, j} }_{\!\!\! 1 \leq i, j \leq N} = \sum_{0 \leq \ell_1 < \cdots < \ell_N  } \det(A_{i, \ell_j})_{i, j} \det(B_{\ell_i, j})_{i, j} 
\end{align*}
and the fact that $ \Delta(1 + T) = \Delta(T) $ using its product form.

Classically writing $ \ell_j := \mu_j + N - j $ so that $ \mu_1 \geq \mu_2 \geq \cdots \geq \mu_N \geq 0 $ yields
\begin{align*}%$
s_\lambda(1 + x_1, \dots, 1 + x_N) & = \sum_{0 \leq \mu_N < \cdots < \mu_1 \leq N + \lambda_1} \det\prth{   {\lambda_j + N - j \choose \mu_r + N - r}  }_{\!\! 1 \leq r, j \leq N} \frac{ \det\prth{  x_i^{\mu_m + N - m} }_{\!\! 1 \leq i, m \leq N} }{\Delta(  x_1, \dots,  x_N) } \\
                   & =: \sum_{\mu \subset (N + \lambda_1)^N } \det\prth{   {\lambda_j + N - j \choose \mu_r + N - r}  }_{\!\! 1 \leq r, j \leq N} s_\mu(  x_1, \dots,  x_N)  
\end{align*}

This proof can easily be modified to any symmetric functions that writes under the form
\begin{align}\label{Def:GeneralFunctionBinomialExpansion}%$
S_{f_1, \dots, f_N}(x_1, \dots, x_N) := \frac{ \det(f_i(x_j))_{1 \leq i, j \leq N} }{ \Delta(x_1, \dots, x_N) }, \qquad   f_i(x) := \sum_{k \geq 0} \widehat{f}_i(k) x^k
\end{align}

As remarked in \cite[(1.11)]{OkounkovOlshanski2} the Cauchy-Binet formula gives then
\begin{align*}%$
\frac{ \det(f_i(x_j))_{1 \leq i, j \leq N} }{ \Delta(x_1, \dots, x_N) } = \sum_{ \ell(\mu) \leq N} \det\prth{  \widehat{ f}_i(\mu_j + N - j)  }_{\!\! 1 \leq i, j \leq N} s_\mu(x_1, \dots,  x_N)  
\end{align*}
the previous case using $ f_i(x) := (1 + x)^{\lambda_i + N - i} $.

As noted in \cite[Comments to (5.7), (3)]{OkounkovOlshanski1}, the novelty in the generalised binomial expansion for Schur functions consists in writing the coefficients $ d_{\lambda, \mu}(N) := \det\prth{   {\lambda_j + N - j \choose \mu_r + N - r}  }_{\!\! 1 \leq r, j \leq N} $ as \textit{shifted Schur functions}. This is the formula used in \cite[(23)]{PODfpsac}. An elegant manipulation performed by Dehaye \cite[Lem. 3]{PODfpsac} allows to express these values in terms of \textit{content polynomials} \cite[I-1 ex. 11]{MacDo}. Dehaye's final result \cite[(10)]{PODfpsac} writes
\begin{align}\label{Eq:POD:derivativeExpansion}%$
\frac{ \Esp{ \abs{Z_{U_N}(1) }^{2k} \prth{ i \frac{Z_{U_N}'(1) }{Z_{U_N}(1) } }^{\! \!r} \, } }{ \Esp{ \abs{Z_{U_N}(1) }^{2k} }  } = \frac{1}{r!} \sum_{\lambda \vdash r} d_\lambda^2 \prod_{\square \in \lambda } \frac{ (k + c(\square) )(-N + c(\square) ) }{2k + c(\square) }
\end{align}

Here, we have used the expression $ d_\lambda := \frac{r!}{H_\lambda} $ where $ H_\lambda $ is the hook length product \cite[(18)]{PODfpsac} and $ d_\lambda $ is the dimension of the irreducible module $ V^\lambda $ of the symmetric group $ \Sg_r $ (denoted by $ f_\lambda $ in \cite[(19)]{PODfpsac}). 

\medskip

The expression \eqref{Eq:POD:derivativeExpansion} is particularly interesting to obtain an asymptotic expansion in powers of $ N $, as one can write
\begin{align}\label{Eq:ExpansionDerivative:Plancherel}%$
\begin{aligned}
\frac{ \Esp{ \abs{Z_{U_N}(1) }^{2k} \prth{ i \frac{Z_{U_N}'(1) }{Z_{U_N}(1) } }^{\! \!r} \, } }{ N^r \, \Esp{ \abs{Z_{U_N}(1) }^{2k} }  } & = \frac{(-1)^r }{r!} \sum_{\lambda \vdash r} d_\lambda^2 \prod_{\square \in \lambda } \frac{ k + c(\square)  }{2k + c(\square) }  \times  \prod_{\square \in \lambda }(1 - N\inv c(\square) ) \\
                 & = \sum_{\ell \geq 0} N^{-\ell} \, \frac{(-1)^r }{r!} \sum_{\lambda \vdash r} d_\lambda^2 \, h_\ell\crochet{ - \Ae^\lambda} \prod_{\square \in \lambda } \frac{ k + c(\square)  }{2k + c(\square) }   
\end{aligned}
\end{align}
where $ \Ae^\lambda $ is the \textit{content alphabet}
\begin{align}\label{Def:ContentAlphabet}%$
\Ae^\lambda := \ensemble{ c(\square), \square \in \lambda} 
\end{align}

Expanding $ \Esp{ \abs{Z_{U_N}(1) }^{2k} } $ using its explicit expression as a product of Gamma factor and multiplying the expansions gives then a complete expansion. 

\medskip

Other types of expansions are possible (although the final expressions must then be equal). For instance, \cite[(4-3), (4-4) \& conj. 1]{BasorBleherBuckinghamGravaIts2Keating} uses a conjectural conformal block expansion coming from the Painlev\'e theory and Riemann-Hilbert problems.

% ==
\medskip
\begin{remark}\label{Rk:ComparisonPODvsB3GI2K}
The rescaling of \eqref{Eq:POD:derivativeExpansion} using the convergence \eqref{Eq:KSmomentsCharpol} and $ M_k $ defined in \eqref{Def:MatrixFactor} gives
\begin{align}\label{Eq:POD:limitDerivatives}%$
\frac{ \Esp{ \abs{Z_{U_N}(1) }^{2k} \prth{ i \frac{Z_{U_N}'(1) }{Z_{U_N}(1) } }^{\! \!r} \, } }{  N^{ k^2 + r }  } \tendvers{ N}{+\infty }  M_k \times   \frac{(-1)^r}{r!} \sum_{\lambda \vdash r} d_\lambda^2 \prod_{\square \in \lambda } \frac{ k + c(\square)   }{ 2k + c(\square) }
\end{align}

In \cite[Conj. 2, (6-32)]{BasorBleherBuckinghamGravaIts2Keating}, the following expression is conjectured to be equivalent to \eqref{Eq:POD:limitDerivatives}~:
\begin{align}\label{Eq:LimitB3GI2K}%$
\hspace{-0.4cm} \frac{ \Esp{ \abs{Z_{U_N}(1) }^{2k} \prth{ i \frac{Z_{U_N}'(1) }{Z_{U_N}(1) } }^{\! \! 2h} \, } }{  N^{ k^2 + 2h }  } \tendvers{ N}{+\infty }  M_k \times   \frac{(-1)^h}{(2h)!} \sum_{\lambda \vdash 2h} d_\lambda^2 \prod_{\square \in \lambda } \frac{ k + c(\square)   }{ (2k + b(\square) )^2 } (2k + c(\square) )
\end{align}
where $ b(\square) := \lambda'_j - j + 1 - i $ if $ \square = (i, j) \in \lambda $. Note that the rational functions in $k$ do not match term-wise since the content alphabet $ \Ae^\lambda $ defined in \eqref{Def:ContentAlphabet} does not have the same range of values as $ \Be^\lambda := \ensemble{ b(\square), \square \in \lambda } $, hence, it is not immediately obvious that they coincide exactly (nevertheless, in the particular case of integer $k$, cancellations due to numerator simplifications may occur).
\end{remark}

% ==
\begin{remark}\label{Rk:DerivativesWithJucysMurphy}
Using the orthogonality of the characters $ (\chi^\lambda)_{\lambda \vdash r} $ of $ \Sg_r $, one gets
\begin{align*}%$
\crochet{\mathrm{id}_r} \sigma = \Unens{\sigma = \mathrm{id}_r} = \frac{1}{r!} \sum_{\lambda \vdash r} \chi^\lambda(\mathrm{id}_r) \chi^\lambda(\sigma) = \frac{1}{r!} \sum_{\lambda \vdash r} d_\lambda^2 \, \widehat{\chi}^\lambda(\sigma)  = \Ee_{\Pe\ell(r)}(\widehat{\chi}^{\lambdab}(\sigma))
\end{align*}
where $ d_\lambda = \chi^\lambda(\mathrm{id}_r) $, $ \widehat{\chi}^\lambda := \chi^\lambda/d_\lambda $ and $ \Pp_{\Pe\ell(r)}(\lambda) := \frac{d_\lambda^2}{r!} $ is the Plancherel measure on $ \Yy_r $ (see e.g. \cite{JohanssonHouches}). Using moreover the \textit{Jucys-Murphy elements} $ \Je := (J_k)_{k \geq 1} $ defined by $ J_1 := 0 $, $ J_k := \sum_{\ell = 1}^{k - 1} (\ell, k) \in \Cc\crochet{\Sg_r} $ and the fact that $ \widehat{\chi}^\lambda(f(\Je)) = f(\Ae^\lambda) $ for all symmetric function $f$ (see e.g. \cite{BianeCharactersCumulants}), one gets with $ f := g^{\otimes r} $ and $ g(x) := \frac{ (k + x) (-N + x) }{2k + x} $
\begin{align}\label{Eq:ExpansionDerivative:JM}%$
\begin{aligned}
\frac{ \Esp{ \abs{Z_{U_N}(1)}^{2k}  \prth{ i \frac{Z_{U_N}'(1)}{Z_{U_N}(1) } }^r } }{ \Esp{ \abs{Z_{U_N}(1)}^{2k} } } & = \crochet{\mathrm{id}_r} \prod_{\ell \in \intcrochet{1, r} } \frac{ (k + J_\ell)(-N + J_\ell) }{2k + J_\ell} \\
               & = (-N)^r \sum_{m \geq 0} (-N)^{-m} \crochet{\mathrm{id}_r} e_m(J_1, \dots, J_r) \prod_{\ell \in \intcrochet{1, r} } \frac{ k + J_\ell }{2k + J_\ell}
\end{aligned}
\end{align}

This equality, equivalent to \eqref{Eq:ExpansionDerivative:Plancherel}, can lead to other interesting combinatorial expressions, using e.g. the recursive method of \cite{FerayJM} or \cite[\S~4 \&~9]{BianeCharactersCumulants} (combinatorics of Jucys-Murphy elements).
\end{remark}

\medskip
% ==
\subsubsection{Expansions with \eqref{Eq:SchurWithHn}}

One can perform an expansion with \eqref{Eq:SchurWithHn} if, instead of just considering the first order of the limit, one continues the expansion of each quantity \textit{ad libitum}. For instance, one has
\begin{align*}%$
s_{N^k}\crochet{\Ae} & = \frac{1}{k!} \oint_{\Uu^{k - 1}} V^{-N} h_{Nk}\crochet{(1 + V) \Ae} \abs{\Delta\crochet{1 + V} }^2 \frac{d^*V}{V} \\
               & = \frac{1}{k!} \int_{ \crochet{-\frac{N}{2}, \frac{N}{2}}^{k - 1} }  e^{-2 i \pi \sum_{j = 2}^k x_j } h_{Nk}\crochet{(1 + e^{2 i \pi X/N}) \Ae} \abs{\Delta\crochet{1 + e^{2 i \pi X/N}} }^2 \frac{dX}{N^{k - 1} } \\
               & = \frac{\Esp{ S_k\crochet{\Ae}^{kN} } }{k! \, (kN)! \, N^{k - 1} } \int_{ \crochet{-\frac{N}{2}, \frac{N}{2}}^{k - 1} }    e^{-2 i \pi \sum_{j = 2}^k x_j }  \Esp{ \prth{ 1 + \frac{ \xi_N(X, \Ae)}{N}  }^{\!\! kN} } \abs{\Delta\crochet{1 + e^{2 i \pi X/N}} }^2 dX  
\end{align*}
where, by analogy with \eqref{Def:XiNandXiInfty}, one defines 
\begin{align*}%$
h_m\crochet{X\Ae} & = \frac{1}{m!} \Esp{ \prth{ \sum_{j = 1}^k x_j Z_j\crochet{\Ae} }^{\!\! m} }, \qquad \Esp{ e^{x Z\,\crochet{\Ae} } } := H\crochet{x \Ae} \\
\xi_N(X, \Ae) & := \sum_{j = 1}^k N (e^{2 i \pi x_j/N} - 1) \frac{Z_j\crochet{\Ae}}{S_k\crochet{\Ae}}, \qquad S_k\crochet{\Ae} := \sum_{j = 1}^k Z_j\crochet{\Ae} \\
\xi_\infty(X, \Ae) & := 2 i \pi \sum_{j = 1}^k   x_j \frac{Z_j\crochet{\Ae}}{S_k\crochet{\Ae}}
\end{align*}
with $ (Z_j\crochet{\Ae})_j $ independent of $ S_k\crochet{\Ae} $.

We now reason formally and explain how to justify the steps at the end. One can push the expansion of $ \log(1 + \xi_N(X, \Ae)/N ) = \sum_{\ell = 1}^L (-1)^{\ell + 1} ( \xi_N(X, \Ae)/N)^\ell / \ell + O(N^{-L - 1} ) $ with a uniform deterministic $ O(N^{-L - 1} )  $ inside the exponential to get
\begin{align*}%$
h_m\crochet{(1 + e^{2 i \pi X/N} ) \Ae} & = \Esp{ e^{ m \log( 1 + \xi_N(X, \Ae)/N ) } } = \Esp{ e^{ m \sum_{\ell = 1}^L (-1)^{\ell + 1} ( \xi_N(X, \Ae)/N)^\ell / \ell + O(m N^{-L - 1} ) } } \\
               & = e^{ O(m N^{-L - 1} ) } \sum_{ \ell(\lambda) \leq L } \frac{(-1)^{\abs{\lambda} - \ell(\lambda) } }{z_\lambda } \frac{m^{\ell(\lambda) }}{ N^{\abs{\lambda} } } \Esp{ \xi_N(X, \Ae)^{\abs{\lambda} } } \\
               & \approx \sum_{ \ell(\lambda) \leq L } \frac{(-1)^{\abs{\lambda} - \ell(\lambda) } }{z_\lambda } \frac{m^{\ell(\lambda) }}{ N^{\abs{\lambda} } } \Esp{ \xi_\infty (X, \Ae)^{\abs{\lambda} } } \prth{ 1 + O(m N^{-L - 1} ) }
\end{align*}

Note that for $ L = 1 $, one recovers \eqref{Eq:Beta_ProbRepr} with an expansion (in the case where $ \Ae = 1^\kappa $). One supposes of course that $ m = \pe{\rho N} $, which gives a general expansion in $N$. In this particular case, one can get the following factorisation~:
\begin{align*}%$
h_{ \pe{\rho N} }\crochet{(1 + e^{2 i \pi X/N} ) \Ae} & =  \Esp{ e^{  \, \rho \sum_{\ell = 0}^L (-1)^\ell ( \xi_N(X, \Ae) )^{\ell + 1} N^{-\ell } / (\ell + 1) + O( N^{ -L - 1} ) } } \\
               & = e^{ O( N^{-L - 1} ) } \sum_{ \ell(\lambda) \leq L } \frac{(-1)^{\abs{\lambda} - \ell(\lambda) } }{z_\lambda } \frac{ 1 }{ N^{ \abs{\lambda} - \ell(\lambda)} } \Esp{ e^{\xi_N(X, \Ae) } \xi_N(X, \Ae)^{\abs{\lambda} } } \\
               & \approx \sum_{ \ell(\lambda) \leq L } \frac{(-1)^{\abs{\lambda} - \ell(\lambda) } }{z_\lambda } \frac{ 1 }{ N^{\abs{\lambda} - \ell(\lambda) } } \Esp{ e^{\xi_\infty (X, \Ae)} \xi_\infty (X, \Ae)^{\abs{\lambda} } } \prth{ 1 + O(  N^{-L } ) } \\
               & =: \sum_{L \geq s \geq \ell \geq 0} a_{s, \ell} \frac{ (-1)^{s - \ell} }{ N^{ s - \ell } } \Esp{ e^{\xi_\infty (X, \Ae)} \xi_\infty (X, \Ae)^s }   + O(  N^{-L } )  
\end{align*}
with explicit combinatorial coefficients $ a_{s, \ell} := \sum_{\ell(\lambda) = \ell, \abs{\lambda} = s} \frac{ 1 }{z_\lambda } $. 

One also needs to expand the Vandermonde part. Remark that
\begin{align*}%$
\frac{e^{ix} - 1}{ix} & = \Esp{ e^{i x U} }, \qquad U \sim \Us([0, 1]) \\
\abs{ \frac{e^{ix} - 1}{ix} }^2 & = \Esp{ e^{i x W} }, \qquad W :\eqlaw U - U', \quad U, U' \sim \Us([0, 1]) \mbox{, independent.}
\end{align*}

This implies with independent %$ (U_j)_{2 \leq j \leq k} $ and $ (U_{j, \ell})_{2 \leq j < \ell \leq k} $ uniform on $ [0, 1] $ and independent 
$ (W_j)_{2 \leq j \leq k} $ and $ (W_{j, \ell})_{2 \leq j < \ell \leq k} $ distributed as $W \eqlaw U - U'$  
\begin{align*}%$
\abs{ \frac{\Delta\crochet{1 + e^{2 i \pi X/N}}}{ \Delta(0, 2 i \pi X/N) } }^2 & = \prod_{j =  2}^k \abs{ \frac{ e^{2 i \pi x_j/N } - 1 }{2 i \pi x_j/N } }^2 \prod_{2 \leq \ell < j \leq k } \abs{ \frac{e^{2 i \pi (x_j - x_\ell)/N } - 1 }{ 2 i \pi (x_j - x_\ell)/N }}^2 \\
%                	& = \prod_{j =  2}^k \abs{ \Esp{ e^{2 i \pi x_j U_j/N} } }^2 \prod_{2 \leq \ell < j \leq k } \abs{ \Esp{ e^{2 i \pi x_{j, \ell} U_{j, \ell} /N } } }^2, \qquad x_{j, \ell} := x_j - x_\ell \\
                	& = \prod_{j =  2}^k  \Esp{ e^{2 i \pi x_j W_j/N} } \prod_{2 \leq \ell < j \leq k } \Esp{ e^{2 i \pi (x_j - x_\ell) W_{j, \ell} /N } }  \\
                	& = \Esp{ e^{2 i \pi \sum_{j = 2}^k x_j \We_j / N } }, \qquad \We_j :\eqlaw W_j + \sum_{\ell \in \intcrochet{2, k}\setminus\ensemble{j} } (W_{j, \ell} - W_{\ell, j})
\end{align*}

Define
\begin{align*}%$
\chi(X) := \sum_{j = 2}^k x_j \We_j
\end{align*}

We then have 
\begin{align*}%$
\abs{ \frac{\Delta\crochet{1 + e^{2 i \pi X/N}}}{ \Delta(0, 2 i \pi X/N) } }^2  & =   \sum_{r = 0}^M \frac{(2 i \pi)^r}{r! } N^{-r} \Esp{ \chi(X)^r  } + O\prth{ N^{-M - 1} \sum_{j = 2}^k \abs{x_j}^{M + 1} } 
%                	& =:   1 + \sum_{r = 1}^M P_r(X) N^{-r}  + O_X(N^{-M - 1})  
\end{align*}
%
%
%
%where $ P_r $ are thus explicit polynomials and $ O_X (N^{-M - 1}) $ has a dependency in $X$ of the type $ \sum_j \abs{x_j}^{M + 1} $, i.e. is equal to $ O ( N^{-M - 1} \sum_j \abs{x_j}^{M + 1}) $.

Multiplying both contributions, using the scaling of the Vandermonde determinant and integrating out gives an expansion of the form~:
\begin{align*}%$
s_{N^k}\crochet{\Ae} & = \frac{\Esp{ S_k\crochet{\Ae}^{kN} } }{k! \, (kN)! \, N^{k - 1} } \int_{ \Rr^{k - 1} }    e^{-2 i \pi \sum_{j = 2}^k x_j }  \Esp{ \prth{ 1 + \frac{ \xi_N(X, \Ae)}{N}  }^{\!\! kN} } \abs{\Delta\crochet{1 + e^{2 i \pi X/N}} }^2 dX  \\
                 & \approx \frac{\Esp{ S_k\crochet{\Ae}^{kN} } }{k! \, (kN)! \, } N^{-k^2 - 2k - 1} \int_{ \Rr^{k - 1} }    e^{-2 i \pi \sum_{j = 2}^k x_j } \Esp{ e^{\xi_\infty (X, \Ae) } } \Delta(0, 2\pi X)^2 \\
                 & \hspace{+4cm} \times \Bigg( \sum_{ r, s, \ell = 0}^M  a_{s, \ell} (-1)^{s - \ell} \frac{\Esp{ e^{\xi_\infty (X, \Ae)} \xi_\infty (X, \Ae)^s \chi(X)^r  } }{\Esp{ e^{\xi_\infty (X, \Ae)}  }  } \frac{1}{N^{r + s - \ell } } \\
                 & \hspace{+8cm}+ O\prth{ \frac{\sum_j \abs{x_j}^{M + 1} }{N^{ M + L + 1 } } } \Bigg) dX
\end{align*}

Note that if $ \xi_\infty(X, \Ae) $ was real, one would have an interpretation of the probabilitic part of the coefficients as a change of probability. This expansion is thus half-probabilistic, half-combinatorial (due to the coefficients $ a_{s, \ell} $). 

To turn these formal computations into a rigorous expansion, one needs an integrable domination akin to Lemmas~\ref{Lemma:DominationSupersym} and~\ref{Lemma:IntegrabilityDomination}. Since the $ \Esp{ \chi(X)^r } $ are polynomials, this only amounts to dominate $ \Esp{ e^{\xi_\infty (X, \Ae)} \xi_\infty (X, \Ae)^s } = \prth{\frac{d}{dt}}^s \Esp{ e^{t \xi_\infty (X, \Ae)} } $ against a polynomial function of the type $ \Esp{\chi(X)^r} \Delta(X)^2 $ whose degree is explicit. This can be done by an exchange of integration and $ \frac{d}{dt} $ in a slight modification of~\eqref{Eq:GegenbauerMultivariateWithTKL}. Details are left to the reader. % Wesh bonbon, amusez-vous bien sur celle-là.

\medskip
% ==
\subsubsection{Expansion of \eqref{Def:JointMomentsDerivative_0_1}/\eqref{Eq:POD:derivativeExpansion}}

In view of \eqref{Eq:ExpansionDerivative:Plancherel} and \eqref{Eq:ExpansionDerivative:JM}, one can apply the previous (formal) methodology to \eqref{Def:JointMomentsDerivative_0_1} to try to find a similar expansion. 

We will use the methodology of remark~\ref{Rk:Alternative:DerivativeOther} with $ K_{1, N}(u, z) := \frac{d}{dx}\prth{\frac{1 - x^{N + 1} }{1 - x} }\big\vert_{x = u z\inv} = C_{N, 1} \Esp{ x^{D_{N, 1} } }\big\vert_{x = uz\inv} $ where $ D_{N, r}, C_{N, r} $ are given in \eqref{Def:DnR}. We get
\begin{align*}%$
\De_N(k, h) & := \Esp{ \abs{Z_{U_N}(1)}^{2(k - h)} \abs{Z'_{U_N}(1)}^{2h} } \\
%               & = \Esp{ \det(U_N)^{ (k - h) } Z_{U_N}(1)^{ 2(k - h) } \prod_{j = 1}^h \frac{\partial}{\partial x_j}   Z_{U_N}(x_j)\bigg\vert_{x_j = 1} \frac{\partial}{\partial y_j}   \overline{Z_{U_N}(y_j)}\bigg\vert_{y_j = 1} } \\
               & = \Esp{   \abs{Z_{U_N}(1)}^{2(k - h)} \abs{ \oint_\Uu K_{N, 1}(u, z) Z_{U_N}(u) \frac{d^*u}{u} }^{2h} } \\
               & = \oint_{ \Uu^{2h} } \prod_{j = 1}^h K_{1, N}(u_j, 1) \overline{K_{1, N}(v_j, 1)} \\
               & \hspace{+4cm} \times \Esp{ Z_{U_N}(1)^{ k - h } \overline{Z_{U_N}(1) }^{ k - h }  \prod_{j = 1}^h Z_{U_N}(u_j) \overline{Z_{U_N}(v_j) } } \frac{d^*\ub}{\ub} \frac{d^*\vb}{\vb}  \\
               & = \oint_{ \Uu^{2h} } \prod_{j = 1}^h K_{1, N}(u_j, 1) \overline{K_{1, N}(v_j, 1)} \prod_{j = 1}^h v_j^{-N} \, s_{N^k}\crochet{ 1^{2(k - h)} + U + V\inv } \frac{d^*\ub}{\ub} \frac{d^*\vb}{\vb} \\
               & = \int_{ \crochet{-\frac{N}{2}, \frac{N}{2} }^{2h} } e^{-2 i \pi \sum_{j = 1}^h \alpha_j} \prod_{j = 1}^h K_{1, N}(e^{2 i \pi \theta_j/N}, 1) \overline{K_{1, N}(e^{2 i \pi \alpha_j/N}, 1)} \\
               & \hspace{+6cm} \times s_{N^k}\crochet{ 1^{2(k - h)} + e^{2 i \pi \theta/N} + e^{-2 i \pi \alpha/N} } \frac{d \thetab}{N^h} \frac{d \alphab}{N^h} \\
               & =  \prth{ \frac{C_{N, 1}}{N} \! }^{\!\! 2h} \!\! \int_{ \crochet{-\frac{N}{2}, \frac{N}{2} }^{2h} } e^{-2 i \pi \sum_{j = 1}^h \alpha_j} \Esp{ e^{ 2 i \pi \sum_{j = 1}^h (\theta_j D_{N, 1}^{(j)}/N - \alpha_j \widetilde{D}_{N, 1}^{(j)}/N) } } \\
               & \hspace{+6cm} \times s_{N^k}\crochet{ 1^{2(k - h)} + e^{2 i \pi \theta/N} + e^{-2 i \pi \alpha/N} }  d \thetab \, d \alphab  
\end{align*}

The expansion of $ \Esp{ e^{i \theta D_{N, 1}/N } } $ can be done in a probabilistic way using an Edgeworth expansion, or with a Riemann sum expansion which amounts to use the Euler-MacLaurin formula (equivalent to the Edgeworth expansion in this setting). This last formula writes for $ f \in \Ce^{K + 1}([a, b]) $ with $ a, b \in \Nn $ (see e.g. \cite[I.0, thm. 4]{Tenenbaum})
\begin{align*}%$
\sum_{a + 1 \leq \ell \leq b} f(\ell) & = \int_a^b f(t) dt + \sum_{r = 0}^K  \frac{(-1)^{r + 1} B_{r + 1} }{(r + 1)!} \prth{ f^{(r)}(b) - f^{(r)}(a) }  \\
               & \hspace{+6cm} + \frac{(-1)^K}{(K + 1) !} \int_a^b B_{K + 1}(t) f^{(K + 1)}(t) dt
\end{align*}
where $ B_r $ is the $r$-th Bernoulli number and $ B_K(t) $ is the Bernoulli polynomial.

Taking $ a = 0 $, $ b = N $ and $ f(x) := \frac{1}{N} g(x/N) $, one gets 
\begin{align*}%$
\frac{1}{N} \sum_{ 1 \leq \ell \leq N} g\prth{ \frac{\ell}{N} } & = \int_0^1 g(u) du + \sum_{r = 0}^K  \frac{(-1)^{r + 1} B_{r + 1} }{(r + 1)!} \frac{1}{N^{r + 1} } \prth{ g^{(r)}(1) - g^{(r)}(0) } \\
                   & \hspace{+5cm} + \frac{(-1)^K}{(K + 1) !} \frac{1}{N^{K + 2} } \int_0^1 B_{K + 1}(\pf{N u}) g^{(K + 1)}(u) du
\end{align*}
where $ \pf{x} := x - \pe{x} $ is the fractional part of $x$, the Bernoulli polynomials in this formula being the periodised version (see e.g. \cite[I.0]{Tenenbaum}).

Remarking that $ C_{N, 1} = \sum_{j = 1}^N j = \frac{N(N + 1)}{2} $, hence that $ \frac{C_{N, 1}}{N} = \frac{N + 1}{2} $, one gets
\begin{align*}%$
\frac{ C_{N, 1} }{N} \, \Esp{ e^{i \theta D_{N, 1}/N } } =  N \times \frac{1}{N} \sum_{\ell = 1}^N \frac{\ell}{N} e^{i \theta \ell/N} =  N \times \frac{1}{N}\sum_{\ell = 1}^N g_\theta(\ell/N), \qquad g_\theta(x) := x e^{2i\pi \theta x}
\end{align*}

An easy induction shows that $ \prth{ \frac{d}{dx} }^r (x e^{\lambda x}) = (\lambda x + r) \lambda^{r - 1} e^{\lambda x} $, hence,  
\begin{align*}%$
\tfrac{ C_{N, 1} }{ N} \Esp{ e^{2i\pi \theta D_{N, 1}/N } } & = \int_0^1 x e^{i \theta x} dx + \sum_{r = 0}^K  \frac{(-1)^{r + 1} B_{r + 1} }{(r + 1)!} \frac{(2 i \pi \theta)^{r - 1} }{N^{r + 1} } \prth{ (2 i \pi \theta + r) e^{2 i \pi \theta} - r } \\
                   & \hspace{+2.5cm} + \frac{(-1)^K}{(K + 1) !} \frac{ (2 i \pi \theta)^K }{N^{K + 2} } \int_0^1 B_{K + 1}(\pf{N u}) (2 i \pi \theta u + K + 1) e^{i \theta u}  du
\end{align*}

One has moreover the Lehmer estimate $ \sup_{x \in [0, 1]}\abs{B_K(x)} \leq 2 \frac{K!}{(2\pi)^K}(1 + \zeta(K)\Unens{ K \equiv 2 \! \mod 4 }) $ which shows that the integral in the RHS of the last equality is bounded by $ C'_K \abs{\theta} $ for a certain constant $ C'_K $. In the end, one has
\begin{align*}%$
\tfrac{ 2 C_{N, 1} }{ N } \Esp{ e^{2i\pi \theta D_{N, 1}/N } } & = N \, \Esp{ e^{2 i \pi \theta D_{\infty, 1} } } + 2 \sum_{r = 0}^K  \frac{(-1)^{r + 1} B_{r + 1} }{(r + 1)!} \frac{(2 i \pi \theta)^{r - 1} }{N^r} \prth{ (2 i \pi \theta + r) e^{2 i \pi \theta} - r } \\
           & \hspace{+10cm} + O\prth{  \frac{ \abs{ \theta}^{K + 1} }{N^{K + 1} } } 
\end{align*}

One is then left to expand $  s_{N^k}\crochet{ 1^{2(k - h)} + e^{2 i \pi \theta/N} + e^{-2 i \pi \alpha/N} } $. This can be done in a way similar to the previous \S, as the representation with random variables $ \xi_N(X, \Ae) $ is also valid in the case of a ``plethystically additive'' specialisation of the type $ s_{N^k}\crochet{\Ae + X} $. In the end, one has two types of integrals whose exchange is to justify~: the ones coming from $ (\alphab, \thetab) $ in the previous formula, and the ones coming from $X$ in the formula of the previous \S. Due to the presence of Bernoulli numbers in the expansion of the Riemann sum, the coefficients of this expansion are very combinatorial in nature. We leave to the interested reader the task of writing exactly these coefficients. % Putain, la loose. Vraiment trop relou à écrire.

\medskip
% ==
\subsection{Other approaches to $ s_\lambda\crochet{\Ae} $}\label{Subsec:Ultimate:OtherApproachesSchur}

As noticed in remark~\ref{Rk:OtherApproachesSchur}, one can also look at the literature on the rescaling of Schur functions outside of the $ CUE $ literature. Dehaye's use of the binomial method \cite{PODfpsac} or the use of conformal block expansion by Basor \textit{et al.} \cite{BasorBleherBuckinghamGravaIts2Keating} are examples of a successful transfer of technology in this setting. 

One can separate this literature into two sub-categories~: on the one hand, the study of spherical integrals of the Harish-Chandra-Itzykson-Zuber (HCIZ) type, and on the other hand problems in relation with asymptotic representation theory. The frontier between these two topics is not clear, and several methods lie at the intersection of the two fields.

% ==
\subsubsection{Spherical integrals}

The HCIZ integral is defined by \cite{HarishChandraInt, ItzyksonZuber}
\begin{align*}%$
HCIZ_N(A, B) := \int_{\Ue_N} e^{\trace(U A U\inv B) } dU, \qquad A, B \in \He_N
\end{align*}

The link with the Schur function is provided with the following explicit formula \cite[(3.4)]{ItzyksonZuber} where one has supposed that $ A $ and $ B $ are diagonal (what we can always suppose by invariance)~: 
\begin{align*}%$
HCIZ_N(A, B) = \prth{ \prod_{j = 1}^{N - 1} j! } \, \frac{\det\prth{ e^{a_i b_j } }_{1 \leq i, j \leq N } }{\Delta(A) \Delta(B) }
\end{align*}

The way to derive this formula is in itself interesting and could eventually be modified to lead to an asymptotic analysis. To the knowledge of the author, five methods are available~: Dyson Brownian motion \cite{ItzyksonZuber, JohanssonHCIZ}, character expansion \cite{ItzyksonZuber, ZinnJustinZuber}, Duistermaat-Heckman stationary phase method \cite{DuistermaatHeckman}, coordinate Bethe ansatz \cite{HarishChandraInt} and path decomposition of integrable processes \cite{KatoriTanemura} (see also \cite{McSwiggenHCIZ, TaoBlogHCIZ} for an exposition). 

The (Weyl dimension) formula $ \frac{\Delta(\lambda + \delta_N ) }{ \prod_{j = 1}^{N - 1}  j! } = \frac{\Delta(\lambda + \delta_N ) }{ \Delta( \delta_N )  } = s_\lambda\crochet{1^N} $ (see \cite[I-1, ex. 1, (4) p. 11 \& I-3 ex. 4 p. 45]{MacDo} or \cite[(4.6)]{KuznetsovSklyaninFactorisation}) yields the more explicit link with the Schur function
\begin{align}\label{Eq:SchurWithHCIZ}%$
HCIZ_N(\lambda + \delta_N , X) = \frac{ \Delta(e^X) }{\Delta(X) } \frac{s_\lambda\crochet{e^X}  }{ s_\lambda\crochet{1^N}}
\end{align}

These last years saw several approaches to tackle the asymptotics of $ HCIZ_N $ in several regimes of $ \lambda \equiv \lambda(N) $ and $ X \equiv X(N) $~:
\begin{itemize}

\item The method used in \cite{BunBouchaudMajumdarPotters, GuionnetMaidaSchur, GuionnetZeitouniSchur, Matytsin, ZinnJustinZuber} lies in the framework of large deviations or equivalently, variational problems on the space of discrete measures that converge to limiting variational problems on a ``continuous'' space. 

\item Note in particular the integrable systems reformulation of \cite{Matytsin} (turned rigorous in \cite{GuionnetZeitouniSchur}) given in \cite{MenonHCIZ} starting from the (quantum) Calogero-Moser system (see also \cite[\S~4.1]{McSwiggenHCIZnew}).

\item The authors of \cite{CollinsHCIZ, CollinsGuionnetMaurelSegala} use free probability (with Weingarten calculus) and free stochastic calculus in the vein of \cite{BianeFreeBM}~; see the survey \cite{GuionnetSurveyPGDinRMT} for a summary and additional references.

\item Novak \cite{NovakHCIZandBGW, NovakHCIZsampler} uses a string expansion governed by the Plancherel measure on partitions defined in remark~\ref{Rk:DerivativesWithJucysMurphy}. 

\item Topological recursion \cite{BorotEynardOrantin, EynardOrantin2009} or Dyson-Schwinger (loop) equations is another method that emerged in a general context these last years and that was applied to the study of $ HCIZ_N $ in \cite{GuionnetNovak}. As noted by Guionnet\footnote{For instance in a conference in Columbia held from August 28 to September 1, 2017, \href{https://www.math.columbia.edu/department/probability/seminar/guionnet.html}{``Dyson-Schwinger equations, topological expansions, and random matrices''}.}, ``the idea is similar to Stein's method in that the observables are approximate solutions of equations that can be solved asymptotically''.

\item \&c.

\end{itemize}

The use of the local CLT approach seems to be new in this context. This is mainly due to the fact that these previous limits are all performed at another scale, in the large deviations regime~: supposing that the empirical measures $ \mu_Z^{(N)} := \frac{1}{N} \sum_{k = 1}^N \delta_{z_k} $ for $ Z \in \ensemble{A, B} $ converges $ \star $-weakly to a limiting measure $ \mu_Z^{(\infty)} = \rho_Z^{(\infty)}\bullet \mathrm{Leb} $ (with some additional hypotheses), the limit is of the form
\begin{align*}%$
HCIZ_N(\sqrt{N} A, \sqrt{N} B) = \exp\prth{ N^2 \Fe(\rho_A^{(\infty)}, \rho_B^{(\infty)}) + O(N) } 
\end{align*}
where $ \Fe $ uses the solution of a Burgers equation.

From the expository point of view, recasting all the $ CUE $ results of the introduction in terms of probability theory has the advantage to position them clearly in this global picture (in terms of probabilistic regime).

\begin{remark}
One can thus write $ s_{N^k}\crochet{X} $ using the HCIZ integral for a ``classical'' alphabet. This yields the formula
\begin{align*}%$
s_{N^k}\crochet{e^X} = s_{N^k}\crochet{1^N} \frac{\Delta(e^X)}{\Delta(X)} \int_{\Ue_N} e^{\trace(U (N^k + \delta_N\!)\, U\inv X) } dU = \int_{\Ue_N} \det(U)^{-k} \prod_{x \in e^X} \det(I_N - x U) dU
\end{align*}
where $ N^k + \delta_N $ is understood in the vectorial sense (and then as a diagonal matrix). A direct proof of such an identity is still missing, but the existence of such a link allows some of the previous techniques such as topological expansion to be used (in the different regime of local CLT, though). We do not pursue here.
\end{remark}

% ==
\subsubsection{Asymptotic representation theory}

Asymptotic representation theory is the study of classical groups when their rank tends to infinity. To quote Vershik \cite{VershikTAR}~:
\begin{quote}
Investigation of classical groups of high ranks leads to two kinds of problems. Questions of the first kind deal with asymptotical properties of groups, their representations, characters and other attributes as group rank grows to infinity. Another kind of questions (in the spirit of infinite dimensional analysis) deal with properties of infinite dimensional analogues of classical groups.
\end{quote}

The study of characters of such groups is performed with a proper normalisation that is of interest in the scope of this article as the characters of $ \Ue_N $ (and also $ \Sg_N $ up to the Frobenius map defined in \cite[I-7 (7.2)]{MacDo}) are the Schur functions. In this last case, one is typically interested in the behaviour of 
\begin{align*}%$
S_{\lambda}\crochet{X^{(k)} + 1^{N - k} } := \frac{ s_{\lambda}\crochet{X^{(k)} + 1^{N - k} }}{ s_{\lambda}\crochet{  1^N } }, \qquad X^{(k)} := \ensemble{x_1, \dots, x_k}
\end{align*}

Note that this rescaled Schur function is exactly the RHS of \eqref{Eq:SchurWithHCIZ}, up to an explicit multiplicative factor. Its study is thus equivalent to the study of the HCIZ integral when one of the matrices has a finite rank $k$.

This topic being very huge, we only list a few methods that were developed in the last 80 years (starting in the 40ies with work of H. Weyl and J. Von Neumann)~: 
\begin{itemize}

\item The case where $ X^{(k)} $ is fixed and $ \abs{\lambda(N)} \sim \pe{ \gamma N } $ is studied in \cite{OkounkovOlshanskiJack, OkounkovOlshanskiBC} in the more general setting of Jack symmetric functions, a particular case of which degenerates to Schur functions and was studied in \cite{KerovVershikTARunitary}. One needs to suppose that $ \lambda_j(N) \sim \pe{N \alpha_j} $ and $ \lambda'_j(N) \sim \pe{N \beta_j} $ with finite $ \alpha_j, \beta_j $ to be in the setting of \textit{Vershik-Kerov}\footnote{The alphabetical order is the one from the Russian alphabet.} sequences. Note that this is the case of $ N^k $ (with $ \beta_j = 0 $) but the completion of $ s_\lambda $ by $ 1^{N - k} $ transforms radically the problem. Supposing that one only has ``positive partitions'', their result writes $ S_{\lambda}\crochet{X^{(k)} + 1^{N - k} }  \sim \prod_{j = 1}^k \phi_{\alpha, \beta, \gamma}(z_j - 1) $ with $ \phi_{\alpha, \beta, \gamma}(x) := H\crochet{ x ( \gamma \Eee + \Ae + \widehat{\omega} \Be) } $ where $ \Ae := \ensemble{ \alpha_j, j \geq 1} $, $ \Be := \ensemble{ \beta_j, j \geq 1} $ and $ \Eee $ is defined in example~\ref{Ex:HfunctorSpecialise}. The method is the binomial expansion. Note also the elementary method of Boyer \cite{Boyer83, Boyer92} to solve this problem (using nevertheless the Voiculescu factorisation theorem to reduce the problem to $ k = 1 $).

\medskip
\item The case of $ s_{\lambda}\crochet{X^{(k)} + \frac{1 - q^{N - k} }{ 1 - q } } / s_{\lambda}\crochet{  \frac{1 - q^{N  } }{ 1 - q } } $ (where $ \frac{1 - t}{1 - q} := \frac{1}{1 - q} - t \frac{1}{1 - q} $ and $ \frac{1}{1 - q} $ is defined in example~\ref{Ex:HfunctorSpecialise}) is treated in \cite{GorinQ} in relation with a quantisation of $ \Ue(\infty) $. The result is a variant of the previous case.

\medskip
\item The case where $ X^{(k)} = e^{Y^{(k)}/\sqrt{N}} $ with $ \abs{\lambda} = O(N^2) $ and $ Y^{(k)} := \ensemble{y_1, \dots, y_k} $ is a fixed alphabet is treated in \cite[thm. 5.1]{GorinPanova} with the additional hypothesis that the step function $ x \mapsto \lambda_{\pe{Nx}}/N = N\inv \sum_{\ell = 1}^N \lambda_\ell \Unens{\ell/N \leq x < (\ell + 1)/N} $ is close to a continuous decreasing function $f$ with error $ o(\sqrt{N}) $. The result is of the type $ S_{\lambda(N)}\crochet{e^{Y/\sqrt{N} } + 1^{N - k} } = \exp\prth{ \sqrt{N} \Eee(f) \sum_j y_j + \frac{1}{2} \Ce(f, f) \sum_j y_j^2 + o(1)} $ which is similar to the HCIZ limit obtained in \cite{GuionnetMaidaSchur} in the finite rank case. It is characteristic of a Gaussian CLT writing $ S_{\lambda(N)}\crochet{e^{Y/\sqrt{N} } + 1^{N - k} } = \Esp{ e^{\sum_{j = 1}^k y_j Z_{j, N} } } $ for given random variables $ Z_{j, N} $~; using the formula $ s_\lambda\crochet{X} = \sum_{T \in SST(\lambda)} X^T $ where $ X^T := \prod_{j \geq 1} x_j^{c_j(T)} $ and $ c_j(T) $ is the number of times $ j $ appears in the semi-standard tableau $T$ of shape $ \lambda $, $ s_\lambda\crochet{e^X + 1^{N - k} }/s_\lambda\crochet{1^N} $ becomes the joint Laplace transform of $ ( c_j(\Tb_{\lambda(N) } ) )_{1 \leq j \leq k} $ for a random tableau $ \Tb_{\lambda(N) } \in SST(\lambda(N)) $. Using one of the numerous bijections between tableaux and other combinatorial objects, such random variables are equivalent to observables of random uniform lozenge tilings or the six-vertex model. The employed method is a variant of the SoV method that we introduce in the next \S.

\medskip
\item  \&c.

\end{itemize}

Here again, the regime is not the one of a local CLT but the one of a classical CLT, with independence at the limit as the limiting function splits (this approximate character factorisation is a general phenomenon, see e.g. \cite{BianeApproximateFactorisation, VoiculescuApproximateFactorisation}).

\medskip
% ==
\subsubsection{The SoV method}\label{Subsubsec:Ultimate:OtherApproachesSchur:SoV}

The method used in \cite{GorinPanova} is a variation of the \textit{Separation of Variables} (SoV in short). This method was developed in \cite{KuznetsovSklyaninFactorisation} using the $Q$-operator formalism of \cite{KuznetsovMangazeevSklyanin} and the earlier work \cite{KuznetsovSklyaninBacklund}. We refer to the introduction of \cite{deGierPonsaing} for further historical information. 

One motivation\footnote{This is more a consequence of the result. The true motivation was to mimic in the $q$-deformed setting of quantum Hamiltonians the classical separation of variables of classical Hamiltonian dynamics such as Liouville integrable systems. The commutation between separation and quantisation does not occur in the case of the Schur and monomial symmetric functions, but happens e.g. for a family of symmetric functions related with the elementary symmetric functions~; see \cite{KuznetsovSklyaninFactorisation}.} of the problem was to compute the asymptotics of the normalised characters $ \widehat{\chi}^{\lambda, N, k}( X^{(k)} ) := \chi^\lambda\crochet{X^{(k)} + 1^{N - k} } / \chi^\lambda\crochet{ 1^N } $ such as described in the previous \S~for various types of characters $ \chi^\lambda $~; the case of $ s_\lambda $ introduced in this paper concerns the unitary group. Since character factorisation is a general phenomenon when $ N \to +\infty $ for such renormalised characters, it is tempting to highlight an approximate factorisation for $ N $ fixed, and see how factorisation occurs starting from it. We refer to \cite[\S~2]{deGierPonsaing} or \cite{KuznetsovSklyaninFactorisation} for a precise description of the $Q$-operator method (it can be seen as a clever change of basis for the Hamiltonians that act diagonally on the considered family of symmetric functions). The resulting formula is of the type
\begin{align*}%$
\widehat{\chi}^{\lambda, N, k}  = \Se_{k, N}\inv q_{\lambda, N}^{\otimes k}, \qquad q_{\lambda, N} := \widehat{\chi}^{\lambda, N, 1}
\end{align*}
where $ \Se_{k, N} $ is the key ``separating'' operator to construct, that writes as the product $ \Se_{k, N} = \Ae_{1, N} \cdots \Ae_{k, N} $ (see \cite[(2.10), (2.11)]{KuznetsovSklyaninFactorisation} for examples of $ \Ae_{k, n} $). Since the product is finite, one sees that the approximate factorisation will be proven if e.g. one proves that $ \Ae_{j, N} \to \mathrm{Id} $ when $ N \to +\infty $ in some operator topology and if $ q_{\lambda, N} \to q_{\lambda, \infty} $ in some functional space. 

The operator $ \Se_{k, N}\inv $ can be found without the previous factorisation, simply starting from the alternant form of the Schur function \eqref{AlternantRepr:Schur} and using the manipulations of \cite[\S~5-VI p. 446]{GelfandNaimarkUnimodular} that end up writing it as 
\begin{align*}%$
s_\lambda\crochet{x_1, \dots, x_N} & = \Mg_{\!\Delta}\inv \Delta(\Tb) \, \Mg_{\!\Delta} f_\lambda (x_1, \dots, x_N), \qquad \Mg_{g} f := g f, \quad T_j := x_j \frac{\partial}{\partial x_j} = \frac{ \partial}{\partial \log(x_j)} \\
                & = \Mg_{\!\Delta}\inv \Delta(\Tb)   g_\lambda^{\otimes N}(x_1, \dots, x_N)
\end{align*}

The key manipulation is the $N$-linearity $ \Delta(\Tb) g_1\otimes \cdots \otimes g_N (\xb) = \det_{N\times N} \prth{ T_i^{N - j} } g_1\otimes \cdots \otimes g_N (\xb) = \det_{N\times N} \prth{ T^{N - j}g_i(x_i) } $ such as highlighted in \cite[proof of Prop. 6.1]{deGierPonsaing} in the symplectic case. As a result, it can be generalised to any symmetric function of the form $ \Mg_{\!\Delta}\inv \Delta(\Tb) g_1 \otimes \cdots \otimes g_N  $ such as noticed in \cite[\S~3.1]{GorinPanova}. Note that these are exactly the functions defined in \eqref{Def:GeneralFunctionBinomialExpansion}. The binomial expansion and the SoV method are thus equivalent for the purpose of aymptotic representation theory.

The other key manipulation is then to explicit the function $ g_\lambda $ and write it in a clever way to perform an asymptotic analysis. As noticed in the previously cited papers, after the proper rescaling to set $ N - k $ variables to $1$, one ends up with $ g_\lambda(x) = S_\lambda\crochet{x + 1^{N - 1} } $. Determinantal manipulations often use the minor expansion used in remark~\ref{Rk:BirkhoffWithH}, itself equivalent to the partial fraction expansion of $ t \mapsto H\crochet{tX} $ or $ t \mapsto H\crochet{t(X - Y) } $ given in \eqref{Eq:PartialFracExpH}, but to perform an asymptotic analysis, it is better suited to use an integral form coming from a residue sum~; in our setting, this means going from \eqref{Eq:SupersymHnWithResidues} to \eqref{Def:SupersymHomogeneousComplete}. Such a manipulation is for instance done in \cite[(3), (4)]{HoughJiang} starting from the residue sum of \cite[Prop. 3.1]{Rosenthal} or in \cite[Prop. 3.3]{GorinPanova}. Note also that writing a sum of residues as an integral is precisely what is performed in \cite{ConreyFarmerKeatingRubinsteinSnaith} to obtain the CFKRS formula (and is constantly done in this whole theory)~; it amounts to an RKHS manipulation such as noticed in \S~\ref{Subsec:Theory:CFKRSwithRKHS} (the residue formula being itself an instance of its RKHS theory described in remark~\ref{Rk:Derivative:RKHS}).

\medskip 

The SoV method could be an interesting alternative to the theory presented in this paper. Philosophically speaking, the Fourier form of the Schur function \eqref{FourierRepr:Schur} is equivalent to the Jacobi-Trudi formula \eqref{Eq:JacobiTrudiSchur} whereas the SoV form is equivalent to the ratio of alternants \eqref{AlternantRepr:Schur}. As a result, it is restricted to classical alphabets $ \ensemble{x_1, \dots, x_k} $, but for some alphabets such as $ X 1^{\kappa} $, the method could be adapted with a careful limiting procedure (for supersymmetric alphabets nevertheless, the right equivalent of \eqref{AlternantRepr:Schur} is given by the Moens-Van Der Jeugt formula \cite[(1.17)]{MoensVanDerJeugt} and a proper analysis will be more involved). Another direction of investigation concerns the factorisation operators $ \Ae_{k, N} $ given in \cite[(4.29)]{KuznetsovSklyaninFactorisation}~; they use a disintegration on a Gelfand-Tsetlin ensemble akin to the polytopial manipulations of \cite{AssiotisKeating, KeatingR3G} and it is likely that their rescaling in the local CLT regime will give a polytope-like limit. 

Nevertheless, this is more on the other direction of investigation that an interesting connection is to be found, i.e. by applying the formalism of this article and in particular \eqref{Eq:SchurWithRKHS} to the problems previously exposed. The article \cite{GorinPanova} is concerned with the asymptotics of $ s_{\lambda_N}[X^{(k)} + 1^{N - k}] $ when $ X^{(k)} $ is an alphabet $ \ensemble{X_{1, N}, \dots, X_{k, N}} $ with fixed length $k$ and $ \lambda_N $ has length $N$ and grows in a certain way \cite[thm. 5.1]{GorinPanova}. The SoV representation of the Schur function is then rescaled with the steepest descent method. In a dual way, the case we consider is that of $ \lambda_N = N^k $ and of an ``abstract'' alphabet $ X $ (for instance a supersymmetric alphabet, see \S~\ref{Subsec:FuncSym}) with a local CLT type of rescaling (originally an alternative to the steepest descent analysis). For the particular alphabets considered in \cite{GorinPanova}, the two problems can be seen equivalent with an application of the following $ (\lambda$-$\mu) $-\textit{duality formula} obtained from \eqref{Eq:SchurWithHCIZ} with $ A := \sqrt{\log(q)} (\lambda + \delta_N) $ and $ B := \sqrt{\log(q)} (\mu + \delta_N) $ (this is suggested in \cite[Rk after Thm. 3.6]{GorinPanova})
\begin{align*}%$
\frac{s_\lambda(q^{\mu + \delta})}{s_\lambda(q^\delta) } =  \frac{s_\mu(q^{\lambda + \delta}) }{ s_\mu(q^\delta) }, \qquad \delta \equiv \delta_N := (N - 1, N - 2, \dots, 1, 0)
\end{align*}

The regime of $ \lambda_N $ and the alphabets considered in \cite{GorinPanova} are nevertheless different, but some cases such as \cite[prop. 5.8]{GorinPanova} are concerned with convergence in distribution towards a multivariate Gaussian (case $ X = e^{Y/\sqrt{N}} $ previously exposed). In this last regime, and after using the $ (\lambda$-$\mu) $-duality formula for the Schur function, the conjunction of \eqref{Eq:SchurWithHn}, lemmas \ref{Lemma:GammaProbRepr} and \ref{Lemma:BetaProbRepr} and the local CLT approach can be investigated. Note that \cite[Rk after Thm. 3.6]{GorinPanova} advocates to use \eqref{Eq:TchebytchevMultivariate} but without any probabilistic representation. We do not pursue here.

% ==
\subsection{Summary of the encountered functionals and limits}\label{Subsec:FinalTable}
We sumarise here the different functionals that we obtained throughout the article. We set $ a_k := \frac{(2\pi)^{k(k - 1)}}{k!} $ for typographical reasons. Note that the limit $ a_k \int_{ \Rr^{k } }  \Phi_{\!\Ze\Te_\rho^{(k)}}( \xb) \Delta(0, \xb)^2 d\xb $ is not a typo~: one has an integral over $ \Rr^k $ with $ \Delta(0, \xb) $ (but $ \xb = (x_2, \dots, x_{k + 1}) $ in this case).

\begin{center} 
\hspace{-1cm}
  \begin{tabular}{ | c |  c | c | c |}
    \hline
    Functional $\raisebox{-1ex}{$ \white{.} $} $ $\raisebox{+2ex}{$ \white{.} $} $ & Order & Limit & $\! \Phi \! $ /  $\! \Phi^{(Alt)} \!$/$ \Phi^{(Rand)} \!$   \\ \hline
% =====================================================================================
    $ \Esp{ \abs{ Z_{U_N}(1) }^{2k} }     \vphantom{\bigg(} $ & $ N^{k^2} $ & $ \!\!a_k\! \int_{ \Rr^{k - 1} } \! \Phi_{\widetilde{L}_1}\!(0, \xb) \Delta(0, \xb)^2 d\xb $       & -/\eqref{EqPhi:KS}/-   \\ \hline  
% =====================================================================================
  $ \Esp{ \abs{\sc_{\pe{\rho N}}(U_N) }^{2k} } \vphantom{\bigg(} $ & $ N^{(k - 1)^2} $   & $\!\!a_k \!\int_{ \Rr^{k - 1} } \! \Phi_{\!\!\Se\Ce_\rho^{(k)}} \!\!\; (0, \xb) \Delta(0, \xb)^2 d\xb $ & \eqref{EqPhi:MidCoeff}/\eqref{EqPhi:Alternative:MidSecularCoeff}/- \\ \hline
% =====================================================================================  
  $ \!\raisebox{-1ex}{$\Ee\Bigg($} \raisebox{-1ex}{$\Bigg\vert$} \hspace{-0.7cm} \underset{ \substack{  \hspace{+0.7cm} 1 \leq j_1, \dots, j_k \leq N \\ \hspace{+0.7cm} j_1 + \cdots + j_k = \pe{cN} }}{ \sum } \hspace{-1cm}  \sc_{j_1}\!(U_N\!) . . .  \sc_{j_k}\!(U_N\!) \vphantom{\bigg(} \raisebox{-1ex}{$\Bigg\vert $} \raisebox{1ex}{$^2 $} \raisebox{2.3ex}{$  $} \raisebox{-1ex}{$ \Bigg)$} \raisebox{4ex}{$ \white{.} $} \raisebox{-3.5ex}{$ \white{.} $} \!\!\! $ & $ N^{ k^2 - 1 } $   & $\!\!a_k\!\int_{ \Rr^{k - 1} } \! \Phi_{\Je_c}\!\!\;(0, \xb) \Delta(0, \xb)^2 d\xb $ & \!\eqref{EqPhi:KR3G}/\eqref{EqPhi:Alternative:KR3G}/\eqref{EqPhi:Random:KR3G} \!\!\!    \\ \hline
% =====================================================================================
  $ \hspace{-0.65cm} \Esp{ \! \underset{ \, j = 1}{ \overset{k}{\prod} }\! Z_{U_N}\!\prth{ \!e^{2 i \pi \frac{x_j}{N}} \! }   \!  \!\! \underset{ \, \ell = 1}{ \overset{m}{\prod} }\! \overline{Z_{U_N}\!\prth{ \!e^{-2 i \pi \frac{y_\ell}{N}} \! } } \! }     \vphantom{\bigg(} \raisebox{4.5ex}{$ \white{_.} $} \raisebox{-4ex}{$ \white{.} $} \hspace{-1cm} $ & $ N^{k m} $ & $ \!a_m\!\int_{ \Rr^{m - 1 } } \!\! \Phi_{\!\Ae(X, Y)\, }\!(0, \tb) \Delta( 0, \tb)^2 d\tb $ & -/\eqref{EqPhi:Autocorrels}/\eqref{EqPhi:Random:Autocorrels},\eqref{EqPhi:Random:AutocorrelsAsFourier}   \\ \hline  
% =====================================================================================
  $ \hspace{-0.7cm} \Esp{ \!\hspace{-0.05cm}|\hspace{-0.05cm}Z_{U_N}(1)\hspace{-0.05cm}|^{2(k - |\hb| )} \!\! \underset{ \, r = 1}{ \overset{m}{\prod} }\!\! \abs{\partial^r\! Z_{U_N}(1)}^{2h_r} \!\! }     \raisebox{4ex}{$ \white{.} $} \raisebox{-3ex}{$ \white{.} $} \hspace{-1cm} $ & $ \!\! N^{k^2 \!+ 2 C(\hb)} \!\! $ & $\!\!a_k\!\int_{ \Rr^{k } } \! \Phi_{\De_\infty(k ; \hb)}\!( \xb) \Delta( \xb)^2 d\xb $ & \!\!\eqref{EqPhi:Derivatives}/\eqref{EqPhi:Alternative:DerivativesMicro}/\eqref{EqPhi:Random:DerivativesMicro}\!\!  \\ \hline  
% =====================================================================================
  $ \Esp{ \abs{ Z_{U_N, \pe{\rho N}}(1) }^{2k} }     \vphantom{\bigg(} \!\!\! $ & $ N^{k^2} $ & $\!\!a_k\!\int_{ \Rr^{k } } \! \Phi_{\!\Ze\Te_\rho^{(k)}}( \xb) \Delta(0, \xb)^2 d\xb $ & \eqref{EqPhi:TruncatedCharpolIn1}/\eqref{EqPhi:TruncatedCharpolIn1Bis}/\eqref{EqPhi:Random:TruncatedCharpolIn1}   \\ \hline  
% ===================================================================================== 
  $ \Esp{  \abs{ \lambda^{-\rho N} Z_{U_N, \pe{\rho N}}(\lambda) }^{2k} }, \lambda > 1     \vphantom{\bigg(} \!\!\! $ & $ N^{(k - 1)^2} $ & $\!\! {}_2 F_1 \prth{ \begin{smallmatrix} k, k \\ 1 \end{smallmatrix} \big\vert \lambda^{-2} } \Se\Ce_{\rho}^{(k)} $ & \eqref{EqPhi:TruncatedCharpolMicro}/-/\eqref{Eq:LimTruncatedCharpolHeapLindqvist}  \\ \hline  
% =====================================================================================
  $ L(\ell, \Se_k)     \vphantom{\bigg(}   $ & $ \ell^{k^2} $ & $\!\!a_k\!\int_{ \Rr^{k } } \! \Phi_{\Se_k}\!( \xb) \Delta( \xb)^2 d\xb $ & -/\eqref{EqPhi:VolumeSubBirkoffPolytopeRMT}/-  \\ \hline  
% =====================================================================================
  $ L(\ell, \Be_k)     \vphantom{\bigg(}   $ & $ \ell^{(k - 1)^2} $ & $\!\!a_k\!\int_{ \Rr^{k - 1 } } \! \Phi_{\Be_k}\!(0, \xb) \Delta(0, \xb)^2 d\xb $ & -/\eqref{EqPhi:VolumeBirkoffPolytopeRMT}/-  \\ \hline  
% =====================================================================================
  $ L(\ell, \Te_{\lambda, \mu} )     \vphantom{\bigg(}   $ & $ \!\!\!\! \ell^{(\hspace{-0.02cm} m \hspace{-0.02cm}-\hspace{-0.02cm} 1 \hspace{-0.02cm} )\hspace{-0.02cm}(\hspace{-0.02cm}n \hspace{-0.02cm}-\hspace{-0.02cm} 1\hspace{-0.02cm}) }    \!\!\!\! $ & $\!\!a_k\!\int_{ \Rr^{k - 1 } } \! \Phi_{\Te_{\lambda, \mu} }\!(0, \xb) \Delta(0, \xb)^2 d\xb $ & -/\eqref{EqPhi:CovolumeTransportationPolytope}/-  \\ \hline  
% =====================================================================================
  $ \Esp{ \! \det(U_N)^{-k}  \frac{\prod_{j = 1}^\ell Z_{U_N}\prth{e^{ 2 i \pi x_j/N} }}{ \prod_{r = 1}^m  Z_{U_N}\prth{ e^{ 2 i \pi y_r/N} } }  }      \raisebox{4ex}{$ \white{.} $} \raisebox{-3ex}{$ \white{.} $} $ & $ \!\!\! N^{ k ( \ell \hspace{-0.02cm}-\hspace{-0.02cm} m \hspace{-0.02cm} - \hspace{-0.02cm}k) }    \!\!\!  $ & $\!\!a_k\!\int_{ \Rr^{k - 1 } } \! \Phi_{\Ree(X, Y) } (0, \tb) \Delta(0, \tb)^2 d\tb $ & -/\eqref{EqPhi:Ratios}/- \\ \hline 
% ===================================================================================== tutu
  $ \Esp{ \prth{ \oint_\Uu \abs{ Z_{U_N}(z) }^{2\beta} \frac{d^*z}{z} }^{\!\! k}\, }    \raisebox{4ex}{$ \white{.} $} \raisebox{-4ex}{$ \white{.} $} $ & $   N^{ (k\beta)^2 \hspace{-0.02cm} + \hspace{-0.02cm} 1 \hspace{-0.02cm} -\hspace{-0.02cm} k }      $ & $\!\!a_{k\beta}\!\int_{ \Rr^{k\beta - 1 } } \! \! \Phi_{\MoMf_+(k, \beta) } (0, \xb)   \raisebox{-3ex}{$ \hspace{-1.3cm} \Delta(0, \xb)^2 d\xb $}  $ & \eqref{EqPhi:MoMo}/\eqref{EqPhi:Alternative:MoMo}/\eqref{EqPhi:Random:MoMo} \\ \hline   
% =====================================================================================
  \end{tabular}
\end{center}
%
%
%$ $
%
%$ $
%
%Note that the limit $ a_k \int_{ \Rr^{k } }  \Phi_{\!\Ze\Te_\rho^{(k)}}( \xb) \Delta(0, \xb)^2 d\xb $ is not a typo~: one has an integral over $ \Rr^k $ with $ \Delta(0, \xb) $ (but $ \xb = (x_2, \dots, x_{k + 1}) $ in this case).
%

$ $

% ==
\subsection{Partial conclusion and future work}\label{Subsec:Ultimate:Conclusion}

Several problems have been already raised throughout the previous chapters, and we list here some extra questions that are left for future work. 

The first possible generalisation of this work concerns of course a change of group, from the unitary to the orthogonal or the symplectic group, as Schur functions and Haar measures are defined in a very similar way, but a better suited generalisation lies in the framework of $ \beta $-ensembles. In addition, two groups of importance should not be forgotten~: the symmetric group $ \Sg_n $ and $ GL_n(\Ff_q) $ for a prime power $q$. A study of the autocorrelations of the characteristic polynomial of a random (uniform) permutation was performed by Dehaye and Zeindler \cite{PODZeindler} and Schur-Weyl duality allows to express such a functional on $ \Ue_n $, by means of a supersymmetric Schur function (with an alphabet composed by means of the points in which the autocorrelations are taken). Since $ \Sg_n $ occurs as a limit $ q \to 1 $ (the conjectural \textit{field with one element}) or $ q \to +\infty $ (after rescaling) for several functionals of $ GL_n(\Ff_q) $, it also seems of interest to look at the equivalent problems on this last field. As described in the introduction, problems over finite fields are often amenable to exact computations (as opposed to number fields), and these analogies can become fruitful when characterising a general phenomenon. Of course, the question of extending these reproducing kernel methods to number theory should not be forgotten, although the general CFKRS philosophy seems to imply that the ``non standard manipulations'' described in \S~\ref{Subsec:Theory:CFKRSwithRKHS} are necessary before applying a reproducing kernel methodology~; an alternative to these manipulations can nevertheless be explored. 

\medskip

Another point of interest that this article highlighted is the importance of the mid-secular coefficients. Being able to find the exact limit in law of such a functional would allow, with the help of the randomisation paradigm, to access the limiting behaviour of several of the studied functionals, if nevertheless the commutation between rescaling and randomisation is ``well-behaved''.

\medskip

Last, a transfer of technology is always possible from several neighbouring fields, and the SoV method described in \S~\ref{Subsubsec:Ultimate:OtherApproachesSchur:SoV} should prove useful in giving new expressions of the limits here computed.

\appendix

$ $

$ $

% ==
\section{Probabilistic representations}\label{Sec:ProbReprGegenbauer}

% ==
\subsection{A symmetric function extension of Gegenbauer polynomials} 

Remark that
\begin{align*}%$
h_n(x, y) = \crochet{t^n }H\crochet{t(x + y)} = \crochet{t^n } \frac{1}{(1 - tx)(1 - ty) } = \sum_{k = 0}^n x^k y^{n - k} = \frac{x^{n + 1} - y^{n + 1} }{x - y} = y^n h_n(x/y, 1)
\end{align*}

In particular, if $ z \in \Uu $, i.e. $ z = e^{2i\pi \theta} $, then, 
\begin{align*}%$
h_n(z, 1) = \sum_{k = 0}^n z^k = \sum_{k = 0}^n e^{2i\pi k \theta} = \frac{1 - e^{2i\pi (n + 1) \theta} }{ 1 - e^{2i\pi \theta} } = e^{2i\pi \frac{n}{2} \theta } \frac{\sin\prth{ 2 \pi\frac{n + 1}{2} \theta }}{\sin\prth{2 \pi\frac{ \theta }{2} } } = e^{i \pi n \theta } U_n( \cos(\pi \theta ) )
\end{align*}
where $ U_n $ is the Tchebychev polynomial of second kind. Thus, 
\begin{align*}%$
U_n( \cos(\pi \theta ) ) = h_n\prth{e^{i \pi \theta}, e^{-i\pi\theta } }
\end{align*}
and $ h_n(e^{i \theta_1},\dots, e^{i \theta_k}) $ is a multivariate generalisation of the Tchebychev polynomial. 

The same can be said about $ h_n\crochet{1^{\plusInOne \kappa} X } $. The classical Gegenbauer (or ultraspherical) polynomials $ G_n^{(\kappa)}(x) $ are defined for all $ \kappa > - \frac{1}{2} $ by their generating series
\begin{align*}%$
\sum_{n \geq 0} t^n G_n^{(\kappa)}(x) = \frac{ 1}{ ( 1 - 2xt + t^2)^\kappa }
\end{align*}

The Tchebychev polynomials of the second kind are given for $ \kappa = 1 $. For $ x = \cos(\theta) $, one has $ 1 - 2xt + t^2 = \abs{1 - te^{i \theta} }^2 $ and
\begin{align*}%$
G_n^{(\kappa)}(\cos(\theta) ) & = \crochet{t^n } \frac{1}{ ( 1 - t e^{i \theta} )^\kappa ( 1 - t e^{-i \theta} )^\kappa }  =  \crochet{t^n } H\crochet{t ( e^{i \theta} + e^{ - i \theta} ) }^\kappa = \crochet{t^n } H\crochet{t 1^{\plusInOne \kappa} ( e^{i \theta} + e^{ - i \theta} ) } \\
              & = h_n\crochet{ 1^{\plusInOne \kappa} ( e^{i \theta} + e^{ - i \theta} ) } % = e^{i n\theta} h_n\crochet{ 1^{\plusInOne \kappa} ( 1 + e^{ - 2 i \theta} ) }
\end{align*}

One can thus define a multivariate extension of Gegenbauer polynomials setting
\begin{align*}%$
h_n^{(\kappa)}(x_1, \dots, x_k) := h_n\crochet{ 1^{\plusInOne \kappa} (x_1 + \dots + x_k) } = \sum_{\lambda \vdash n} \frac{\kappa^{\ell(\lambda) } }{z_\lambda} p_\lambda(x_1, \dots, x_k) %= \crochet{t^n} \prod_{j = 1}^k \frac{1}{ (1 - t x_j)^\kappa }
\end{align*}
where the second equality comes from \eqref{Eq:PlethisticHnSchur} and the fact that $ p_\lambda\crochet{1^{\plusInOne \kappa } } = \kappa^{\ell(\lambda) } $. Note that the associated multivariate generalisation of Tchebychev polynomials is then $ h_n(x_1, \dots, x_k) $.

% ==
\subsection{An integral representation of $ h_n\crochet{\Ae} $}

We first introduce an important simple lemma that will be of constant use throughout this article:

\begin{lemma}[Crucial trivial lemma]\label{Lemma:ShallowCrucialLemma} For all $ n, r \geq 1 $ and for any alphabets $ Z := \ensemble{z_1, z_2, \dots } $ and $ U := \ensemble{u_1, u_2, \dots } $, we have
\begin{align}\label{Eq:t_n_trick}%$
\crochet{t^n} H\crochet{t U}  = \crochet{t^n} H\crochet{tU - t^{n + r} Z }
\end{align}
\end{lemma}

% ==

\begin{proof}
The proof is straightforward as 
\begin{align*}%$
H\crochet{tU - t^{n + r} Z } = H\crochet{tU }H\crochet{ - t^{n + r} Z } = \sum_{k_1, k_2 \geq 0} t^{k_1} h_{k_1}\crochet{U} (-t)^{k_2(n + r)} e_{k_2}\crochet{Z}
\end{align*}
hence, taking the $n$-th Fourier coefficient only gives $ h_n\crochet{U} $.
\end{proof}

We have now

\begin{lemma}[Integral representation of $ h_n^{(\kappa)}(u_1, \dots, u_k) $] We have for all $ \kappa > 0 $
\begin{align}\label{Eq:GegenbauerMultivariateWithTKL}%$
h_n^{(\kappa)}(e^{2 i \pi \theta_1}, \dots, e^{2 i \pi \theta_k})  = e^{   i \pi n \kappa \sum_j \theta_j  }  \int_0^1  e^{   i \pi n \alpha (k \kappa - 2)  }  \prod_{j = 1}^k \prth{ \frac{\sin\prth{   \pi (\alpha + \theta_j)  (n + 1)   } }{ \sin\prth{  \pi (\alpha + \theta_j) } } }^\kappa   d\alpha
\end{align}

\end{lemma}

% ==

\begin{proof}
Using lemma \ref{Lemma:ShallowCrucialLemma} with $r = 1$ and $ Z := 1^{\plusInOne \kappa} U^{n + 1} := 1^{\plusInOne \kappa} (u_1^{n + 1}+ \dots + u_k^{n  + 1}) $, one gets
\begin{align*}%$
h_n^{(\kappa)}(u_1, \dots, u_k)  = \crochet{t^n} H\crochet{t 1^{\plusInOne \kappa} U}  = \crochet{t^n} H\crochet{t 1^{\plusInOne \kappa} U - t^{n + 1} U^{n + 1}1^{\plusInOne \kappa} }  =  \crochet{t^n} \prod_{j = 1}^k \prth{  \frac{ 1 - t^{n + 1} u_j^{n + 1} }{  1 - t u_j } }^\kappa 
\end{align*}

This implies, on the unit circle 
\begin{align*}%$
h_n^{(\kappa)}(e^{2 i \pi \theta_1}, \dots, e^{2 i \pi \theta_k})  & =  \int_0^1  e^{ -2 i \pi \alpha    n}  \prod_{j = 1}^k \prth{ \frac{1 -  e^{ 2 i \pi (\alpha + \theta_j)  (n + 1)   } }{ 1 - e^{ 2 i \pi (\alpha + \theta_j) } } }^\kappa  d\alpha \\
              & = \int_0^1  e^{ -2 i \pi \alpha    n}  \prod_{j = 1}^k \prth{  e^{i \pi n (\alpha + \theta_j) } \frac{\sin\prth{   \pi (\alpha + \theta_j)  (n + 1)   } }{ \sin{  \pi (\alpha + \theta_j) } }  }^\kappa  d\alpha \\
              & = \int_0^1  e^{   i \pi ( n \alpha (k \kappa - 2) + n \kappa \sum_j \theta_j) }  \prod_{j = 1}^k \prth{ \frac{\sin\prth{   \pi (\alpha + \theta_j)  (n + 1)   } }{ \sin\prth{  \pi (\alpha + \theta_j) } } }^\kappa   d\alpha
\end{align*}
which is the desired expression.
\end{proof}

\begin{remark}
A particular case of interest is given by
\begin{align}\label{Eq:TchebytchevMultivariate}%$
h_n(e^{2 i \pi \theta_1}, \dots, e^{2 i \pi \theta_k})  = e^{   i \pi n \sum_{j = 1}^k \theta_j  } \int_0^1  e^{   i \pi  n \alpha (k - 2)  }  \prod_{j = 1}^k  U_n( \cos[ \pi (\alpha + \theta_j) ] ) d\alpha 
\end{align}
with $ U_n(\cos(\theta) ) = \frac{\sin(n\theta)}{\sin(\theta) } $ the Tchebychev polynomial of second kind.
\end{remark}

% ==

The advantage of the integral representation \eqref{Eq:GegenbauerMultivariateWithTKL} lies in the following lemma:

\begin{lemma}[Rescaled Gegenbauer functions in the microscopic regime]\label{Lemma:RescaledGegenbauer}
Let $ k \geq 2 $ and $ \kappa \geq 1 $ be integers. We have, locally uniformly in $ (x_1,\dots,x_k) \in \Rr^k $ and for all $ c \in \Rr_+^* $
\begin{align*}%$
\frac{1}{ N^{\kappa k - 1} } h^{(\kappa)}_{ \pe{ c N} }\prth{ e^{2 i \pi \frac{x_1}{N} }, \dots, e^{2 i \pi \frac{x_k}{N} } }   \tendvers{N}{+\infty} e^{   i \pi c \kappa \sum_{j = 1}^k x_j  } \, h^{(\kappa)}_{c, \infty}(x_1, \dots, x_k)
\end{align*}
with $ h^{(\kappa)}_{c, \infty} $ defined by  
\begin{align}\label{Def:HKappaCInfty}%$
h^{(\kappa)}_{c, \infty}( x_1, \dots, x_k) := c^{ k\kappa} \int_\Rr   e^{   i \pi c (k\kappa - 2) y } \prod_{j = 1}^k  \sinc\prth{   \pi c (y + x_j) }^\kappa dy
\end{align}
\end{lemma}

% ==

\begin{proof}
Using \eqref{Eq:GegenbauerMultivariateWithTKL}, we get 
\begin{align*}%$
h^{(\kappa)}_{ \pe{ c N} }\prth{  e^{2 i \pi \frac{x_1}{N} }, \dots, e^{2 i \pi \frac{x_k}{N} } }   & = e^{   i \pi  \kappa \frac{\pe{ c N}}{N} \sum_j x_j  }  \int_{-\frac{1}{2} }^{\frac{1}{2} }   e^{   i \pi \pe{ c N} \alpha (k \kappa - 2)  }  \prod_{j = 1}^k \prth{ \frac{\sin\prth{   \pi (\alpha + \frac{x_j}{N} )  \pe{ c N+ 1}  } }{ \sin\prth{  \pi (\alpha + \frac{x_j}{N}) } } }^\kappa   d\alpha \\
            & =  e^{   i \pi  \kappa \frac{\pe{ c N}}{N} \sum_j x_j  }  \int_{-\frac{N}{2} }^{\frac{N}{2} }   e^{   i \pi \frac{\pe{ c N}}{N} y (k \kappa - 2)  }  \prod_{j = 1}^k \prth{ \frac{\sin\prth{   \pi   \frac{y + x_j}{N}   \pe{ c N+ 1}  } }{ \sin\prth{  \pi  \frac{y + x_j}{N}  } } }^\kappa   \frac{dy}{N} \\
            &  \hspace{-0.25cm} \equivalent{N\to +\infty } \ N^{k\kappa - 1} e^{   i \pi c \kappa \sum_{j = 1}^k x_j  } \!\! \int_{ \Rr }  e^{   i \pi c (k\kappa - 2) y } \prod_{j = 0}^k \prth{  \frac{\sin\prth{   \pi c (y + x_j) }  }{  \pi  (y + x_j)  } }^\kappa dy \\
           & =:  N^{ k\kappa - 1} e^{   i \pi c \kappa \sum_{j = 1}^k x_j  } \, h^{(\kappa)}_{c, \infty}( x_1, \dots, x_k)
\end{align*}

Here, we have applied dominated convergence to pass to the limit. The required estimate needed to apply it comes from the continuity of sinc in a neighbourhood of $0$ and its decay at infinity (since $ \kappa \geq 1 $ and $ k \geq 2 $).
\end{proof}

% ==

\begin{remark}
One can change variable $ y' := -y $ in \eqref{Def:HKappaCInfty} in which case, one gets a product of ``sine kernel'' $ \sinc(\pi c(x_i - y') ) $. This would be interesting to see if such a limit also holds for other matrix models using a symmetric function approach such as initiated in \cite{JonnadulaKeatingMezzadri}, as the sine kernel is a universal kernel occuring in the bulk of several matrix models. A positive answer to this question would mean that the reproducing kernel method employed in this article is a manifestation of the determinantal nature of the eigenvalues and that the limiting form \eqref{Def:FormPhi} goes far beyond the scope of the $ CUE $. We will treat the case of orthogonal and symplectic ensembles in a subsequent publication.
\end{remark}

\medskip
% ==
\subsection{Probabilistic representations of $ h^{(\kappa)}_{c, \infty} $}\label{Subsec:ProbaHn}

% ==
\subsubsection{Beta-Gamma representations}

We show here that the functions $ h^{(\kappa)}_{c, \infty} $ are, up to a constant, the Fourier transforms of a particular random variable. This representation allows furthermore to get the speed of convergence in lemma \ref{Lemma:RescaledGegenbauer}.

\begin{lemma}[Gamma representation of $ h^{(\kappa)}_n $]\label{Lemma:GammaProbRepr}
Let $ \kappa > 0 $ and $ ( \gammab_\kappa^{(j)} )_{1 \leq j \leq k} $ be a sequence of i.i.d. random variables $ \Gammab(\kappa) $-distributed, i.e. $ \Prob{ \gammab_\kappa \in dt } = e^{ - t} t_+^{\kappa - 1} \frac{dt}{\Gamma(\kappa) } $. Then, for all $ x_1, \dots, x_k \in \Cc^k $
\begin{align}\label{Eq:Gamma_ProbRepr}%$
h_n^{(\kappa)}(x_1, \dots, x_k) = \frac{1}{n!} \Esp{ \prth{ \sum_{j = 1}^k x_j \gammab_\kappa^{(j)} }^{\! \! n} }
\end{align}
\end{lemma}

% ==

\begin{proof}
We have $ \Esp{ e^{-t \gammab_\kappa} } = (1 + t)^{-\kappa} $ for all $ t \in \Rr_+ $. Let $t \in \Rr_+ $ be such $ \abs{t x_i} < 1 $. Then, 
\begin{align*}%$
H\crochet{t 1^{\plusInOne \kappa} X} = \prod_{j = 1}^k \frac{1}{(1 - t x_j)^\kappa} = \prod_{j = 1}^k \Esp{ e^{t x_j \gammab_\kappa} } = \Esp{ e^{t \sum_{j = 1}^k y_j \gammab_\kappa^{(j)} } }
\end{align*}
hence, taking the $n$-th Fourier coefficient inside the expectation gives the result.
\end{proof}

This last representation yields the following:

% ==
\begin{lemma}[Beta representation of $ h^{(\kappa)}_{c, \infty} $]\label{Lemma:BetaProbRepr} Let $ (\betab_{\kappa, 1}^{(j)})_{1 \leq j \leq k} $ be a sequence of i.i.d. $ \Bd( \kappa, 1) $-distributed random variables, i.e. $ \Prob{\betab_{\kappa, 1} \in dt } = \kappa t^{\kappa - 1} \Unens{0 \leq t \leq 1} dt  $ (or equivalently, $ \betab_{\kappa, 1} \eqlaw U^{1/\kappa} $ where $U$ is uniform on $ \crochet{0, 1} $). Then,
\begin{align}\label{Eq:Beta_ProbRepr}%$
h^{(\kappa)}_{c, \infty}( x_1, \dots, x_k) = \frac{c^{k\kappa - 1} }{ \Gamma\prth{ k\kappa  } } \, \Esp{ e^{ 2 i \pi c \sum_{j = 1}^k x_j ( \betab_{\kappa, 1}^{(j)} - \frac{\kappa}{2})  } \Bigg\vert \sum_{j = 1}^k \betab_{\kappa, 1}^{(j)}  = 1 }
\end{align}

Moreover, setting $ \norm{\xb}_p := \prth{ \sum_{j = 1}^k \abs{x_j}^p }^{\! 1/p} $ for all $ p > 0 $,  one has 
\begin{align}\label{Ineq:SpeedOfConvergenceHinfty}%$
\hspace{-0.4cm} N^{1 - k\kappa } h_{ \pe{c N} }^{(\kappa)}\!\prth{ \hspace{-0.05cm} e^{2 i \pi \frac{x_1}{N} }\!, \dots, e^{2 i \pi\frac{x_k}{N}} \! } = h^{(\kappa)}_{c, \infty} ( x_1, \dots, x_k) \prth{\! 1 +  O\prth{ \!\frac{c\inv\! + \! (c + 1)\!\norm{\xb}_2^2 + \norm{\xb}_1 }{N} \! } \! }
\end{align}
where the implicit constant in $ O $ is independent of $ c, x_1, \dots, x_k $.
\end{lemma}

% ==

\begin{proof}
Let $ S_k := \sum_{j = 1}^k  \gammab_\kappa^{(j)} $. A classical result on the beta-gamma algebra (see \cite[(3.b)]{CarmonaPetitYor}) gives the independence of $ S_k $ and $ \prth{ \gammab_\kappa^{(j)} / S_k }_{1 \leq j \leq k} $ and the fact that
\begin{align}\label{Eq:BetaGammaEqlaw}%$
\prth{ \gammab_\kappa^{(j)} / S_k }_{1 \leq j \leq k} \eqlaw \prth{ (\betab_{1, \kappa}^{(j)})_{1 \leq j \leq k} \bigg\vert \sum_{j = 1}^k \betab_{1, \kappa}^{(j)} = 1 }
\end{align}

We can thus write \eqref{Eq:Gamma_ProbRepr} as  
\begin{align*}%$
h_{ \pe{c N} }^{(\kappa) }\prth{ e^{2 i \pi x_1/N }, \dots, e^{2 i \pi x_k/N} }  & = \frac{1}{\pe{ c N}!} \Esp{ \prth{ S_k +  \sum_{j = 1}^k (e^{2 i \pi x_j/N } - 1) \gammab_\kappa^{(j)} }^{\!\!\pe{ c N}} } \\
                 & = \frac{\Esp{ S_k^{\pe{c N}} } }{\pe{c N}!} \Esp{ \prth{ 1 +  \sum_{j = 1}^k (e^{2 i \pi x_j/N } - 1) \frac{\gammab_\kappa^{(j)}  }{S_k} }^{\!\!\pe{c N}} }
\end{align*}
by independence. Now, recall that $ S_k \sim \operatorname{Gamma}(k\kappa ) $, namely $ S_k \sim t_+^{k\kappa - 1} e^{-t} \frac{dt}{\Gamma(k\kappa)} $ and that $ \Esp{S_k^m} = \frac{\Gamma(k \kappa + m ) }{ \Gamma(k \kappa ) } $. We thus get
\begin{align*}%$
h_{ \pe{c N} }^{(\kappa) }\prth{  e^{2 i \pi x_1/N }, \dots, e^{2 i \pi x_k/N} }  = \frac{ \Gamma( \pe{c N} + k\kappa ) }{ \pe{c N}! \Gamma(k \kappa ) } \Esp{ \prth{ 1 + \frac{1}{N} \sum_{j = 0}^k N(e^{2 i \pi x_j/N } - 1) \frac{ \gammab_\kappa^{(j)} }{S_k} }^{\!\!\pe{c N}} }
\end{align*}

Define
\begin{align}\label{Def:XiNandXiInfty}%$
\begin{aligned}
\xi_N(X) & := \sum_{j = 0}^k N(e^{2 i \pi x_j/N } - 1) \frac{ \gammab_\kappa^{(j)} }{S_k} \\
\xi_\infty(X) & := 2 i \pi \sum_{j = 0}^k  x_j \frac{ \gammab_\kappa^{(j)} }{S_k}
\end{aligned}
\end{align}
so that
\begin{align*}%$
h_{ \pe{c N} }^{(\kappa) }\prth{  e^{2 i \pi x_1/N }, \dots, e^{2 i \pi x_k/N} }  = \frac{ \Gamma( \pe{c N} + k\kappa ) }{ \pe{c N}! \Gamma(k \kappa ) } \, \Esp{ \prth{ 1 + \frac{\xi_N(X) }{N}}^{\pe{c N}} }
\end{align*}

Since $ N(e^{2 i \pi x /N } - 1) = 2i \pi x \, \Esp{ e^{2i\pi x U/N} } $ with $ U $ uniform on $ \crochet{0, 1} $, we have $ \abs{ N(e^{2 i \pi x /N } - 1) } \leq 2\pi \abs{x} $ and since $ \abs{ \log(1 + x) - x } \leq \abs{x}^2/2 $ for all $ x \in \Cc $ such that $ \abs{x} < 1 $, we get, almost surely and for $ N \geq 2\pi \sum_{i = 1}^k \abs{x_i} $ big enough 
\begin{align}\label{Ineq:BoundXi}%$
\begin{aligned}
\abs{ \xi_N(X) } & \leq 2\pi \sum_{j = 1}^k \abs{x_j } \\
\abs{\log \prth{ 1 + \frac{\xi_N(X) }{N} } - \frac{\xi_N(X)}{N} } & \leq \frac{4\pi^2 }{N^2}  \sum_{j = 1}^k \abs{x_j }^2
\end{aligned}
\end{align}
where we have used the branch of the logarithm with a cut on $ \Rr_- $. This implies
\begin{align*}%$
\Esp{ \prth{ 1 + \frac{ \xi_N(X) }{N} }^{\!\!\pe{c N} } } & =  \Esp{ e^{ \pe{c N} \log\,\prth{ 1 + \frac{\xi_N(X) }{N} } } } = \Esp{ e^{ \pe{c N}  \frac{\xi_N(X) }{N} + O\,\prth{ \frac{\pe{c N}}{N^2} \sum_{j = 1}^k \abs{x_j }^2 } }  } \\
              & = \Esp{ e^{ \pe{c N}  \frac{\xi_N(X) }{N} } } e^{  O\,\prth{ \frac{ c }{N} \sum_{j = 1}^k \abs{x_j }^2 } }  
\end{align*}
since the constant in the $ O $ is deterministic by the second inequality in \eqref{Ineq:BoundXi}. Note also that this constant is independent of $ c, x_1, \dots, x_k $.

Now, we have $ \pe{c N} = c N - \pf{cN} $ with $ \pf{cN} \in [0, 1) $ and $ \xi_N(X) = \xi_\infty(X) - O\prth{ \frac{1}{N} \sum_{j = 1}^k \abs{x_j}^2 } $ by the inequality $ \abs{ N(e^{2 i \pi x /N } - 1) - 2 i \pi x } \leq 4\pi^2 \abs{x}^2/ N $ and the fact that $  \gammab_\kappa^{(j)} / S_k < 1 $ almost surely. Thus, with deterministic uniform $ O $, we get
\begin{align*}%$
\Esp{ e^{ \pe{c N}  \frac{\xi_N(X) }{N} } } = \Esp{ e^{ c \, \xi_\infty(X) + O\,\prth{\frac{ 1 }{N} \sum_{j = 1}^k \abs{x_j}  } + O\,\prth{ \frac{1}{N} \sum_{j = 1}^k \abs{x_j}^2 }   } }
\end{align*}

Multiplying all the contributions yields
\begin{align*}%$
\Esp{ \prth{ 1 + \frac{ \xi_N(X) }{N} }^{\!\!\pe{c N} } }  & =  \Esp{ e^{ c \, \xi_\infty(X) } } e^{ O\,\prth{ \frac{ c }{N} \sum_{j = 1}^k \abs{x_j }^2 } + O\,\prth{\frac{1}{N}\sum_{j = 1}^k \abs{x_j}  } + O\,\prth{ \frac{1}{N} \sum_{j = 1}^k \abs{x_j}^2 }   }  \\
              & =  \Esp{ e^{ c \, \xi_\infty(X) } } \prth{ 1 + O \prth{ \frac{ (c + 1) \norm{\xb}_2^2 + \norm{\xb}_1 }{N}  } }
\end{align*}

Last, when $ N \to +\infty $, the Stirling approximation gives
\begin{align*}%$
\frac{ \Gamma( \pe{c N} + k\kappa ) }{ \pe{c N}! \Gamma(k \kappa ) } = \frac{ \Gamma ( \pe{c N} + k\kappa )  }{ \Gamma( \pe{ c N} + 1) \Gamma(k \kappa ) } = \frac{ (cN)^{k\kappa - 1} }{ \Gamma(k \kappa ) } \prth{1 + \frac{(k\kappa)^2}{2 c N} + O_\kappa\prth{\frac{1}{c^2 N^2} } }
\end{align*}
implying
\begin{align*}%$
N^{1 - k\kappa} h_{ \pe{c N} }^{(\kappa) }\prth{ e^{2 i \pi x_1/N }, \dots, e^{2 i \pi x_k/N} }  & = \frac{ c^{k\kappa - 1} }{ \Gamma(k \kappa ) }  \Esp{ e^{ c \, \xi_\infty(X) } }\prth{ 1 + O\prth{ \frac{c\inv + (c + 1) \norm{\xb}_2^2 + \norm{\xb}_1 }{N} } \! }
\end{align*}

By lemma \ref{Lemma:RescaledGegenbauer}, we thus have
\begin{align*}%$
e^{ i \pi c \kappa \sum_{j = 1}^k x_j }  h_{ c, \infty }^{(\kappa) }\prth{  x_1, \dots, x_k }  = \frac{ c^{k\kappa - 1} }{ \Gamma(k \kappa ) }  \Esp{ e^{ c \, \xi_\infty(X) } } 
\end{align*}

Last, \eqref{Eq:BetaGammaEqlaw} yields
\begin{align*}%$
\Esp{ e^{ c \, \xi_\infty(X) } } = \Esp{ e^{ 2 i \pi c \sum_{j = 1}^k x_j \betab_{\kappa, 1}^{(j)}  } \bigg\vert \sum_{j = 1}^k \betab_{\kappa, 1}^{(j)}  = 1 }
\end{align*}
concluding the proof.
\end{proof}

% ==

% ==
\subsubsection{A remark on the uniform representation}
The case $ \kappa = 1 $ is of particular interest. We have with $ (\ee_j)_{1 \leq j \leq k} $ a sequence of i.i.d. $ \Exp(1) $-distributed random variables
\begin{align}\label{Eq:ProbabilisticRepresentation_h_n}%$
h_n\prth{y_1, \dots, y_k} = \frac{1}{n!} \Esp{ \prth{  \sum_{j = 1}^k y_j \ee_j }^{\!\! n} } 
\end{align}
and, with $ V_j = 2 U_j - 1 $ i.i.d random variable uniformly distributed on $ \crochet{-1, 1} $
\begin{align}\label{Eq:ProbabilisticRepresentation_h_rho}%$
h_{\rho, \infty}( x_1, \dots, x_k ) = \frac{\rho^{k - 1} }{ (k - 1) ! } \Esp{ e^{  i \pi \rho \sum_{j = 1}^k x_j V_j } \bigg\vert \sum_{j = 1}^k V_j = 2 - k }
\end{align}

This implies an additional probabilistic representation:
\begin{align*}%$
h_{\rho, \infty}(x_1, \dots, x_k) & = \rho^{k - 1} \int_{\crochet{0, 1}^{k - 1} } e^{ i \pi \rho \sum_{j = 1}^k x_j ( t_j - t_{j - 1} ) } \Unens{ 0 = t_0 < t_1 < \dots < t_{k - 1} < t_k = 1 } dt_1\dots dt_{k -1} \\
                & = \frac{\rho^{k - 1} }{ (k - 1) ! } \Esp{ e^{  i \pi \rho \sum_{j = 1}^k x_j (U_j - U_{j -1}) } \bigg\vert 0 = U_0 < U_1 < \dots < U_{k - 1} < U_k = 1  }
\end{align*}

It can be proven starting from \eqref{Eq:ProbabilisticRepresentation_h_rho} and using the equality in law 
\begin{align*}%$
\prth{   ( V_j )_{1 \leq j \leq k} \bigg\vert \sum_{j = 1}^k V_j = 2 - k } \eqlaw \prth{ (U_j - U_{j -1})_{1 \leq j \leq k}   \bigg\vert 0 = U_0 < U_1 < \dots < U_{k - 1} < U_k = 1  }
\end{align*}

%$ a_{\raisebox{+3ex}{$\scriptstyle S$} } $ vs $ a_{\raisebox{-3ex}{$\scriptstyle S$} }$

\medskip
% ==
\subsubsection{Negative Binomial representation}

We now give an alternative probabilistic representation of $ h_{[cn]} \! \hspace{-0.45cm} \vphantom{a}^{(\kappa)} $ that shows that the form of the beta representation of $ h_{c, \infty}^{(\kappa)} $ is already present for fixed $n$. We focus on $ c = 1 $, leaving details of the general case to the reader.

% ==

\begin{lemma}[Negative Binomial representation of $ h_n^{(\kappa)} $]\label{Lemma:NegBinRepresentation_h_n}
Recall that $ a^{\uparrow n} = a(a + 1)\dots (a + n - 1) $. Let $ (G_\ell(r, \kappa))_{1 \leq \ell \leq k} $ be a sequence of i.i.d. random variables with Negative Binomial distribution of parameter $ r < \min\!\ensemble{1, \max_{1 \leq \ell \leq k} \abs{ x_\ell }\inv } $ and $ \kappa \in \Rr_+^* $. Then, 
\begin{align}\label{Eq:NegBinRepresentation_h_n}%$
h_n^{(\kappa)}(x_1, \dots, x_k) = \frac{(k\kappa)^{\uparrow n}}{n!} \, \Esp{ \prod_{\ell = 1}^k x_\ell^{ G_\ell(r, \kappa) } \Bigg\vert \sum_{\ell = 1}^k G_\ell(r, \kappa) = n }
\end{align}
\end{lemma}

% ==

\begin{proof}
Recall that a Negative Binomial random variable $ G(r, \kappa) \sim \NegBin(r, \kappa) $ is given for $ 0 < r < 1 $ and $ \kappa \in \Rr_+^* $ by the Fourier-Laplace transform
\begin{align*}%$
\Esp{ z^{G(r, \kappa)} } = \prth{\frac{1 - r}{1 - rz} }^\kappa, \qquad z \in (0, 1) \mbox{ or } z\in \Uu
\end{align*}

For $ \kappa \in \Nn^* $ (which is the only case we consider in this article), $ G(r, \kappa) $ is a sum of $ \kappa $ i.i.d. Geometric random variables of parameter $r$ (i.e. $ G(r, 1) $).

Then, using \eqref{Eq:FourierCoeff} for all $ r < \min\!\ensemble{1, \max_{1 \leq \ell \leq k} \abs{ x_\ell }\inv } $, one gets
\begin{align*}%$
h_n^{(\kappa)}(x_1, \dots, x_k) & = \crochet{t^n} H\crochet{t X 1^{\plusInOne \kappa } } = r^{-n} \oint_{\Uu} z^{-n} H\crochet{r z X}^\kappa \frac{d^* z}{z} = r^{-n} \oint_{\Uu} z^{-n} \prod_{\ell = 1}^k \frac{1}{(1 - r z x_\ell)^\kappa} \frac{d^* z}{z} \\
                & = r^{-n}  (1 - r)^{-k\kappa} \oint_{\Uu} z^{-n} \prod_{\ell = 1}^k \prth{ \frac{1 - r}{1 - r z x_\ell} }^\kappa \frac{d^* z}{z} \\
                & =  r^{-n}  (1 - r)^{-k \kappa} \oint_{\Uu} z^{-n}  \Esp{ \prod_{\ell = 1}^k (z x_\ell)^{G_\ell(r, \kappa)} } \frac{d^* z}{z} \\
                & = r^{-n}  (1 - r)^{-k\kappa} \, \Esp{ \prod_{\ell = 1}^k x_\ell^{G_\ell(r, \kappa) } \Unens{ \sum_{\ell = 1}^k \! G_\ell(r, \kappa) \, = \, n } } \qquad \mbox{(Fubini)} \\
                & = r^{-n}  (1 - r)^{-k\kappa} \, \Prob{ \sum_{\ell = 1}^k G_\ell(r, \kappa) = n } \, \Esp{ \prod_{\ell = 1}^k x_\ell^{G_\ell(r, \kappa) }  \Bigg\vert \sum_{\ell = 1}^k G_\ell(r, \kappa) = n  }
\end{align*}

One concludes with
\begin{align*}%$
\Prob{ \sum_{\ell = 1}^k G_\ell(r, \kappa) = n \! } & = \Esp{ \crochet{t^n} t^{\sum_{\ell = 1}^k G_\ell(r, \kappa)} } \\
                    & = \crochet{t^n} \prth{ \frac{1 - r}{1 -rt} }^{k\kappa}  = \Prob{ G(r, k\kappa) = n  } \\
                    & = (1 - r)^{k \kappa} r^n h_n\crochet{ 1^{\plusInOne  k \kappa} } = (1 - r)^{k\kappa} r^n \frac{(k\kappa)^{\uparrow n}}{n!}
\end{align*}
since $ (1 - t)^{-k} = \sum_{n \geq 0} h_n\crochet{1^k} t^n = \sum_{n \geq 0} k^{\uparrow n} \frac{t^n}{n!} $.
\end{proof}

% ==

\begin{remark}\label{Rk:DegreeFreedomReprHn}
One may think that there is an additional degree of freedom in \eqref{Eq:NegBinRepresentation_h_n}, but this is not the case due to the scaling of $ h_n $. This degree of freedom was in fact present in the lemmas~\ref{Lemma:GammaProbRepr} and \ref{Lemma:BetaProbRepr} by considering $ h_n(\lambda X) = \lambda^n h_n(X) $ for a good choice of $ \lambda $ but was ``artificially'' suppressed by choosing the radius of convergence equal to $1$ in \eqref{Eq:FourierCoeff}. This choice of radius is equivalent to the scaling of $ h_n $ since one can always dilate and then take the Fourier coefficient on a circle of radius $1$. The relevant dilation in all the treated problems is found by considering $ h_n\crochet{(1 + X)1^{\plusInOne \kappa} } $ in the previous computations (leading to the ``hyperplane concentration'' phenomenon).
\end{remark}

% ==

\begin{remark}
One can pass to the limit in \eqref{Eq:NegBinRepresentation_h_n} using the following convergence in law, easily seen with the Fourier-Laplace transform of $ \NegBin(r, \kappa) $
\begin{align*}%$
\frac{1}{n} G\prth{ e^{-\nu/n} , \kappa } \cvlaw{n}{+\infty} \nu\inv \gammab_{\kappa, 1}
\end{align*}

The conditionning is a non trivial one at the limit since one has an event of probability 0 (it nevertheless works in the same way since the event is monotonous in $N$ hence converges monotonically). Setting $ r := e^{-\nu/n} $, one thus gets at the limit
\begin{align*}%$
h_n^{(\kappa)}\prth{ e^{\frac{x_1}{n} }, \dots, e^{ \frac{x_k}{n}} } & \equivalent{n\to + \infty} \, \frac{ n^{k\kappa - 1} }{\Gamma(k \kappa) } \, \Esp{ e^{ \sum_{\ell = 1}^k \frac{x_\ell}{n} G_\ell(e^{-\nu/n},\, \kappa) } \Bigg\vert \sum_{\ell = 1}^k n\inv G_\ell\prth{ e^{-\nu/n}, \kappa } = 1 } \\
                  & \equivalent{n\to + \infty} \, \frac{ n^{k\kappa - 1} }{\Gamma(k \kappa) } \, \Esp{ e^{ \sum_{\ell = 1}^k x_\ell \nu\inv \betab_{\kappa, 1}^{(\ell)} } \Bigg\vert \sum_{\ell = 1}^k \nu\inv \betab_{\kappa, 1}^{(\ell)}  = 1 }
\end{align*}

This is the most general form of $ h_{1, \infty}^{(\kappa)} $ with any type of scaling of the form $ \lambda^{1/n} $. 
\end{remark}

\medskip
% ==
\subsubsection{A remark on the Poisson representation}

A probabilistic theory of $ h_N\crochet{\Ae} $ emerged these last years with the study of the \textit{Generalised Ewens measure} (see e.g. \cite{BetzSchaeferZeindler, BetzUeltschi, BetzUeltschi2, BetzUeltschiVelenik, BogachevZeindler, CiprianiZeindler, HughesNajnudelNikeghbaliZeindler, NikeghbaliZeindler, RoblesZeindler} and references cited\footnote{This measure already appeared in the Russian literature under the name \textit{Kolchin model}, see e.g. \cite{KerovEwensPitman, Kolchin1971, KolchinSevastyanovChistyakov}.}). The link comes from 
%
%
%the following formula \cite[I-2 (2.10)]{MacDo} expressing \eqref{Def:Hfunctor} with the generating series of $ (\frac{p_k}{k})_{k \geq 1} $~:
%%
%%
%%
%\begin{align}\label{Eq:HwithP}%$
%H\crochet{ \Ae} = e^{ P\,\crochet{ \Ae} }, \qquad P\crochet{ \Ae} := \sum_{k \geq 1} \frac{p_k\crochet{\Ae} }{k}
%\end{align}
%%
%%
%%
%itself equivalent to \eqref{Eq:PlethisticHnSchur} after replacing $ \Ae $ by $ t \Ae $ and taking the $ \crochet{t^n} $-Fourier coefficient.
%
%
%Note that we slightly changed the convention of \cite[I-2 (2.10)]{MacDo} where $ P_{\mathrm{Macdo}}\crochet{\Ae} := \sum_{k \geq 1} p_k\crochet{\Ae} $. The advantage is the simpler for given by \eqref{Eq:HwithP}. 
%
% ==
%
the first equality in \eqref{Eq:PlethisticHnSchur} that defines the following probability measure on the set $ \Yy_n $ of partitions of $ n $, or equivalently on the symmetric group $ \Sg_n $ for positive specialisations of $ (p_k)_k $~:
\begin{align}\label{Def:EwensGen}%$
\Pp_{\! \Ae, \Yy_n}(\lambda) := \frac{1}{h_n\crochet{\Ae} } \frac{p_\lambda\crochet{\Ae} }{z_\lambda} \Unens{\lambda \vdash n} \quad \Longleftrightarrow \quad \Pp_{\! \Ae, \Sg_n}(\sigma) := \frac{1}{h_n\crochet{\Ae} } \prod_{k \geq 1} p_k\crochet{\Ae}^{c_k(\sigma)} \Unens{\sigma \in \Sg_n}
\end{align}

Here, $ c_k(\sigma) $ designates the number of $k$-cycles of $ \sigma \in \Sg_n $, i.e. the number of cycles of size $k$. Note that for $ \Ae = 1^\kappa $, one gets the classical Ewens measure \cite{ArratiaBarbourTavare, Ewens72} and for $ \kappa = 1 $, one gets the uniform measure on $ \Sg_n $ (which is not uniform on $ \Yy_n $).

As a result, one has
\begin{align*}%$
h_n\crochet{X \Ae} = h_n\crochet{\Ae} \frac{h_n\crochet{X\Ae}}{h_n\crochet{\Ae} } = h_n\crochet{\Ae} \Espr{\! \Ae, \Sg_n}{ \prod_{k \geq 1} p_k\crochet{X}^{c_k(\sigmab_n )}  }
\end{align*}
where $ \sigmab_n : \sigma \in \Sg_n \mapsto \sigma $ is the canonical evaluation of the probability space $ (\Sg_n, \Pp_{\! \Ae, \Sg_n}) $. In particular, as $ h_n\crochet{1^\kappa} = \frac{\kappa^{\uparrow n }}{ n! } $, one gets
\begin{align*}%$
h_n^{(\kappa)}\prth{ e^{\frac{x_1}{n} }, \dots, e^{ \frac{x_k}{n}} } =  \frac{\kappa^{\uparrow n }}{ n! } \, \Espr{  1^\kappa \! , \, \Sg_n}{ \prod_{k \geq 1} p_k\crochet{e^{X/n}}^{c_k(\sigmab_n )}  }, \qquad X := \ensemble{x_1, \dots, x_k}
\end{align*}

Now, a classical result on the cycle structure of a random Ewens-distributed permutation $ \sigmab_n^{(\kappa)} \sim \Pp_{1^\kappa, \Sg_n} $ ensures that \cite[ch. 2.3, (2.5) \& (2.7)]{ArratiaBarbourTavare}
\begin{align}\label{EqLaw:CyclesEwensNandT}%$
\prth{ c_k(\sigmab_n^{(\kappa)}) }_{1 \leq k \leq n}  \eqlaw \prth{ \prth{ \Pb\prth{ \tfrac{\kappa}{\ell} } }_{1 \leq \ell \leq n} \Bigg\vert \sum_{\ell \geq 1} \ell \Pb\prth{ \tfrac{\kappa}{\ell} } = n } 
\end{align}
where the Poisson random variables $ \prth{ \Pb\prth{ \tfrac{\kappa}{\ell} } }_{1 \leq \ell \leq n} $ in the RHS are independent.

As a result, one has
\begin{align}\label{Eq:PoissonRepresentation_h_n}%$
h_n^{(\kappa)}\crochet{e^{X/n}} =  \frac{\kappa^{\uparrow n }}{ n! } \, \Esp{ \prod_{\ell \geq 1} p_\ell\crochet{e^{X/n}}^{\Pb(\kappa / \ell )} \Bigg\vert \sum_{\ell = 1}^n \ell \Pb(\kappa/\ell) = n  } 
\end{align}
which is very similar to \eqref{Eq:NegBinRepresentation_h_n} in the form. Note nevertheless that the normalisation $ h_n\crochet{1^\kappa} $ in \eqref{Eq:PoissonRepresentation_h_n} is different from $ h_n\crochet{1^{k \kappa} } $ in \eqref{Eq:NegBinRepresentation_h_n}. The limiting functional, although existing by lemma~\ref{Lemma:NegBinRepresentation_h_n} (after renormalisation) seems complicated to find with this Poisson representation. The random variable $ \sum_{\ell = 1}^n \log(p_\ell\crochet{e^{X/n} } ) c_\ell(\sigmab_n^{(\kappa)}) $ is a well-studied \textit{linear statistic of the cycles} and several general theorems are available in the literature for its study \cite{BaksajevaManstavicius, Manstavicius}, but the coefficients $ \log(p_\ell\crochet{e^{X/n} } ) $ display a behaviour of the form $ \log(p_\ell\crochet{e^{X/n} } ) \sim \frac{\ell p_1(X)}{n} $ which does not seem to appear in the literature, as far as the author is aware of. Since the limit in lemma~\ref{Lemma:BetaProbRepr} uses beta random variables in \eqref{Eq:Beta_ProbRepr} that are used in one possible definition of the Poisson-Dirichlet point process, it is likely that the limit of the linear statistic will use this universality class that describes the behaviour of the largest cycles (rescaled by $N$).

\medskip

% ==
\subsection{Integrability of $ h_{c, \infty}^{(\kappa) } $}

The domination \eqref{Ineq:SpeedOfConvergenceHinfty} writes for a certain constant $ C > 0 $
\begin{align}\label{Ineq:SpeedOfConvergenceHinftyBis}%$
\abs{ N^{1 - k\kappa  } h_{ \pe{c N} }\prth{ e^{2 i \pi \frac{x_1}{N} }\!, \dots, e^{2 i \pi\frac{x_k}{N}} }  -  \widetilde{h}^{(\kappa)}_{c, \infty} ( x_1, \dots, x_k) } \leq  \frac{C}{N} \abs{  h^{(\kappa)}_{c, \infty} ( x_1, \dots, x_k) } \varphi_c(\xb)
\end{align}
with $ \varphi_c(\xb) := c\inv\! + \! (c + 1)\norm{\xb}_2^2 + \norm{\xb}_1 $ and 
\begin{align}\label{Def:hTildeInfty}%$
\widetilde{h}^{(\kappa)}_{c, \infty} ( x_1, \dots, x_k) := e^{i \pi c \kappa \sum_j x_j } h^{(\kappa)}_{c, \infty} ( x_1, \dots, x_k) 
\end{align}

Hence, to apply dominated convergence, it is important to study the integrability of $ \xb \mapsto \norm{\xb}_p^p \abs{  h^{(\kappa)}_{c, \infty} ( \xb ) } $ for $ p \in \ensemble{0, 1, 2} $ (with the convention that $ \norm{\xb}_0^0 := 1$). We will moreover have to study powers of $ h^{(\kappa)}_{c, \infty} $. 

\begin{lemma}[Integrability of the domination]\label{Lemma:IntegrabilityDomination}
Let $ M, M', \kappa, \kappa' \geq 1 $, $ K \geq 2 $ and $ p \in \ensemble{0, 1, 2} $. Define moreover
\begin{align*}%$
A & := \int_{\Rr^K} \abs{  h^{(\kappa)}_{c, \infty} ( x_1, \dots, x_K) }^M \norm{\xb}_p^p \Delta(x_1, \dots, x_K)^2 \, dx_1\dots dx_K  \\
B & := \int_{\Rr^K} \abs{  h^{(\kappa)}_{c, \infty} ( x_1, \dots, x_K) }^M \abs{  h^{(\kappa')}_{c', \infty} ( x_1, \dots, x_K) }^{M'} \norm{\xb}_p^p \Delta(x_1, \dots, x_K)^2 \, dx_1\dots dx_K 
\end{align*}

Then, we have
\begin{align}\label{Ineq:IntegrabilityDomination}%$
\begin{aligned}
A & < \infty \qquad \Longleftrightarrow \qquad K( \kappa M - K) > p + M \\
B & < \infty \qquad \Longleftrightarrow \qquad K (\kappa M + \kappa' M' - K) > p + M + M' 
\end{aligned}
\end{align}

\end{lemma}

% ==

\begin{proof}
Using \eqref{Def:HKappaCInfty}, we have
\begin{align*}%$
\abs{  h^{(\kappa)}_{c, \infty}( x_1, \dots, x_k) }  & = c^{ k\kappa} \abs{ \int_\Rr   e^{   i \pi c (k\kappa - 2) y } \prod_{j = 1}^k  \sinc\prth{   \pi c (y + x_j) }^\kappa dy } \\
                  & \leq c^{ k\kappa} \int_\Rr \prod_{j = 1}^k \abs{ \sinc\prth{ \pi c (y + x_j) } }^\kappa dy  = \int_\Rr \prod_{j = 1}^k \abs{ \frac{\sin \prth{ \pi c (y + x_j) }}{ \pi (y + x_j)   } }^\kappa dy
\end{align*}
thus 
\begin{align*}%$
A  \leq \int_{ \Rr^{K + M} } \prod_{j = 1}^K \prod_{r = 1}^M  \abs{ \sinc\prth{ \pi c (y_r + x_j) } }^\kappa \Delta(\xb)^2  \norm{\xb}_p^p d\xb d\yb
\end{align*}

The function $ (\xb, \yb) \mapsto \prod_{j = 1}^K \prod_{r = 1}^M  \abs{ \sinc\prth{ \pi c (y_r + x_j) } }^\kappa \Delta(\xb)^2  \norm{\xb}_p^p $ is continuous on any compact of $ \Rr^{K + M} $, hence is integrable on such a compact. The only possible problem is in a neighbourhood of infinity, and due to the decay of $ \sinc $, this will be integrable if and only if $ \kappa K M - K(K - 1) - p > K + M $ which amounts to $ K( \kappa M - K) > p + M $.

In the same vein, 
\begin{align*}%$
B \leq \int_{ \Rr^{K + M + M'} } \prod_{j = 1}^K \prod_{r = 1}^M  \abs{ \sinc\prth{ \pi c (y_r + x_j) } }^\kappa \prod_{r' = 1}^{M'}  \abs{ \sinc\prth{ \pi c (y'_{r'} + x_j) } }^{\kappa'} \Delta(\xb)^2  \norm{\xb}_p^p d\xb d\yb d\yb'
\end{align*}
and the underlying function will be integrable if and only if $ K (\kappa M + \kappa' M') - K(K - 1) - p > K + M + M' $ which amounts to $ K (\kappa M + \kappa' M' - K) > p + M + M' $.
\end{proof}

% ==
\subsection{Integrability in $c$}

The randomisation paradigm of section~\ref{Subsec:RandomisationParadigm} requires to study the dependency of $ h_{c, \infty}^{(\kappa) } $ in $c$.

\begin{lemma}[Domination in $ c $]\label{Lemma:IntegrabilityDominationInC}
The domination \eqref{Ineq:SpeedOfConvergenceHinfty}/\eqref{Ineq:SpeedOfConvergenceHinftyBis} is integrable in $ c \in [0, 1] $ (jointly with $ \xb \in \Rr^M \!$) if $ k \kappa \geq 2 $. The same holds with powers or products of $ h^{(\kappa)}_{c, \infty} $.
\end{lemma}

% ==

\begin{proof}
Since we carefully kept the dependency in $c$ in all previous estimates, we only need to check that the RHS in \eqref{Ineq:SpeedOfConvergenceHinfty}/\eqref{Ineq:SpeedOfConvergenceHinftyBis} is in $ L^1([0, 1]) $ in $c$. The theorem of joint integrability of functions $ (\xb, c) \mapsto f(\xb, c) $ in case of an $ L^1 $ domination in both variables will then allow to conclude (and since we already have the integrability in $ \xb $, we only need the one in $c$).

The triangle inequality in the probabilistic representation \eqref{Eq:Beta_ProbRepr} gives
\begin{align*}%$
\abs{  h^{(\kappa)}_{c, \infty} ( x_1, \dots, x_k) } = \abs{ \frac{c^{k\kappa - 1} }{ \Gamma\prth{ k\kappa  } } \Esp{ e^{ 2 i \pi c \sum_{j = 1}^k x_j ( \betab_{\kappa, 1}^{(j)} - \frac{\kappa}{2})  } \bigg\vert \sum_{j = 1}^k \betab_{\kappa, 1}^{(j)}  = 1 } } \leq  \frac{c^{k\kappa - 1} }{ \Gamma\prth{ k\kappa  } }
\end{align*}
and multiplying this last function by $ c\inv + c + 1 $ gives an integrable function if $ k\kappa \geq 2 $. Moreover, taking powers of $ h^{(\kappa)}_{c, \infty} $ or considering a product of $ h^{(\kappa)}_{c \alpha_j, \infty} $ for some $ \alpha_j \in (0, 1] $ does not change anything to the integrability (and in fact weakens the condition of integrability since a power $ M $ gives the condition $ (k\kappa - 1)M - 1 \geq 0 $ equivalent to $ k \kappa \geq 1 + M\inv $, weaker than $ k \kappa \geq 2 $).
\end{proof}

% ==
\subsection{The supersymmetric case} 

The supersymmetric specialisation of the homogeneous complete symmetric function $ h_N $ is defined by \cite[I-3 ex. 23 p. 58]{MacDo}
\begin{align}\label{Def:SupersymHomogeneousComplete}%$
h_N\crochet{X - Y} := \crochet{t^N} H\crochet{t (X - Y) }
\end{align}

Since $ X - X = \emptyset $, we can suppose that $ x_i \neq y_j $ for all $ i, j $.

\begin{lemma}[Rescaled supersymmetric Chebychev polynomials in the microscopic regime]\label{Lemma:RescaledSupersym} 
Let $ k, m \geq 2 $ and $ X, Y $ be alphabets such that $ x_i \neq y_j $ for all $ i, j $. Suppose moreover that $ k - m \geq 2 $. Then, we have, locally uniformly in $ X + Y \in \Rr^{k + m} $ and for all $ c \in \Rr_+^* $
\begin{align*}%$
\frac{ 1}{ N^{ k - m - 1 } } h_{ \pe{c N} }\crochet{ e^{ 2 i \pi X/N } - e^{ 2 i \pi Y/N } } \tendvers{N}{+\infty} e^{i \pi c \sum_j x_j} h_{c, \infty}\crochet{X - Y}
\end{align*}
with 
\begin{align}\label{Def:SupersymHinfty}%$
h_{c, \infty}\crochet{X - Y} :=  ( 2 i \pi)^m c^k  \int_\Rr e^{ i \pi c (k - 2) \theta }  \prod_{j = 1}^k  \sinc\prth{ \pi c (x_j + \theta) } \prod_{\ell = 1}^m (\theta + y_j) d\theta
\end{align}
\end{lemma}

\begin{remark}
If $ k - m = 1 $, the convergence still holds weakly, and the limiting integral is semi-convergent. We will see another convergence in lemma \ref{Lemma:RescaledSupersym2}.
\end{remark}

\begin{proof}
Using \ref{Lemma:ShallowCrucialLemma}, we have
\begin{align*}%$
h_N\crochet{X - Y} := \crochet{t^N} H\crochet{t (X - Y) - t^{N + 1} X^{N + 1} }
\end{align*}
hence
\begin{align*}%$
h_{ \pe{c N} } \crochet{ e^{ 2 i \pi X/N } - e^{ 2 i \pi Y/N } }  & =  \crochet{t^{ \pe{c N} } } H\crochet{ t e^{ 2 i \pi X/N } - t^{\pe{c N} + 1} e^{ 2 i \pi X(\pe{c N} + 1) /N }  - t Y } \\
            & = \int_{  - \frac{N}{2} }^{ \frac{N}{2} } e^{ - 2 i \pi \frac{\pe{cN} }{N} \theta } \prod_{j = 1}^k \frac{ e^{ i \pi \frac{\pe{c N} + 1}{N} (x_j + \theta) } \sin\prth{ \pi \frac{\pe{c N} + 1}{N} (x_j + \theta) } }{ e^{ i \pi \frac{x_j + \theta}{N}  } \sin\prth{ \pi \frac{x_j + \theta }{N} } } \\
            & \hspace{+7cm } \prod_{\ell = 1}^m \prth{ 1 - e^{ 2 i \pi \frac{y_\ell + \theta}{N} } } \frac{d\theta}{N} \\
            & \equivalent{N\to +\infty} N^{k - m - 1} e^{i \pi c \sum_j x_j} \int_\Rr e^{ i \pi c (k - 2) \theta }  \prod_{j = 1}^k  c \sinc\prth{ \pi c (x_j + \theta) } \\
            & \hspace{+8cm } \prod_{\ell = 1}^m 2 i \pi (\theta + y_\ell) d\theta
\end{align*}
where we have used dominated convergence, due to the continuity and the decay of the functions (the integrability comes from the fact that $ k - m \geq 2 $).
\end{proof}

One can give another expression of $ h_{c, \infty}\crochet{X - Y} $. Suppose first that all $ x_i, y_j $ are distinct. Using the partial fraction decomposition of $ t \mapsto H\crochet{t (X - Y) } $, namely
\begin{align}\label{Eq:PartialFracExpH}%$
H\crochet{t (X - Y) } =  \sum_{j = 1}^k \frac{1}{1 - t x_j} \frac{ \prod_{r = 1}^m (1 - y_r x_j\inv)  }{ \prod_{1 \leq i \leq k, i \neq j }  (1 - x_i x_j\inv) }
\end{align}
we get
\begin{align}\label{Eq:SupersymHnWithResidues}%$
h_N\crochet{X - Y} = \sum_{j = 1}^k  x_j^N \frac{ \prod_{r = 1}^m (1 - y_r x_j\inv) }{ \prod_{1 \leq i \leq k, i \neq j } (1 - x_i x_j\inv) }
\end{align}

This implies the following:

\begin{lemma}[Rescaled supersymmetric Chebychev polynomials in the microscopic regime]\label{Lemma:RescaledSupersym2} 
Let $ k, m \geq 2 $ such that $ k - m \geq 0 $ and $ X, Y $ be alphabets such that $ X + Y $ is composed with distinct elements. We have, locally uniformly in $ X + Y \in \Rr^{k + m} $ and for all $ c \in \Rr_+^* $
\begin{align*}%$
\frac{ 1}{ N^{ k - m - 1 } } h_{ \crochet{c N} }\crochet{ e^{ 2 i \pi X/N } - e^{ 2 i \pi Y/N } } \tendvers{N}{+\infty} (2 i \pi)^{ k - m - 1 } \sum_{j = 1}^k e^{2 i \pi c x_j} \frac{ \prod_{r = 1}^m ( y_r - x_j ) }{ \prod_{1 \leq i \leq k, i \neq j } ( x_i - x_j ) }
\end{align*}
namely 
\begin{align*}%$
h_{c, \infty}\crochet{X - Y} = (2 i \pi)^{ k - m - 1 } e^{- i \pi c \sum_j x_j} \sum_{j = 1}^k e^{2 i \pi c x_j} \frac{ \prod_{r = 1}^m ( y_r - x_j ) }{ \prod_{1 \leq i \leq k, i \neq j } ( x_i - x_j ) }
\end{align*}
\end{lemma}

% ==

\begin{proof}
This is a straightforward rescaling of \eqref{Eq:SupersymHnWithResidues}. Details are left to the reader.
\end{proof}

\begin{remark}
We see that with the constraint of all elements of $ X + Y $ different, we do not need to suppose $ k - m \geq 2 $. The drawback is that one cannot see easily if the limiting function in lemma \ref{Lemma:RescaledSupersym2} is integrable in $X$. Nevertheless, one has also 
\begin{align*}%$
e^{ i \pi c \sum_j x_j} h_{c, \infty}\crochet{X - Y} & = (- 2 i \pi)^{ k - m - 1 } \sum_{j = 1}^k e^{2 i \pi c x_j} x_j^{m - k + 1} \frac{ \prod_{r = 1}^m ( 1 - y_r  x_j\inv ) }{ \prod_{1 \leq i \leq k, i \neq j } ( 1 - x_i x_j\inv ) } \\
              & = (- 2 i \pi)^{ k - m - 1 } \sum_{ \ell \geq 0 } \frac{ (2 i \pi c)^\ell }{\ell ! } h_{ \ell + m - k + 1 } \crochet{X - Y} \\
              & = (- 2 i \pi)^{ k - m - 1 } \oint_{ r \Uu } e^{ 2 i \pi c z\inv } z^{ m - k + 1 } H\crochet{z (X - Y) } \frac{d^*z}{z}, \quad r < \min_j\ensemble{ \abs{x_j}\inv }
\end{align*}

Note that this last expression is defined for possibly equal $ x_i $'s (but with $ x_i \neq y_j $ for all $i, j$). We thus have
\begin{align}\label{Eq:HcWithResidues}%$
h_{c, \infty}\crochet{X - Y} = (- 2 i \pi)^{ k - m - 1 } e^{ - i \pi c \sum_j x_j} \crochet{ z^{ k - m - 1 }  } e^{ 2 i \pi c z\inv } H\crochet{z (X - Y) }
\end{align}
and this last expression is clearly seen integrable in $ X $ if $ k - m \geq 2 $. 
\end{remark}

In the supersymmetric case, we set
\begin{align*}%$
\widetilde{h}_{c, \infty}\crochet{X - Y}  := e^{i \pi c \sum_j x_j} h_{c, \infty}\crochet{X - Y}
\end{align*}

The equivalent of the domination \eqref{Ineq:SpeedOfConvergenceHinftyBis} or \eqref{Ineq:SpeedOfConvergenceHinfty} is given by:

\begin{lemma}[Domination and speed of convergence, supersymmetric case]\label{Lemma:DominationSupersym}
Let $ X, Y $ be alphabets satisfying the hypotheses of lemma \ref{Lemma:RescaledSupersym} and recall that $ \norm{X + Y}_1 := \sum_{j = 1}^k \abs{x_j} + \sum_{\ell = 1}^m \abs{y_\ell} $. Then, 
\begin{align*}%$
N^{-k + m + 1} h_{\pe{c N} }\crochet{e^{2 i \pi \frac{X}{N} } - e^{2 i \pi \frac{Y}{N} } } = O\prth{ h_{c, \infty}\crochet{X - Y} }
\end{align*}
with a constant in the $ O(\cdot) $ independent of $ X, Y $, i.e. there exists $ C > 0 $ independent of $ X, Y $ such that
\begin{align*}%$
N^{-k + m + 1} \abs{ h_{\pe{c N} }\crochet{e^{2 i \pi \frac{X}{N} } - e^{2 i \pi \frac{Y}{N} } } } \leq C \abs{ h_{c, \infty}\crochet{X - Y} }
\end{align*}

Suppose moreover that $ k - m \geq 3 $. Then, 
\begin{align*}%$
N^{-k + m + 1} h_{\pe{c N} }\crochet{e^{2 i \pi \frac{X}{N} } - e^{2 i \pi \frac{Y}{N} } } - \widetilde{h}_{c, \infty}\crochet{X - Y} = O\prth{ \frac{ h_{c, \infty}\crochet{X - Y - \norm{X + Y}_1} }{N} }
\end{align*}
with a constant in the $ O(\cdot) $ independent of $ X, Y $, i.e. there exists $ C > 0 $ independent of $ X, Y $ such that
\begin{align*}%$
\abs{ N^{-k + m + 1} h_{\pe{c N} }\crochet{e^{2 i \pi \frac{X}{N} } - e^{2 i \pi \frac{Y}{N} }} - \widetilde{h}_{c, \infty}\crochet{X - Y} } \leq \frac{C}{N} \abs{ h_{c, \infty}\crochet{X - Y - \norm{X + Y}_1} }   
\end{align*}
\end{lemma}

%\begin{remark}
%%
%Recall that $ \norm{X + Y}_1 := \sum_{j = 1}^k \abs{x_j} + \sum_{\ell = 1}^m \abs{y_\ell} $. Then, 
%%
%%
%%
%\begin{align*}%$
%\abs{ h_{c, \infty}\crochet{X - Y - \norm{X + Y}_1 } } = \abs{ c^k (2\pi)^m \int_\Rr e^{i \pi c (k - 2) \theta } \prod_{j = 1}^m \sinc(\pi c (x_j + \theta) ) \prod_{\ell = 1}^m (\theta + y_\ell) \times (\theta +  \norm{X + Y}_1) d\theta    }
%\end{align*}
%%
%\end{remark}

% ==

\begin{proof}
Let $  U_N $ be a uniform random variable in $ \intcrochet{0, N} $. We have
\begin{align*}%$
H\crochet{t X - t^{N + 1} X^{N + 1} } = \prod_{j = 1}^k \sum_{\ell = 0}^N (tx_j)^\ell =  \prod_{j = 1}^k N\Esp{ (tx_j)^{U_N} } = N^k \Esp{ \prod_{j = 1}^k (t x_j)^{U_N^{(j)} } }
\end{align*}
where $ (U_N^{(j)})_{1 \leq j \leq k} $ is a sequence of i.i.d. random variables uniformly distributed in $ \intcrochet{0, N} $. 

Consider now a sequence $ (U^{(\ell + k)})_{1 \leq \ell \leq m} $ of i.i.d. random variables uniformly distributed in $ \crochet{0, 1} $ and note that for $ U $ such a random variable
\begin{align*}%$
1 - e^{i a } = -i a \int_0^1 e^{i u } du = - ia \Esp{e^{i a U} }
\end{align*}

Then, one has
\begin{align*}%$
h_{\pe{cN} }\crochet{ e^{2 i \pi \frac{X}{N} } - e^{2 i \pi \frac{Y}{N} } }  & = \int_{-\frac{1}{2} }^{ \frac{1}{2} } e^{ - 2 i \pi \alpha \pe{cN} }   N^k \Esp{ e^{ 2 i \pi \sum_{j = 1}^k \prth{ \alpha + \frac{x_j}{N} } U_N^{(j)}  }   }  \\
                   & \hspace{+4.5cm } \times \prod_{\ell = 1}^m (-2 i \pi) \prth{ \alpha + \frac{y_\ell}{N} } \Esp{ e^{ 2 i \pi \, \prth{ \alpha + \frac{y_\ell}{N} } U^{(k + \ell)} } } d\alpha \\
                   & = N^{k - m - 1} (-2 i \pi )^m  \int_{-\frac{N}{2} }^{ \frac{N}{2} } e^{ - 2 i \pi \theta \frac{\pe{cN}}{N } } \Esp{ e^{ 2 i \pi \sum_{j = 1}^k \frac{\theta + x_j}{N}  U_N^{(j)}  }   } \prod_{\ell = 1}^m  \prth{ \theta + y_\ell } \\
                   & \hspace{+7.8cm } \times  \Esp{ e^{ 2 i \pi \sum_{\ell = 1}^m \frac{\theta + y_\ell}{N} U^{(k + \ell)} } } d\theta 
\end{align*}

Recall that $ \frac{\pe{cN}}{N} = c - \frac{\pf{cN} }{N} = c - O(N\inv) $ since $ \pf{x}\in [0, 1) $ for all $ x \in \Rr $. Moreover, we have the classical coupling between $ U_N $ and $ U := \lim_{N\to+\infty} \frac{U_N}{N} $ (limit in law), namely
\begin{align}\label{Eq:CouplingUniform}%$
U_N \eqlaw \pe{ N U }, \qquad U \sim \Us([0, 1]) %\mbox{$U$ uniformly distributed in $ \crochet{0, 1} $}
\end{align}

Thus, defining $ U_N^{(j)} $ by $ \pe{N U^{(j)} } $ for a sequence of i.i.d. uniform random variables in $ \crochet{0, 1} $ and using the fact that the fractional part is uniformly bounded by 1 (a deterministic uniform constant), we get
\begin{align*}%$
F_N(\theta ; X, Y, c) & := e^{ - 2 i \pi \theta \frac{\pe{cN}}{N } } \Esp{ e^{ 2 i \pi \sum_{j = 1}^k (\theta + x_j) \frac{U_N^{(j)}}{N}   + 2 i \pi \sum_{\ell = 1}^m \frac{\theta + y_\ell}{N} U^{(k + \ell)} } }  \\
                   & = \Esp{ e^{ 2 i \pi \, \prth{  - c \theta +  \sum_{j = 1}^k (\theta + x_j) U^{(j)} +  \frac{1}{N} O\,\prth{ 1 + \sum_{\ell = 1}^m (\theta + y_\ell) +  \sum_{j = 1}^k (\theta + x_j)   } } } }
\end{align*}

Let us prove the first statement: for $ \theta \in \crochet{ - \frac{N}{2}, \frac{N}{2} } $, one has $ \abs{\frac{\theta}{N}} \leq \frac{1}{2} $ i.e. $ \frac{\theta}{N} = O(1) $, thus
\begin{align*}%$
F_N(\theta ; X, Y, c) = O\prth{  \Esp{ e^{ 2 i \pi \, \prth{  - c \theta +  \sum_{j = 1}^k (\theta + x_j) U^{(j)}  } } } } e^{ \frac{1}{N} O\,\prth{ 1 + \sum_{\ell = 1}^m  y_\ell  +  \sum_{j = 1}^k  x_j  } }
\end{align*}
and, by integrating on $ \crochet{ - \frac{N}{2}, \frac{N}{2}} $
\begin{align*}%$
h_{\pe{cN} }\crochet{ e^{2 i \pi \frac{X}{N} } - e^{2 i \pi \frac{Y}{N} } }   =    O\prth{ \int_{-\frac{N}{2} }^{ \frac{N}{2} } \Esp{ e^{ 2 i \pi \, \prth{  - c \theta +  \sum_{j = 1}^k (\theta + x_j) U^{(j)}  } } } \prod_{\ell = 1}^m (\theta + y_\ell) d\theta }  
\end{align*}

Last, using the fact that $ \Un_{ \crochet{ - \frac{N}{2}, \frac{N}{2}} } = \Un_{\Rr_+} + O(N^{-r}) $ for all $ r \geq 1 $, we get
\begin{align}\label{Ineq:DominationSupersymWithoutSpeed}%$
\frac{1}{N^{k - m + 1}} h_{\pe{cN}}\crochet{ e^{2 i \pi \frac{X}{N} } - e^{2 i \pi \frac{Y}{N} } } = O\prth{ h_{c, \infty }\crochet{ X - Y }  }  
\end{align}
which is the first inequality. This domination allows to use dominated convergence if the limiting function is proven integrable. 

Let us now prove the second statement. We have
\begin{align*}%$
F_N(\theta ; X, Y, c)  & = \Esp{ e^{ 2 i \pi \, \prth{  - c \theta +  \sum_{j = 1}^k (\theta + x_j) U^{(j)}  } } } e^{ \frac{1}{N} O\,\prth{ 1 + \sum_{\ell = 1}^m (\theta + y_\ell) +  \sum_{j = 1}^k (\theta + x_j)   }} \\
                & = \Esp{ e^{ 2 i \pi \, \prth{  - c \theta +  \sum_{j = 1}^k (\theta + x_j) U^{(j)}  } } } \prth{1 + \frac{1}{N} O \prth{ \theta + \sum_{\ell = 1}^m \abs{y_\ell} + \sum_{j = 1}^k \abs{x_j} }}
\end{align*}

Integrating on $ \Rr $ against $ \Un_{ \crochet{ - \frac{N}{2}, \frac{N}{2}} } = \Un_{\Rr_+} + O(N^{-r}) $ for all $ r \geq 1 $, one thus gets
\begin{align}\label{Ineq:DominationSupersymWithSpeed}%$
\frac{1}{N^{k - m + 1}} h_{\pe{cN}}\crochet{ e^{2 i \pi \frac{X}{N} } - e^{2 i \pi \frac{Y}{N} } } - \widetilde{h}_{c, \infty }\crochet{ X - Y }= O\prth{ \frac{h_{c, \infty }\crochet{ X - Y - \norm{X + Y}_1 } }{N} }  
\end{align}
where the constant in the $ O(\cdot) $ is uniform in $ X, Y $, i.e. there exists $ C > 0 $ s.t. 
\begin{align}\label{Ineq:DominationSupersymWithSpeedBis}%$
\abs{ \frac{1}{N^{k - m + 1}} h_{\pe{cN}}\crochet{ e^{2 i \pi \frac{X}{N} } - e^{2 i \pi \frac{Y}{N} } } - \widetilde{h}_{c, \infty }\crochet{ X - Y } } \leq \frac{C }{N} \abs{ h_{c, \infty }\crochet{ X - Y - \norm{X + Y}_1 }  }  
\end{align}
which concludes the proof.
\end{proof}

% ==

\begin{remark}
A probabilistic representation in the same vein as \eqref{Eq:NegBinRepresentation_h_n} can be given for $ h_N\crochet{X - Y} $: Let $ 0 < r < \min\!\ensemble{1, \max_{1\leq \ell \leq k} \abs{x_\ell}\inv} $, let $ (B_j(r))_{j \geq 1} $ a sequence of i.i.d. $ \ensemble{0, 1} $-Bernoulli random variables of parameter $ \frac{r}{1 + r} $ and $ (G_\ell(r))_{\ell \geq 1} $ a sequence of i.i.d. Geometric random variables of parameter $r$ (i.e. negative binomial random variables with $ \kappa = 1 $), both sequences being independent. Then,
\begin{align*}%$
h_N\crochet{X - Y} & = \crochet{t^N}H\crochet{t(X - Y)} = r^{-N} \oint_{\Uu} z^{-N} \prod_{j = 1}^m (1 - r z y_j) \prod_{\ell = 1}^k \frac{1}{1 - r z  x_\ell } \frac{d^*z}{z} \\
               & = r^{-N} \frac{ (1 + r)^m}{(1-r)^k } \oint_{\Uu} z^{-N} \prod_{j = 1}^m \frac{1 - r z y_j}{1 + r } \prod_{\ell = 1}^k \frac{1 - r}{1 - r z  x_\ell } \frac{d^*z}{z} \\
               & = r^{-N} \frac{ (1 + r)^m}{(1-r)^k } \oint_{\Uu} z^{-N} \, \Esp{ \prod_{j = 1}^m (-z y_j)^{B_j(r) } \prod_{\ell = 1}^k (z x_\ell)^{G_\ell(r)} } \frac{d^*z}{z} \\
               & = r^{-N} \frac{ (1 + r)^m}{(1-r)^k } \, \Esp{ \prod_{j = 1}^m (- y_j)^{B_j(r) } \prod_{\ell = 1}^k  x_\ell^{G_\ell(r)} \Unens{ \sum_{j = 1}^m B_j(r) + \sum_{\ell = 1}^k G_\ell(r) = N } } \\
               & = A_{m, k}(N) \, \Esp{ \prod_{j = 1}^m (- y_j)^{B_j(r) } \prod_{\ell = 1}^k  x_\ell^{G_\ell(r)} \Bigg\vert \sum_{j = 1}^m B_j(r) + \sum_{\ell = 1}^k G_\ell(r) = N  }
\end{align*}
with $ A_{m, k}(N) := \crochet{t^N} \frac{(1 + t)^m }{ (1 - t)^k } = h_N\crochet{ 1^{\plusInOne k} - \varepsilon 1^{\plusInOne m} } $.

Nevertheless, this expression does not seem suitable for asymptotic analysis.

\end{remark}

% ==
$ $

We conclude this part with the integrability of $ U \mapsto h_{c, \infty} \crochet{ U \tensorsum (X - Y) } $ with $ U := \ensemble{ u_1, \dots, u_K} $, $ X := \ensemble{ x_1, \dots, x_R} $ and $ Y := \ensemble{ y_1, \dots, y_L } $ as this will be the relevant function at stake in theorems \ref{Theorem:Ratios} and \ref{Theorem:Derivatives}.

\begin{lemma}[Integrability of the domination, supersymmetric case]\label{Lemma:IntegrabilitySupersym}
Let $ M \geq 1 $, $ R, K \geq 1 $ and $ L \geq 0 $ with $ K(R - L) \geq 2 $. Define
\begin{align*}%$
A' & := \int_{\Rr^K} \abs{  h_{c, \infty} \crochet{ U\tensorsum (X - Y)} }^M \Delta(u_1, \dots, u_K)^2 \, du_1\dots du_K  \\
B' & := \int_{\Rr^K} \abs{  h_{c, \infty} \crochet{ U\tensorsum (X - Y - \norm{X +Y}_1)} }^M \Delta(u_1, \dots, u_K)^2 \, du_1\dots du_K 
\end{align*}

Then, 
\begin{align}\label{Ineq:IntegrabilitySupersymDomination}%$
\begin{aligned}
& A' < \infty \qquad \Longleftrightarrow \qquad K( M (R - L) - K ) >  M  \\  
& B' < \infty \qquad \Longleftrightarrow \qquad K( M (R - L - 1) - K ) >  M
\end{aligned}
\end{align}
\end{lemma}

% ==

\begin{proof}
Define $ [n] := \intcrochet{1, n} $ for all $ n \in \Nn^* $. Using \eqref{Def:SupersymHinfty}, we have
\begin{align*}%$
\abs{  h_{c, \infty}\crochet{U \tensorsum (X - Y)} }  & = c^{K R} (2\pi)^{K L} \abs{ \int_\Rr   e^{   i \pi c (k  - 2) t } \!\!\!\prod_{j \in [K], r \in [R]} \!\!\!  \sinc\prth{   \pi c (t + u_j + x_r) } \!\!\!\prod_{j \in [K], \ell \in [L]}\!\!\! (t + u_j + y_\ell) dt } \\
                  & \leq c^{K R} (2\pi)^{K L} \int_\Rr   \prod_{j \in [K], r \in [R]}  \abs{ \sinc\prth{   \pi c (t + u_j + x_r) } } \prod_{j \in [K], \ell \in [L]} \abs{ t + u_j + y_\ell  } dt
\end{align*}
thus, setting $ A'' := \frac{A'}{ c^{K R M} (2\pi)^{K L M}  } $, one gets
\begin{align*}%$
A''  \leq \int_{ \Rr^{K + M} } \prod_{j \in [K], r \in [R], m \in [M] }  \abs{ \sinc\prth{   \pi c (t_m + u_j + x_r) } } \prod_{j \in [K], \ell \in [L], m \in [M]} \abs{ t_m + u_j + y_\ell  } \Delta(U)^2 dU dT
\end{align*}

The function 
\begin{align*}%$
(U, T) \mapsto \prod_{j \in [K], r \in [R], m \in [M] }  \abs{ \sinc\prth{   \pi c (t_m + u_j + x_r) } } \prod_{j \in [K], \ell \in [L], m \in [M]} \abs{ t_m + u_j + y_\ell  } \Delta(U)^2 
\end{align*}
is continuous on any compact of $ \Rr^{K + M} $, hence is integrable on such a compact. The only possible problem is in a neighbourhood of infinity, and due to the decay of $ \sinc $, this will be integrable if and only if $  K R M - K L M - K(K - 1) > K + M $ which amounts to $ K( R M - L M - K ) >  M $.

The second assertion amounts to use the first with the alphabet $ Y' := Y + \norm{X + Y}_1 $ which has exactly $ L + 1 $ elements. Hence the result.
\end{proof}

\begin{remark}\label{Rk:IntegrabilityDominationSupersym}
In the case of a Vandermonde determinant $ \Delta\crochet{0 + U} $, i.e.
\begin{align*}%$
A''  & := \int_{\Rr^K} \abs{  h_{c, \infty} \crochet{ U\tensorsum (X - Y)} }^M \Delta(u_1, \dots, u_K, 0)^2 \, du_1\dots du_K  \\
             & =  \int_{\Rr^K} \abs{  h_{c, \infty} \crochet{ U\tensorsum (X - Y)} }^M \prod_{j = 1}^k u_j^2 \, \Delta(u_1, \dots, u_K )^2 \, du_1\dots du_K
\end{align*}
one has to add $ 2 $ to $ L $ as there is an additional polynomial of degree $ 2K $. The criteria \eqref{Ineq:IntegrabilitySupersymDomination} thus becomes
\begin{align}\label{Ineq:IntegrabilitySupersymDominationBis}%$
A'' < \infty \qquad \Longleftrightarrow \qquad K( M (R - L - 2) - K ) >  M
\end{align}
\end{remark}

$ $

$ $

% ==
\section{Truncated $ \zeta $ function in the microscopic scaling}\label{Sec:NumberTheory}

% ==
\subsection{Reminders}

Let $ (Z_p)_{p \in \Pe} $ be a sequence of i.i.d random variables uniformly distributed on the unit circle $ \Uu $. Recall that for $ X \geq 1 $ and $ \sigma \in \Rr_+ $, we have defined
\begin{align*}%$
\Ze_{X, \sigma} := \sum_{k = 1}^X \frac{1}{k^\sigma} \prod_{p \in \Pe } Z_p^{v_p(k)}
\end{align*}

The moments of this random variable have been successively investigated by Conrey-Gamburd for the case $ \sigma = \frac{1}{2} $ (see \cite{ConreyGamburd}), by  Harper-Nikeghbali-Radziwi\l\l~ for the case $ \sigma = 0 $ (see \cite{HarperNikeghbaliRadziwill}) and by Heap-Lindqvist for the general case (see \cite{HeapLindqvist}). The case $ \sigma > 1/2 $ gives a convergence in $ L^2 $ for $ \Ze_{X, \sigma} $, a particular case of the Bohr-Jessen theorem (see \cite{BohrJessen}) in the case of an additional truncation ; it is interesting to remark that in this case,
\begin{align*}%$
\Ze_{X, \sigma} \cvlaw{X}{+\infty }  \Ze_{\infty, \sigma} = \sum_{n \geq 1} \frac{1}{k^\sigma} \prod_{p \in \Pe } Z_p^{v_p(k)} = \prod_{p \in \Pe} \frac{1}{1 - p^{-\sigma} Z_p } 
\end{align*}
and that 
\begin{align*}%$
\Esp{ \abs{ \Ze_{\infty, \sigma} }^{2k} } & =  \Esp{ \abs{ \prod_{p \in \Pe} \frac{1}{1 - p^{-\sigma} Z_p } }^{2k} } = \prod_{p \in \Pe} \Esp{ \abs{ 1 - p^{-\sigma} Z_p  }^{-2k} } \\
                 & = e^{ \zeta_\Pe(\sigma) k^2 } \prod_{p \in \Pe} e^{ - \frac{k^2}{p^{2\sigma}} } \Esp{ \abs{ 1 - p^{-\sigma} Z_p  }^{-2k} }, \qquad \zeta_\Pe(\sigma) := \sum_{p \in \Pe} p^{-2\sigma}.
\end{align*}

The arithmetic factor $ a_k $ defined in \eqref{Def:ArithmeticFactor} is thus given by the previous product for $ \sigma = 1/2 $.

% ==
\subsection{A phase transition}

In view of \eqref{Eq:BehaviourRandomMultiplicativeFunction}, a natural question concerns the phase transition between the log-normal regime $ \sigma = \frac{1}{2} $ and the non-Gaussian regime $ \sigma < 1/2 $. Define the random variable 
\begin{align*}%$
\Ze^*_X(\lambda) := \Ze_{X, \frac{1}{2} - \frac{\lambda}{\log(X)} } := \sum_{k = 1}^X \frac{1}{k^{ \frac{1}{2} - \frac{\lambda}{\log(X)} } } \prod_{p \in \Pe } Z_p^{v_p(k)} 
\end{align*}

\begin{theorem}[Moments of $ \Ze^*_X(\lambda) $]\label{Theorem:PhaseTransition}
For all $ \lambda \geq 0 $, we have when $ X \to +\infty $
\begin{align*}%$
\Esp{ \abs{\Ze_X^*(\lambda) }^{2k} } \sim a_k \gamma_k \lambda^{k^2} e^{2k \lambda } (\log X)^{ k^2 }
\end{align*}
\end{theorem}

% ==

\begin{proof}
Using the Perron inversion formula for $ y \notin \Nn $ as in \cite{ConreyGamburd, HeapLindqvist}
\begin{align}\label{Eq:IntegralRepresentationIndicator}%$
\Unens{y > 1 } = \int_{b + i \Rr} y^s \frac{d^*s}{s}, \qquad b > 0
\end{align}
we get
\begin{align*}%$
\Ze_{X, \sigma} & : = \sum_{1 \leq m \leq X} \frac{1}{m^\sigma} \prod_{ p \in \Pe } Z_p^{v_p(m)}  = \int_{b + i \Rr} \sum_{m \geq 1} \prth{ \frac{X}{m } }^s \frac{1}{m^\sigma} \prod_{ p \in \Pe } Z_p^{v_p(m)} \frac{d^*s}{s}   \\
               & = \int_{b + i \Rr} \prod_{ p \in \Pe } \frac{1}{1 - \frac{Z_p}{p^{\sigma + s} } } X^s \frac{d^*s}{s} 
\end{align*}

This implies for $ b_j > 0 $
\begin{align*}%$
\Esp{ \abs{\Ze_{X, \sigma} }^{2k} } =  \int_{ \prod_{j = 1}^{2k} (b_j + i \Rr) }  \prod_{ p \in \Pe } \Esp{ \prod_{j = 1}^k \frac{1}{1 - \frac{Z_p}{p^{\sigma + s_j} } } \frac{1}{1 - \frac{\overline{Z_p}}{p^{\sigma + \overline{s_{j + k} } } } } }  \, X^{\sum_{j = 1}^k ( s_j + \overline{s_{j + k} } ) } \frac{d^*\sbb}{\sbb} 
\end{align*}
with $ \sbb := (s_1, \dots, s_{2k}) $. Now, changing $ s_j $ with $ \overline{s_j} $ does not change the integral, hence, we get
\begin{align*}%$
\boxed{ \Esp{ \abs{\Ze_{X, \sigma} }^{2k} } =  \int_{ \prod_{j = 1}^{2k} (b_j + i \Rr) }  \prod_{ p \in \Pe } \Esp{ \prod_{j = 1}^k \frac{1}{1 - \frac{Z_p}{p^{\sigma + s_j} } } \frac{1}{1 - \frac{\overline{Z_p}}{p^{\sigma + s_{j + k}  } } } }  \, X^{\sum_{j = 1}^{2k}  s_j  } \frac{d^*\sbb}{\sbb} }
\end{align*}

We now adapt the arguments of \cite{ConreyGamburd, HeapLindqvist}. Define
\begin{align*}%$
A_{k, \sigma}(\sbb) & :=  \prod_{ p \in \Pe } \Esp{ \prod_{j = 1}^k \frac{1}{1 - \frac{Z_p}{p^{\sigma + s_j} } } \frac{1}{1 - \frac{\overline{Z_p}}{p^{\sigma + s_{j + k}  } } }   } \prod_{j, \ell = 1}^k \prth{1 - \frac{1}{p^{ 2\sigma + s_j + s_{\ell + k} } } } \\
B_{k, \sigma}(\sbb) & :=  A_{k, \sigma}(\sbb)\prod_{j, \ell = 1}^k (2\sigma - 1 + s_j + s_{\ell + k} ) \prod_{ p \in \Pe }   \frac{1}{1 - \frac{1}{p^{2\sigma + s_j + s_{\ell + k} } } }
\end{align*}
so that 
\begin{align*}%$
\Esp{ \abs{\Ze_{X, \sigma} }^{2k} } =  \int_{ \prod_{j = 1}^{2k} (b_j + i \Rr) } B_{k, \sigma}(\sbb)  \prod_{j, \ell = 1}^k \frac{1}{2\sigma - 1  + s_j + s_{\ell + k}}  \, X^{\sum_{j = 1}^{2k}  s_j  } \frac{d^*\sbb}{\sbb}
\end{align*}

Using the Euler product of the Zeta function, classical arguments show that $ A_{k, \sigma} $ is an absolutely convergent product if $ \Re(s_i + s_j) > \sigma $ and that it is holomorphic in a neighbourhood of $ \sbb = - \sigma \Un_k $.

We now change variable, setting $ s'_j = \log(X) s_j $, $ b'_j = \log(X) b_j $ and, using $ 2 \sigma_X - 1 = - \frac{2\lambda }{\log(X) } $, we get 
\begin{align*}%$
\Esp{ \abs{\Ze_{X, \sigma_X} }^{2k} }  & =  \int_{ \prod_{j = 1}^{2k} (b_j + i \Rr) } B_{k, \sigma_X }(\sbb)  \prod_{j, \ell = 1}^k \frac{1}{2\sigma_X - 1  + s_j + s_{\ell + k}}  \, e^{ \log(X) \sum_{j = 1}^{2k}  s_j  } \frac{d^*\sbb}{\sbb} \\
               & = \int_{ \prod_{j = 1}^{2k} (b'_j + i \Rr) } B_{k, \sigma_X }\prth{ \frac{ \sbb' }{\log(X) } } \prod_{j, \ell = 1}^k \frac{ \log(X) }{ -2 \lambda  + s'_j + s'_{\ell + k}}  \, e^{  \sum_{j = 1}^{2k}  s'_j  } \frac{d^*\sbb'}{\sbb'}
\end{align*}

Using holomorphicity, we can translate the contours up to $ c_j $ (independent of $X$). Truncating the integrals up the height $ T = \sqrt{ \log(X) } $ and taking a Taylor approximation of $ B_{k, \sigma_X }\prth{ \frac{ \sbb }{\log(X) } } $ around $ (0, \dots, 0) $ ($k$ times), we get, remarking that $ \prod_{p \in \Pe} (1 - p^{-s})\inv = \zeta(s) $
\begin{align*}%$
B_{k, \sigma_X }\prth{ \frac{ \sbb }{\log(X) } } & = A_{k, \sigma_X}\prth{ \frac{ \sbb }{\log(X) } } \prod_{j, \ell = 1}^k \prth{ 2\sigma_X - 1 + \frac{ s_j + s_{\ell + k} }{\log(X) } }   \zeta\prth{2\sigma_X + \frac{s_j + s_{\ell + k}}{\log(X) } } \\
               & = A_{k, \sigma_X}\prth{ \frac{ \sbb }{\log(X) } } \prod_{j, \ell = 1}^k   \frac{   s_j + s_{\ell + k} - 2 \lambda }{ \log(X) }    \zeta\prth{1 +  \frac{   s_j + s_{\ell + k} - 2 \lambda }{ \log(X) }     }\\
               & \equivalent{X \to +\infty } A_{k, \sigma_X}\prth{ \frac{ \sbb }{\log(X) } }
\end{align*}
as $ \alpha \zeta(1 + \alpha) \to 1 $ when $ \alpha \to 0 $. %\ybcomm{Reference ? }

Last, using dominated convergence and continuity
\begin{align*}%$
A_{k, \sigma_X}\prth{ \frac{ \sbb }{\log(X) } } & = \prod_{ p \in \Pe } \Esp{ \prod_{j = 1}^k \frac{1}{1 - \frac{Z_p}{p^{\sigma_X + s_j / \log(X) } } } \frac{1}{1 - \frac{\overline{Z_p}}{p^{\sigma_X + s_{j + k} / \log(X)  } } }   } \prod_{j, \ell = 1}^k \prth{1 - \frac{1}{p^{ 2\sigma_X +  \frac{s_j + s_{\ell + k}}{ \log(X) }   } } } \\
                & \equivalent{X \to +\infty } \prod_{ p \in \Pe } \Esp{ \prod_{j = 1}^k \frac{1}{1 - \frac{Z_p}{p^{1/2 } } } \frac{1}{1 - \frac{\overline{Z_p}}{p^{ 1/2  } } }   } \prod_{j, \ell = 1}^k \prth{1 - \frac{1}{ p  } } \\
                &  =  \prod_{ p \in \Pe } e^{  k^2 \log(1 - 1/p) } \Esp{  \abs{1 - \frac{Z_p}{ \sqrt{p} } }^{- 2 k } } \\
                & = a_k
\end{align*}

This implies that
\begin{align*}%$
\Esp{ \abs{\Ze_{X, \sigma_X} }^{2k} }  & \equivalent{X \to +\infty } a_k \log(X)^{k^2} \int_{ \prod_{j = 1}^{2k} (c_j + i \Rr) }  \prod_{j, \ell = 1}^k \frac{ 1 }{ -2 \lambda  + s_j + s_{\ell + k}}  \, e^{  \sum_{j = 1}^{2k}  s_j  } \frac{d^*\sbb}{\sbb} \\
                & =: a_k \gamma_k(\lambda) e^{2k \lambda }  \log(X)^{k^2}
\end{align*}
with 
\begin{align*}%$
\gamma_k(\lambda) := e^{- 2k \lambda } \int_{ \prod_{j = 1}^{2k} (c_j + i \Rr) }  \prod_{j, \ell = 1}^k \frac{ 1 }{ -2 \lambda  + s_j + s_{\ell + k}}  \, e^{  \sum_{j = 1}^{2k}  s_j  } \frac{d^*\sbb}{\sbb}  
\end{align*}

It remains to show that $ \gamma_k(\lambda) = \vol( \lambda \Be_k ) = \lambda^{k^2} \gamma_k $. Setting $ s_j = \lambda (s_j' + 1) $ and $ d_j := \lambda c_j + 1 > 0 $, we have 
\begin{align*}%$
\gamma_k(\lambda) :=  \int_{ \prod_{j = 1}^{2k} (d_j + i \Rr) }  \prod_{j, \ell = 1}^k \frac{ 1 }{ s'_j + s'_{\ell + k}}  \, e^{  \sum_{j = 1}^{2k} \lambda s'_j  } \frac{d^*\sbb}{\sbb}  
\end{align*}

Setting $ x_j = s_j $ if $ j \leq k $ and $ y_j = s_j $ for $ j \geq k + 1 $, and using the formula $ \frac{1}{x} = \int_{\Rr_+} e^{-tx } dt $ valid if $ \Re(x) > 0 $, one has, using the Fubini theorem
\begin{align*}%$
\gamma_k(\lambda) & :=  \int_{ \Rr_+^{k^2 } }  \int_{ \prod_{j = 1}^{2k} (d_j + i \Rr) }  \exp\prth{ -\sum_{i, j = 1}^k t_{i, j} (x_i + y_j) + \lambda \sum_{i = 1}^k (x_i + y_j)  }  \frac{d^*X}{X}  \frac{d^*Y}{Y} d\Tb \\
                  & =  \int_{ \Rr_+^{k^2 } } \int_{ \prod_{j = 1}^{2k} (d_j + i \Rr) }   \exp\prth{ \sum_{i = 1}^k x_i \prth{ -\sum_{j = 1}^k t_{i, j} + \lambda } + \sum_{j = 1}^k  y_j \prth{ -\sum_{i = 1}^k t_{i, j} + \lambda }  }  \frac{d^*X}{X}  \frac{d^*Y}{Y} d\Tb \\
                  & =  \int_{ \Rr_+^{k^2 } }  \prod_{i = 1}^k \Unens{ -\sum_{j = 1}^k t_{i, j} + \lambda > 0 }  \prod_{j = 1}^k \Unens{ -\sum_{i = 1}^k t_{i, j} + \lambda > 0 }  d\Tb \quad \mbox{using \eqref{Eq:IntegralRepresentationIndicator} with $ y = e^{x} $}  \\
                  & = \vol(\lambda \Be_k ) = \lambda^{k^2} \vol(\Be_k)
\end{align*}
by definition of $ \Be_k $ and of the $ \lambda $-dilation. 
\end{proof}

% ==

\begin{remark}
The trick $ \frac{1}{x} = \int_{\Rr_+} e^{-tx } dt $ allows to bypass the combinatorial analysis of \cite[(52), (54)]{ConreyGamburd}. The advantage of this trick is to make the equations defining $ \Be_k $ appear immediately. Note that it works for any polytope, in particular the Gelfand-Tsetlin one given in \cite{AssiotisBaileyKeating, AssiotisKeating, BaileyKeating}.
\end{remark}

Remarking that we can adapt the whole proof to the case $ \lambda \equiv \lambda_X = o(\log(X)) $, we get

\begin{corollary}[Phase transition with a drifted log-normal random variable]\label{Lemma:PhaseTransition}
We have, when $ X \to +\infty $ and for all $ \lambda_X > 0 $, $ \lambda_X \to +\infty $ and $ \lambda_X = o(\log(X)) $
\begin{align*}%$
\Esp{ \abs{\Ze_X^*(\lambda_X) }^{2k} } \sim a_k \gamma_k  e^{k^2 \log(\lambda_X \log(X)) + 2k \lambda_X } 
\end{align*}
\end{corollary}

$ $

% ==
\section*{Acknowledgements}

The author thanks Adam Harper, Jon Keating, Olivier H\'enard, Sho Matsumoto, Pierre-Lo\"ic M\'eliot, Joseph Najnudel, Ashkan Nikeghbali, Elliot Paquette and Oleg Zaboronsky for interesting discussions, references or remarks concerning previous versions of this work. A particular thanks is given to Matthias Beck for explanations on polytopes, to Valentin F\'eray and Paul-Olivier Dehaye for explanations concerning the plethystic formalism, and to Persi Diaconis for several remarks and useful criticism. A very particular thanks is given to Nick Simm for numerics and crucial help in understanding the hyperplane concentration phenomenon.

%\commentaire{To send to : Conrey, Snaith, Farmer, Chris, Brad Rodgers, Paul Bourgade, etc.} 

\bibliographystyle{amsplain}

\begin{thebibliography}{10}


\bibitem{AdlerVanMoerbeke}
\rm M. Adler, P. Van Moerbeke,
\it Integrals over classical groups, random permutations, Toda and Toeplitz lattices,
\rm \CPAM 54(2):153-205 (\textbf{2001}).


\bibitem{AdlerShiotaVanMoerbeke}
\rm M. Adler, T. Shiota, P. Van Moerbeke,
\it Random matrices, vertex operators and the Virasoro algebra,
\rm Phys. Lett. A 208:101-112 (\textbf{1995}).


\bibitem{AkemannVernizzi}
\rm G. Akemann, G. Vernizzi,
\it Characteristic polynomials of complex random matrix models,
\rm Nuclear Physics B, 660(3):532-556 (\textbf{2003}).


\bibitem{AkemannPottier}
\rm G. Akemann, G. Vernizzi,
\it Ratios of characteristic polynomials in complex matrix models,
\rm J.Phys. A 37:L453-L460 (\textbf{2004}).


\bibitem{AndreLMCI}
\rm Y. Andr\'e,  
\it Le\c{c}ons de math\'ematiques contemporaines \`a l'IRCAM, 
\rm \href{https://cel.archives-ouvertes.fr/cel-01359200}{cel:01359200} (\textbf{2009}).


\bibitem{ArguinBeliusBourgade}
\rm L.-P. Arguin, D. Belius, P. Bourgade,
\it Maximum of the characteristic polynomial of random unitary matrices, 
\rm \CMP 349(2):703-751 (\textbf{2017}).


\bibitem{ArguinBeliusBourgadeRadziwillSoundararajan}
\rm L.-P. Arguin, D. Belius, P. Bourgade, M. Radziwi\l\l, K. Soundararajan, 
\it Maximum of the Riemann zeta function on a short interval of the critical line, 
\rm \CPAM 72(1):500-535 (\textbf{2019}).


\bibitem{ArguinBourgadeRadziwill1}
\rm L.-P. Arguin, P. Bourgade, M. Radziwi\l\l,  
\it The Fyodorov-Hiary-Keating Conjecture. I, 
\rm \href{https://arxiv.org/abs/2007.00988}{arXiv:2007.00988} (\textbf{2020}).


\bibitem{Aronszajn1}
\rm N. Aronszajn, 
\it La th\'eorie des noyaux reproduisants et ses applications, premi\`ere partie, 
\rm In: Mathematical Proceedings of the Cambridge Philosophical Society 39(3):133-153, Cambridge University Press (\textbf{1943}).

\bibitem{Aronszajn2}
\rm N. Aronszajn, 
\it Theory of reproducing kernels, 
\rm Trans. Amer. Math. Soc. 68(3):337-404 (\textbf{1950}).


\bibitem{ArratiaBarbourTavare}
\rm R. Arratia, A. D. Barbour, and S. Tavar\'e, 
\it Logarithmic combinatorial structures, a probabilistic approach, 
\rm EMS Monographs in Mathematics, Z\"urich, Europ. Math. Soc. (\textbf{2003}).


\bibitem{AssiotisIntertwining}
\rm T. Assiotis,
\it Intertwinings for General $ \beta $-Laguerre and $ \beta $-Jacobi Processes, 
\rm \JTP 32(4):1880-1891 (\textbf{2019}).


\bibitem{AssiotisHuaPickrell}
\rm T. Assiotis,
\it Hua-Pickrell diffusions and Feller processes on the boundary of the graph of spectra, 
\rm Ann. Inst. Henri Poincar\'e Probab. Stat. 56(2):1251-1283 (\textbf{2020}).


\bibitem{AssiotisNajnudel}
\rm T. Assiotis,
\it The boundary of the orbital beta process, 
\rm \href{https://arxiv.org/abs/1905.08684}{1905.08684} (\textbf{2019}).


\bibitem{AssiotisBaileyKeating}
\rm T. Assiotis, E. C. Bailey, J. P. Keating,
\it On the moments of the moments of the characteristic polynomials of Haar distributed symplectic and orthogonal matrices, 
\rm \href{https://arxiv.org/abs/1910.12576}{arXiv:1910.12576} (\textbf{2019}).


\bibitem{AssiotisKeating}
\rm T. Assiotis, J. P. Keating,
\it Moments of moments of characteristic polynomials of random unitary matrices and lattice point counts, 
\rm \href{https://arxiv.org/abs/1905.06072}{arXiv:1905.06072} (\textbf{2019}).


\bibitem{AssiotisKeatingWarren}
\rm T. Assiotis, J. P. Keating, J. Warren,
\it On the joint moments of the characteristic polynomials of random unitary matrices, 
\rm \href{https://arxiv.org/abs/2005.13961}{arXiv:2005.13961} (\textbf{2020}).


\bibitem{BaezDuartePn}
\rm L. Baez-Duarte,
\it Hardy-Ramanujan's asymptotic formula for partitions and the central limit theorem, 
\rm Adv. Math. 125(1):114-120 (\textbf{1997}).


\bibitem{BaileyKeating}
\rm E. C. Bailey, J. P. Keating,
\it On the moments of the moments of the characteristic polynomials of random unitary matrices, 
\rm \CMP 371(2):689-726 \href{https://arxiv.org/abs/1807.06605}{arXiv:1807.06605} (\textbf{2018}).


\bibitem{BaileyBettinBlowerConreyProkhorovRubinsteinSnaith}
\rm E. C. Bailey, S. Bettin, G. Blower, J. B. Conrey, A. Prokhorov, M. O. Rubinstein, N. C. Snaith, 
\it Mixed moments of characteristic polynomials of random unitary matrices , 
\rm J. Math. Phys. 60(8):083509 \href{https://arxiv.org/abs/1901.07479}{arXiv:1901.07479} (\textbf{2019}).


\bibitem{BaikDeiftStrahov}
\rm J. Baik, P. Deift, E. Strahov, 
\it Products and ratios of characteristic polynomials of random Hermitian matrices,
\rm J. Math. Phys. 44(8)3657-3670 (\textbf{2003}).


\bibitem{BaksajevaManstavicius}
\rm T. Bak\v{s}ajeva, E. Manstavi\v{c}ius,
\it On statistics of permutations chosen from the Ewens distribution, 
\rm Combinatorics, Probability \& Computing, CPC Volume, 23(6):889-913 (\textbf{2014}).


\bibitem{BarhoumiHughesNajnudelNikeghbali}
\rm Y. Barhoumi-Andr\'eani, C. P. Hughes, J. Najnudel, A. Nikeghbali,
\it On the number of zeros of linear combinations of independent characteristic polynomials of random unitary matrices,
\rm Int. Math. Res. Not. IMRN 2015(23):12366-12404 (\textbf{2015}). 


\bibitem{BarhoumiThesis}
\rm Y. Barhoumi-Andr\'eani,
\it Some problems in probability theory motivated by number theory and mathematical physics,
\rm Ph.D. thesis, university of Z\"urich, \href{https://www.zora.uzh.ch/id/eprint/88064/1/20142169.pdf}{Zora:20142169} (\textbf{2013}).


\bibitem{BasorBleherBuckinghamGravaIts2Keating} 
\rm E. Basor, P. Bleher, R. Buckingham, T. Grava, A. Its, E. Its, J. P. Keating,
\it A representation of joint moments of CUE characteristic polynomials in terms of Painlev\'e functions, 
\rm Nonlinearity 32(10):4033-4078 (\textbf{2019}).


\bibitem{BasorForrester}
\rm E. Basor, P. Forrester,
\it Formulas for the Evaluation of Toeplitz Determinants with Rational Generating Functions, 
\rm Math. Nachr. 170:5-18 (\textbf{1994}).


\bibitem{BasorGeRubinstein}
\rm E. Basor, F. Ge, M. O. Rubinstein,
\it Some multidimensional integrals in number theory and connections with the Painlev\'e V equation, 
\rm J. Math. Phys. 59(9):091404 (\textbf{2018}).


% ==

\bibitem{BeckResidue}
\rm M. Beck,
\it Counting lattice points by means of the residue theorem, 
\rm The Ramanujan Journal 4(3):299-310 (\textbf{2000}).


\bibitem{BeckPixton}
\rm M. Beck, D. Pixton,
\it The Ehrhart polynomial of the Birkhoff polytope, 
\rm Discrete \& Computational Geometry 30(4):623-637 (\textbf{2003}).


\bibitem{BeckRobins}
\rm M. Beck, S. Robins,
\it Computing the continuous discretely, 
\rm Springer Science+ Business Media, LLC (\textbf{2007}).

% ==

\bibitem{Bergere}
\rm M. C. Berg\`ere,
\it Correlation functions of complex matrix models,
\rm J. Phys. A: Math. Gen. 39:8749-8773 (\textbf{2006}).


\bibitem{BergereEynard}
\rm M. C. Berg\`ere, B. Eynard, 
\it Determinantal formulae and loop equations,
\rm \href{https://arxiv.org/pdf/0901.3273.pdf}{arXiv:0901.3273} (\textbf{2009}).


\bibitem{BettinConrey}
\rm S. Bettin, J. B. Conrey, 
\it Averages of long Dirichlet polynomials,
\rm \href{https://arxiv.org/pdf/2002.09466.pdf}{arXiv:2002.09466} (\textbf{2020}).


\bibitem{BetzSchaeferZeindler}
\rm V. Betz, H. Sch\"afer, D. Zeindler,
\it Random permutations without macroscopic cycles,
\rm Ann. Appl. Probab. 30(3):1484-1505 (\textbf{2020}). 


\bibitem{BetzUeltschi}
\rm V. Betz, D. Ueltschi,
\it Spatial random permutations and infinite cycles,
\rm \CMP 285(2):469-501 (\textbf{2009}).


\bibitem{BetzUeltschi2}
\rm V. Betz, D. Ueltschi,
\it Spatial random permutations with small cycle weights,
\rm \PTRF 149(1):191-222 (\textbf{2011}). 


\bibitem{BetzUeltschiVelenik}
\rm V. Betz, D. Ueltschi, Y. Velenik,
\it Random permutations with cycle weights,
\rm Ann. Appl. Probab. 21(1):312-331 (\textbf{2011}).

% == Biane

\bibitem{BianeCUE}
\rm P. Biane,
\it Representations of unitary groups and free convolution, 
\rm Publications of the Research Institute for Mathematical Sciences, 31(1):63-79 (\textbf{1995}).


\bibitem{BianeFreeBM}
\rm P. Biane,
\it Free Brownian Motion, free stochastic calculus and random matrices, 
\rm Fields Institute Communications 12(1):1-19. (\textbf{1997}).


\bibitem{BianeApproximateFactorisation}
\rm P. Biane,
\it Approximate factorization and concentration for characters of symmetric groups,
\rm Int. Math. Res. Not. 178(4):179-192 (\textbf{2001}). 


\bibitem{BianeCharactersCumulants}
\rm P. Biane, 
\it Characters of symmetric groups and free cumulants, 
\rm in A. Borodin and al., A. M. Vershik eds, LN 1817:185-200 (\textbf{2003}).

% ==

\bibitem{BirkhoffPolytope}
\rm G. Birkhoff,
\it Tres observaciones sobre el algebra lineal, 
\rm Univ. Nac. Tucum\'an. Revista A. 5:147-151 (\textbf{1946}).


\bibitem{BogachevZeindler}
\rm L.V. Bogachev, D. Zeindler,
\it Asymptotic statistics of cycles in surrogate-spatial permutations,
\rm \CMP 334(1):39-116 (\textbf{2015}). 


\bibitem{BohrJessen}
\rm H. Bohr, B. Jessen,
\it \"Uber die Wertverteilung der Riemannschen Zetafunktion, 
\rm \textit{I.}, Acta Mathematica 54:1-35 (\textbf{1930}), \textit{II.}, Acta Mathematica 58:1-55 (\textbf{1932}).

% == Borodin

\bibitem{BorodinInfSym}
\rm A. Borodin,
\it Point processes and the infinite symmetric group, part ii, higher correlation functions,
\rm \href{http://arxiv.org/abs/math/9804087}{arXiv:9804087} (\textbf{1998}).


\bibitem{BorodinCorwinSasamoto}
\rm A. Borodin, I. Corwin, T. Sasamoto, 
\it From duality to determinants for $q$-TASEP and ASEP,
\rm \AOP 42(6):2314-2382 (\textbf{2014}).


\bibitem{BorodinOlshanskiZmeas1}
\rm A. Borodin, G. Olshanski,
\it $Z$-measures on partitions, Robinson-Schensted-Knuth correspondence and $ \beta = 2 $ ensembles,
\rm in ``Random matrix models and their applications'', P. M. Bleher and A. R. Its eds., MSRI Publications 40:71-94 \href{http://arxiv.org/abs/math/9905189}{arXiv:9905189} (\textbf{2001}).


\bibitem{BorodinOlshanskiErg}
\rm A. Borodin, G. Olshanski,
\it Infinite random matrices and ergodic measures,
\rm \CMP 223(1):87-123 (\textbf{2001}).


\bibitem{BorodinOlshanskiZmeas2}
\rm A. Borodin, G. Olshanski,
\it $Z$-measures on partitions and theirscaling limits,
\rm European J. Combin. 26(6):795-834 (\textbf{2005}).


\bibitem{BorodinOlshanskiStrahov}
\rm A. Borodin, G. Olshanski, E. Strahov,
\it Giambelli compatible point processes, 
\rm Adv. Appl. Math. 37(2):209-248 (\textbf{2006}).


\bibitem{BorodinPetrov}
\rm A. Borodin, L. Petrov, 
\it Integrable probability: from representation theory to Macdonald processes, 
\rm \href{http://arxiv.org/abs/math/1310:8007v3}{arXiv:1310:8007} (\textbf{2014}).


\bibitem{BorodinStrahov}
\rm A. Borodin, E. Strahov, 
\it Averages of characteristic polynomials in random matrix theory, 
\rm \CPAM 59(2):161-253 (\textbf{2006}).

% ==

\bibitem{BorotEynardOrantin}
\rm G. Borot, B. Eynard, M. Mulase, N. Orantin,
\it Abstract loop equations, topological recursion and applications, 
\rm Commun. Number Theory Phys. 9(1):51-187 (\textbf{2015}).

% == Bourgade

\bibitem{BourgadeCondHaar}
\rm P. Bourgade,
\it Conditional Haar measures on classical groups,
\rm \AOP 37(4):1566-1586 (\textbf{2009}).


\bibitem{BourgadeMesoscopic}
\rm P. Bourgade,
\it Mesoscopic fluctuations of the zeta zeros,
\rm Probab. Theory Related Fields 148(3-4):479-500 (\textbf{2010}).


\bibitem{BHNY}
\rm P. Bourgade, C. Hughes, A. Nikeghbali, M. Yor,
\it The characteristic polynomial of a random unitary matrix : a probabilistic approach,
\rm Duke Math. J. 145:45-69 (\textbf{2008}).


\bibitem{BourgadeNajnudelNikeghbali}
\rm P. Bourgade, J. Najnudel, A. Nikeghbali,
\it A unitary extension of virtual permutations,
\rm Int. Math. Res. Not. 2012(18):4101-4134 (\textbf{2012}).


\bibitem{BourgadeNikeghbaliRouault}
\rm P. Bourgade, A. Nikeghbali, A. Rouault,
\it Ewens measures on compact groups and hypergeometric kernels,
\rm S\'eminaire de Probabilit\'es XLIII, Lecture Notes in Math. 2006:351-377 (\textbf{2010}).

% ==

\bibitem{Boyer83}
\rm R. P. Boyer,
\it Infinite traces of $AF$-algebras and characters of $U(\infty)$,
\rm J. Operator Theory 9(1):205-236 (\textbf{1983}).


\bibitem{Boyer92}
\rm R. P. Boyer,
\it Characters and factor representations of the infinite-dimensional classical groups,
\rm J. Operator Theory 28(1):281-307 (\textbf{1992}).


\bibitem{BrezinHikamiCharpol}
\rm E. Br\'ezin, S. Hikami,
\it Characteristic polynomials of random matrices,
\rm \CMP 214:111-135 (\textbf{2000}).


\bibitem{BrezinHikamiCharpolEdge}
\rm E. Br\'ezin, S. Hikami,
\it Characteristic polynomials of random matrices at edge singularities,
\rm Phys. Rev. E 62(3-A):3558-3567 (\textbf{2000}).


\bibitem{BrezinHikamiAutocorrs}
\rm E. Br\'ezin, S. Hikami,
\it New correlation functions for random matrices and integrals over supergroups,
\rm J. Phys. A. 36:711-751 (\textbf{2003}).


\bibitem{BrezinHikamiDuality}
\rm E. Br\'ezin, S. Hikami,
\it Intersection theory from duality and replica,
\rm \CMP 283: 507-521 (\textbf{2008}).

\bibitem{BruDiff}
\rm M. F. Bru,
\it Diffusions of perturbed principal component analysis,
\rm J. Maltivated Anal. 29:127-136 (\textbf{1989}). 


%\bibitem{BruWishart}
%\rm M. F. Bru,
%\it Wishart process,
%\rm J. Theoret. Probab. 3:725-751 (\textbf{1991}).


\bibitem{BumpGamburd}
\rm D. Bump, A. Gamburd, 
\it On the averages of characteristic polynomials from classical groups, 
\rm \CMP 265:227-274 (\textbf{2006}).


\bibitem{BunBouchaudMajumdarPotters}
\rm J. Bun, J. P. Bouchaud, S. N. Majumdar, M. Potters,
\it Instanton approach to large $n$ Harish-Chandra-Itzykson-Zuber integrals, 
\rm Physical review letters 113(7):070201 (\textbf{2014}).


\bibitem{CanfieldPn}
\rm E. R. Canfield,
\it From recursions to asymptotics: on Szekeres' formula for the number of partitions, 
\rm Electron. J. Combin. 4(2):Research paper 6 (\textbf{1997}).


\bibitem{CanfieldMcKay}
\rm E. R. Canfield, B. D. McKay,
\it The asymptotic volume of the Birkhoff polytope, 
\rm \href{https://arxiv.org/abs/0705.2422}{arXiv:0705.2422} (\textbf{2007}).


\bibitem{CarmonaPetitYor}
\rm P. Carmona, F. Petit, M. Yor,
\it Beta-gamma random variables and intertwining relations between certain Markov processes,
\rm Rev. Mat. Iberoamericana, 14(2):311-367 (\textbf{1998}). 


\bibitem{ChhaibiMadauleNajnudel}
\rm R. Chhaibi, T. Madaule, J. Najnudel, 
\it On the maximum of the C$ \beta $E field,
\rm Duke Math. J. 167(12):2243-2345 (\textbf{2018}).


\bibitem{ChhaibiNajnudelNikeghbali}
\rm R. Chhaibi, J. Najnudel, A. Nikeghbali, 
\it The circular unitary ensemble and the Riemann Zeta function: the microscopic landscape and a new approach to ratios, 
\rm \Inventiones 207(1):23-113 (\textbf{2017}).


\bibitem{CiprianiZeindler}
\rm A. Cipriani, D. Zeindler,
\it The limit shape of randompermutations with polynomially growing cycle weights,
\rm ALEA Lat. Am. J. Probab. Math. Stat. 12(2):971-999 (\textbf{2015}). 

% == Collins

\bibitem{CollinsHCIZ}
\rm B. Collins,
\it Moments and cumulants of polynomial random variables on unitary groups, the Itzykson-Zuber integral and free probability,
\rm Int. Math. Res. Not. 2003(17):953-982 (\textbf{2003}). 


\bibitem{CollinsSniady}
\rm B. Collins, P. Sniady,
\it Integration with respect to the Haar measure on unitary, orthogonal and symplectic group,
\rm \CMP 264(3):773-795 (\textbf{2006}). 


\bibitem{CollinsGuionnetMaurelSegala}
\rm B. Collins, A. Guionnet, E. Maurel-Segala,
\it Asymptotics of unitary and orthogonal matrix integrals, 
\rm Adv. Math. 222(1):172-215 (\textbf{2009}).


% == Conrey

\bibitem{ConreyFourth}
\rm J. B. Conrey,
\it The fourth moment of derivatives of the Riemann Zeta Function,
\rm Quart. J. Math. Oxford Ser. 2, 39(1):21-36 (\textbf{1988}).


\bibitem{ConreyFarmerKeatingRubinsteinSnaith}
\rm J. B. Conrey, D. W. Farmer, J.P. Keating, M. O. Rubinstein, N.C. Snaith,
\it Autocorrelation of random matrix polynomials,
\rm \CMP 237(3):365-395 (\textbf{2003}).


\bibitem{CFKRSZeta}
\rm J. B. Conrey, D. W. Farmer, J.P. Keating, M. O. Rubinstein, N.C. Snaith,
\it Integral moments of $L$-functions,
\rm Proc. Lond. Math. Soc. 91(3):33-104 (\textbf{2005}).


\bibitem{ConreyFarmerZirnbauer}
\rm J. B. Conrey, D. W. Farmer, M. R. Zirnbauer,
\it Howe pairs, supersymmetry, and ratios of random characteristic polynomials for the unitary groups $ U(N) $,
\rm preprint \href{https://arxiv.org/abs/math-ph/0511024}{arXiv:0511024} (\textbf{2005}).


\bibitem{ConreyForresterSnaith}
\rm J. B. Conrey, P. J. Forrester, N.C. Snaith,
\it Averages of ratios of characteristic polynomials for the compact classical groups,
\rm Int. Math. Res. Not. 7:397-431 (\textbf{2005}).


\bibitem{ConreyGamburd}
\rm J. B. Conrey, A. Gamburd, 
\it Pseudomoments of the Riemann zeta-function and pseudomagic squares, 
\rm Journal of Number Theory 117(2):263-278 (\textbf{2006}).


\bibitem{ConreyRubinsteinSnaith}
\rm J. B. Conrey, M. O. Rubinstein, N. C. Snaith,
\it Moments of the derivative of characteristic polynomials with an application to the Riemann Zeta Function,
\rm \CMP 267(3):611-629 (\textbf{2006}).

% ==

\bibitem{DalqvistBMU}
\rm A. Dalqvist,
\it Free energies and fluctuations for the unitary Brownian motion,
\rm \CMP 348:395-444 (\textbf{2016}).


\bibitem{Day}
\rm K. M. Day,
\it Toeplitz matrices generated by the Laurent series expansion of an arbitrary rational function,
\rm Trans. American Math. Soc. 206:224-245 (\textbf{1975}).


\bibitem{deGierPonsaing}
\rm J. de Gier, A. Ponsaing,
\it Separation of variables for symplectic characters,
\rm Lett. Math. Phys. 97(1):61-83 (\textbf{2011}).


% == POD

\bibitem{POD}
\rm P.-O. Dehaye,
\it Joint moments of derivatives of characteristic polynomials,
\rm Algebra and Number Theory, 2(1):31-68 (\textbf{2008}).


\bibitem{PODfpsac}
\rm P.-O. Dehaye,
\it A note on moments of derivatives of characteristic polynomials,
\rm proceedings of FPSAC 2010, San Francisco, USA (\textbf{2010}).


\bibitem{PODmomentsCombin}
\rm P.-O. Dehaye,
\it Combinatorics of lower order terms in the moment conjectures for the Riemann zeta function,
\rm \href{https://arxiv.org/abs/1201.4478}{arXiv:1201.4478} (\textbf{2012}).


\bibitem{PODZeindler}
\rm P.-O. Dehaye, D. Zeindler,
\it On averages of randomized class functions on the symmetric groups and their asymptotics,
\rm Annales de l'Institut Fourier 63(4):1227-1262 (\textbf{2013}).

% ==

\bibitem{DeiftIntegrable}
\rm P. Deift,
\it Integrable operators,
\rm American Mathematical Society Translations, 189(2):69-84 (\textbf{1999}).


\bibitem{DeiftItsKrasovsky}
\rm P. Deift, A. Its, I. Krasovsky,
\it Toeplitz matrices and Toeplitz determinants under the impetus of the Ising model: some history and some recent results,
\rm \CMP 66(9):1360-1438 (\textbf{2013}).


\bibitem{DeLoeraLiuYoshida}
\rm J.-A. De Loera, F. Liu, R. Yoshida,
\it A generating function for all semi-magic squares and the volume of the Birkhoff polytope,
\rm J. Algebraic Combin. 30:113-139 (\textbf{2009}).


\bibitem{DesrosiersDuality}
\rm P. Desrosiers, 
\it Duality in random matrix ensembles for all $\beta$, 
\rm Nucl. Phys. B 817:224-251 (\textbf{2009}).


\bibitem{DesrosiersLiu2011}
\rm P. Desrosiers, D.-Z. Liu,
\it Selberg integrals, super hypergeometric functions and applications to $\beta$ ensembles of random matrices, 
\rm Random Matrices: Theory and Applications, 4(02):1550007 \href{http://arxiv.org/abs/1109.4659v1}{arXiv:1109.4659} (\textbf{2011}).

% == Diaconis

\bibitem{DiaconisRMTSurvey}
\rm P. Diaconis,
\it Patterns in eigenvalues : the 70th Josiah Willard Gibbs lecture, 
\rm Bulletin of the AMS (new series), 40(2):155-178 (\textbf{2003}).


\bibitem{DiaconisGamburd}
\rm P. Diaconis, A. Gamburd,
\it Random matrices, magic squares and matching polynomials,
\rm Electronic Journal of combinatorics, Vol. 11, Part 2 \# R2 (\textbf{2004}).


\bibitem{DiaconisGangolli}
\rm P. Diaconis, A. Gangolli,
\it Rectangular arrays with fixed margins,
\rm Discrete probability and algorithms, Springer, New York, NY, pp. 15-41 (\textbf{1995}).


\bibitem{DiaconisFill}
\rm P. Diaconis, J. A. Fill,
\it Strong stationary times via a new form of duality,
\rm \AOP 18:1483-1522 (\textbf{1990}).


\bibitem{DiaconisShahshahani} 
\rm P. Diaconis, M. Shahshahani,
\it On the Eigenvalues of Random Matrices,
\rm Journal of Applied Probability 31:49-62 (\textbf{1994}).

% ==

\bibitem{DuistermaatHeckman}
\rm J. J. Duistermaat, G. J. Heckman,
\it On the variation in the cohomology of the symplectic form of the reduced phase,
\rm \Inventiones 69(2):259-268 (\textbf{1982}).


\bibitem{Dyson}
\rm F. Dyson,
\it Correlations between eigenvalues of a random matrix, 
\rm \CMP 19(3):235-250 (\textbf{1970}).


\bibitem{Efetov}
\rm K. Efetov,
\it Supersymmetry in Disorder and Chaos,
\rm Cambridge Univ. Press (\textbf{1996}).


\bibitem{Ehrhart62}
\rm E. Ehrhart,
\it Sur les poly\`edres rationnels homoth\'etiques \`a $n$ dimensions,
\rm C. R. Acad. Sci. Paris  254: 616-618 (\textbf{1962}).


\bibitem{EynardOrantin2009}
\rm B. Eynard, N. Orantin,
\it Topological recursion in random matrices and enumerative geometry, 
\rm J. Phys. A: Mathematical and Theoretical 42(29):293001 (\textbf{2009}).


\bibitem{Ewens72}
\rm W. Ewens,
\it The sampling theory of selectively neutral alleles, 
\rm Theoretical Population Biology 3(1):87-112 (\textbf{1972}).


\bibitem{Fahs}
\rm B. Fahs,
\it Uniform asymptotics of Toeplitz determinants with Fisher-Hartwig singularities,
\rm \href{https://arxiv.org/pdf/1909.07362.pdf}{arXiv:1909.07362} (\textbf{2019}).


\bibitem{FerayJM}
\rm V. F\'eray,
\it On complete functions in Jucys-Murphy elements,
\rm Ann. Comb. 16(4):677-707 (\textbf{2012}).

% == Forrester

\bibitem{ForresterGamburd}
\rm P. Forrester, A. Gamburd,
\it Counting formulas associated with some random matrix averages,
\rm \href{http://arxiv.org/pdf/math/0503002v1.pdf}{arXiv:0503002} (\textbf{2005}).


\bibitem{ForresterOleWarnaar}
\rm P. J. Forrester, S. Ole Warnaar, 
\it The importance of the Selberg integral, 
\rm Bull. Amer. Math. Soc. 45(2):489-534 (\textbf{2008}).


\bibitem{ForresterRainsFuchsian}
\rm P. J. Forrester, E. Rains, 
\it A Fuchsian matrix differential equation for Selberg correlation integrals, 
\rm \CMP 309(3):771-792 (\textbf{2012}).


\bibitem{ForresterWitteTau2}
\rm P. J. Forrester, N. S. Witte, 
\it Application of the $ \tau $-function theory of Painlev\'e Equations to random matrices: P-V, P-III, the LUE, JUE, and CUE, 
\rm \CPAM 55(6):679-727 (\textbf{2002}).


\bibitem{ForresterWittePainleve3p5}
\rm P. J. Forrester, N. S. Witte, 
\it Boundary conditions associated with the Painlev\'e III' and V evaluations of some random matrix averages, 
\rm J. Phys. A: Math. Gen. 39(28):8983-8995 (\textbf{2006}).


% == Fyodorov


\bibitem{FyodorovSupersym}
\rm Y.V. Fyodorov, 
\it Negative moments of characteristic polynomials of random matrices: Ingham-Siegel integral as an alternative to Hubbard-Stratonovich transformation,
\rm Nucl. Phys. B 621:643-674 (\textbf{2002}).


\bibitem{FyodorovHiaryKeating}
\rm Y.V. Fyodorov, G.A. Hiary, J.P. Keating,
\it Freezing Transition, Characteristic Polynomials of Random Matrices, and the Riemann Zeta-Function,
\rm preprint \href{http://arxiv.org/abs/1202.4713}{arXiv:1202.4713} (\textbf{2012}).


\bibitem{FyodorovKeating}
\rm Y.V. Fyodorov, J.P. Keating,
\it Freezing transitions and extreme values: Random Matrix Theory, $ \zeta(1/2+it) $, and disordered landscapes,
\rm \href{http://arxiv.org/abs/1211.6063}{arXiv:1211.6063} (\textbf{2012}).


\bibitem{FyodorovKhoruzhenkoSimm}
\rm Y.V. Fyodorov, B. Khoruzhenko, N. Simm,
\it Fractional Brownian motion with Hurst index $ H = 0 $ and the Gaussian Unitary Ensemble,
\rm \AOP 44(4):2980-3031 (\textbf{2016}).


\bibitem{FyodorovStrahovCorrChiralGUE}
\rm Y. V. Fyodorov, E. Strahov,
\it On correlation functions of characteristic polynomials for chiral Gaussian unitary ensemble,
\rm Nuclear Physics B 647(3):581-597 (\textbf{2002}).


\bibitem{FyodorovStrahovCorrGUE}
\rm Y. V. Fyodorov, E. Strahov,
\it An exact formula for general spectral correlation function of random Hermitian matrices,
\rm Journal of Physics A: Mathematical and General, 36(12):3203-3224 (\textbf{2003}).


\bibitem{FyodorovStrahovRHP}
\rm Y. V. Fyodorov, E. Strahov,
\it Universal results for correlations of characteristic polynomials: Riemann-Hilbert approach,
\rm Commun. Math. Phys. 241(1):343-382 (\textbf{2003}).

% ==

\bibitem{GarsiaHaimanTesler}
\rm A. M. Garsia, M. Haiman, G. Tesler,
\it Explicit plethystic formulas for Macdonald $q,t$-Kostka coefficients, 
\rm In The Andrews Festschrift, Springer, Berlin, Heidelberg 253-297 (\textbf{2001}).


\bibitem{GelfandNaimarkUnimodular}
\rm I. M. Gel'fand, M. A. Naimark,
\it Unitary representations of a unimodular group containing an identity representation of the unitary subgroup, 
\rm Tr. Mosk. Mat. Obs. 1(1):423-475 (\textbf{1952}).


\bibitem{GiardinaKurchanRedigVafayi}
\rm C. Giardin\`a, J. Kurchan, F. Redig, K. Vafayi, 
\it Duality and hidden symmetries in interacting particle systems, 
\rm J. Stat. Phys. 135(1):25-55 (\textbf{2009}).


\bibitem{GonekDerivatives1}
\rm S.M. Gonek,
\it Mean values of the Riemann Zeta Function and its Derivatives,
\rm \Inventiones 75(1):123-141 (\textbf{1984}).


\bibitem{GonekDerivatives2}
\rm S.M. Gonek,
\it On Negative Moments of the Riemann Zeta-Function,
\rm Mathematika 36(1):71-88 (\textbf{1989}).


%\bibitem{GonekDerivatives3}
%\rm S.M. Gonek,
%\it The Second Moment of the Reciprocal of the Riemann Zeta Function and its Derivative,
%\rm Talk at Mathematical Sciences Research Institute, Berkeley (\textbf{1999}).


\bibitem{GorinQ}
\rm V. Gorin,
\it The $q$-Gelfand-Tsetlin graph, Gibbs measures and $q$-Toeplitz matrices, 
\rm Adv. Math. 229(1):201-266 (\textbf{2012}).


\bibitem{GorinPanova}
\rm V. Gorin, G. Panova,
\it Asymptotics of symmetric polynomials with applications to statistical mechanics and representation theory, 
\rm Ann. Probab. 43(6):3052-3132 (\textbf{2015}).


\bibitem{GorodetskyRodgers}
\rm O. Gorodetsky, B. Rodgers,
\it The variance of the number of sums of two squares in $ \mathbb{F}_q[T] $ in short intervals, 
\rm \href{https://arxiv.org/abs/1810.06002}{arXiv:1810.06002} (\textbf{2018}).


\bibitem{GrinbergReiner}
\rm D. Grinberg, V. Reiner,
\it Hopf Algebras in Combinatorics, 
\rm \href{https://arxiv.org/abs/1409.8356v6}{arXiv:1409.8356v6} (\textbf{2020}).


\bibitem{GronqvistGuhrKohler}
\rm J. Gronqvist, T. Guhr, H. Kohler,
\it The $k$-point random matrix kernels obtained from one-point supermatrix models, 
\rm J. Phys. A. 37:2331-2344 (\textbf{2004}).

% == Guionnet

\bibitem{GuionnetSurveyPGDinRMT}
\rm A. Guionnet,
\it Large deviations and stochastic calculus for large random matrices, 
\rm Probab. Surv. 1(1):72-172 (\textbf{2004}).

\bibitem{GuionnetMaidaSchur}
\rm A. Guionnet, M. Ma\"ida,
\it A Fourier view on the $R$-transform and related asymptotics of spherical integrals, 
\rm J. Funct. Anal. 222(2):435-490 (\textbf{2005}).

\bibitem{GuionnetNovak}
\rm A. Guionnet, J. Novak,
\it Asymptotics of unitary multimatrix models: the Schwinger-Dyson lattice and topological recursion, 
\rm J. Funct. Anal. 268(10):2851-2905 (\textbf{2015}).

\bibitem{GuionnetZeitouniSchur}
\rm A. Guionnet, O. Zeitouni,
\it Large deviations asymptotics for spherical integrals, 
\rm J. Funct. Anal. 188(2):461-515 (\textbf{2002}).

% ==

\bibitem{HaakeEtAl}
\rm F. Haake, M. Kus, H.-J. Sommers, H. Schomerus, K. Zyckowski,
\it Secular determinants of random unitary matrices, 
\rm J. Phys. A.: Math. Gen. 29(13):3641-3658 (\textbf{1996}).


\bibitem{HaimanMacDo}
\rm M. Haiman,
\it Macdonald Polynomials and Geometry, 
\rm New Perspectives in Algebraic Combinatorics, MSRI Publications, \url{https://math.berkeley.edu/~mhaiman/ftp/nfact/msri.pdf}  (\textbf{1999}).


\bibitem{HardyLittlewoodMoments}
\rm G. H. Hardy, J. E. Littlewood,
\it Contributions to the theory of the Riemann zeta-function and the theory of the distribution of primes, 
\rm Acta Math. 41:119-196 (\textbf{1918}).


\bibitem{HarishChandraInt}
\rm Harish-Chandra,
\it Differential operators on a semi-simple Lie algebra,
\rm Amer. J. Math. 79(1):87-120 (\textbf{1957}).


\bibitem{HarperZetaFK}
\rm A. J. Harper, 
\it On the partition function of the Riemann zeta function, and the Fyodorov-Hiary-Keating conjecture, 
\rm \href{https://arxiv.org/abs/1906.05783v1}{arXiv:1906.05783} (\textbf{2019}).


\bibitem{HarperNikeghbaliRadziwill}
\rm A. J. Harper, A. Nikeghbali, M. Radziwi\l{ll}, 
\it A note on Helson's conjecture on moments of random multiplicative functions,
\rm in Analytic Number Theory 145-169, Springer International Publishing (\textbf{2015}).


\bibitem{Hayman}
\rm W. K. Hayman,
\it A generalisation of Stirling's formula,
\rm Proc. London Math. Soc. 196(1 \& 2):67-95 (\textbf{1956}). 


\bibitem{HeapLindqvist}
\rm W. P. Heap, S. Lindqvist,
\it Moments of random multiplicative functions and truncated characteristic polynomials,
\rm The Quarterly Journal of Mathematics 2016:1-32 (\textbf{2016}).


\bibitem{HejhalLogZetaPrime}
\rm D. A. Hejhal,
\it On the distribution of $ \log\abs{\zeta'\prth{\frac{1}{2} + i t} } $,
\rm Number Theory, Trace Formulas, and Discrete Groups , K.E. Aubert, E. Bombieri and D.M. Goldfeld, eds., Proceedings of the 1987 Selberg Symposium, Academic Press, pp. 343-370 (\textbf{1989}).


\bibitem{HiaryOdlyzko}
\rm G. Hiary, A. Odlyzko,
\it The zeta function on the critical line: numerical evidence for moments and random matrix theory models,
\rm Mathematics of Computation, 81(279):1723-1752 (\textbf{2012}).


\bibitem{Hisakado}
\rm M. Hisakado,
\it Unitary matrix models and Painlev\'e III,
\rm Modern Physics Letters A 11(38):3001-3010 (\textbf{1996}).


\bibitem{HoughJiang}
\rm B. Hough, Y. Jiang,
\it Cut-off phenomenon in the uniform plane Kac walk, 
\rm \AOP 45(4):2248-2308 (\textbf{2017}).


\bibitem{HuaHarmonic}
\rm L. K. Hua,
\it Harmonic analysis of functions of several complex variables in the classical domains, 
\rm in Chinese (\textbf{1958}) ; English translation: Transl. Math. Monogr., 75, Amer. Math. Soc., Providence, RI (\textbf{1963}).

% == Hughes

\bibitem{HughesKeatingOConnell2}
\rm C.P. Hughes, J.P. Keating,  N. O'Connell,
\it Random matrix theory and the derivative of the Riemann Zeta Function,
\rm R. Soc. Lond. Proc. Ser. A Math. Phys. Eng. Sci. 456(2003):2611-2627 (\textbf{2000}).


\bibitem{HughesNajnudelNikeghbaliZeindler}
\rm C. P. Hughes, J. Najnudel, A. Nikeghbali, D. Zeindler,
\it Random permutation matrices under the generalized Ewens measure, 
\rm Ann. of App. Prob. 23(3):987-1024 (\textbf{2013}).


\bibitem{HughesThesis}
\rm C.P. Hughes,
\it On the Characteristic Polynomial of a Random Unitary Matrix and the Riemann Zeta Function,
\rm PhD thesis, University of Bristol (\textbf{2001}).

% ==

\bibitem{ImamuraSasamotoGOE}
\rm T. Imamura, T. Sasamoto, 
\it Polynuclear growth model, $GOE^2$ and random matrix with deterministic source, 
\rm Phys. Review E 71:041606 (\textbf{2005}).


\bibitem{Ingham}
\rm A. E. Ingham,
\it Mean-value theorems in the theory of the Riemann zeta-function, 
\rm Proc. of the London Math. Soc. 27(1):273-300 (\textbf{1926}).


\bibitem{ItzyksonZuber}
\rm C. Itzykson, J. B. Zuber, 
\it The planar approximation, II, 
\rm J.Math. Phys. 21(3):411-421 (\textbf{1980}). 


\bibitem{JingQhypergeomMacDo}
\rm N. Jing, 
\it $q$-hypergeometric series and Macdonald functions, 
\rm Journal of Algebraic Combinatorics 3(1):291-305 (\textbf{1994}).

% == Johansson

\bibitem{JohanssonSzego}
\rm K. Johansson, 
\it On Szeg\"o's asymptotic formula for Toeplitz determinants and generalisations, 
\rm Bull. Sc. Math. 112:257-304 (\textbf{1988}).


\bibitem{JohanssonGroupe}
\rm K. Johansson, 
\it On random matrices from the compact classical groups, 
\rm Annals of Mathematics 153:259-296 (\textbf{2001}).


\bibitem{JohanssonHCIZ}
\rm K. Johansson, 
\it Universality of the local spacing distribution in certain ensembles of Hermitian Wigner matrices, 
\rm \CMP 215(3):683-705 (\textbf{2001}).


\bibitem{JohanssonHouches}
\rm K. Johansson, 
\it Random matrices and determinantal processes, 
\rm Les Houches lecture, \href{https://arxiv.org/abs/math-ph/0510038}{arXiv:0510038} (\textbf{2005}).

% ==

\bibitem{JonnadulaKeatingMezzadri}
\rm B. Jonnadula, J. Keating, F. Mezzadri,
\it Symmetric function theory and unitary invariant ensembles, 
\rm \href{https://arxiv.org/abs/2003.02620}{arXiv:2003.02620} (\textbf{2020}).


\bibitem{KarlinMcGregor}
\rm S. Karlin, J. McGregor,  
\it The classification of birth and death processes,
\rm Trans. Am. Math. Soc. 86(2):366-400 (\textbf{1957}).


\bibitem{KatoriTanemura}
\rm M. Katori, H. Tanemura, 
\it Noncolliding Brownian motions and Harish-Chandra formula, 
\rm Electron. Commun. Probab. 8(13):112-121 (\textbf{2003}).


\bibitem{KatzCUE}
\rm N. M. Katz,
\it Witt vectors and a question of Keating and Rudnick,
\rm Int. Math. Res. Not. IMRN 2013(16):3613-3638 (\textbf{2013}).


\bibitem{KatzSarnak}
\rm N. Katz, P. Sarnak, 
\it Random matrices, Frobenius eigenvalues and monodromy, 
\rm AMS Colloquium Publications, vol. 45 (\textbf{1999}).


\bibitem{KeatingR3G}
\rm J.P. Keating, B. Rodgers, E. Roditty-Gershon, Z. Rudnick,
\it Sums of divisor functions in $ \Ff_q[t] $ and matrix integrals,
\rm Mathematische Zeitschrift 1-32 (\textbf{2017}).


\bibitem{KeatingSnaith}
\rm J.P. Keating, N.C. Snaith,
\it Random Matrix Theory and $ \zeta(1/2 + i t)$,
\rm \CMP 214:57-89 (\textbf{2000}).

% == Kerov

\bibitem{Kerov}
\rm S. V. Kerov, 
\it Asymptotic representation theory of the symmetric group and its applications in analysis, 
\rm Translations of Mathematical Monographs n. 219, American Mathematical Society, Providence, RI (\textbf{2003}).


\bibitem{KerovEwensPitman}
\rm S. V. Kerov, 
\it Coherent random allocation and the Ewens-Pitman formula, 
\rm J. Math. Sci. 138(3):5699-5710 (\textbf{2006}).


\bibitem{KerovOlshanskiVershik}
\rm S. V. Kerov, G. Olshanski, A. Vershik,
\it Harmonic analysis on the infinite symmetric group, 
\rm \Inventiones 158(3):551-642 (\textbf{2004}).


\bibitem{KerovVershikTARunitary}
\rm S. V. Kerov, A. Vershik,
\it Characters and factor representations of the infinite unitary group, 
\rm Soviet Math. Doklady 267(2):570-574 (\textbf{1982}).

% ==

\bibitem{KillipNenciu}
\rm R. Killip, I. Nenciu,
\it Matrix models for circular ensemble,
\rm Int. Math. Res. Not. IMRN 2004:2665-2701 (\textbf{2004}).


\bibitem{KillipRyckman}
\rm R. Killip, E. Ryckman,
\it Autocorrelations of the characteristic polynomial of a random matrix under microscopic scaling,
\rm \href{http://arxiv.org/abs/1004.1623}{arXiv:1004.1623} (\textbf{2010}).


%\bibitem{KirillovReshetikhin}
%\rm A. N. Kirillov, N. Y. E. Reshetikhin,
%\it Bethe ansatz and combinatorics of Young tableaux, 
%\rm Zapiski Nauchnykh Seminarov POMI, 155:65-115 (\textbf{1986}) 


\bibitem{Kolchin1971}
\rm V. F. Kolchin,
\it A certain problem of the distribution of particles in cells, and cycles of random permutations, 
\rm Teor. Veroyatn. Primen. 16(1):67-82 (\textbf{1971}).


\bibitem{KolchinSevastyanovChistyakov}
\rm V. F. Kolchin, B. A. Sevast'yanov, and V. P. Chistyakov,
\it Random Allocations, 
\rm Wiley, New York (\textbf{1978}).


\bibitem{Kontsevich1992}
\rm M. Kontsevich,
\it Intersection theory on the moduli space of curves and the matrix Airy function,  
\rm \CMP 147(1):1-23 (\textbf{1992}).


\bibitem{KostovCFTtoRMT}
\rm I. K. Kostov,
\it Conformal Field Theory techniques in Random Matrix Theory,
\rm Claude Itzykson Meeting, Paris, July 1998 \href{http://arxiv.org/abs/hep-th/9907060}{arXiv:9907060} (\textbf{1999}).


%\bibitem{KostantKostka}
%\rm B. Kostant,
%\it A formula for the multiplicity of a weight,  
%\rm Trans. Amer. Math. Soc. 93:53-73 (\textbf{1959}).


\bibitem{KowalskiRandonnee}
\rm E. Kowalski, 
\it Arithmetic randonn\'ee, an introduction to probabilistic number theory, 
\rm Amer. Math. Soc. Open Math Notes, course given at the ETH Z\"urich (\textbf{2015}).


\bibitem{KuanUqAn}
\rm J. Kuan,
\it An algebraic construction of duality functions for the stochastic $ U_q(A_n^{(1)}) $ vertex model and its degenerations,
\rm \CMP 359(1):121-187 (\textbf{2018}).


\bibitem{KuanAsep}
\rm J. Kuan,
\it Stochastic duality of ASEP with two particle types via symmetry of quantum groups of rank two,
\rm J. Phys. A 49(11):115002-115033 (\textbf{2016}).


\bibitem{KuanMarkovSchurWeyl}
\rm J. Kuan,
\it Two dualities: Markov and Schur-Weyl,
\rm \href{https://arxiv.org/pdf/2006.13879.pdf}{arXiv:2006.13879} (\textbf{2020}).


\bibitem{KummerHypergeom}
\rm E. E. Kummer,
\it \"Uber die hypergeometrische Reihe ${\displaystyle 1\!+\! {\tfrac {\alpha \cdot \beta }{1\cdot \gamma }}x \!+\! {\tfrac {\alpha (\alpha +1)\beta (\beta +1)}{1\cdot 2\cdot \gamma (\gamma +1)}}x^{2} \!+\! {\tfrac {\alpha (\alpha +1)(\alpha +2)\beta (\beta +1)(\beta +2)}{1\cdot 2\cdot 3\cdot \gamma (\gamma +1)(\gamma +2)}}x^{3}}$ $+\mbox{etc.} $,
\rm J. Reine Angew. Math. 15(1):39-83 (\textbf{1836}).


\bibitem{KuznetsovMangazeevSklyanin}
\rm V.B. Kuznetsov, V. V. Mangazeev, E. K. Sklyanin,
\it $Q$-operator and factorised separation chain for Jack polynomials,
\rm Indagat. Math. 14(3-4):451-482 (\textbf{2003}).


\bibitem{KuznetsovSklyaninBacklund}
\rm V.B. Kuznetsov, E. K. Sklyanin,
\it On B\"acklund transformations for many-body systems,
\rm J. Phys. A 31(9):2241-2251 (\textbf{1998}).


\bibitem{KuznetsovSklyaninFactorisation}
\rm V.B. Kuznetsov, E. K. Sklyanin,
\it Factorization of symmetric polynomials,
\rm Contemp. Math. 417(1):239-256 (\textbf{2006}).


\bibitem{LascouxChern}
\rm A. Lascoux,
\it Classes de Chern d'un produit tensoriel,
\rm C. R. Acad. Sci. Paris 286(1):385-387 (\textbf{1978}).


\bibitem{LascouxSym}
\rm A. Lascoux, 
\it Symmetric functions, 
\rm Course Nankai University \url{http://www.emis.de/journals/SLC/wpapers/s68vortrag/ALCoursSf2.pdf} (\textbf{2001}).


%\bibitem{Lederer}
%\rm M. Lederer,
%\it On a formula for the Kostka numbers, 
%\rm Annals of Combinatorics 10(3):389-394 (\textbf{2006}).


\bibitem{Liggett}
\rm T. M. Liggett,
\it Interacting Particle Systems, 
\rm Springer-Verlag, New York (\textbf{1985}).


\bibitem{MacDo}
\rm I. G. Macdonald, 
\it Symmetric functions and Hall polynomials, 
\rm Oxford Mathematical Monographs, second edition, The Clarendon Press Oxford University Press (\textbf{1995}).


\bibitem{MadauleGaussien}
\rm T. Madaule, 
\it Maximum of a log-correlated Gaussian field, 
\rm Ann. Inst. H. Poincar\'e Probab. Statist. 51(4):1369-1431 (\textbf{2015}).


\bibitem{Major}
\rm P. Major, 
\it The limit behavior of elementary symmetric polynomials of IID random variables when their order tends to infinity, 
\rm \AOP 27(4):1980-2010 (\textbf{1999}).


\bibitem{Manstavicius}
\rm E. Manstavi\v{c}ius, 
\it Additive and multiplicative functions on random permutations, 
\rm Lith. Math. J. 36(4):400-408 (\textbf{1996}).


\bibitem{MatsumotoNovak}
\rm S. Matsumoto, J. Novak,
\it Jucys-Murphy elements and unitary matrix integrals, 
\rm Int. Math. Res. Not. IMRN 2013(2):362-397. (\textbf{2013}).


\bibitem{MatsumotoCbetaE}
\rm S. Matsumoto, 
\it Averages of ratios of characteristic polynomials in circular beta-ensembles and super-Jack polynomials, 
\rm \href{https://arxiv.org/abs/0805.3573}{arXiv:0805.3573} (\textbf{2008}).


\bibitem{Matytsin}
\rm A. Matytsin,
\it On the large $N$ limit of the Itzykson-Zuber integrals, 
\rm Nuclear Phys. B 411(2-3):805-820 (\textbf{1994}).


\bibitem{McSwiggenHCIZ}
\rm C. McSwiggen,
\it The Harish-Chandra integral, an introduction with examples, 
\rm \href{https://arxiv.org/pdf/1806.11155.pdf}{arXiv:1806.11155} (\textbf{2018}).


\bibitem{McSwiggenHCIZnew}
\rm C. McSwiggen,
\it A new proof of Harish-Chandra's integral formula, 
\rm \CMP 365(1):239-253 (\textbf{2019}).


\bibitem{MeliotCutOff}
\rm P.-L. M\'eliot, 
\it The cut-off phenomenon for Brownian motions on symmetric spaces of compact type,
\rm Potential Anal. 40(4):427-509 (\textbf{2014}).


\bibitem{MeliotReprSym}
\rm P.-L. M\'eliot, 
\it Representation theory of symmetric groups,
\rm Discrete Mathematics and its Applications, CRC Press (\textbf{2017}).


\bibitem{MenonHCIZ}
\rm G. Menon, 
\it The complex Burgers equation, the HCIZ integral and the Calogero-Moser system, 
\rm in Meeting on Random Matrix Theory at CMSA (Harvard), January 2017 \href{http://www.dam.brown.edu/people/menon/talks/cmsa.pdf}{http://www.dam. brown.edu/people/menon/talks/cmsa.pdf} (\textbf{2017}).


\bibitem{MezzadriSnaith}
\rm F. Mezzadri, N.C. Snaith (editors),
\it Recent Perspectives in Random Matrix Theory and Number Theory,
\rm London Mathematical Society Lecture Note Series (CUP), vol. 322 (\textbf{2005}).


\bibitem{MezzadriDerivatives}
\rm F. Mezzadri,
\it Random matrix theory and the zeros of $\zeta'(s)$,
\rm J. Phys. A: Mathematical and General 36(12):2945-2962 (\textbf{2003}).


\bibitem{MoensVanDerJeugt}
\rm E. M. Moens, J. Van der Jeugt,
\it A determinantal formula for supersymmetric Schur polynomials,
\rm Journal of Algebraic Combinatorics, 17(3):283-307 (\textbf{2003}).


\bibitem{Montgomery}
\rm H. L. Montgomery, 
\it The pair correlation of zeros of the zeta function, 
\rm Proc. of Symposia in Pure Math. 24:181-193, American Mathematical Society (\textbf{1973}).


\bibitem{MorozovRMTsystint}
\rm A. Morozov, 
\it Matrix models as integrable systems, 
\rm Particles and fields, Springer New York, 127-210 \href{https://arxiv.org/abs/hep-th/9502091}{arXiv:9502091} (\textbf{1999}).


\bibitem{MorozovCUE}
\rm A. Morozov, 
\it Unitary integrals and related matrix models, 
\rm Theoret. Math. Phys. 162(1):1-33 (\textbf{2010}).

\bibitem{NajnudelExtremeZeta}
\rm J. Najnudel, 
\it On the extreme values of the Riemann zeta function on random intervals of the critical line, 
\rm Probab. Theory Relat. Fields. 172(1):387-452 (\textbf{2018}).


\bibitem{Neretin}
\rm Yu. A. Neretin, 
\it Hua type integrals over unitary groups and over projective limits of unitary groups, 
\rm Duke Math. J. 114(2):239-266 (\textbf{2002}).

%\bibitem{NikeghbaliNajnudelCoeffs}
%\rm A. Nikeghbali, J. Najnudel,
%\it private communication 
%\rm (\textbf{2014}).

\bibitem{NikeghbaliZeindler}
\rm A. Nikeghbali, D. Zeindler,  
\it The generalized weighted probability measure on the symmetric group and the asymptotic behavior of the cycles,
\rm Ann. Inst. Henri Poincar\'e Probab. Stat. 49(4):961-981 (\textbf{2013}).


\bibitem{NovakHCIZandBGW} 
\rm J. Novak, 
\it On the Complex Asymptotics of the HCIZ and BGW Integrals,
\rm \href{https://arxiv.org/pdf/2006.04304.pdf}{arXiv:2006.04304} (\textbf{2020}).


\bibitem{NovakHCIZsampler}
\rm J. Novak, 
\it A Tale of Two Integrals,
\rm Notices Amer. Math. Soc. 66(10):1696:1697 (\textbf{2020}).

%\bibitem{OSZ}
%\rm N. O'Connell, T. Sepp\"al\"ainen, N. Zygouras, 
%\it Geometric RSK correspondence, Whittaker functions and symmetrized random polymers, 
%\rm \Inventiones 197:361-416 (\textbf{2014}).

% == Okounkov

\bibitem{OkounkovZmeas}
\rm A. Okounkov, 
\it $SL(2)$ and $z$-measures, 
\rm In: Random matrix models and their applications, Math. Sci. Res. Inst. Publ. 40:407-420, Cambridge Univ. Press, Cambridge (\textbf{2001}).


\bibitem{OkounkovOlshanski1}
\rm A. Okounkov, G. Olshanski,
\it Shifted Schur functions, 
\rm Algebra i Analiz 9(2):73-146 \textbf{1997}).


\bibitem{OkounkovOlshanski2}
\rm A. Okounkov, G. Olshanski,
\it Shifted Schur functions II, the binomial formula for characters of classical groups and its applications, 
\rm Amer. Math. Soc. Trans. 181(2):245-271 \textbf{1998}).


\bibitem{OkounkovOlshanskiJack}
\rm A. Okounkov, G. Olshanski,
\it Asymptotics of Jack polynomials as the number of variables goes to infinity, 
\rm Int. Math. Res. Not. 13(3):641-682 \textbf{1998}).


\bibitem{OkounkovOlshanskiBC}
\rm A. Okounkov, G. Olshanski,
\it Limits of BC-type orthogonal polynomials as the number of variables goes to infinity, 
\rm Contemp. Math. 417(1):281-318 \textbf{2006}).

% ==

\bibitem{OlshanskiSymInf1}
\rm G. Olshanski,
\it Point processes and the infinite symmetric group, part i: The general formalism and the density function, 
\rm \href{http://arxiv.org/abs/math/9804086}{arXiv:9804086} (\textbf{1998}).


\bibitem{OlshanskiSymInf2}
\rm G. Olshanski,
\it An introduction to harmonic analysis on the infinite symmetric group, 
\rm In: Asymptotic Combinatorics with Applications to Mathematical Physics, p. 127-160, Springer (\textbf{2003}).


\bibitem{OlshanskiVershik}
\rm G. Olshanski, A. Vershik,
\it Ergodic unitarily invariant measures on the space of infinite Hermitian matrices, 
\rm Contemp. Math., Amer. Math. Soc. Ser. 2, 175:137-175 (\textbf{1996}).


\bibitem{PakBirkhoff}
\rm I. Pak, 
\it Four questions on Birkhoff polytope, 
\rm Ann. Comb. 4(1):83-90 (\textbf{2000}).


\bibitem{PaquetteZeitouni}
\rm E. Paquette, O. Zeitouni, 
\it The maximum of the CUE field, 
\rm Int. Math. Res. Not. 2018(16):5028-5119 (\textbf{2017}).


\bibitem{Pickrell}
\rm D. Pickrell,
\it Mackey analysis of infinite classical motion groups,
\rm Pacific J. Math. 150(1):139-166 (\textbf{1991}).


\bibitem{Ram1991}
\rm A. Ram, 
\it A Frobenius formula for the characters of the Hecke algebra, 
\rm \Inventiones 106:61-488 (\textbf{1991}).


\bibitem{RogersPitman}
\rm L.C.G. Rogers, J.W. Pitman,
\it Markov functions, 
\rm \AOP 9(4):573-582 (\textbf{1981}).


\bibitem{PanchenkoBook}
\rm D. Panchenko,
\it The Sherrington-Kirkpatrick model,
\rm Springer Science \& Business Media (\textbf{2013}).


\bibitem{RiedtmannMixedRatios}
\rm H. Riedtmann,
\it A Combinatorial Approach to Mixed Ratios of Characteristic Polynomials, 
\rm \href{https://arxiv.org/abs/1805.07261v1}{arXiv:1805.07261} (\textbf{2018}). 


\bibitem{RoblesZeindler}
\rm N. Robles, D. Zeindler,
\it Random permutations with logarithmic cycle weights, 
\rm Ann. Inst. Henri Poincar\'e Probab. Stat. 56(3):1991-2016 (\textbf{2020}). 


\bibitem{RomikPn}
\rm D. Romik, 
\it Partitions of $n$ into $t\sqrt{n}$ parts, 
\rm Europ. J. Comb. 26(1):1-17 (\textbf{2005}).


\bibitem{Rosenbloom}
\rm P. C. Rosenbloom,
\it Probability and Entire Functions,
\rm Studies in Mathematical Analysis and Related Topics, Essays in Honor of George Polya, 45:325-332, Stanford Univ. Press, Palo Alto, CA (\textbf{1962}). 


\bibitem{Rosenthal}
\rm J. S. Rosenthal, 
\it Random rotations: Characters and random walks on $SO(N)$, 
\rm \AOP 22(1):398-423 (\textbf{1994}).


%\bibitem{RevuzYor}
%\rm D. Revuz, M. Yor, 
%\it Continuous martingales and Brownian motion, 
%\rm Springer-Verlag, third edition (\textbf{1999}).


\bibitem{SabbaghOnMontgomeryDyson}
\rm K. Sabbagh,
\it Dr. Riemann's Zeros,
\rm Atlantic Books, pp. 134-136 \href{http://empslocal.ex.ac.uk/people/staff/mrwatkin//zeta/dyson.htm}{http://empslocal.ex.ac.uk/people/ staff/mrwatkin/dyson.htm} (\textbf{2002}).


\bibitem{SantilliTierz}
\rm L. Santilli, M. Tierz,
\it Exact equivalences and phase discrepancies between random matrix ensembles,
\rm \href{https://arxiv.org/pdf/2003.10475.pdf}{arXiv:2003.10475} (\textbf{2020}).


\bibitem{SawinShusterman}
\rm W. Sawin, M. Shusterman, 
\it On the Chowla and twin primes conjectures over $ \Ff_q\crochet{T} $, 
\rm \href{https://arxiv.org/abs/1808.04001}{arXiv:1808.04001} (\textbf{2018}).


\bibitem{SchraffThibon}
\rm T. Schraff, J.-Y. Thibon, 
\it A Hopf algebra approach to inner plethysm, 
\rm Adv. Math. 104(1):30-58 (\textbf{1994}).


\bibitem{SchutzDuality}
\rm G.M. Sch\"utz, 
\it Duality relations for asymmetric exclusion processes, 
\rm J. Stat. Phys. 86(5-6):1265-1287 (\textbf{1997}).


\bibitem{SchwartzRKHS}
\rm L. Schwartz, 
\it Sous-espaces hilbertiens d'espaces vectoriels topologiques et noyaux associ\'es (noyaux reproduisants), 
\rm J. Anal. Math. 13(1):115-256 (\textbf{1964}).


\bibitem{SchwartzDistributions}
\rm L. Schwartz, 
\it Th\'eorie des distributions, 
\rm vol. 2 Hermann, Paris (\textbf{1966}).


\bibitem{Spitzer}
\rm F. Spitzer,
\it Interaction of Markov processes, 
\rm Adv. Math. 5(2):246-290 (\textbf{1970}).


%\bibitem{StanleyEC2}
%\rm R. P. Stanley, 
%\it Enumerative Combinatorics. Vol. 2, 
%\rm Cambridge Studies in Advanced Mathematics (\textbf{1999}).


\bibitem{SnaithDerivatives}
\rm N.C. Snaith,
\it Derivatives of random matrix characteristic polynomials with applications to elliptic curves,
\rm J. Phys. A 38(48):10345-10366 (\textbf{2005}).


\bibitem{Speiser}
\rm A. Speiser,
\it Geometrisches zur Riemannschen Zetafunktion,
\rm Math. Ann. 110:514-521 (\textbf{1935}).


\bibitem{SzekeresPn}
\rm G. Szekeres, 
\it Some asymptotic formulae in the theory of partitions (II),
\rm Quart. J. Math., Oxford Ser. 4(2):96-111 (\textbf{1953}).


\bibitem{TaoBlogHCIZ}
\rm T. Tao,
\it The Harish-Chandra-Itzykson-Zuber integral formula,
\rm \href{https://terrytao.wordpress.com/2013/02/08/the-harish-chandra-itzykson-zuber-integral-formula/}{https://terrytao.wordpress.com/2013/ 02/08/the-harish-chandra-itzykson-zuber-integral-formula/}, Accessed Sept. 10, 2020 (\textbf{2013}).


\bibitem{Tenenbaum}
\rm G. Tenenbaum,
\it Introduction to Analytic and Probabilistic Number Theory,
\rm Cambridge studies in advanced mathematics, vol. 46, Cambridge University Press (\textbf{1995}).


\bibitem{Titchmarsh}
\rm E. E. C. Titchmarsh,
\it The theory of the Riemann Zeta-function,
\rm Oxford University Press (\textbf{1986}) [first edition: \textbf{1951}]. 


\bibitem{TribeZaboronskiPfaffianCoalescing}
\rm R. Tribe, O. Zaboronski,
\it Pfaffian formulae for one dimensional coalescing and annihilating systems, 
\rm Electronic J. Probab. 16(P 76):2080-2103 (\textbf{2011}).

% Ausbilderegnung Kompakt ---> Souvenir de ma voisine (blonde pulpeuse) dans l'avion.

\bibitem{UlrichUnitary}
\rm M. Ulrich, 
\it Construction of a free L\'evy process as high-dimensional limit of a Brownian Motion on the unitary group, 
\rm Infinite Dimensional Analysis, Quantum Probability and Related Topics, 18(03):1550018 (\textbf{2015}).


\bibitem{VanMoerbeke}
\rm P. Van Moerbeke, 
\it Integrable lattices: random matrices and permutations,
\rm MSRI-volume on Random matrices and exactly solvable models, Eds.: P. Bleher, A. Its, Oxford University press (\textbf{2000}).


\bibitem{VershikTAR} 
\rm A. M. Vershik,
\it Two lectures on the asymptotic representation theory and statistics of Young diagrams, 
\rm in: Vershik A.M., Yakubovich Y. (eds) Asymptotic Combinatorics with Applications to Mathematical Physics, Lecture Notes in Mathematics, vol 1815, Springer, Berlin, Heidelberg (\textbf{2003}).


\bibitem{VoiculescuApproximateFactorisation} 
\rm D. Voiculescu, 
\it Sur les repr\'esentations factorielles finies de $ U(\infty)$ et autres groupes semblables, 
\rm C. R. Acad. Sci. Paris S\'er. A 279(3):945-946 (\textbf{1974}).

% https://gallica.bnf.fr/ark:/12148/bpt6k6226936t/f229.image.r=Comptes%20rendus%20hebdomadaires%20des%20s%C3%A9ances%20de%20l'Acad%C3%A9mie%20des%20sciences

\bibitem{Weingarten}
\rm D. Weingarten, 
\it Asymptotic behavior of group integrals in the limit of infinite rank,
\rm J. Mathematical Phys. 19(5):999-1001 (\textbf{1978}).


\bibitem{Winn}
\rm B. Winn, 
\it Derivative Moments for Characteristic Polynomials from the CUE, 
\rm \CMP 315(2):531-562 (\textbf{2012}).


\bibitem{ZabrockiThesis}
\rm M. Zabrocki, 
\it On the action of the Hall-Littlewood vertex operator, 
\rm Ph.D. thesis, University of California, San Diego (\textbf{1998}). 


\bibitem{ZinnJustinZuber}
\rm P. Zinn-Justin, J.-B. Zuber, 
\it On some integrals over the $U(N)$ unitary group and their large $N$ limit, 
\rm J. Phys. A 36(12):3173-3193 (\textbf{2003}). 

\end{thebibliography}

% ==

\end{document}